\documentclass[11pt]{article}

\usepackage{pratik}

\begin{document}

\newcommand{\titletext}{Corrected generalized cross-validation\\ for finite ensembles of penalized estimators}

\title{\titletext}

\author{
    \setcounter{footnote}{-1}
    \blfootnote{Author names sorted alphabetically. $^\ast$Corresponding authors.}
    \setcounter{footnote}{1}
    Pierre {C.} Bellec\footremember{rutgersstats}{Department of Statistics, Rutgers University, New Brunswick, NJ 08854, USA.}
    \\ {\small \href{mailto:pierre.bellec@rutgers.edu}{{pierre.bellec@rutgers.edu}}}
    \and
    Jin-Hong Du\footremember{cmustats}{Department of Statistics and Data Science, Carnegie Mellon University, Pittsburgh, PA 15213, USA.}\footremember{cmumld}{Machine Learning Department, Carnegie Mellon University, Pittsburgh, PA 15213, USA.}
    \\ {\small \href{mailto:jinhongd@andrew.cmu.edu}{{jinhongd@andrew.cmu.edu}}}
    \and
    Takuya Koriyama$^\ast$\footrecall{rutgersstats}
    \\ {\small \href{mailto:tk691@stat.rutgers.edu}{{takuya.koriyama@rutgers.edu}}}
    \and
    Pratik Patil$^\ast$\footremember{berkeleystats}{Department of Statistics, University of California, Berkeley, CA 94720, USA.}
    \\ {\small \href{mailto:pratikpatil@berkeley.edu}{{pratikpatil@berkeley.edu}}}
    \and
    Kai Tan\footrecall{rutgersstats}
    \\ {\small \href{mailto:kt536@stat.rutgers.edu}{{kai.tan@rutgers.edu}}}
}

\date{\today}

\maketitle

\begin{abstract}
Generalized cross-validation (GCV) is a widely-used method for estimating the squared out-of-sample prediction risk that employs a scalar degrees of freedom adjustment (in a multiplicative sense) to the squared training error. In this paper, we examine the consistency of GCV for estimating the prediction risk of arbitrary ensembles of penalized least-squares estimators. We show that GCV is inconsistent for any finite ensemble of size greater than one. Towards repairing this shortcoming, we identify a correction that involves an additional scalar correction (in an additive sense) based on degrees of freedom adjusted training errors from each ensemble component. The proposed estimator (termed CGCV) maintains the computational advantages of GCV and requires neither sample splitting, model refitting, or out-of-bag risk estimation. The estimator stems from a finer inspection of the ensemble risk decomposition and two intermediate risk estimators for the components in this decomposition. We provide a non-asymptotic analysis of the CGCV and the two intermediate risk estimators for ensembles of convex penalized estimators under Gaussian features and a linear response model. Furthermore, in the special case of ridge regression, we extend the analysis to general feature and response distributions using random matrix theory, which establishes model-free uniform consistency of CGCV.
\end{abstract}

\section{Introduction}

Ensemble methods are an important part of the toolkit in current machine learning practice.
These methods combine multiple models and improve predictive accuracy and stability compared to any single component model \citep{hastie2009elements,dietterich1998experimental,dietterich2000ensemble}.
Bagging and its variants form an important class of ensemble methods that average predictors fitted on different subsamples of the full dataset.
The classical literature on bagging includes work by \citet{breiman_1996,buhlmann2002analyzing,friedman_hall_2007}, among others.
There has been a flurry of recent work on characterizing the risk behavior of such subsampling-based ensemble methods in high dimensions; see, e.g., \citet{loureiro2022fluctuations,adlam2020understanding,mucke_reiss_rungenhagen_klein_2022,patil2022bagging,du2023subsample}.
These works demonstrate that the risk can be significantly reduced by properly choosing the size of the ensemble. 
Complementary to characterizing and understanding the risk behavior, it is important to have a reliable method for estimating these risks when implementing ensemble methods.
Accurate risk estimation provides a benchmark for comparing the performance of different ensemble configurations and provides insights into the expected performance of the combined model on unseen data. 
In particular, \citet{du2023subsample,patil2023generalized} demonstrate that the implicit regularization provided by ensembling amounts to ridge regularization, and optimizing over both the ensemble and subsample size helps to mitigate double descent behaviors.
In this paper, our primary focus is on risk estimation techniques suitable for ensemble methods, when these ensembles are composed of penalized least-squares estimators.
Several tuning parameters define an ensemble even for a single penalty function, including the regularization parameter and the subsample size.
Developing consistent risk estimators informs the selection of the optimal ensemble (regularization parameters and subsample size) and leads to better generalization performance.

One widely recognized technique for risk estimation is Generalized Cross-Validation (GCV). 
GCV adjusts the training error multiplicatively based on a factor known as degrees of freedom correction. 
GCV is conceptually motivated as an approximation to leave-one-out cross-validation and provides consistent risk estimation without sample splitting and refitting; see, e.g., Chapter 7 of \citet{hastie2009elements} and Chapter 5 of \citet{wasserman2006all}.
The consistency of GCV has been extensively studied, initially for fixed-X design settings and linear smoothers \citep{golub_heath_wabha_1979, craven_wahba_1979}.
Subsequent work has extended this to the random-X settings, which are arguably more relevant in modern machine learning applications where observations are drawn from some underlying distribution rather than being predetermined.
In particular, the consistency of GCV for ridge regression has been shown in \citet{patil2021uniform, patil2022estimating, wei_hu_steinhardt} in various settings. 
For ridge regression, the consistency holds for arbitrary response models with only the assumption of bounded moments.
Beyond ridge regression, the consistency of GCV is shown for Lasso in \citep{miolane2021distribution,celentano2020lasso} and for convex regularized estimates in \cite[Section 3]{bellec2020out}.
For linear base estimators (e.g., least squares estimator, ridge, etc.), the degrees of freedom of the ensemble estimator is simply the average of degrees of freedom for all base predictors.
The first question that this paper asks is the following:
\begin{center}
    Is GCV still consistent for ensembles of general penalized estimators?
\end{center}

Perhaps the simplest base predictor to study this question is ridge regression, which has emerged as a simple but powerful tool, particularly when dealing with high-dimensional data without any special sparsity structure.
Recent work by \cite{du2023subsample} considers the problem of risk estimation for ensembles of ridge predictors.
The authors analyze the use of GCV for the full ensemble ridge estimator when the number of ensemble components goes to infinity, demonstrating that GCV is consistent for the prediction risk.
However, an intriguing surprise from their work is the inconsistency of GCV for any finite ensemble:
The intricate correlation caused by overlapping samples in different subsampled datasets complicates risk estimation and necessitates a correction.
This finding suggests a limitation of GCV and prompts the question of whether this inconsistency is special to ridge regression or is a broader issue with GCV's applicability to other models as well.
And, in particular, this prompts the second key question of our paper: 
\begin{center}
    Is there a GCV correction that is consistent for all possible ensemble sizes?
\end{center}

Understanding these two questions forms the basis of our work.
We propose a novel risk estimator that addresses the limitations of naive GCV
for ensembles of penalized estimators.
Focusing on convex penalized predictors, we show that for a finite ensemble size larger than one, the GCV is always inconsistent and deviates from the true risk by an additive error that we characterize explicitly.
Understanding this additive error leads to a novel correction to GCV, and we provide both non-asymptotic and asymptotic analyses of this corrected GCV under two different sets of data-generating assumptions.
Our non-asymptotic analysis is restricted to Gaussian features and well-specified linear models, allowing general strongly convex penalty functions.
Our asymptotic analysis is more general in accommodating random matrix theory features beyond Gaussianity and arbitrary data models while being restricted to ridge predictors.
Our asymptotic analysis also enables the analysis of the ``ridgeless'' predictors (minimum-norm interpolators) in the overparameterized regime.
For the case of ridge regression, this also allows us to show the uniform consistency in the regularization parameter.
Before diving into the precise details of our contributions, we present our results in \Cref{fig:Fig1_ridge_lasso}, which demonstrates the clear gap between GCV and true risk. Our proposed corrected GCV (CGCV), on the other hand, accurately estimates the risk across varying settings for ensembles of both ridge and lasso predictors. 

\begin{figure}[!t]
    \centering
    \includegraphics[width=0.95\textwidth]{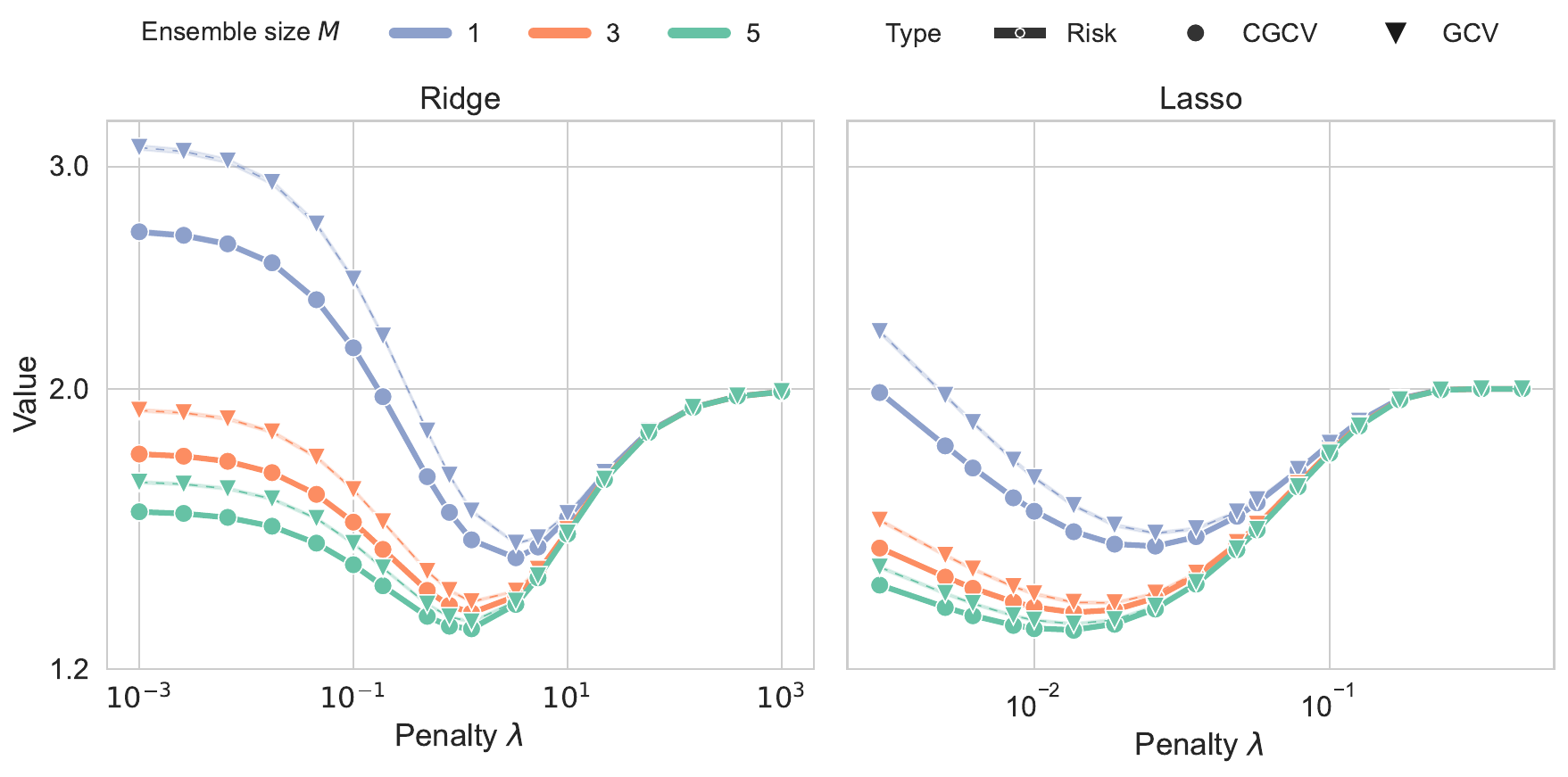}
    \caption{
    \textbf{CGCV is consistent for finite ensembles of penalized estimators while GCV is not.}
    We show a numerical comparison between the true squared risk (Risk), GCV error (GCV), and corrected GCV (CGCV) for ensembles across different tuning parameters $\lambda$ and ensemble sizes $M$, over 1000 repetitions of the datasets.
    The left panel shows ridge ensembles with different $\lambda$ and $M$.
    Data is generated from a linear model with Gaussian design, $(n,p) = (2000, 500)$, and a signal-to-noise ratio of $1$. 
    The subsample size is fixed at $k=800$. 
    The right panel shows the plot of lasso ensembles under the same setting as the left panel except for a different range of $\lambda$. 
    Further details of the experimental setup are given in \Cref{sec:numerical-illustrations-gaussian}. 
    }
    \label{fig:Fig1_ridge_lasso}
\end{figure}

\subsection{Summary of results}
In the classical low-dimensional regime, where the feature dimension \( p \) is fixed, the excess prediction risk goes to zero as the sample size \(n\) diverges. 
Thus, the average training error itself already serves as a consistent estimator of the risk. 
In the proportional asymptotic regime, where the feature size \( p \) scales proportionally with the sample size \( n \), the prediction risk typically converges to a non-vanishing constant \citep[among others]{bayati2011lasso,karoui_2013,thrampoulidis2018precise}. 
Therefore, the regime of primary interest for our results is the proportional asymptotic regime. 
More formally, we implicitly consider a sequence of regression problems, say, indexed by $n$. 
The dimension $p{(n)}$, the distribution of the observations, the estimators and penalty under consideration, and the subsample size $k{(n)}$ all implicitly depend on $n$ with $p{(n)},k{(n)}$ growing with $n$. 
We omit the dependence on $n$ to lighten the notation and write simply, e.g., $p$ for the dimension and $k$ for the subsample size. 
Throughout the paper, the ratio $p/n$ is bounded from above by a constant independent of $n$, so that by extracting a subsequence if necessary, we may assume that $p/n$ has a finite limit.
While several of our results are non-asymptotic with explicit dependence on certain problem parameters, we also state consistency results asymptotically using $\pto$/$\asto$ (convergence in probability/almost sure convergence), $\op(\cdot)$ and $\Op(\cdot)$ notation, which are defined with respect to the aforementioned sequence of regression problems.
Our contributions are four-fold, as summarized below. 

\begin{enumerate}[leftmargin=7mm]
    \item
    \textbf{Inconsistency of GCV for finite ensembles.}
    We establish that ordinary GCV is inconsistent in estimating the squared out-of-sample prediction risk of finite ensembles with more than one penalized estimator (\Cref{prop:gcv_inconsistency}).
    Here, we assume strongly convex penalties that include ridge regression and elastic net, among others; and the inconsistency result of \Cref{prop:gcv_inconsistency} applies to base estimators that have normalized degrees of freedom bounded away from 0 (see \Cref{rem:df-nonnegligibility} for details).
    
    \item
    \textbf{Corrected GCV for finite ensembles.}
    We introduce a novel risk estimator, termed corrected GCV (CGCV), designed to estimate the prediction risk of the arbitrary ensemble of penalized estimators (\Cref{def:cgcv-proposal}).
    The proposed estimator employs a scalar data-dependent correction term (in an additive sense) that incorporates degrees of freedom adjusted training errors of individual component estimators (\Cref{fig:cgcv-illustration}).
    Importantly, CGCV preserves all the computational advantages inherent to GCV, since it requires neither sample splitting nor model refitting nor out-of-bag risk estimation.
    
    \item
    \textbf{Non-asymptotic analysis under Gaussian designs.}
    The structure of CGCV stems from inspecting the ensemble risk decomposition and two intermediate risk estimators that we construct for the components in this decomposition (\Cref{eq:risk-decomp-est-1,eq:risk-decomp-est-2}).
    For Gaussian features and a well-specified linear model, we provide non-asymptotic bounds for these component estimators (\Cref{th:sub_full}).
    Building on these, we provide a non-asymptotic bound for the proposed corrected GCV estimator (\Cref{th:correction_gcv}).
    Our error bounds decrease at the $n^{-1/2}$ rate and provide explicit dependence on problem parameters such as the ensemble size $M$, the subsample size $k$, and the strong convexity parameter of the penalty (\Cref{eq:sub_full_relative_error,eq:cgcv_relative_error}).
    The derived rates are pertinent within the scope of the proportional asymptotic regime.
    
    \item
    \textbf{Asymptotic analysis under general designs.}
    For general feature structure (multiplicatively) composed of an arbitrary covariance matrix and independent components with bounded $4+\eta$ moments and response vector with bounded $4+\eta$ moments for some $\eta > 0$, we show that the intermediate risk estimators are asymptotically consistent for ridge ensemble (\Cref{thm:risk-est}).
    We further show uniform consistency in the regularization level $\lambda$ that includes the case of ridgeless regression for both intermediate estimators and, consequently, CGCV (\Cref{thm:unifrom-consistency-lambda} and \Cref{prop:correct-gcv}).
    The asymptotic analysis allows us to relax the data and model assumptions while obtaining model-free uniform consistency of CGCV in the regularization parameter.
\end{enumerate}

\subsection{Related work}
\label{sec:related_work}

In this section, we describe related work on ensemble risk analysis and risk estimation, as well as a detailed comparison to other baseline risk estimation approaches based on sample splitting and out-of-bag risk estimation.

\paragraph{GCV for linear smoothers.\hspace{-0.5em}}
Risk estimation is a crucial aspect of statistical learning that plays an important role in model evaluation and selection.
Over the years, a myriad of methods have been proposed, each with its own strengths and limitations.
Among these, GCV has emerged as a widely adopted technique for risk estimation. 
GCV is conceptually motivated as an approximation to leave-one-out cross-validation, a method that provides unbiased estimates of the true prediction error but can be computationally expensive \citep{hastie2009elements}.
The origins of GCV can be traced back to the work of \citet{golub_heath_wabha_1979} and \citet{craven_wahba_1979}, where it was studied in the context of fixed-X design settings for linear smoothers. 
These settings, where the predictors are considered fixed and non-random, are common in experimental designs. 
The consistency of GCV was subsequently investigated in a series of papers by \citet{li_1985, li_1986, li_1987}.
More recently, the focus has shifted toward the random-X and high-dimensional setting.
This change in focus is motivated by modern machine learning applications, which often deal with large datasets and where predictors are typically considered to be random variables drawn from some distribution rather than fixed values. 
In this context, GCV has been the subject of considerable research attention.
In particular, \cite{leeb2008evaluation} establishes the consistency of GCV for the ordinary least-squares, and the consistency of GCV for ridge regression has been established in a number of recent studies in various data settings by \citet{adlam_pennington_2020neural,patil2021uniform,patil2022estimating,wei_hu_steinhardt,han2023distribution,misiakiewicz2024non}, which provide different flavors of results (both asymptotic and non-asymptotic).

\paragraph{GCV beyond linear smoothers.\hspace{-0.5em}}
Generalized cross-validation was initially defined for linear smoothers, where the sum of training errors is adjusted multiplicatively by the trace of the smoothing matrix. 
There is a general way of understanding this estimator, and the definition of GCV for estimators that are not linear smoothers can be extended by using degrees of freedom adjustments. 
The notion of degrees of freedom of linear estimate $\hbbeta$ dates back to the pioneering paper of \cite{stein1981estimation}. 
This literature established the multivariate version of Stein's formula and proposed an unbiased estimator for estimating the mean square error of a Gaussian multivariate mean, which is often called Stein's Unbiased Risk Estimate ($\SURE$). 
For example, under the linear model $\by = \bX \bbeta_0 + \bvarepsilon$ with Gaussian covariates $\bX\in\RR^{n\times p}$ and Gaussian noise $\bvarepsilon\in\RR^{n}$, the Stein's Unbiased Risk Estimate has expression $\hat{\SURE} = \|\by - \bX\hbbeta\|_2^2 + 2\sigma^2 \df - n \sigma^2$, where $\sigma^2$ is the noise variance, and 
\begin{equation}
 \df = \trace[(\partial/\partial \by)\bX \hbbeta]
 \label{df}
\end{equation}
is the degrees of freedom of the estimate $\hbbeta$. 
The $\SURE$ is unbiased in the sense that 
$
    \EE [\hat{\SURE}]
    = \EE [\|\bX \hbbeta - \bX \bbeta_0\|_2^2].
$
This definition of degrees of freedom as the trace of the Jacobian yields a natural generalization of the GCV formula, where the trace of the smoothing matrix (for linear smoothers) is replaced by $\df$, i.e., the trace of the Jacobian $(\partial/\partial\by)\bX \hbbeta$.
Moving beyond the estimation of the in-sample error $\|\bX\hbbeta - \bX\bbeta_0\|_2^2$, this generalization of GCV beyond linear smoothers is known to be consistent in the sense:
\[
    \frac{\|\by - \bX\hbbeta\|_2^2/n}{(1 - \df/n)^2}
    \Big/
    \Bigl(
    \|\bSigma^{1/2}(\hbbeta - \bbeta_0)\|_2^2 + \sigma^2
    \Bigr)
    \pto 1,
\]
under linear models with jointly normal observations $(\bx_i, y_i)$ and $\EE[\bx_i\bx_i^\top]=\bSigma$.
This consistency is proven for the lasso \citep{bayati2011lasso, bayati2013estimating,miolane2021distribution,celentano2020lasso}, and for regularized least-squares with general convex penalty \citep[Section 3]{bellec2020out}.
\citet{bellec2022derivatives} generalized the above GCV estimator for the risk $\|\bSigma^{1/2}(\hbbeta - \bbeta_0)\|_2^2 + \|\bvarepsilon\|_2^2/n$ in the context of regularized M-estimation where the noise has possibly infinite variance. 
\citet{tan2022noise} studied a generalization of GCV to estimate the out-of-sample error in a multi-task regression model where each response $y_i$ is vector-valued.

Closely related to GCV, \cite{rad2018scalable} introduced the Approximate Leave-One Out (ALO) estimator for the out-of-sample error of regularized M-estimators and showed the consistency of ALO under smoothness assumptions on the data-fitting loss and the regularizer. 
\cite{xu2019consistent} used Approximate Message Passing to develop
risk estimates for estimators constructed with separable loss and separable regularizer, and showed that these estimates, including ALO, enjoy finite error bounds.
ALO has been extended to non-differentiable regularizers \cite[Section 2.2]{rad2018scalable} as well as non-differentiable losses and non-differentiable regularizers \citep{wang2018approximate}. 
Its theoretical accuracy has recently been demonstrated for LASSO and elastic net \citep{auddy2023approximate}.

\paragraph{Ensemble risk characterization.\hspace{-0.5em}}
Ensemble methods combine weak predictors to produce strong predictors and are a common part of the statistical machine learning toolkit \citep{hastie2009elements}.
Several variants exist such as bagging \citep{breiman_1996,buhlmann2002analyzing}, random forests \citep{breiman2001random}, stacking \citep{wolpert1992stacked}, among others.
These methods are widely used because of their ability to improve prediction accuracy and robustness.
The risk behavior of ensemble methods has been a topic of recent interest.
\cite{patil2022bagging} provide a strategy for analyzing the risk of an ensemble of general predictors and present exact risk asymptotics for an ensemble of ridge predictors.
Other related work on ensemble risk analysis include \citet{lejeune2020implicit,ando2023high,loureiro2022fluctuations,adlam2020understanding}, among others, which provide insights into the fluctuations in risk of ensemble methods, demonstrating that the risk can be significantly reduced by properly choosing the ensemble size. 
They also highlight the importance of the correlation structure among the predictors in determining the risk of the ensemble.
Recently, \citet{du2023subsample} show that the naive extension of the GCV estimator for the ridge ensemble is generally inconsistent for a finite ensemble size $M>1$, but it becomes consistent again when $M$ approaches infinity.
The follow-up work of \cite{patil2023generalized} provides certain structural and risk equivalences between subsampling and ridge regularization.
This work identifies an inflation factor of $\Delta / M$, for some finite bias term $\Delta$ (independent of $M$), between the risk of the ensemble at a given regularization level and full ridge at a larger ``induced'' regularization level. 
This inflation vanishes as $M \to \infty$. 
Even though GCV is consistent for ridge regression $(M=1)$ or the infinite-ensemble ridge $(M = \infty)$, these two works suggest that the naive extension of GCV for finite-ensemble ridge may be inconsistent for $1< M < \infty$.

\paragraph{Comparison with sample splitting and out-of-bag risk estimation.\hspace{-0.5em}}
In the context of ensemble learning, other types of cross-validation methods include the sample-split CV and out-of-bag CV, which utilize test samples to estimate the out-of-sample risk.
The sample-split CV is the most common strategy for risk estimation \citep{patil2022bagging}.
However, this approach presents several drawbacks.
The estimator can suffer from finite-sample efficiency due to the inherent sample splitting: $V$-fold CV estimates the risk of an estimator trained on $(V-1)n$ observations only instead of the risk of the estimator trained on the full dataset of size $n$.
Figure 1 of \citet{rad2018scalable} provides a clear illustration of this drawback.
Furthermore, the accuracy of the estimator can suffer when dealing with large subsample sizes.
On the other hand, \citet{lopes2019estimating,lopes2020measuring,du2023extrapolated} use out-of-bag risk estimates to extrapolate the risk estimation.
Their estimators provide a more efficient risk estimator without sample splitting but do not perform well with large subsample sizes, as the out-of-bag sample size becomes small.
In light of the aforementioned limitations of existing methods, our aim is to develop a risk estimator that fulfills certain desiderata.
First, we would like an estimator that does not require sample splitting. 
This would address the issue of finite-sample efficiency faced by split cross-validation.
Second, we aim to extend the applicability of GCV to accommodate any ensemble size. 
This would overcome the limitation of GCV being inconsistent for finite ensembles, as identified in \cite{du2023subsample}.
In this paper, we present two intermediate risk estimators that fulfill both these desiderata, which ultimately lead to our final corrected GCV.

\subsection{Organization}

The rest of the paper is organized as follows.
In \Cref{sec:preliminaries}, we set our notation and define the ensemble estimators and their prediction risks.
In \Cref{sec:risk-estimators}, we define our risk estimators and provide a qualitative comparison to other related risk estimators.
In \Cref{sec:finite-sample-analysis}, we provide a finite-sample analysis under stronger assumptions on the feature and response distributions.
In \Cref{sec:asm-consistency}, we show that the risk estimators are consistent under mild assumptions on the feature and response distributions for ridge ensembles.
\Cref{sec:discussion} concludes the paper and offers several follow-up directions.
Proofs of all the theoretical results are included in the supplement, which also contains a summary of the notational conventions used throughout the paper.
The source code for generating all the experimental figures in this paper can be accessed at: \href{https://github.com/kaitan365/CorrectedGCV}{[code repository]}.

\section{Background}\label{sec:preliminaries}

Consider the standard supervised regression setting where we observe $n$ i.i.d.\ samples $\{(\bx_i, y_i) : i \in [n] \}$ in $\RR^{p} \times \RR$, where $[n]$ stands for the index set $\{1, 2, \ldots, n\}$.  
Let $\bX\in\RR^{n\times p}$ denote the feature matrix whose $i$-th row contains $\bx_{i}^\top$ and $\by\in\RR^n$ denote the response vector whose $i$-th entry is $y_{i}$.

\subsection{Ensemble estimator and prediction risk}

For each $m\in [M]$, let $I_m$ be a non-empty subset of $[n]$.
Let $(\bX_{I_m}, \by_{I_m})$ denote the corresponding random design matrix and response vector associated with the subsampled dataset $\{(\bx_i, y_i) : i \in I_m\}$.
For each of the subsampled datasets $(\bX_{I_m}, \by_{I_m})$ for $m \in [M]$, we consider a penalized least-squares estimator:
\begin{equation}\label{eq:def-hbeta}
    \hbbeta_m
    \in \argmin_{\bb \in \RR^{p}}
    \frac{1}{2|I_m|}
    \sum_{i\in I_m}
    (y_i - \bx_i^\top\bb)^2
     + g_m(\bb)
     =
     \argmin_{\bb \in \RR^{p}}
    \frac{1}{2|I_m|}
    \|\bL_{I_m}(\by - \bX\bb)\|_2^2
     + g_m(\bb)
\end{equation}
where
$g_m:\R^p\to\R$ is a prespecified convex penalty and $\bL_{I_m}$ is the diagonal projection matrix with entries $(\bL_{I_m})_{ii}=I\{i\in I_m\}$.
Following \cite{stein1981estimation}, we define the degrees of freedom of the estimator $\hbeta_m$ as: 
\begin{equation}\label{eq:def-df}
    \df_m = \trace[(\partial/\partial\by) \bX\hbeta_m]. 
\end{equation}
For the estimator defined in \eqref{eq:def-hbeta} using the full data ${(\bx_i, y_i)}_{i\in[n]}$ with particular choices of the penalty function $g$, the quantity $\df$ has explicit expressions; see, for example, \citet{zou2007degrees,tibshirani2012degrees,dossal2013degrees,vaiter2012degrees,vaiter2017degrees,bellec2022derivatives}) among others. 
\Cref{tab:examples} presents the explicit known formulae for the lasso, ridge, and elastic net. 

\begin{table}[t]
    \centering
    \caption{
    Explicit formulae for the degrees of freedom in \eqref{eq:def-df} of the estimator \eqref{eq:def-hbeta} for commonly used penalty functions.
    Here $\hat S=\{j\in [p]:e_j^\top \hbeta_m \ne 0\}$ is the set of active variables, $\bX_{\hat S}$ is the submatrix of $\bX_{I_m}$ made of columns indexed in $\hat S$.
    See \cite{zou2007degrees,tibshirani2012degrees,dossal2013degrees,bellec2021second} among 
    others.
    }\label{tab:examples}
    \begin{tabular}{c c c}
         \toprule
         \textbf{Estimator} $\hbbeta_m$  & \textbf{Regularizer} $g_m(\bb)$  & \textbf{Degrees of freedom} $\df_m$ \\ 
         \midrule
         Lasso &
         $\lambda\|\bb\|_1$ & $|\hat S|$  \\[1pt]
         Ridge &
         $\frac{\lambda}{2} \|\bb\|_2^2$ & $\trace\big[\bX_{I_m}\big(\bX_{I_m}^\top\bX_{I_m} + {|I_m|}\lambda \bI\big)^{-1}\bX_{I_m}^\top\big]$  \\
         Elastic net &
         $\lambda_1\|\bb\|_1 + \frac{\lambda_2}{2} \|\bb\|_2^2$ &
         $\trace\big[\bX_{\hat S}\big(\bX_{\hat S}^\top\bX_{\hat S} + {|I_m|}\lambda_2\bI\big)^{-1}\bX^\top_{\hat S}\big] $ \\
        \bottomrule
    \end{tabular}    
\end{table}

The ensemble estimator constructed using \eqref{eq:def-hbeta} from the subsample datasets $\{(\bX_{I_m}, \by_{I_m})\}_{m\in[M]}$ is defined as:
\begin{align}\label{eq:def-M-ensemble}
    \tbeta_{M}\bigl(\{I_m\}_{m=1}^M\bigr)
    &:= \frac{1}{M} \sum_{m=1}^M
    \hbeta_m.
\end{align}
Though it is possible to extend the analysis by using unequal weights for the ensemble estimator \eqref{eq:def-M-ensemble}, we focus on the equal-weighted ensemble estimator for simplicity; see \Cref{rmk:unequal-weight} for more discussion.
For brevity, we will omit the dependency on $\{ I_{m} \}_{m\in[M]}$ for the ensemble estimator and simply write $\tbeta_M$ when it is clear from the context.
By the linearity property of the trace operator, the degrees of freedom of the ensemble estimator can be easily shown to be the average of the individual degrees of freedom:
\begin{align}
    \tdf_M = \frac{1}{M} \sum_{m=1}^M \df_{m}.%
    \label{eq:def-tdf}
\end{align}
We assess the performance of the $M$-ensemble predictor $\tbeta_{M}$ via conditional squared prediction risk:
\begin{align}
    R_{M}&:=\EE_{(\bx_0,y_0)}\big[(y_0-\bx_0^{\top} \tbeta_{M})^2\mid (\bX, \by), \{I_{m}\}_{m = 1}^M\big], \label{eq:R_M}
\end{align}
where $(\bx_0, y_0)$ is an independent test point sampled from the same distribution as the training dataset $(\bX, \by)$.
Note that the conditional risk $R_{M}$ is a scalar random variable that depends on both the dataset $(\bX, \by)$ and the random samples $\{I_{m}: m\in [M]\}$.
Our goal is to construct estimators for $R_M$ that work for all choices of ensemble sizes $M$.

\subsection{GCV for ensemble estimator}
\label{subsec:naive_gcv}

Before we delve into our proposed estimators, let us first consider the ordinary Generalized Cross-Validation (GCV) estimator.
Suppose that we have a linear predictor $\hf(\bx) = \bx^{\top}\hat{\bbeta}$ where $\hat{\bbeta}$ is obtained through \eqref{eq:def-hbeta} based on the dataset $(\bX, \by)$.
With the notation $\df$ in \eqref{df}, the ordinary GCV estimator for the prediction risk of this estimate $\hbeta$ is defined as:
\begin{equation}
    \label{eq:ogcv}
    \hR^\gcv = \frac{\|\by - \bX \hat{\bbeta} \|_2^2 / n}{(1 - \df/n)^2}.
\end{equation}
The numerator of the GCV estimator, representing the average training error, usually underestimates the true prediction risk. The denominator of the GCV estimator, smaller than one, is designed to correct this underestimation.
On the other hand, the denominator of the GCV estimator accounts for this training ``optimism''.
Note that this definition of the degree-of-freedom adjusted estimator coincides with the classic definition of GCV for linear smoother \citep{golub_heath_wabha_1979,craven_wahba_1979}, where the trace of the smoothing matrix represents the degrees of freedom of the linear predictor.
Existing literature provides a rich theoretical background for the consistency and efficiency of the ordinary GCV estimator.
For instance, its behaviors have been analyzed by
\citet{bayati2011lasso, bayati2013estimating,miolane2021distribution,celentano2020lasso} for the lasso and in Section 3 of \citet{bellec2020out} for penalized M-estimator. 
For ridge regression, \cite{patil2021uniform} show the uniform consistency of the ordinary GCV even in the overparameterized regimes when the feature size $p$ is larger than the sample size $n$.
Beyond ridge regression, the consistency of the estimator \eqref{eq:gcv-naive} is shown to be consistent for other convex penalized estimators under suitable design conditions, see for example, \cite{bayati2011lasso,bayati2013estimating,bellec2020out}.

Two natural extensions of GCV for bagging can be considered.
For the first extension (\eqref{eq:gcv-naive} below), the subsamples $I_1,\dots,I_M$ are sampled without looking at the data, and subsequently only the training data $(\bx_i,y_i)_{i\in I_{1:M}}$ is used for training, where $I_{1:M}=\cup_{m=1}^MI_{m}$. 
Then, since the training data is $(\bx_i,y_i)_{i\in I_{1:M}}$, a natural extension of the ordinary GCV for the ensemble estimator \eqref{eq:def-M-ensemble} is the following estimator:
\begin{align}
 \frac{\|\bL_{I_{1:M}} (\by-\bX\tbeta_{M})\|_2^2 / |I_{1:M}| }{(1 - \tdf_M / |I_{1:M}| )^2}, \label{eq:gcv-naive}
\end{align}
where $\bL_I$ is a diagonal matrix whose $i$-th diagonal entry is one if $i\in I$ and zero otherwise.
In the special case when the ensemble size is $M = 1$ or the subsample size is $k = n$, the definition reduces to the ordinary GCV \eqref{eq:ogcv}, for which consistency has been established under various settings.
For ensemble size $M>1$, there is less understanding of the behaviors of this estimator for general predictors.
However, it is shown to be inconsistent when the ensemble size is $M=2$ for ridge predictors \citep{du2023subsample}.
The inconsistency in large part happens because, for a finite ensemble size $M$, the residuals computed using the bagged predictor contain non-negligible fractions of out-of-sample and in-sample, and all of them are treated equally.

For the second extension, we consider that all data
$(\bx_i,y_i)_{i\in[n]}$ is used even if some $i\in[n]$ does not belong
to any of the subsamples $I_1,\ldots,I_M$.
This leads to the ensemble GCV:
\begin{align}
    \tR_M^{\gcv} = \frac{\| (\by-\bX\tbeta_{M})\|_2^2 / n }{(1 - \tdf_M / n )^2}. \label{eq:fgcv-naive}
\end{align}
As $M\to \infty$ with uniformly sampled $I_1,\ldots,I_M$, the union $I_{1:M}$ closely approaches the full set $[n]$ and the difference between \eqref{eq:gcv-naive} and \eqref{eq:fgcv-naive} vanishes.
Both \eqref{eq:gcv-naive} and \eqref{eq:fgcv-naive} are consistent as $M$ tends to infinity for ridge predictors \citep{du2023subsample}; namely, $\tR_{\infty}^{\gcv} \pto R_{\infty}$.
For finite ensemble sizes, however, \eqref{eq:fgcv-naive} is generally preferred over \eqref{eq:gcv-naive} because one may gain efficiency from more observations (except when considering ensemble sizes that are neither extremely large nor when considering no ensemble).
As we will show in \Cref{prop:gcv_inconsistency}, the ensemble GCV estimator, as defined in \eqref{eq:fgcv-naive}, is generally inconsistent for finite ensemble sizes.
The primary goal of this paper is to fully understand this phenomenon, correct the inconsistency of GCV estimate \eqref{eq:fgcv-naive}, and develop a risk estimator that is consistent for all ensemble sizes $M$.

\section{Proposed GCV correction}
\label{sec:risk-estimators}

In this section, we aim to address the limitations of the naive GCV estimators \eqref{eq:gcv-naive} and \eqref{eq:fgcv-naive}, which, as we demonstrate, fail to produce consistent estimates for the true prediction risk \eqref{eq:R_M} under finite ensemble sizes $M > 1$. 
By introducing term-by-term adjustments for each term of the decomposition \eqref{eq:risk-decomposition} below, we derive corrected and consistent risk estimators that hold for any ensemble size $M$.

\subsection{Intermediate risk estimators}
\label{sec:intermediate_estimators}

We begin by considering the decomposition of \eqref{eq:R_M} into its constituent components.
The motivation behind our proposed risk estimators stems from this decomposition.
It is easy to see that the risk of the $M$-ensemble estimator can be decomposed as follows:
\begin{align}
    R_{M} 
    &= 
    \frac{1}{M^2}
    \sum_{m, \ell \in [M]}
    \EE_{\bx_0, y_0}\big[(y_0 - \bx_0^\top \hbbeta_m ) (y_0 - \bx_0^\top \hbbeta_\ell ) \mid (\bX, \by), \{ I_m, I_\ell \}\big] \nonumber \\
    &=
    \frac{1}{M^2}\sum_{m, \ell \in [M]}
    \underbrace{(\hbeta_{m} - \bbeta_0)^\top \bSigma (\hbeta_{\ell} - \bbeta_0) + \sigma^2}_{R_{m, \ell}}.
    \label{eq:risk-decomposition}
\end{align}
Here, $\bbeta_0 := \EE[\bx_0 \bx_0^{\top}]^{-1} \EE[\bx_0 y_0]$ denotes the coefficient vector of linear projection of $y_0$ onto $\bx_0$, and $\sigma^2 := \EE[(y_0 -  \bx_0^\top \bbeta_0)^2]$ denotes the energy in the residual component in the response. 
As indicated in \eqref{eq:risk-decomposition}, we will denote the component in the risk decomposition corresponding to $\hat{\bbeta}_m$ and $\hat{\bbeta}_{\ell}$ by $R_{m,\ell}$.
The fundamental idea behind our risk estimators is to estimate each component $R_{m,\ell}$ individually.
The final risk estimator is then the sum of these estimated risk components.

We propose an adjustment for each term $R_{m,\ell}$ in the risk decomposition, which leads to a $\sqrt n$-consistent estimator.
This adjustment is a key component of our proposed risk estimators and contributes to their improved performance over the ordinary GCV estimator.
Before we present the corrected GCV estimators, we first introduce two intermediate estimators that give rise to our final correction.

\paragraph{Estimator using overlapping observations.\hspace{-0.5em}}
Our first estimator is defined as analogous to $R_M$, with each of $R_{m, \ell}$ replaced by its estimate $\hR^\sub_{m, \ell}$ defined below:
\begin{align}
    \tR_{M}^{\sub} 
    =  \ddfrac{1}{M^2}\sum_{m,\ell\in[M]}
    \hR^\sub_{m, \ell},
    \quad
    \hR^\sub_{m, \ell}
    =
    \frac{
    (\by_{I_m\cap I_\ell} - \bX_{I_m\cap I_\ell}\hbeta_{m})^{\top}(\by_{I_m\cap I_\ell} - \bX_{I_m\cap I_\ell}\hbeta_{\ell})
    / | I_m \cap I_\ell |
    }
    {
     {(1  - {\df_m} / {|I_m|}) 
     (1  - {\df_\ell}/ {|I_\ell|})}
     }.
    \label{eq:risk-decomp-est-1}
\end{align}
Here, the superscript ``\ovlp'' is used because the numerator of $\hR^\sub_{m, \ell}$ uses only the overlapping observations of subsets $I_m$ and $I_{\ell}$ when computing the residuals. 
In contrast, our next estimator will use all the available observations when computing similar residuals. 

The intuition behind the \ovlp-estimator \eqref{eq:risk-decomp-est-1} is that the summand in \eqref{eq:risk-decomp-est-1} is consistent for the corresponding summand in \eqref{eq:risk-decomposition}. 
This result is already known for $m=\ell$ under suitable conditions; see, for example, \cite{bellec2020out} and references therein. 
In the present paper, we extend this consistency result from $m = \ell$ to any pair of $(m,\ell)_{m,\ell\in[M]}$, see \Cref{th:sub_full} for a formal statement.

\begin{remark}[Special cases of the \ovlp-estimator]\label{rem:sub}
We consider two special cases under equal subsample sizes with $|I_m| = k$ for all $m\in[M]$:
(a) When $k = n$, the estimator in \eqref{eq:risk-decomp-est-1} matches with ordinary GCV estimator \eqref{eq:ogcv}, which is consistent.
(b) When $M = 1$, notice that $\tR_{M}^{\sub}$ is equal to the GCV in \eqref{eq:gcv-naive} restricted to $k$ subsampled observations. 
Thus, \eqref{eq:risk-decomp-est-1} is consistent for $M = 1$.
\end{remark}

It is worth noting that the definition of the \ovlp-estimator implicitly requires $| I_m \cap I_\ell |>0$, which necessitates a large enough subsample size $k$ to guarantee the overlapping sets $I_m\cap I_\ell$ are not empty. 
Furthermore, when the regularization is near zero (e.g., in ridgeless estimators), the term $1 - \df_m/|I_m|$ in the denominator can approach zero, further degrading the estimator's performance. 
To address these shortcomings, it is beneficial to use data from both $I_m$ and $I_{\ell}$, as well as other observations, to estimate the individual terms in the decomposition \eqref{eq:risk-decomposition}. This motivates us to construct the next estimator. 

\paragraph{Estimator using all observations.\hspace{-0.5em}}
We propose the following estimator, which estimates each of $R_{m,l}$ using all the data $(\bX, \by)$. The expression is given below:
\begin{align}
    \tR_M^{\full} 
    &= \ddfrac{1}{M^2} \sum_{m, \ell\in [M]}\hR_{m,\ell}^{\full}, 
    \quad
    \qquad
    \hR_{m,\ell}^{\full}  =  \frac{
    (\by - \bX \hbeta_{m})^\top (\by - \bX \hbeta_{\ell}) / n}{1 - {\df_m} / {n} - {\df_{\ell}} / {n} + ( \df_{m} \df_{\ell} / |I_m| |I_{\ell}| ) \cdot  | I_m \cap I_\ell | / {n} }.
    \label{eq:risk-decomp-est-2}
\end{align}
In contrast to the ``\ovlp'' estimator in \eqref{eq:risk-decomp-est-1}, we use superscript ``full'' in \eqref{eq:risk-decomp-est-2}, which refers to the risk estimator that uses all the available observations ($\by$ and $\bX$ are used in the numerator to compute residuals). 
Because the two intermediate estimators use different proportions of the samples to quantify the optimism, they require different degrees of freedom adjustments presented in the denominators.

Compared to $\tR^{\sub}_M$, the full-estimator offers several advantages.
Firstly, it utilizes all available data, which is beneficial, especially when the subsample sizes $|I_m|$ and $|I_{\ell}|$ are small, or the union of subsets is limited.
Secondly, its denominator is lower bounded by a positive constant, ensuring numerical stability, especially for base estimators with regularization close to zero.
It is worth noting that hybrid estimators can be constructed that utilize data points in a range between the extreme cases of the intersection and the full set of observations. We will not focus on such hybrid strategies in this paper.

\begin{remark}[Special cases of the full-estimator]
\label{rem:full}
As with \Cref{rem:sub}, we consider two special cases:
(a) When $k = n$, observe from \eqref{eq:est-full-M=1_1} that \eqref{eq:risk-decomp-est-2} matches with \eqref{eq:fgcv-naive}, which is consistent, as argued previously.
(b) When $M = 1$, the full-estimator in \eqref{eq:risk-decomp-est-2}, reduces to:
\begin{align}
    \tR_{1}^{\full}  
    &=  
    \frac{
    (\by - \bX \tbeta_{1})^\top (\by - \bX \tbeta_{1})/n}{{1 -  2{\tdf_1}/n  + (\tdf_{1})^2/(kn)}} \label{eq:est-full-M=1_1} \\
    &=  
    \frac{
    (\by - \bX \tbeta_{1})^\top \bL_{1} (\by - \bX \tbeta_{1}) / n
    + (\by - \bX \tbeta_{1})^\top \bL_{1^c} (\by - \bX \tbeta_{1}) / n}{
    1-2{\tdf}_1/n + ({\tdf}_1)^2/(kn)
    }, \label{eq:est-full-M=1_2}
\end{align}
where $\bL_{1^c}=\bI_n-\bL_{1}$ and $\bL_1 = \bL_{I_1}$ for brevity. 
To see that \eqref{eq:est-full-M=1_1} is consistent for any subsample size $k$, from the consistency of ordinary GCV for a single base predictor and no bagging, with observations $(\bx_i, y_i)_{i\in I_1}$ and sample size $k=|I_1|$, from Section 3 of \citet{bellec2020out}, we have
\begin{equation}
    \label{eq:est-full-M=1_numer1}
    \frac{
    (\by - \bX \tbeta_{1})^\top \bL_{1} (\by - \bX \tbeta_{1}) / k
    }
    {
    (1 -  {\tdf_1} / {k})^2 R_1
    }
    \pto
    1
    ,
\end{equation}
where $R_1$ is the risk of $\tbeta_1$.
Moreover, by the law of large numbers as $(n-k)\to\infty$,
\begin{equation}
    \label{eq:est-full-M=1_numer2}
    \frac{
    (\by - \bX \tbeta_{1})^\top \bL_{1^c} (\by - \bX \tbeta_{1})}{
    R_1
    (n - k)
}
\pto 1.
\end{equation}
Writing \eqref{eq:est-full-M=1_numer1} and \eqref{eq:est-full-M=1_numer2}
as $1+\op(1)$ and substituting 
into \eqref{eq:est-full-M=1_2}, we get that
\begin{align*}
    \frac{\tR_1^\full}{R_1}
    &=
    \frac{[1+\op(1)]\cdot (1 -  {\tdf_1} / {k})^2 \cdot ({k} / {n}) + [1+\op(1)]\cdot ({n-k})/{n}}{
        {1-2{\tdf}_1/n + ({\tdf}_1)^2/(kn)}
    }\pto 1.
\end{align*}
\end{remark}

\subsection{Corrected GCV for ensembles}

Though the two intermediate estimators are consistent, the nature of the term-by-term correction requires enumeration over all pairs of $m,\ell\in[M]$. 
This quadratic complexity in $M$ prevents them from being computationally tractable and practical for large $M$.
To mitigate the computational drawbacks of the intermediate estimators, we further propose a corrected GCV estimator as below.

\begin{figure}[!t]
    \centering
    \includegraphics[width=0.9\textwidth]{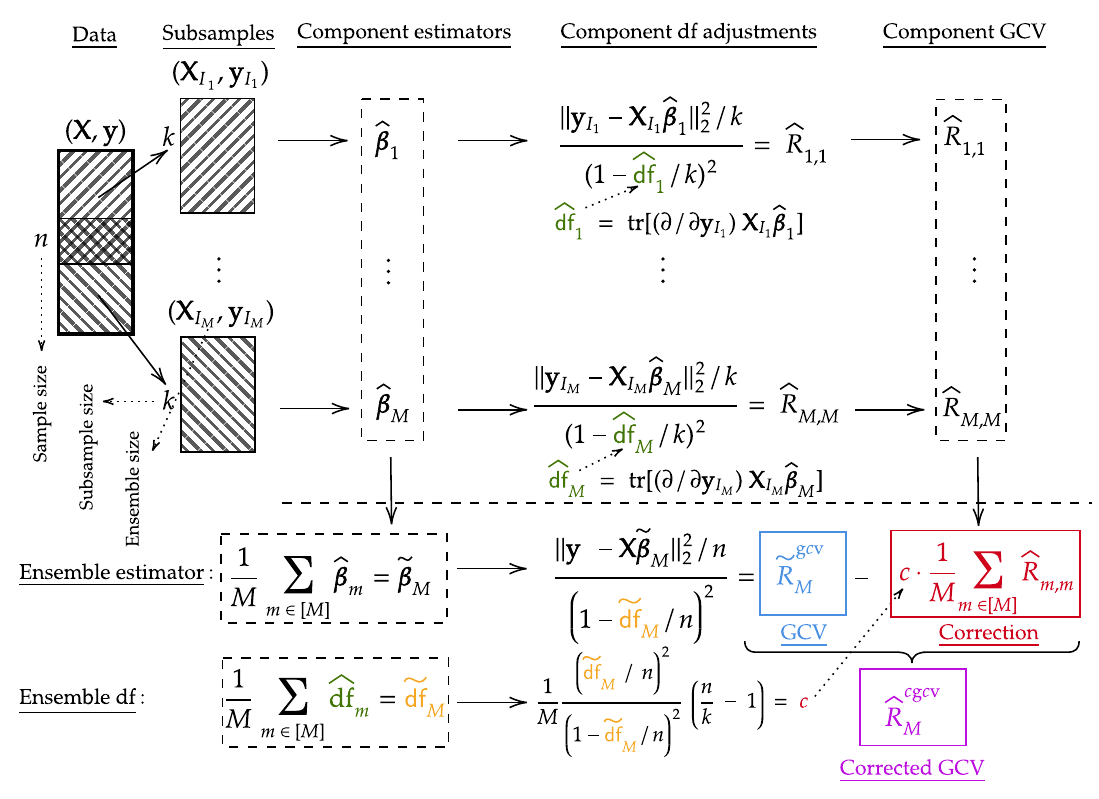}
    \caption{
    Illustration of the corrected generalized cross-validation for subsampled ensembles with equal subsample sizes $k$.
    }
    \label{fig:cgcv-illustration}
\end{figure}

\begin{definition}
[Corrected GCV for ensemble estimator]
\label{def:cgcv-proposal}
    The corrected GCV estimator for ensemble size $M$ is defined as:
    \begin{align}
        \label{eq:cgcv-proposal-sub}
        \tR^{\cgcv,\sub}_{M}
        = \underbrace{\frac{ \| \by - \bX \tbeta_{M} \|_2^2/n}{(1-\tdf_M/n)^2}}_{\tR_M^{\gcv}}
        -  \underbrace{\frac{1}{M} \Biggl\{ \frac{(\tdf_M/n)^2}{(1-\tdf_M/n)^2} \frac{1}{M}\sum_{m \in [M]} \bigg(\frac{n}{|I_m|}-1\bigg) \hat{R}_{m, m}^\sub \Biggr\}}_{\mathrm{correction}}.
    \end{align}
    We will also consider the alternative expression:
    \begin{align}
        \label{eq:cgcv-proposal}
        \tR^{\cgcv,\full}_{M}
        = \underbrace{\frac{ \| \by - \bX \tbeta_{M} \|_2^2/n}{(1-\tdf_M/n)^2}}_{\tR_M^{\gcv}}
        -  \underbrace{\frac{1}{M} \Biggl\{ \frac{(\tdf_M/n)^2}{(1-\tdf_M/n)^2} \frac{1}{M}\sum_{m \in [M]} \bigg(\frac{n}{|I_m|}-1\bigg) \hat{R}_{m, m}^\full \Biggr\}}_{\mathrm{correction}},
    \end{align}
    where the only difference is the use of 
    $\hat{R}_{m,m}^\full$ instead of $\hR^{\sub}_{m,m}$ in the rightmost sum in \eqref{eq:cgcv-proposal}. 
\end{definition}

The evaluation of \eqref{eq:cgcv-proposal} (and \eqref{eq:cgcv-proposal-sub}) only requires computing three components: (1) the full training of the ensemble estimator; (2) the average $\tdf_M$ defined in \eqref{eq:def-tdf}; and (3) the correction term that is built on the intermediate estimators. %
More importantly, this proposed estimator can be evaluated only in linear time with respect to $M$.
See \Cref{fig:cgcv-illustration} for an illustration of computing \eqref{eq:cgcv-proposal}.

The correction term in \eqref{eq:cgcv-proposal} is derived from the insights gained through the intermediate estimators $\tR_{M}^{\sub}$ and $\tR_M^{\full}$. 
It allows us to alleviate the computational challenges posed by the pairwise enumeration in the intermediate estimators, leading to an estimator that retains the advantages of both while being more computationally tractable.
While the correction is defined with respect to GCV in \eqref{eq:fgcv-naive}, one can also derive another correction with respect to \eqref{eq:gcv-naive}.
We prefer to present the correction with respect to \eqref{eq:fgcv-naive} because it leads to a slightly simpler expression and because, in this case, the correction is always positive.
Our theoretical results in the next section show the corrected GCV estimator \eqref{eq:cgcv-proposal} is consistent (\Cref{th:correction_gcv}), and
the naive extension of GCV in \eqref{eq:fgcv-naive} \emph{over-estimates the risk} with positive probability (\Cref{prop:gcv_inconsistency}).

\begin{remark}[Special cases of corrected GCV]
\label{rem:cgcv}
As with \Cref{rem:sub,rem:full}, we consider two special cases:
(a) When $k = n$, observe that the correction term is zero and the first term $\tR_1^\gcv$ is consistent as argued previously; thus, $\tR^{\cgcv,\full}_M$ is consistent.
(b) When $M = 1$, it is also not hard to see that $\tR^{\cgcv,\full}_1$ is consistent.
As argued in \Cref{sec:intermediate_estimators}, both $\hR^\sub_{1,1}$ and $\hR^\full_{1,1}$ are consistent.
Now writing \eqref{eq:est-full-M=1_numer1} and \eqref{eq:est-full-M=1_numer2} as $1+\op(1)$, 
\[
    \frac{\tR^{\cgcv,\full}_1}{R_1}
    =
    \frac{
        [1+\op(1)]
        (1 - \tdf_1/k)^2 {k}/{n} 
        +
        [1+\op(1)]
        ({n-k})/{n} 
    }{(1- \tdf_1/n)^2} 
    -
    \frac{
    [1+\op(1)] 
    (\tdf_1/n)^2 ({n-k})/{k}}{(1- \tdf_1/n)^2} \pto 1.
\]
Here, the last equality follows because
\begin{align*}
    \frac{k}{n}(1 - \tdf_1/k)^2 + \frac{n-k}{n} - \frac{n-k}{k}(\tdf_1/n)^2  &= 1 - 2 \tdf_1/n + \tdf_1^2/(kn) - \frac{n-k}{k}(\tdf_1/n)^2 \\
    &= 1 - 2 \tdf_1/n + \tdf_1^2/n^2
    = (1- \tdf_1/n)^2.
\end{align*}
\end{remark}
\begin{remark}[Unequal weighted ensemble]\label{rmk:unequal-weight}
    When defining the ensemble estimator in \eqref{eq:def-M-ensemble} with unequal weights, one can derive a similar corrected GCV estimator.
    This requires modifying the risk expansion in \eqref{eq:risk-decomposition} and using the intermediate estimators presented in \Cref{sec:intermediate_estimators}.
    For simplicity, we present and analyze the corrected GCV with equal weights.
    Because the individual estimators are i.i.d.\ conditioned on data, among all choices of deterministic weights that sum to $1$, the optimal combination of weights is simply $1/M$ for the lowest expected risk.

    One way to see this is to note that, since the penalties $g_1,\dots,g_M$ are the same, the component estimators $\hbeta_1,\dots,\hbeta_M$ are exchangeable.
    Because of exchangeability, for weights $(w_1,\dots,w_M)$, the expected risk is a convex polynomial of degree $2$ that is invariant by permutation of the $M$ inputs.
    By symmetry, a minimum is achieved at equal weights $w_1=\cdots=w_M$.
    Indeed, if a minimum is achieved at $w_1^*,\ldots,w_M^*$, it is also achieved at the average over all permutations $\frac{1}{M!}\sum_{\pi}(w_{\pi(1)}^*,\ldots,w_{\pi(M)}^*)$ by Jensen's inequality, and this average is proportional to $(1,1,\ldots,1)$.
    More broadly, this is a consequence of the Purkiss Principle \citep{waterhouse1983symmetric}.
    Thus, if we further restrict the weights to sum to $1$, the optimal choice for each weight is $1/M$.
\end{remark}

Our goal in the rest of the paper is to establish consistency properties of the corrected GCV for any $M$ and to provide convergence rates under various assumptions on the data and component estimators.
We will provide a non-asymptotic analysis in \Cref{sec:finite-sample-analysis} and an asymptotic analysis in \Cref{sec:asm-consistency}.
In both sections, we assume the following sampling scheme. 

\begin{assumption}[Sampling strategy]\label{assu:sampling}
The subsample sets $\{I_m\}_{m=1}^M$ are independent of $(\bX, \by)$ and i.i.d., uniformly distributed over subsets of $[n]$ with cardinality $k$. 
\end{assumption}

Before we dive into our analysis, it is worth reiterating that the intermediate estimators \eqref{eq:risk-decomp-est-1} and \eqref{eq:risk-decomp-est-2} are defined for ensembles fitted on subsamples $\cD_{I_m}$ for $m \in [M]$ of arbitrary sizes and component estimators fitted on different penalty functions $g_m$ for $m \in [M]$.
While our proofs are simplified by assuming equal subsample sizes, this condition can be relaxed for certain theoretical claims to follow in \Cref{sec:intermediate-guarantees-gaussian}. 
The formulation of the modified GCV, on the other hand, assumes that the same penalty function is used for estimators across subsamples and that subsample sizes do not vary much to ensure the validity of certain degrees of freedom concentrations, as we will explain in \Cref{sec:cgcv-guarantees-gaussian}.

\section{Non-asymptotic analysis for general ensembles}
\label{sec:finite-sample-analysis}

In this section, we provide a non-asymptotic analysis of the estimators introduced in the previous section.
Our analysis is applicable to general convex penalized estimators such as ridge regression and elastic net, among others.

To facilitate our analysis, we will need to impose certain assumptions on the data generation process.
Specifically, we will assume Gaussian features and linear models as encapsulated in the assumptions below.
These assumptions serve as a starting point but are not indispensable for more specialized estimators, as we discuss in \Cref{sec:asm-consistency}.

\begin{assumption}[Feature structure]
\label{assu:Gaussian-feature}
     The design matrix $\bX\in \R^{n\times p}$ has i.i.d. rows $\bx_i$ for $i \in [n]$ each drawn from a Gaussian distribution: $\bx_i \sim \cN(\bm{0},\bSigma)$.
\end{assumption}

\begin{assumption}[Response structure]
    \label{assu:Gaussian-response}
    The response vector $\by$ has i.i.d.\ entries $y_i$ for $i \in [n]$ each drawn from a linear model with Gaussian noise: $y_i = \bx_i^\top\bbeta_0 + \varepsilon_i$, where $\varepsilon_i\sim \cN(0,\sigma^2)$ is independent of~$\bx_i$.
\end{assumption}

Although these assumptions may appear restrictive, they serve to simplify our analysis and enable us to derive the estimate of risk as well as prove explicit non-asymptotic bounds.
These assumptions can be substantially relaxed for specialized base estimators.
In \Cref{sec:asm-consistency}, we will demonstrate that we can significantly relax these assumptions by focusing on a specific base estimator of ridge(less) regression.

\begin{assumption}
[Penalty structure]
\label{assu:penalty}
The penalty function $g_{m}: \R^p \to \R$ in \eqref{eq:def-hbeta} is $\mu$-strongly convex with respect to $\bSigma$ for $\mu > 0$, i.e., the function $\bb \mapsto g_{m}(\bb) - (\mu/2) \bb^\top \bSigma \bb$ is convex in $\bb \in \R^p$.
\end{assumption}

If $g_m$ is twice continuously differentiable, \Cref{assu:penalty} is equivalent to $\inf_{\bb\in \R^p} \nabla^2 g_m(\bb) \succeq \mu \bSigma$, where $\nabla^2 g_m(\bb) \in \R^{p\times p}$ is the Hessian matrix. 
For instance, if $g_m(\bb) = \bb^\top \bSigma_{w}\bb/2$ (generalized ridge penalty) with a positive definite matrix $\bSigma_{w}\in \R^{p\times p}$, \Cref{assu:penalty} holds with 
$
\mu = \sigma_{\min}(\bSigma_{w})/\sigma_{\max}(\bSigma).
$
In particular, $\mu=\lambda/\sigma_{\max}(\bSigma)$ when $g_m(\bb) = \lambda\|\bb\|_2^2/2$. 
Note that the strong convexity penalty guarantees a certain differentiable structure of the penalized least-squares estimator \eqref{eq:def-hbeta}  
with respect to $(\bX,\by)$ (see, for example, Theorem 1 of \citet{bellec2022derivatives} and \Cref{sec:derivatives_formulae}), which helps us in the proof to bound certain quantities in terms of $\mu$.
This strong convexity assumption can be relaxed in certain cases, including the lasso $g_m(\bb) = \lambda \|\bb\|_1$, as studied in \cite{celentano2020lasso,bellec2020out}; however, the proof argument will be largely different from the proof used in the current paper. 
Finally, note that in the case of $p/n < 1$, if the smallest eigenvalue of $\bSigma$ is uniformly lower bounded by a constant away from $0$, then the strongly convex assumption is not needed; we only need the penalty to be convex.
For an explicit statement and proof of this fact in the context of \Cref{th:sub_full}, see \Cref{sec:relaxing_strong_convexity}.

\subsection{Guarantees for intermediate risk estimators}
\label{sec:intermediate-guarantees-gaussian}

Our first result provides a non-asymptotic guarantee on the intermediate estimators.
Recall the notation in \eqref{eq:risk-decomposition}: the estimation target is $R_M = M^{-2}\sum_{m, \ell}R_{m,\ell}$, where $R_{m, \ell} = (\hat{\bbeta}_m - \bbeta_0)^\top\bSigma (\hat{\bbeta}_\ell - \bbeta_0) + \sigma^2$. 
For notational convenience, we introduce the following notation for each $m, \ell\in[M]$,
\begin{align}
 d_{m,\ell}^\sub:= 
        |I_m\cap I_\ell| \bigg(1-\frac{\df_m}{|I_m|}\bigg) \bigg(1-\frac{\df_\ell}{|I_\ell|}\bigg),
        \quad
        \qquad
 d_{m,\ell}^\full:= 
        n-\df_m - \df_\ell + 
    \frac{\df_m \df_\ell}{|I_m| |I_{\ell}|}
    | I_m \cap I_\ell |, 
    \label{eq:d_sub_and_full}
\end{align}
which are the denominators in the definitions of $\hat{R}_{m, \ell}^\sub$ and $\hat{R}_{m, \ell}^\full$ in \eqref{eq:risk-decomp-est-1} and \eqref{eq:risk-decomp-est-2}. 
Now, we are ready to provide the theoretical results of $\hat{R}_{m, \ell}^\sub$ and $\hat{R}_{m, \ell}^\full$ in \Cref{th:sub_full}.

\begin{restatable}[Finite-sample bounds for intermediate estimators]{theorem}{ThmSubFull}
\label{th:sub_full}
Suppose \Cref{assu:sampling,assu:Gaussian-feature,assu:Gaussian-response,assu:penalty} hold.
Let $c=k/n$, $\gamma = \max(1, p/n)$, and $\tau = \min(1, \mu)$. 
Then there exists an absolute constant $C>0$ such that the following holds:
\begin{align}
    \EE \bigg[\Big|\frac{d_{m, \ell}^{\sub}}{n} \cdot \frac{\hat{R}_{m, \ell}^{\sub} -R_{m,\ell}}{\sqrt{R_{m,m} R_{\ell,\ell}}}\Big| \bigg]
    \le C
    \frac{\gamma^{7/2}}{\tau^2 c^3 \sqrt n},
    \quad
    \qquad
    \EE \bigg[\Big|\frac{d_{m, \ell}^{\full}}{n} \cdot \frac{\hat{R}_{m, \ell}^{\full} -R_{m,\ell}}{\sqrt{R_{m,m} R_{\ell,\ell}}}\Big| \bigg]\le
    C
    \frac{\gamma^{5/2}}{\tau^2 c^2 \sqrt n}.
\label{eq:sub_full_moment_ineq}
\end{align}
Furthermore, if the same penalty $g_m$ is used for subsample estimate \eqref{eq:def-hbeta} across all $m\in [M]$, 
    for any $\epsilon\in(0,1)$, we have
    \begin{align}
    \PP\bigg(
    \Big|
    1-
    \frac{\tR_{M}^\sub}{R_M}
    \Big|
    >
    \epsilon
    \bigg)
    \le C \frac{M^3 \gamma^{11/2}}{\epsilon \tau^4 c^7 \sqrt{n}},
    \quad
    \qquad
    \PP\bigg(
    \Big|
    1-
    \frac{\tR_{M}^\full}{R_M}
    \Big|
    >
    \epsilon
    \bigg)
    \le C \frac{M^3 \gamma^{9/2}}{\epsilon \tau^4 c^2 \sqrt{n}}.
\label{eq:sub_full_relative_error}
\end{align}
Thus, if $(M, \mu^{-1}, p/n, n/k)$ are bounded from above by a constant independent of $n$,
we have 
$$\tR_M^\sub/R_M = 1+\Op(n^{-1/2}),
\quad
\qquad
\tR_M^\full/R_M = 1+\Op(n^{-1/2}).
$$
\end{restatable}
\Cref{th:sub_full} shows the rate of convergence for both the \ovlp- and full-estimator are $n^{-1/2}$.
Fixing other problem parameters $(\tau,c,\gamma,M)$, this bound in $n^{-1/2}$ is tight up to constants; e.g., see argument around equation (2.6) of \cite{bellec2020out} in the case of $M = 1$ and the Ordinary Least-Squares.
It is worth noting that the upper bounds regarding $\tR_M^\sub$ and $\tR_M^\full$ in \eqref{eq:sub_full_moment_ineq} and \eqref{eq:sub_full_relative_error} are slightly different, especially in the dependence of the ratio $c=k/n$ and $\gamma$; 
the upper bounds of the full-estimator are tighter than the \ovlp-estimator in terms of its dependencies on $c$ and $\gamma$. 
While these upper bounds may not be optimal in the dependence on $c$ and $\gamma$, we empirically observe that the relative error of the full-estimator is smaller than the \ovlp-estimator, especially when $c$ is small---see \Cref{fig:k}.
We do not claim that the dependence on the parameter \(M\) is tight; rather, it appears to be an artifact of our proof technique. Specifically, while the current bound suggests that \(M \ll n^{1/6}\) is needed for the bound to approach zero, this condition is not necessary in practice. Empirical results in Section \ref{sec:simulations-large-M} indicate that the proposed estimators perform well for all values of \(M\).
A major improvement for the current polynomial dependence in $M$ would come from replacing the bounds \eqref{eq:sub_full_moment_ineq} by novel exponential probability bounds for a single pair $(m,\ell)\in[M]^2$. 
Union bounds over $[M]^2$ pairs would then only induce a logarithmic dependence in $M$.
Exponential bounds for the risk $\|\bSigma^{1/2}(\hbeta_m-\bbeta_0)\|^2$ and the residual $\|\by-\bX\hbeta_m\|^2$ are available, thanks to the CGMT \citep{celentano2020lasso,loureiro2021learning}.
To derive exponential bounds for \eqref{eq:sub_full_relative_error}, one would also need similar exponential concentration bounds for the inner products $(\hbeta_m-\bbeta_0)^\top\bSigma(\hbeta_\ell-\bbeta_0)$ and for the degrees of freedom $\df_m$.
One avenue to obtain exponential bounds for the inner products is the conditional Gordon inequality developed in \cite[Lemma F.2 and appendix L]{celentano2021cad}.
On the other hand, we are not aware at this point of any tool that can provide exponential inequalities for the degrees of freedom $\df_m$, except in the special case of the Lasso \cite[Theorem 8]{celentano2020lasso}.
The reason that the Lasso case is special and that this argument does not generalize to penalty functions different from the $\ell_1$ norm is discussed in \cite[around equation (3.8)]{bellec2020out}.

Note that the guarantees in \Cref{th:sub_full} are in a multiplicative form.
This is preferable to the more common additive bounds because, in the regimes we study, the risk $R_M$ may have varied scales with different subsample sizes; for example, when the subsample size is near the feature size and the regularization parameter is small, in which cases the risk can be high \citep{patil2022bagging}.
It is also worth noting that we do not assume either the pure signal energy (i.e., $\| \bbeta_0 \|_2^2$) or the effective signal energy (i.e., $\| \bSigma^{1/2} \bbeta_0 \|_2^2$)  is bounded. 
This is possible because of the multiplicative bounds.

\begin{remark}
    [Diminishing strong convexity parameter]\label{rmk:strong-convexity}
    The upper bounds in \eqref{eq:sub_full_moment_ineq} and \eqref{eq:sub_full_relative_error} give explicit dependence on the strong convexity parameter $\mu$ of \Cref{assu:penalty} through $\tau=\min(1,\mu)$.
    Our results thus allow to choose $\mu$ decreasing in $n$, for instance $\mu=n^{-1/8-\epsilon}$ is sufficient to ensure that the upper bounds of \eqref{eq:sub_full_moment_ineq} and \eqref{eq:sub_full_relative_error} go to 0 as $n\to\infty$ while $M,c^{-1},\gamma$ are bounded by constants.
\end{remark}

\begin{remark}
    [Extensions beyond strongly convex regularizers]
    Without any strong convexity assumption on the penalty (i.e., $\mu=0$), results such as \Cref{thm:risk-est} for $M=1$ cannot be established using known techniques even for the well-studied lasso, unless additional assumptions are made on the subsample size.
    In particular, Proposition 4 of \citet{celentano2020lasso} shows that in the current proportional asymptotic regime with $n\asymp p$, the risk of a single lasso may be unbounded if the number of samples used to train the lasso falls below the Donoho-Tanner phase transition described in Section 3 of \citep{celentano2020lasso}.
    While the consistency of GCV for a single lasso is a consequence of this theory above this phase transition \cite[Theorem 9]{celentano2020lasso}, below the phase transition both $\|\bSigma^{1/2}(\hbbeta-\bbeta_0)\|_2$ and $1/(1-\df/n)$ are unbounded for certain $\bbeta_0$.
    Furthermore, the constants appearing in the upper bounds in \cite{celentano2020lasso} explode as the sample size approaches the phase transition. 
    We are not aware of any currently known techniques suitable for studying GCV below this phase transition, and we believe new ideas are needed.
    Thus, in the present context, where we subsample $I_m$ of size $k=|I_m|$, small values of $k$ would eventually fall below the phase transition, in which case the theoretical analysis of bagging GCV is currently out of reach.
\end{remark}

We numerically compare the performance of the \ovlp- and full-estimators for the ridge ensemble in \Cref{fig:k}.
It is clear that the relative error of the full-estimator is smaller than the \ovlp-estimator across different subsample size $k$ and tuning parameter $\lambda$, especially when $k$ is small and the tuning parameter $\lambda$ is small.
We also observe that, as $k$ increases towards the full sample size $n$, the performance of \ovlp-estimator and full-estimator are getting closer; this makes sense because the expressions of \ovlp-estimator and full-estimator will be the same when $k=n$. 
For comparison between the \ovlp- and full-estimators in the context of the elastic net ensemble and the lasso ensemble, we observe a similar comparative trend, and the results are shown in \Cref{app:experiments-gaussian-sub-vs-full}. 

\subsection{Guarantees for corrected GCV}
\label{sec:cgcv-guarantees-gaussian}

The next result provides non-asymptotic guarantees for corrected GCV.
Recall the form of the two variants of CGCV, $\tR^{\cgcv,\full}_{M}$ and $\tR^{\cgcv,\sub}_{M}$ from \eqref{eq:cgcv-proposal} and \eqref{eq:cgcv-proposal-sub}, respectively:
\begin{align*}
    \tR^{\cgcv,\est}_{M} :=\frac{\|\by - \bX \tbeta_M\|_2^2/n}{(1-\tdf_M/n)^2} -  
    \frac{(\tdf_M/n)^2}{(1-\tdf_M/n)^2} \frac{1}{M^2} \left(\frac{n}{k}-1\right)\sum_{m=1}^M \hat{R}_{m,m}^\est,
\end{align*}
where $\est\in\{\full,\sub\}$, depending on which estimate is used
in the rightmost sum to estimate $R_{m,m}$.

\begin{figure}[!t]
    \centering
    \includegraphics[width=0.75\textwidth]{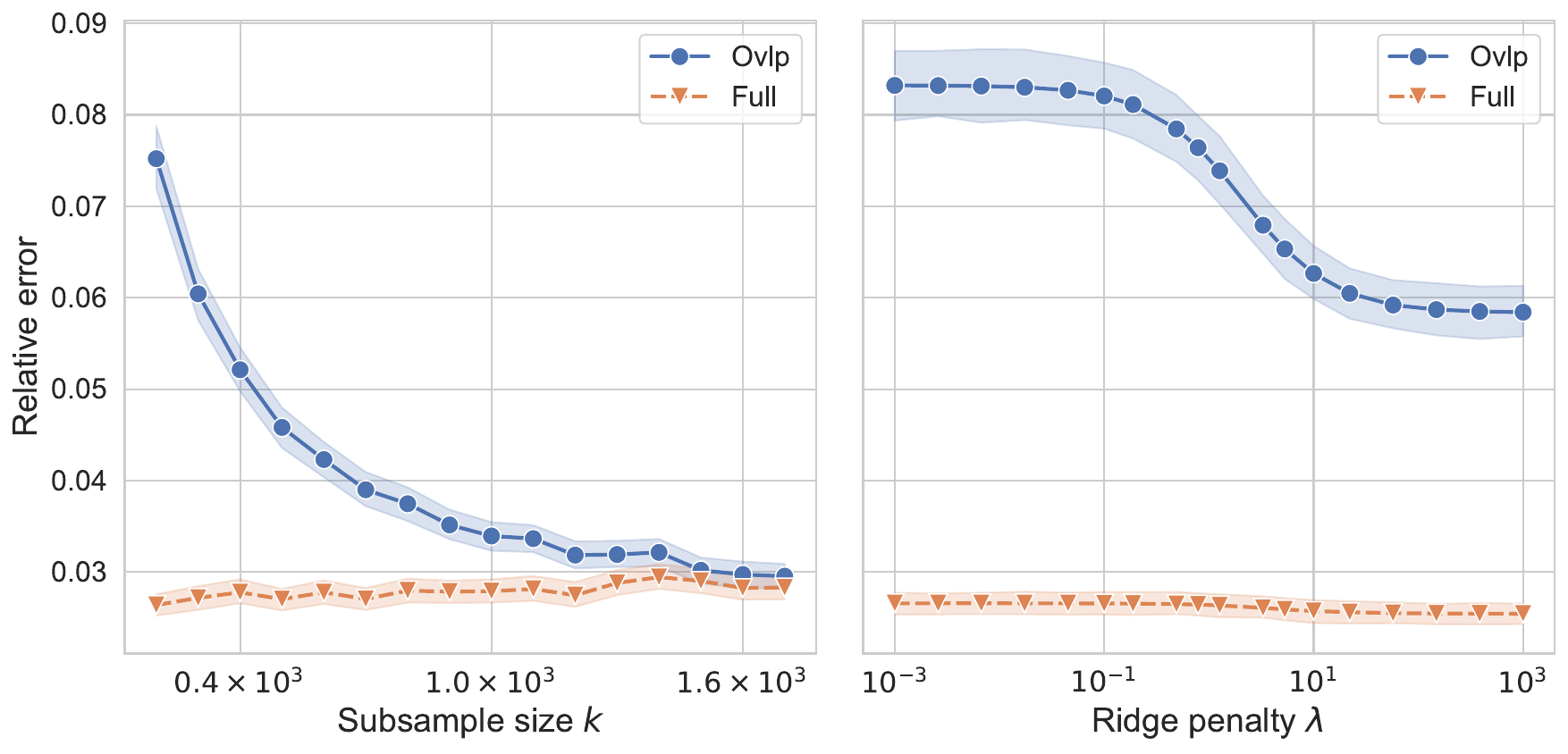}
    \caption{
    \textbf{Full-estimator $\tR_{M}^{\full}$  performs better than \ovlp-estimator $\tR_{M}^{\sub}$ across different regularization and for small subsample size $k$.}
    Plots of relative errors of the ``\ovlp'' and ``full'' estimators for ridge ensemble with ensemble size $M=10$.
    The left panel shows the results with ridge penalty $\lambda=1$ and varying subsample size $k$.
    The right panel shows the results with subsample size $k=300$ and varying ridge penalty $\lambda$.
    The data generating process is the same as in \Cref{fig:Fig1_ridge_lasso}.
    More details on the experimental setup can be found in \Cref{sec:numerical-illustrations-gaussian}.
    }\label{fig:k}
\end{figure}

\begin{restatable}[Finite-sample bounds for corrected GCV]{theorem}{ThmCorrectionGCV}\label{th:correction_gcv}
Suppose \Cref{assu:sampling,assu:Gaussian-feature,assu:Gaussian-response,assu:penalty} hold.
Let $c=k/n$, $\gamma = \max(1, p/n)$, and $\tau = \min(1, \mu)$.
Assume that the same penalty $g_m$ is used for each $m\in[M]$ in \eqref{eq:def-hbeta}.
Then, for any $\epsilon\in (0,1)$, there exists an absolute constant $C>0$ such that the following holds:
\begin{equation}
\label{eq:cgcv_relative_error}
\PP\Bigl(\Big|1-\frac{\tR_M^{\cgcv,\sub}}{R_M}\Big| > \epsilon\Bigr) \le C \frac{M^4 \gamma^{15/2}}{\epsilon\tau^6 c^6 \sqrt{n}},
\quad
\qquad
\PP\Bigl(\Big|1-\frac{\tR_M^{\cgcv,\full}}{R_M}\Big| > \epsilon\Bigr) \le C \frac{M^4 \gamma^{13/2}}{\epsilon\tau^6 c^4 \sqrt{n}}.
\end{equation}
Thus,
if $(M, \mu^{-1}, p/n, n/k)$ are bounded by constants independent of $n$, we have
$$
\tR^{\cgcv,\sub}_M/R_M = 1+ \Op(n^{-1/2}),
\quad
\qquad
\tR^{\cgcv,\full}_M/R_M = 1+ \Op(n^{-1/2}).
$$
\end{restatable}

The upper bounds presented in \Cref{th:correction_gcv} provide dependence of the accuracy of corrected GCV on several problem parameters, such as $M$, $\tau$, $c$, and $\gamma$. 
The key message of \Cref{th:correction_gcv} is 
that the corrected GCV estimators (both $\tR^{\cgcv,\sub}_M$ and $\tR^{\cgcv,\full}_M$) enjoy the rate of convergence $n^{1/2}$. 
As expected, numerical comparisons show that full-CGCV outperforms \ovlp-CGCV, see \Cref{fig:Fig2_sub_full_ridge-cgcv.pdf,fig:Fig2_sub_full_ElasticNet-cgcv.pdf,fig:Fig2_sub_full_lasso.pdf}. 

\begin{remark}[GCV is consistent for infinite ensembles]
Note that for $M \to \infty$, it is easy to show that the correction term becomes 0 if $1-\tdf_M/n \ge \delta$ for some positive constant $\delta$ and $M^{-1}\sum_{m=1}^M \hR_{m,m}^{\est}$ is bounded from above.  
A special case of this result for ridge ensemble is proved in \cite{du2023subsample}. 
While this literature showed that GCV is not consistent for ridge ensemble with $M=2$, our results imply more than inconsistency by providing the correction term to make GCV consistent, and our results are applicable not only to ensemble ridge estimates but also to ensembles of any strongly-convex regularized least-squares, including elastic net estimates.
In the next theorem, we show that GCV is inconsistent by providing a lower bound of the relative error of the GCV estimator. 
\end{remark}

\begin{restatable}[Inconsistency of GCV for finite ensembles]{theorem}{PropGCVInconsistency}\label{prop:gcv_inconsistency}
    Suppose \Cref{assu:sampling,assu:Gaussian-feature,assu:Gaussian-response,assu:penalty} hold, and assume that the same penalty $g_m$ is used for each $m\in[M]$ in \eqref{eq:def-hbeta}.
    Let $c=k/n$, $\gamma = \max(1, p/n)$, $\tau = \min(1, \mu)$. 
    Then, there exists an absolute constant $C>0$ such that the following holds:
    The probability that $\tR^\gcv_M$ in \eqref{eq:fgcv-naive} over-estimates the risk $R_M$ is lower bounded as:
    $$
    \text{for all } \delta\in(0,1), \quad \PP\bigg(
    \frac{\tR_M^{\gcv}}{R_M}
    \ge 1
    +\delta^2 \frac{c(1-c)}{2M}
    \bigg) \ge \PP\bigg(\frac{\tdf_M}k \ge \delta\bigg) - C \frac{M^5\gamma^{15/2}}{\sqrt{n}\tau^5c^5 (1-c)\delta^2}. 
    $$
    Therefore, if there exists a positive constant $\delta_0$ independent of $n$ such that $\PP(\tdf_M/k \ge \delta_0)\ge \delta_0$, and $\{M, \mu^{-1}, p/n, c^{-1}, (1-c)^{-1}\}$ are bounded by a constant independent of $n$, 
    $\tR_M^{\gcv}$
    is inconsistent and over-estimates
    the risk $R_M$ with positive probability, in the sense that
    $$
    \liminf_{n\to\infty} \ \PP\bigg(\frac{\tR^\gcv_M}{R_M} \ge 1 + \frac{\delta_0^2c(1-c)}{2M}\bigg) \ge \delta_0.
    $$
\end{restatable}

In the context of ridge ensembles, \citet{du2023subsample} show that the GCV variant \eqref{eq:gcv-naive} is not consistent when ensemble size is $M=2$.
It is possible to extend the inconsistency results for the variant \eqref{eq:gcv-naive} for finite ensemble sizes greater than $1$.
We provide numerical experiments showing the inconsistency of this variant in \Cref{sec:naive-gcv-inconsistency}.
As discussed in \Cref{subsec:naive_gcv}, for a moderate ensemble size $M$ and subsample size $k$, we expect the GCV variant \eqref{eq:gcv-naive} to behave similarly to \eqref{eq:fgcv-naive}.
Furthermore, the inconsistency of GCV implied by \Cref{prop:gcv_inconsistency} applies even for other penalized least-squares estimator \eqref{eq:def-hbeta}, extending the previous result to more generality.

\begin{remark}[On the assumption of non-negligibility of degrees of freedom]
    \label{rem:df-nonnegligibility}
    We argue here that assumption $\PP(\tdf_M/k>\delta_0) > \delta_0$ in \Cref{prop:gcv_inconsistency} is unavoidable.
    If no such constant $\delta_0$ exists as $n,p,k\to\infty$, by extracting a subsequence if necessary, we may assume $\tdf_M/k\pto 0$. Herein, let $m\in[M]$ be fixed (e.g., take $m=1$ throughout this remark). Then the positiveness of $(\df_{m})_{m=1}^M$ (see \eqref{eq:property_tr} in \Cref{lem:derivative}) implies $\df_m/k \le M\cdot \tdf_M/k = o_p(1)$.  
    The fact that GCV is consistent for $\hbbeta_m$,  combined with $\df_m/k\pto 0$, gives $k^{-1} \sum_{i\in I_m}(y_i-\bx_i^\top\hbbeta_m)^2 / R_{m,m} \pto 1$.
    Now consider the deterministic $\bar\bbeta$ defined as $\bar\bbeta = \argmin_{\bb \in \R^p} (\sigma^2 + \|\bSigma^{1/2}(\bbeta_0 - \bb)\|_2^2)/2 + g_m(\bb),$ where we replaced the objective function in \eqref{eq:def-hbeta} inside the minimum by its expectation.
    The strong convexity of the objective function of $\bar\bbeta$ and the optimality condition of $\hbeta_m$ give
    \begin{align*}
        \mu
        \|\bSigma^{1/2}(\hbbeta_m - \bar\bbeta)\|_2^2
        &\le
        \bigl[
        \|\bSigma^{1/2}(\hbbeta_m - \bbeta_0)\|_2^2
        +
        \sigma^2
        \bigr]
        -
        \bigl[
        \|\bSigma^{1/2}(\bar\bbeta-\bbeta_0)\|_2^2
        +\sigma^2
        \bigr]
        + 2(g_m(\hbbeta_m) - g_m(\bar\bbeta)),
        \\
        0 &\le 
        - \|\bL_{I_m} (\by - \bX\hbbeta_m)\|_2^2/k
        +
        \|\bL_{I_m} (\by - \bX\bar\bbeta)\|_2^2/k
        + 2(g_m(\bar\bbeta)-g_m(\hbbeta_m)).
    \end{align*}
    Summing these two inequalities, the penalty terms cancel out,
    and using the law of large numbers for 
    $\|\bL_m (\by - \bX\bar\bbeta)\|_2^2/k$, we get
    $
        \|\bSigma^{1/2}(\hbbeta_m - \bar\bbeta)\|_2^2
        \le \op(1)
        [
        \sigma^2 \|\bSigma^{1/2}(\hbbeta_m - \bbeta_0)\|_2^2
        +
        \sigma^2
        \|\bSigma^{1/2}(\bar\bbeta - \bbeta_0)\|_2^2
        ]
    $ so that
    $\|\bSigma^{1/2}(\hbbeta_m - \bar\bbeta)\|_2^2/
    [\sigma^2 + 
        \|\bSigma^{1/2}(\bar\bbeta - \bbeta_0)\|_2^2]
        \pto 0.
    $
    This means according to the norm $\|\bSigma^{1/2}(\cdot)\|_2$,
    $\hbbeta_m$ behaves like the point mass $\bar\bbeta$ up to
    a negligible multiplicative factor. 
    This contradicts the typical asymptotics seen in the proportional regime with $p \asymp n$ \citep[among others]{bayati2011lasso,thrampoulidis2018precise,miolane2021distribution,loureiro2022fluctuations,celentano2020lasso,wang2020bridge},
    so that the assumption $\PP(\tdf_M/k>\delta_0)>\delta_0$ is intrinsic to the proportional regime assumed throughout the paper.
\end{remark}

\subsection{Proof outlines}
\label{sec:proof_outlines}

In this section, we provide the proof outlines and the ideas that lead to the definition of the \sub and \full estimators. 
Then we provide a heuristic derivation of the corrected GCV estimator from the \full estimator.

By the risk decomposition \eqref{eq:risk-decomposition}, it is sufficient to study a single term in the double sum over $[M]^2$ and to focus on two estimators $\hbbeta$ and $\tbbeta$, trained respectively on subsampled datasets $(\bx_i,y_i)_{i\in I}$ and $(\bx_i,y_i)_{i\in\tilde I}$ for two subsets $I\subset [n]$ and $\tilde I\subset [n]$ independent of $(\bX,\by)$, that is,
\begin{equation*}
    \hbeta := \argmin_{\bbeta\in \R^p} \frac{1}{k} \sum_{i\in I} (y_i - \bx_i^\top \bbeta)^2 + g(\bbeta), 
    \quad
    \text{and}
    \quad
    \tbeta := \argmin_{\bbeta\in \R^p} \frac{1}{k} \sum_{i\in \tilde I} (y_i - \bx_i^\top \bbeta)^2 + \tilde{g}(\bbeta).
\end{equation*}
To study the success or failure of GCV in \eqref{eq:ogcv} or \eqref{eq:gcv-naive},
it is natural to study the inner products between the residuals of $\hbbeta$ and $\tbbeta$, of the form
\begin{equation}
\sum_{i\in J} \Bigl(y_i-\bx_i^\top\hbbeta\Bigr)\Bigl(y_i-\bx_i^\top\tbbeta\Bigr)
\label{eq:inner_product_res_J}
\end{equation}
for some subset $J\subset [n]$, as such inner products naturally appear by expanding the square in the numerators of the GCV estimates in \eqref{eq:ogcv} or \eqref{eq:gcv-naive}.
There are several natural candidates for the set $J$ of observations to use in the inner product above, including $J=[n]$ or $J=\cup_{m\in [M]}I_m$, or, as we study specific estimators trained over $I$ and $\tilde I$, the intersection $I\cap \tilde I$ as well as $I$ or $\tilde I$ themselves.
The following proof sketch explains why the intermediate estimators \eqref{eq:risk-decomp-est-1}-\eqref{eq:risk-decomp-est-2} were defined in this manner, including the choice of set $J$ in residual inner products of the form \eqref{eq:inner_product_res_J} and the various degrees of freedom adjustments visible in the denominators of \eqref{eq:risk-decomp-est-1}-\eqref{eq:risk-decomp-est-2}.

\paragraph{The \sub estimator.\hspace{-0.5em}}
Define $V_{ii} = 1 - \frac{\partial \bx_i^\top \hbbeta}{\partial y_i}$ as well as $\tilde V_{ii} = 1- \frac{\partial \bx_i^\top \tbbeta}{\partial y_i}$ for all $i\in[n]$.
The degrees of freedom for the individual estimators $\hbeta$ and $\tbeta$ are given by $\df=|I| - \sum_{i\in I} V_{ii}$ and $\tdf=|\tilde I| - \sum_{i\in\tilde I}\tilde V_{ii}$, respectively.
By variants of Stein's formula with respect to $(\bX,\bepsilon)$ formally stated in \Cref{lem:psi_bound,lem:trB_df}, the inner product \eqref{eq:inner_product_res_J} is well approximated as follows,
\begin{equation}
\begin{aligned}
& \sum_{i\in J} \Bigl(y_i-\bx_i^\top\hbbeta\Bigr)\Bigl(y_i-\bx_i^\top\tbbeta\Bigr)
\\ &\approx
-\Bigl(\frac{\tdf}{|\tilde I| - \tdf} \Bigr)
\sum_{i\in J\cap\tilde I} \Bigl(y_i-\bx_i^\top\hbbeta\Bigr)\Bigl(y_i-\bx_i^\top\tbbeta\Bigr)
+
\Bigl(\sum_{i\in J}V_{ii}\Bigr)
\Bigl(\sigma^2+(\hbbeta-\bbeta_0)^\top\bSigma(\tbbeta-\bbeta_0)\Bigr).
\end{aligned}\label{approx_J}
\end{equation}
The right-most term is the unknown quantity of interest, $(\sigma^2+(\hbbeta-\bbeta_0)^\top\bSigma(\tbbeta-\bbeta_0))$, that appears in the risk decomposition \eqref{eq:risk-decomposition}.
This motivates rearranging \eqref{approx_J} as 
\begin{equation}
\begin{aligned}
&\Bigl(\sum_{i\in J}V_{ii}\Bigr)
\Bigl(\sigma^2+(\hbbeta-\bbeta_0)^\top\bSigma(\tbbeta-\bbeta_0)\Bigr)\\
\approx &
\sum_{i\in J} \Bigl(y_i-\bx_i^\top\hbbeta\Bigr)\Bigl(y_i-\bx_i^\top\tbbeta\Bigr)
+ 
\Bigl(\frac{\tdf}{|\tilde I| - \tdf} \Bigr)
\sum_{i\in J\cap\tilde I} \Bigl(y_i-\bx_i^\top\hbbeta\Bigr)\Bigl(y_i-\bx_i^\top\tbbeta\Bigr).
\end{aligned}
\label{approx_J-rearrange}
\end{equation}
Now, to simplify the second line in \eqref{approx_J-rearrange}, we consider  $J\subset[n]$ such that $J = J\cap \tilde I$ (e.g., $J=I\cap \tilde I$ or $J=\tilde I$), so that the second line simplifies to $\bigl(\frac{|\tilde I|}{|\tilde I| - \tdf} \bigr)\sum_{i\in J} \bigl(y_i-\bx_i^\top\hbbeta\bigr)\bigl(y_i-\bx_i^\top\tbbeta\bigr)$. 
Dividing both sides by $(\sum_{i\in J} V_{ii})$ then gives an estimator of the prediction risk. 
Specifically, we first consider $J=I\cap \tilde I$. If given $(|\tilde I\cap I|,I)$ the intersection $\tilde I \cap I$ is uniformly distributed over subsets of $I$ with cardinality $|\tilde I\cap I|$, then  \Cref{lemma:sample_without_rep} (on the concentration of sampling without replacement) gives
$$
\frac{1}{|I\cap \tilde I|}\sum_{i\in I\cap \tilde I} V_{ii} \approx \frac{1}{|I|} \sum_{i\in I} V_{ii}
= 1-\frac{\df}{|I|}
$$
to approximate the factor in front of $(\sigma^2+(\hbbeta-\bbeta_0)^\top\bSigma(\tbbeta-\bbeta_0))$.
Combining these approximations with $J=I\cap \tilde I$ yields
$$
|I\cap \tilde I| \Bigl(1-\df/|I|\Bigr)
\Bigl(\sigma^2+(\hbbeta-\bbeta_0)^\top\bSigma(\tbbeta-\bbeta_0)\Bigr)
\approx 
\frac{1}{1 - \tdf/|\tilde I|}
\sum_{I\cap \tilde I}
\Bigl(y_i-\bx_i^\top\hbbeta\Bigr)\Bigl(y_i-\bx_i^\top\tbbeta\Bigr),
$$
which gives rise to the definition of overlapping estimator \eqref{eq:risk-decomp-est-1}.
The superscript ``\sub'' in \eqref{eq:risk-decomp-est-1} indicates that only overlapping observations are used in the residual inner product.
These successive approximations motivate the definition of $\Rem_1,\Rem_2,\Rem_3$ in \eqref{eq:ovlp_Rem1_2_3} of the formal proof of \Cref{th:sub_full} for the ``\sub'' estimator.

\paragraph{The \full estimator.\hspace{-0.5em}}
When the subsample size is small, the choice of $J=I\cap \tilde I$ may not be desirable because it only uses residuals of the overlapping data, and $J$ can be an empty set when $I$ and $\tilde I$ are disjoint. 
Observations that are neither in $I$ nor in $\tilde I$, i.e., $i\in [n]\setminus (I\cap \tilde I)$ would be particularity useful since they produce exactly unbiased estimates, with $\EE[(y_i-\bx_i^\top\hbbeta)(y_i-\bx_i^\top\tbbeta)\mid (\hbbeta,\tbbeta)]=\sigma^2+(\hbbeta-\bbeta_0)^\top\bSigma(\tbbeta-\bbeta_0)$ for $i\notin I\cap \tilde I$ by independence.
Ideally, in order to reduce variance, we would leverage a residual inner product over all $[n]$ samples. This motivates us to consider the approximation \eqref{approx_J} with $J=[n]$ and $J=\tilde I$.
By choosing a specific weighted combination of these two approximations that make all residual inner product over $\tilde I$ cancel out, we are left with an approximation between the residual inner product over $[n]$ and the target of interest, namely,
$$\sum_{i\in [n]} \Bigl(y_i-\bx_i^\top\hbbeta\Bigr)\Bigl(y_i-\bx_i^\top\tbbeta\Bigr)
\approx
\Bigl[
    -
\frac{\tdf}{|\tilde I|}
\Bigl(\sum_{i\in \tilde I}V_{ii}\Bigr)
+
\Bigl(\sum_{i\in [n]}V_{ii}\Bigr)
\Bigr]
\Bigl(\sigma^2+(\hbbeta-\bbeta_0)^\top\bSigma(\tbbeta-\bbeta_0)\Bigr).
$$
If $I$ and $\tilde I$ are two different independent subsets $I_m,I_\ell$ (corresponding to an off-diagonal term $m\le \ell$ in \eqref{eq:risk-decomposition}), then by the independence of $\tilde I$ and $(V_{ii})_{i\in[n]}$ the concentration of sampling without replacement in \Cref{lemma:sample_without_rep} further gives $\frac{1}{|\tilde I|}\sum_{i\in\tilde I}V_{ii}\approx \frac{1}{n}\sum_{i=1}^n V_{ii} = 1-\df/n$ so that the square bracket factor on the right-hand side simplifies to $n(-\tdf/n(1-\df/n)+ 1-\df/n)=n(1-\tdf/n)(1-\df/n)$.
This is only valid if $I$ and $\tilde I$ are independent.
On the other hand, if $I=\tilde I$,
(corresponding to a diagonal term $m =\ell$ in \eqref{eq:risk-decomposition}),
then $\frac{1}{|\tilde I|}\sum_{i\in \tilde I} V_{ii} = 1 - \df/|I|$
and the square bracket above equals $n-2\df + \df^2/k$.
This means that a different correction needs to be used for the diagonal and off-diagonal terms. 
These approximations motivate the definition of the ``\full'' estimator in \eqref{eq:risk-decomp-est-2}.

\paragraph{The CGCV estimator.\hspace{-0.5em}}
Next, we provide some heuristics behind the derivation of the corrected GCV from the full estimator.
The key ingredients are the concentration of $|I_m\cap I_\ell|$ and $\df_m$ below (formally stated in \Cref{lem:concentration-intersection} and \Cref{lem:concentration_df}, respectively):
\[
    | I_m \cap I_\ell | \approx \frac{|I_m||I_\ell|}{n} \quad\text{ for all }
    m \neq \ell, 
    \qquad
    \df_m \approx \tdf_M = \frac{1}{M} \sum_{m' \in [M]} \df_{m'} \quad\text{ for all } m.
\]
By those concentration properties, we observe that $d_{m,\ell}^\full$, the denominator in the full-estimator defined in \eqref{eq:d_sub_and_full}, has the following approximation depending on whether $m=\ell$ or not:
\begin{align*}
    \frac{d_{m,\ell}^\full}{n} 
    \approx 
    \begin{dcases}
        1 - 2 \, \frac{\tdf_M}{n} + \bigg(\frac{\tdf_M}{n}\bigg)^2 & \text{ if } m \neq \ell \\
         1 - 2 \, \frac{\tdf_M}{n} + \bigg(\frac{\tdf_M}{n}\bigg)^2 \frac{n}{|I_m|} &  \text{ if } m=\ell
    \end{dcases}
    = \bigg(1-\frac{\tdf_M}{n}\bigg)^2 + \ind_{\{m=\ell\}} \bigg(\frac{n}{|I_m|}-1\bigg) \bigg(\frac{\tdf_M}{n}\bigg)^2. 
\end{align*}
Combining the above display and $\hat{R}_{m, \ell}^\full\approx R_{m,\ell}$ by \eqref{eq:sub_full_moment_ineq}, we have 
\begin{align*}
    \frac{\|\by - \bX\tbeta_M\|_2^2}{n} &= \frac{1}{M^2} \sum_{m, \ell}\frac{(\by - \bX\hbeta_m)^\top(\by-\bX\hbeta_\ell)}{n} = \frac{1}{M^2}\sum_{m, \ell} \frac{d_{m,\ell}^\full}{n} \cdot \hat{R}_{m, \ell}^\full\\
    &\approx \bigg(1-\frac{\tdf_M}{n}\bigg)^2 R_M + \bigg(\frac{\tdf_M}{n}\bigg)^2 \frac{1}{M^2} \sum_{m=1}^M  \bigg(\frac{n}{|I_m|}-1\bigg) R_{m,m}.
\end{align*}
Dividing by $(1-\tdf_M/n)^2$ and plugging in $\hat{R}_{m, m}^\sub$ or $\hat {R}_{m,m}^\full$ to estimate $R_{m,m}$ inside the rightmost sum, we obtain the corrected GCV estimators given in \Cref{def:cgcv-proposal}.

\subsection{Numerical illustrations}
\label{sec:numerical-illustrations-gaussian}

In this section, we corroborate our theoretical results with numerical experiments on synthetic data, showing the validity of our proposed corrected GCV ($\tR_M^{\cgcv,\full}$ in \eqref{eq:cgcv-proposal}).
The purpose is to empirically show our proposal, the corrected GCV estimator (shown as CGCV), accurately estimates the true risk $R_M$ (shown as Risk), while the ordinary GCV in \eqref{eq:fgcv-naive} (shown as GCV) is a poor estimator of the true risk $R_M$. 

\begin{figure}[!t]
    \centering
    \includegraphics[width=0.95\textwidth]%
    {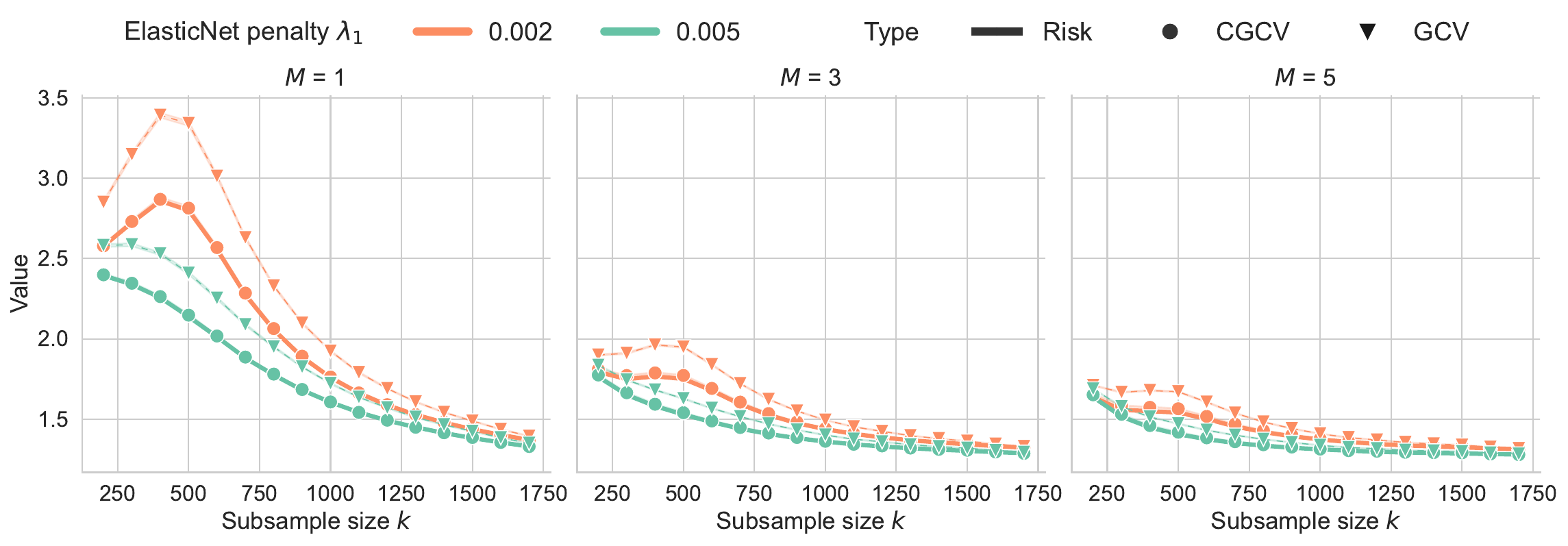}
    \caption{
    \textbf{CGCV is consistent for finite ensembles across different subsample sizes.}
    Plot of risks versus the subsample size $k$ for Elastic Net ensemble with different $\lambda$ and $M$ and fixed $\lambda_2 = 0.01$. 
    }\label{fig:ridge-k}
\end{figure}

\paragraph{Linear models with Gaussian designs.\hspace{-0.5em}}
We first describe the simulation setting. 
We set $(n, p) =(2000, 500)$ and consider different combinations of $(k, M)$
and tuning parameters $(\lambda,\lambda_1,\lambda_2)$ for the three penalty functions in \Cref{tab:examples}.
We consider a linear model satisfying \Cref{assu:Gaussian-feature,assu:Gaussian-response} with noise variance $\sigma^2=1$. To be specific, we consider a non-isotropic covariance matrix $\bSigma$, whose diagonal entries are evenly spaced in $[0.1, 10]$ and other entries are 0. We generate the true regression coefficient $\bbeta_0$ as follows: generate $\bb_0\in \RR^p$ with its first $p-100$ rows sampled from \iid $\cN(0,1)$, and set the remaining rows to $0$. 
We then scale $\bb_0$ to obtain the coefficient $\bbeta_0$ by $\bbeta_0= \bb_0(\sigma^2 / \bb_0^\top\bSigma \bb_0)^{1/2}$ so that the signal-to-noise ratio $\bbeta_0^\top\bSigma \bbeta_0 /\sigma^2$ is $1$. 
For the ridge and lasso ensembles, the tuning parameter $\lambda$ of the corresponding penalty in \Cref{tab:examples} is fixed across all the subsample estimates $\hbeta_m$. 
For the elastic net ensemble, the tuning parameter $\lambda_2$ is fixed as $0.01$ and the same sequence of $\lambda_1$ is used across all subsample estimates $\hbeta_m$. 
For each simulation setting, we perform $1000$ dataset repetitions and report the average of relevant risks.

We present the simulation results of the elastic net ensemble in \Cref{fig:ridge-k,fig:M}, and leave the results of the ridge and lasso ensembles to the \Cref{app:experiments-gaussian} of the supplement. 
\Cref{fig:ridge-k} shows the risk curve and its estimators as a function of the ensemble size $M$. 
We observe a clear gap between the curves of the GCV estimator and the true risk for different settings of $(M, k, \lambda_1)$, especially for small subsample size $k$ or small ensemble size $M$. 
In stark contrast, our proposed corrected GCV aligns closely with the true risk curve. 
\Cref{fig:M} again confirms that our proposal outperforms the GCV in estimating the true risk across different values of $M$ and different combinations of $k$ and $\lambda_1$. 
We provide more simulation results for the ridge and lasso ensembles in \Cref{app:experiments-gaussian}.
 
\begin{figure}[!ht]
    \centering
    \includegraphics[width=0.95\textwidth]%
    {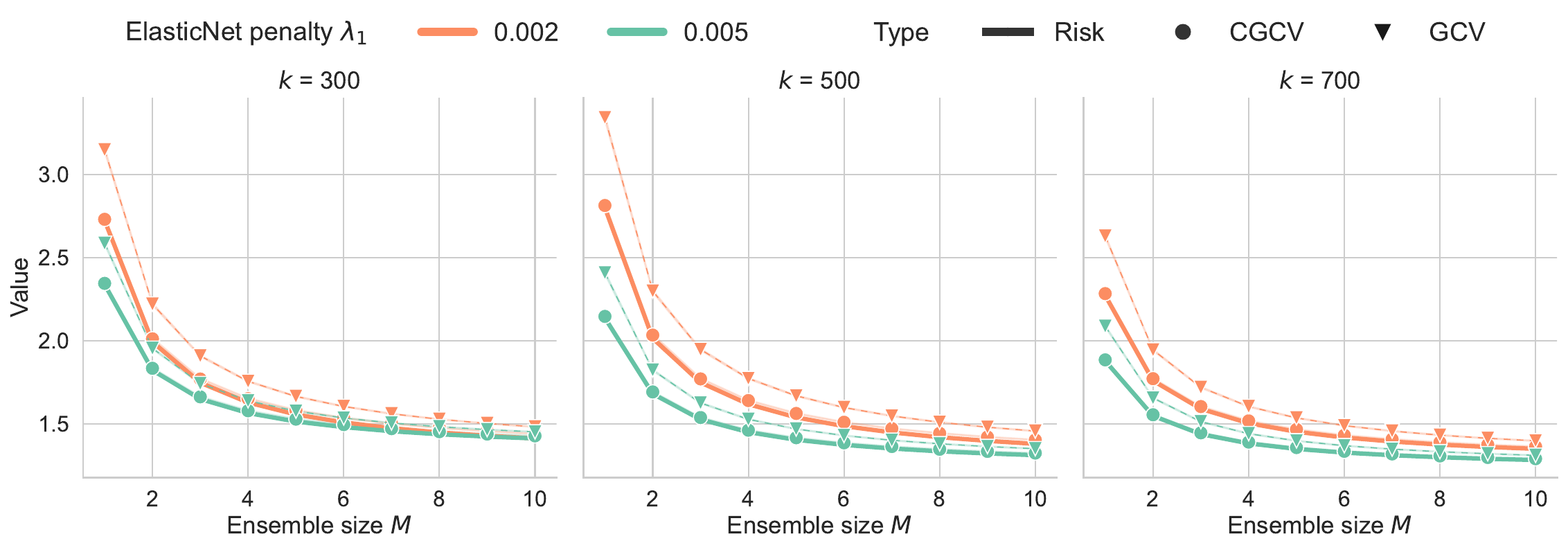}
    \caption{
    \textbf{GCV gets closer to the risk as ensemble size increases across different subsample sizes.}
    Plot of risks versus the Elastic Net ensemble size $M$ for ridge ensemble with different $\lambda$ and $k$ and fixed $\lambda_2 = 0.01$.
    }
    \label{fig:M}
\end{figure}

\paragraph{Linear models with non-Gaussian designs.\hspace{-0.5em}}
To demonstrate the robustness of our results beyond Gaussian distribution, 
we also consider scenarios where $\bX$ is not Gaussian distributed but is sub-Gaussian, for example, the Rademacher and uniform distributions. 
The simulation results for these non-Gaussian features closely resemble those for Gaussian features presented in the current section, confirming that the proposed CGCV accurately estimates the actual risk, whereas GCV exhibits a non-ignorable bias. 
We provide the simulation results for non-Gaussian features in \Cref{sec:numerical-non-gaussian}.

\section{Asymptotic analysis for ridge ensembles}
\label{sec:asm-consistency}

For ridge ensembles, the GCV estimator \eqref{eq:gcv-naive} for ridge ensembles can be inconsistent when the ensemble size is greater than $1$, as mentioned in \Cref{subsec:naive_gcv}. 
This inconsistency still exists when features are sampled from non-Gaussian distributions \citep{du2023subsample}.
In this section, we will show that the result of the previous section can be generalized to a wide class of data distributions when focusing on ridge predictors and using standard results from random matrix theory.

Recall the base estimator $\hbeta_m$ defined in \eqref{eq:def-hbeta} for $m=1,\ldots,M$.
If we use the ridge penalty with regularization parameter $\lambda>0$, then the base estimator reduces to the \emph{ridge} estimator fitted on $\cD_{I_m}$:
\begin{equation}
    \hbeta_{m,\lambda}
    = \argmin\limits_{\bbeta\in\RR^p}
    \sum_{i \in I_m} (y_i  - \bx_i^\top \bbeta)^2 / k
    + \lambda \| \bbeta \|_2^2
    = (\bX^{\top} \bL_{I_m} \bX / k  + \lambda\bI_p)^{-1}{\bX^{\top} \bL_{I_m} \by}/{k} .\label{eq:ingredient-estimator}
\end{equation}
Taking $\lambda\rightarrow0^+$, $\hbeta_{m,0}:=(\bX^{\top} \bL_{I_m} \bX/|I_m|)^{+}\bX^{\top} \bL_{I_m} \by/|I_m|$ becomes the so-called \emph{ridgeless} estimator, or minimum-norm interpolator, where $\bA^+$ denotes the Moore-Penrose inverse of matrix $\bA$.

To prepare for our upcoming results, we impose the following assumptions on the responses, features, and subsample aspect ratios.

\begin{assumption}[Response structure]\label{asm:model}
    The response $y$ satisfies $\EE[y]=0$ and $\EE[y^{4+\delta}] \leq C_0$ for some $\delta>0$ and $C_0>0$.
\end{assumption}
    
\begin{assumption}
[Feature structure]
\label{asm:feat}
    The feature vector decomposes as $\bx = \bSigma^{1/2}\bz$, where $\bz\in\RR^{p}$ contains i.i.d.\ entries with mean $0$, variance $1$, bounded moments of order $4+\delta$ for some $\delta > 0$, and $\bSigma \in \RR^{p \times p}$ is deterministic and symmetric with eigenvalues uniformly bounded between $r_{\min}>0$ and $r_{\max}<\infty$.
\end{assumption}

We make several remarks about the data assumptions in our analysis. 
First, we do not impose any specific model assumptions between the response variable $y$ and the feature vector $\bx$. 
Thus, our results are model-free, relying solely on the bounded moment assumption stated in \Cref{asm:model}, which ensures the generality of our subsequent findings.
Although we assume zero mean for $y$, this is purely for simplicity, as centering can always be achieved by subtracting the sample mean in practice. 
Furthermore, the condition of bounded moments can also be satisfied by imposing a stronger distributional assumption, such as sub-Gaussianity.

\begin{assumption}
[Proportional asymptotics]
\label{asm:prop-asym}
    The sample size $n$, subsample size $k$, and feature dimension $p$ tend to infinity such that $p/n\rightarrow\phi\in(0,\infty)$ and $p/k\rightarrow\psi\in[\phi,\infty]$.
\end{assumption}

Even though we assume that both the data aspect ratio $p/n$ and the subsample aspect ratio $p/k$ converge to fixed constants, this is for simplicity of exposition and proofs, and is not essential.
It suffices only to require that they scale proportionally such that $0 < \liminf p/n \le \liminf p/n \leq \liminf p/k < \infty$ as, by compactness, we can always extract a subsequence of regression problems in which $p/n$ and $p/k$ have a finite limit.

Both \Cref{asm:feat}, concerning the feature vector, and \Cref{asm:prop-asym}, clarifying the proportional asymptotics regime, are assumptions commonly employed in the study of random matrix theory \citep{bai2010spectral,el2010spectrum} and the analysis of ridge and ridgeless regression \citep{karoui2011geometric,karoui_2013,dobriban_wager_2018,hastie2022surprises}.

\begin{remark}[Extreme points in proportional asymptotics]
    In \Cref{asm:prop-asym}, we allow the limiting data aspect ratio $p/n$ and subsample aspect ratio $p/k$ to be in $(0,\infty)$ and $[\phi,\infty]$, respectively.
    In the extreme case where $p / k \to \infty$, the prediction risks of the ridge ensembles admit simple asymptotic formulas.
    More specifically, both the prediction risk and the estimates converge to the null risk $\EE[y^2]$, the risk of a full predictor that always outputs zero.
\end{remark}

\subsection{Guarantees for intermediate risk estimators}

Since the ridge predictor is a linear smoother, its degrees of freedom $\df_{m}$ defined in \eqref{eq:def-df} is equivalent to the trace of its smoothing matrix $\bS_m = \bX_{I_m}(\bX_{I_m}^{\top}\bX_{I_m}/k + \lambda\bI_p)^{-1}\bX_{I_m}^{\top}/k$; namely, $\df_m = \trace[\bS_m]$.
For ridge regression, the \ovlp-estimator $\hR_{m,\ell,\lambda}^{\sub}$ and the full-estimator $\hR_{m,\ell,\lambda}^{\full}$ as defined in \eqref{eq:risk-decomp-est-1} and \eqref{eq:risk-decomp-est-2} are respectively given explicitly as follows:
\begin{align*}
    \hR_{m,\ell,\lambda}^{\sub}
    &=
    \ddfrac{
    (\by - \bX\hbeta_{m,\lambda})^{\top}\bL_{m}\bL_{\ell}(\by - \bX\hbeta_{\ell,\lambda}) / |I_m\cap I_{\ell}|}{    
     1  - \frac{\tr[\bS_m]}{|I_m|} - \frac{\tr[\bS_\ell]}{|I_{\ell}|} + \frac{\tr[\bS_m]}{|I_m|} \frac{\tr[\bS_\ell]}{|I_{\ell}|}
    }, \\
    \hR_{m,\ell,\lambda}^{\full}
    &=
    \ddfrac{
    (\by - \bX \hbeta_{m,\lambda})^\top (\by - \bX \hbeta_{\ell,\lambda}) / n}{1 - \frac{\tr[\bS_m]}{n} - \frac{\tr[\bS_\ell]}{n} 
    + \frac{\tr[\bS_m]}{|I_m|}
    \frac{\tr[\bS_\ell]}{|I_{\ell}|}
    \cdot \frac{| I_m \cap I_\ell |}{n}}.
\end{align*}
Here, we use the subscript $\lambda$ to indicate the dependence of the estimators for the risk component on the ridge penalty parameter $\lambda$.
Similarly, we write $\tR_{M,\lambda}^{\sub} $, $\tR^{\full}_{M,\lambda}$, and $R_{M,\lambda}$ for the two risk estimators and the risk, respectively.
The following theorem shows that both the two estimators are pointwise consistent for the risk component $R_{m,\ell,\lambda}$.

\begin{restatable}[Pointwise consistency of intermediate estimators in $\lambda$]{theorem}{ThmRiskEst}\label{thm:risk-est}
    Under \Cref{assu:sampling,asm:model,asm:feat,asm:prop-asym} with $\sqrt{n}/k=\cO(1)$ for the \ovlp-estimator, for $\lambda\geq 0$, for any $M \in \NN$ and $m, l \in [M]$, it holds that
    $|\hR_{m,\ell,\lambda}^{\sub} - R_{m, \ell,\lambda}| \asto 0$
    and $|\hR_{m,\ell,\lambda}^{\full} - R_{m, \ell,\lambda}| \asto 0$.
    Consequently, it holds that 
    $| \tR_{M,\lambda}^{\sub} - R_{M,\lambda} | \asto 0$ 
    and 
    $| \tR^{\full}_{M,\lambda} - R_{M,\lambda} | \asto 0$.
\end{restatable}

For the first estimator to work in the extreme case when $\psi = \infty$, we also need $\sqrt{n}/k=\cO(1)$.
We note that this is not required for the full-estimator precisely because the full-estimator uses both subsample observations and out-of-subsample observations. So, even if the number of overlapping subsample observations is small, the out-of-subsample observations are large, and the estimator is able to track the risk.  
The regime when $k$ is small is the regime where the full-estimator is substantially better than the \ovlp-estimator, as we have seen in \Cref{fig:k}.

One natural question that remains for ridge predictors is how to select the ridge penalty $\lambda$.
Building on the previous results, we are able to show a stronger notion of the consistency of the two proposed estimators over $\lambda\in\Lambda:=[0,\infty]$.
When $\lambda=\infty$, the ensemble ridge predictor reduces to a null predictor that always outputs zero.
Under the above assumptions, uniform consistency is established through the following theorem.
\begin{restatable}[Uniform consistency of intermediate estimators in $\lambda$]{theorem}{ThmUniformConsistencyLambda}\label{thm:unifrom-consistency-lambda}
    Under same conditions in \Cref{thm:risk-est}, when $\psi\neq 1$, it holds that $\sup_{\lambda\in\Lambda}|\tR_{M,\lambda}^{\sub} -R_{M,\lambda}| \asto  0$ and $\sup_{\lambda\in\Lambda}|\tR_{M,\lambda}^{\full} -R_{M,\lambda}| \asto  0$.
\end{restatable}
The restriction on $\psi$ is because both the risk and the risk estimate diverge when $\psi = 1$ for the ridgeless predictor when $\lambda=0$. 
This divergence phenomenon has been studied in several recent works; see, e.g., \cite{hastie2022surprises}. 

\subsection{Guarantees for corrected GCV}

The average degrees of freedom \eqref{eq:def-tdf} reduces to $\tdf_M = M^{-1}\sum_{m=1}^M \tr(\bS_m)$.
The corrected GCV from \eqref{eq:cgcv-proposal} and \eqref{eq:cgcv-proposal-sub} for ridge regression are given for $\est\in\{\sub,\full\}$ by:
\begin{equation}
    \hR_{M,\lambda}^{\cgcv,\est}:=
    \tR^{\gcv}_{M,\lambda}
        - \frac{1}{M} \left\{ \frac{(M^{-1}\sum_{m=1}^M \tr(\bS_m)/n)^2}{(1-M^{-1}\sum_{m=1}^M \tr(\bS_m)/n)^2} 
        \bigg(\frac{n}{k} - 1\bigg)
        \frac{1}{M}\sum_{m=1}^M \hR_{m,m}^\est \right\},\label{eq:asym-corr-term}
\end{equation}
where $\tR^{\gcv}_{M,\lambda}=  (1-M^{-1}\sum_{m=1}^M \tr(\bS_m)/n)^{-2} \|\by-\bX \tbeta_{M}\|_2^2/n$ is the GCV estimator and the correction term $\hR_{m,m, \lambda}^\est$ is given, for $\est\in\{\sub,\full\}$, by:
\begin{align*}
    \hat{R}_{m, m, \lambda}^\sub
    &=
        \frac{\|\by_{I_m} - \bX_{I_m} \hat{\bbeta}_{m,\lambda}\|_2^2/k}{(1-M^{-1}\sum_{m=1}^M \tr(\bS_m)/k)^2},\\
    \hat{R}_{m, m, \lambda}^\full
    &=
        \frac{\|\by-\bX\hbeta_{m,\lambda}\|_2^2/n}{(1-M^{-1}\sum_{m=1}^M \tr(\bS_m)/n)^2 + (M^{-1}\sum_{m=1}^M \tr(\bS_m)/n)^2 \cdot ({n}/{k} - 1)}.
\end{align*}

\begin{restatable}[Uniform consistency of corrected GCV in $\lambda$]{theorem}{PropCorrectGCV}\label{prop:correct-gcv}
    Under the same conditions in \Cref{thm:risk-est},
    it holds that
    $| \hR_{M,\lambda}^{\cgcv,\sub} - R_{M,\lambda} | \asto 0$ 
    and 
    $| \hR_{M,\lambda}^{\cgcv,\full} - R_{M,\lambda} | \asto 0$.
    Moreover, when $\psi\neq 1$, it holds that 
    $\sup_{\lambda\in\Lambda}|\hR_{M,\lambda}^{\cgcv,\sub} - R_{M,\lambda}| \asto  0$
    and
    $\sup_{\lambda\in\Lambda}|\hR_{M,\lambda}^{\cgcv,\full} - R_{M,\lambda}| \asto  0$.
\end{restatable}

\Cref{prop:correct-gcv} specializes the results of \Cref{th:correction_gcv} for a general convex penalty to the ridge predictor.
In particular, the correction term decreases as $1/M$.
On the other hand, compared to \Cref{th:correction_gcv}, it also extends the corrected GCV to ridgeless predictors.
Some remarks are as follows.

\begin{remark}[Consistency of GCV under general models]
    The GCV estimator $\tR^{\gcv}_{M,\lambda}$ is shown to be consistent for estimating the prediction risk $R_M$ of the ridge predictors \citep{patil2021uniform} (when $M=1$ or $k=n$) and in the infinite ensemble \citep{du2023subsample}.
    Therein, the infinite ensemble is defined by letting the ensemble size $M$ tend to infinity for any given sample size $n$ and the feature dimension $p$.
    Since the correction term \eqref{eq:asym-corr-term} vanishes as $M$ tends to infinity, \Cref{prop:correct-gcv} above also indirectly shows that the corrected GCV estimator is consistent in the infinite ensemble under more general data models.
    On the other hand, the correction term in \eqref{eq:asym-corr-term} implies that the GCV estimator $\tR^{\gcv}_{M,\lambda}$ is generally inconsistent for finite ensemble sizes when $n>k$ and $\tdf_M > 0$.
    In the case of ridge regression, the latter condition is generally satisfied because for $\lambda \in [0, \infty)$ because
    \[\tdf_M=M^{-1}\sum_{m=1}^M\tr(\bS_m) = M^{-1} \sum_{m=1}^{M} \tr[\bX_{I_m}^{\top} \bX_{I_m}/k (\bX_{I_m}^{\top} \bX_{I_m}/k+\lambda \bI_p )^{-1}]
    > 0,\]
    unless all the singular values of $\bX_{I_m}$ are zero for all $m\in[M]$.
    Compared to \citet[Proposition 3]{du2023subsample}, which shows the inconsistency of $\tR^{\gcv}_{M,\lambda}$ for $M=2$ and $\lambda=0$ under a well-specified linear model, our result is more general because it allows for any ridge parameter $\lambda \ge 0$, any ensemble size $M \ge 2$, and an arbitrary response model.
\end{remark}

\begin{remark}[Model selection]
    The ensemble ridge predictors involve three key parameters: the ridge penalty parameter $\lambda$, the subsample size $k$, and the ensemble size $M$.
    Questions related to the selection of the model of ridge ensembles over the subsample size $k$ and the ensemble size $M$ have already been discussed by \citet{patil2022bagging,du2023extrapolated}.
    Tuning over $M$ is relatively easy because the risk is decreasing in $M$ \citep{patil2022bagging}.
    For tuning over $\lambda$, \Cref{prop:correct-gcv} implies that the tuned risk using the corrected GCV estimator converges to the optimal risk in a grid of ridge penalties used to fit the ensemble.
    In other words, the optimality of the data-dependent model selection can be guaranteed by the corrected GCV estimators.
\end{remark}

\begin{remark}[Non-asymptotic versus asymptotic analyses]
    When restricted to ridge estimators, the asymptotic analysis presented in the current section reveals some benefits compared to the finite-sample analysis in \Cref{sec:finite-sample-analysis}.
    First, the asymptotic analysis in the current section relaxes the assumptions on the data-generating process.
    More specifically, no explicit response-feature model is needed, except for the bounded moment assumptions and the feature structure assumed in \Cref{asm:model,asm:feat}.
    Second, the consistency provided in the current section does not require strong convexity as in \Cref{assu:penalty}.
    It also applies to the ridgeless ensembles when $\lambda = 0 $, in which case the base predictor reduces to the minimum $\ell_2$-norm least squares.
    Last but not least, the asymptotic analysis enables the uniformity over $\lambda \in [0,\infty]$.
\end{remark}

\subsection{Proof outlines}

There are three key steps involved in proving the uniform consistency (in $\lambda$) of the intermediate and corrected GCV estimator in \Cref{thm:unifrom-consistency-lambda,prop:correct-gcv}, respectively, for ridge ensembles.

\begin{enumerate}[leftmargin=7mm]
    \item
    Deriving the asymptotic limit of the prediction risk:
    The first step involves deriving asymptotic equivalents for the risk of the ridge ensemble predictor.
    We build upon prior results on the risk analysis of ridge ensembles in \cite{du2023subsample,patil2023generalized} and derive these asymptotic risk equivalences, as done in \Cref{lem:risk}.

    \item
    Deriving the asymptotic limit of the intermediate and CGCV estimates:
    The second step involves deriving asymptotic equivalents for each of the proposed intermediate and corrected GCV estimators.
    Since each of these estimators is a ratio of terms that involve a version of training error and a denominator correction, we obtain each of these asymptotics separately.
    This is done in \Cref{lem:num-est-1,lem:den-est-1,lem:num-est-2,lem:den-est-2}.
    
    \item
    Establishing the pointwise consistency in $\lambda$ by matching the two limits and then lifting to uniform convergence in $\lambda$:
    In the last step, we show that the asymptotic limits obtained in the first two steps match with each other, which shows the pointwise consistency in $\lambda$ (\Cref{thm:risk-est}).  
    The matching uses the analytic properties of parameters related to certain fixed-point equations.
    To establish uniform convergence, we lift the pointwise consistency by showing that the family of functions under consideration forms an equicontinuous family and appeals to a stochastic version of the Arzela-Ascoli theorem.
    This is done in \Cref{lem:boundary}. 
\end{enumerate}

Throughout the proofs, the notion of asymptotic equivalence plays a crucial role (see \Cref{sec:preparation-asm-consistency} for more details).
We refer the reader to \cite{patil2022bagging,patil2023generalized} for background and further calculus of asymptotic equivalents.
In our proofs, we build on these equivalents to obtain equivalents of various bias and variance components of the intermediate and corrected GCV estimators (see \Cref{app:subsec:asym-technical-lemmas} for more details).
Finally, all our results are model-free.
A component of this analysis involves an intermediate step showing quadratic concentration with uncorrelated components in \Cref{lem:quad-uncorr} that extends the previous results of \citet[Lemma A.16]{bartlett2021deep} and \citet[Lemma D.3]{patil2023generalized}, which may be of independent interest.

\subsection{Numerical illustrations: Non-linear models with non-Gaussian designs}
\label{subsec:asym-numerical}

To numerically evaluate our proposed corrected GCV estimator, we generate data from a nonlinear model $y=\bx^{\top}\bbeta_0 + (\|\bx\|_2^2/p - 1 ) + \epsilon$, where $\bx = \bSigma^{-1/2}\bz$ with $z_1,\ldots,z_p\overset{\text{i.i.d.}}{\sim}t_5$ follows a $t$-distribution with degrees of freedom 5 and $\bSigma=\bSigma_{\textsc{ar1},\rho=0.25}=(\rho^{|i-j|})_{i,j\in[p]}$ is the covariance matrix of the AR(1) process, $\bbeta_0$ is the average of the top-5 eigenvectors of $\bSigma$, and $\epsilon\sim\cN(0,1)$.
This setup has a linearized signal-to-noise ratio of 1.67.
We set the sample size $n=6000$ and the feature dimension $p=1200$.

Our first experiment is to compare the naive GCV estimator $\tR_M^{\gcv}$ and the corrected GCV estimator $\tR_M^{\cgcv,\full}$ as a function of the subsample size $k$, or equivalently, the subsample aspect ratio $p/k$.
In \Cref{fig:ridge-lasso-psi}, we visualize the results for both the ridge and lasso predictors with different ensemble sizes $M$.
For ridge predictors, we see that the naive GCV estimator is consistent when $M=1$.
However, it is extremely unstable when the subsample aspect ratio $p/k$ is close to 1, which is also the interpolation threshold for the ridgeless predictor.
For the lasso predictors, the naive GCV estimates are not even close to the prediction risks when the subsample size is small, even for ensemble size $M=1$.
Nevertheless, the corrected GCV estimates are in line with the true prediction risks for different ensemble sizes $M$.
Although our theoretical results do not cover the lasso, the numerical results in the current subsection indicate the generality and robustness of our proposed corrected GCV estimator beyond ridge ensembles.
On the other hand, as shown in \Cref{fig:ridge-M}, the inconsistency of the naive GCV vanishes as the ensemble size $M$ tends to infinity.

\begin{figure}[!t]
    \centering
    \includegraphics[width=0.9\textwidth]{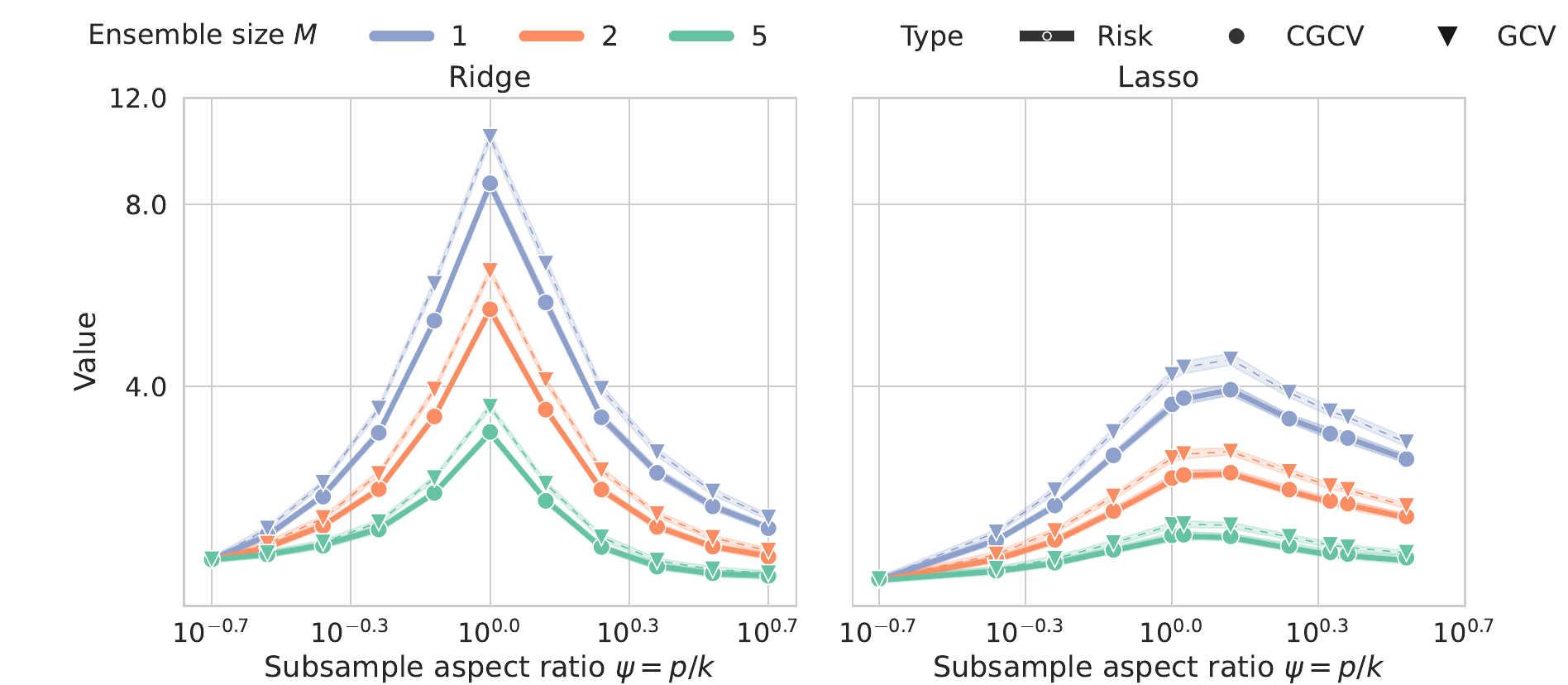}
    \caption{
    \textbf{CGCV is consistent for finite ensembles across different subsample aspect ratios.}
    The GCV estimates for the ridge and lasso ensembles with penalty $\lambda=10^{-2}$ and varying subsample size in a problem setup of \Cref{subsec:asym-numerical}, over 50 repetitions of the datasets.
    The left panel shows the results for ridge predictors.
    The right panel shows the results for the lasso predictors.
    }\label{fig:ridge-lasso-psi}
\end{figure}

Finally, we compare the two estimators as a function of the ridge penalty $\lambda$ in \Cref{fig:ridge-lambda}.
As we can see, for ensemble size $M=1$, both estimators can track the true prediction risk well in either underparameterized or overparameterized regimes.
However, when $M>1$, the naive GCV estimator does not provide a valid estimate for the true risk.
The inconsistency becomes more significant when either the subsample size $k$ is small relative to the feature dimension $p$ or the ridge penalty is close to zero.
Similar observations are also present for the lasso estimators; see \Cref{fig:lasso-lambda}.
On the other hand, our corrected GCV estimator is uniformly consistent in various values of $\lambda$ for all ensemble sizes $M$, as proved in \Cref{prop:correct-gcv}. 

\begin{figure}[!t]
    \centering
    \includegraphics[width=0.9\textwidth]{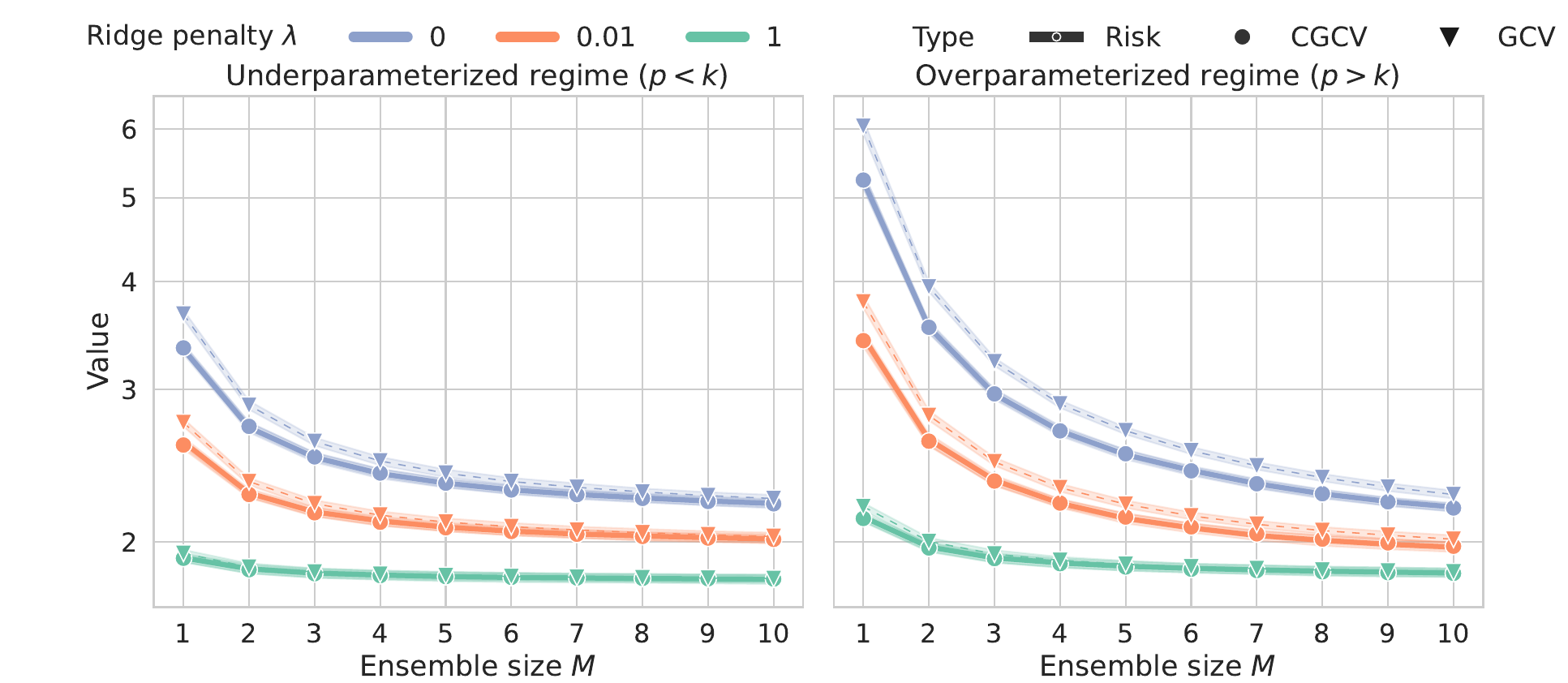}
    \caption{
    \textbf{GCV gets closer to the risk as ensemble size increases.}
    The GCV estimates for ridge ensemble with varying ensemble size in a problem setup of \Cref{subsec:asym-numerical} over 50 repetitions of the datasets.
    Here, the feature dimension is $p=1200$.
    The left panel shows the case when the subsample sizes are $k=2400$ (an underparameterized regime).
    The right panel shows the case when the subsample sizes are $k=800$ (an overparameterized regime).
    }\label{fig:ridge-M}
\end{figure}

\begin{figure}[!t]
    \centering
    \includegraphics[width=0.9\textwidth]{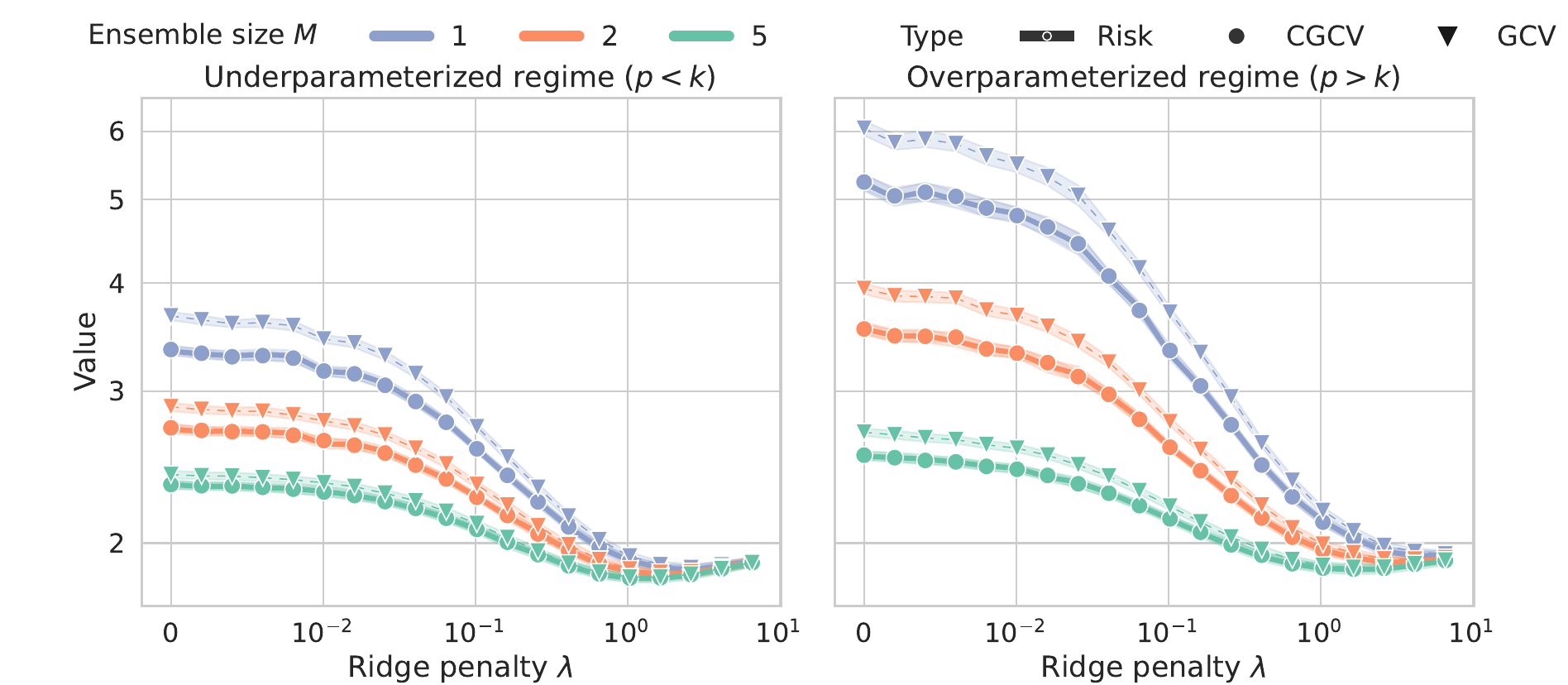}
    \caption{
    \textbf{CGCV is consistent for finite ensembles across different regularization levels.}
    The GCV estimates for ridge ensemble with varying ridge penalty $\lambda$ in a problem setup of \Cref{subsec:asym-numerical} over 50 repetitions of the datasets.
    Here, the feature dimension is $p=1200$.
    The left panel shows the case when the subsample sizes are $k=2400$ (an underparameterized regime).
    The right panel shows the case when the subsample sizes are $k=800$ (an overparameterized regime).
    }
    \label{fig:ridge-lambda}
\end{figure}

\section{Discussion}
\label{sec:discussion}

The primary goal of this paper has been to examine GCV for the risk estimation of ensembles of penalized estimators. 
Our main contribution is to show that ordinary GCV fails to be consistent for finite ensembles and to formulate a corrected GCV that is consistent for any finite ensemble.
The main high-level idea in constructing the correction is deriving term-by-term corrections for the ordinary GCV estimator, which subsequently leads to the form of corrected GCV. 
Assuming Gaussianity and linear model, we derived non-asymptotic bounds for our corrected GCV estimator that enjoys $\sqrt{n}$-consistency and provides explicit dependence on the subsample size $k$, ensemble size $M$, and the form and level of regularization penalty $\lambda$. 
These assumptions, although seemingly stringent, are amenable to a wide range of general convex penalized estimators. 
Moreover, for ensembles of ridge regression, we prove the asymptotic consistency of our estimator under very mild assumptions and, in particular, without assuming Gaussianity and linear response model.

As alluded to earlier, our intermediate risk estimators can incorporate several extensions and modifications, thereby further broadening their applicability.
First, we can use different subsample sizes for different ensemble estimators.
Second, we can use different types of penalties or different regularization parameters for different base estimators.
In such cases, the components of the \ovlp-estimator \eqref{eq:risk-decomp-est-1} and the full-estimator \eqref{eq:risk-decomp-est-2} are still consistent.
Note that additive bounds, as indicated in \Cref{th:sub_full,th:correction_gcv}, already accommodate this flexibility.
It is also possible to establish the multiplicative bounds assuming certain restrictions on the component risks.
Finally, provided that the average cardinality concentrates and the cardinality of the intersection is lower-bounded with high probability, a slightly general form of CGCV is expected to work.

Although the techniques of the present paper can be applied to
estimate the risk of hybrid ensembles (with the $I_m$ having different cardinalities and the $\hbeta_m$ having different penalty functions), it poses challenges when it comes to tuning the hyperparameters.
Tuning an ensemble model, especially when equipped with a plethora of hyperparameters, is computationally challenging. 
In such cases, a greedy strategy often employed is to minimize the risk for each subsample estimator separately. 
However, it is important to note that the optimal regularization parameter $\lambda$ for an individual ensemble component (i.e., for a fixed $m\in [M]$) may not be the best choice for the entire ensemble (i.e., after averaging over $m\in[M]$).
This computational complexity of tuning ensembles becomes especially challenging, as ensemble components employ different penalties or regularization amounts.
For this reason, we recommend training each individual estimate $\hbbeta_m$ with the same penalty and tuning parameter (say, $\lambda$) and selecting the tuning parameter $\lambda$ for the ensemble average $\tbeta_{M} = {M}^{-1}\sum_{m=1}^M \hbbeta_m$ in \eqref{eq:def-M-ensemble} that yields the smallest corrected GCV criterion. 
Furthermore, one can also combine CGCV with the extrapolated CV method of \citet{du2023extrapolated} to tune over $k$ and $M$.

Our work opens up several avenues for future investigation.
We discuss three of them below.

\paragraph{Beyond squared loss for training.\hspace{-0.5em}}
We have considered penalized estimators trained using squared loss in \eqref{eq:def-hbeta}. 
One may consider penalized estimators trained on different loss functions, such as the Huber loss or other robust loss functions, especially if the noise random variables $\epsilon_i$ are heavy-tailed and/or have infinite variance.
Estimators of the squared risk in this context are obtained in \cite{rad2018scalable,wang2018approximate,bellec2020out,bellec2022derivatives} for the non-ensemble setting, with the method in \citep{rad2018scalable,wang2018approximate} being applicable beyond the squared risk.
It would be of interest to extend (with necessary modifications) our corrected GCV criterion to enable ensembles of base estimators $\hbbeta_m$ trained on robust loss functions. 
We leave this research direction for future work.

\paragraph{Beyond squared loss for testing.\hspace{-0.5em}}
Our primary focus in this paper has been centered on evaluating the squared prediction risk.
The special form of the squared risk allows the useful decomposition in \eqref{eq:risk-decomposition}, which subsequently leads to the estimators \eqref{eq:risk-decomp-est-1}, \eqref{eq:risk-decomp-est-2} and corrected GCV.
In practice, one may be interested in error metrics other than squared error.
More broadly, one may be interested in functionals of the out-of-sample error distribution, such as quantiles of the error distribution.
For regular penalized estimators, one can construct estimators for such functionals by constructing estimators for the out-of-sample distribution first and then subsequently using the plug-in functionals of these distributions.
See \cite{patil2022estimating} for an example that uses GCV-based correction of the in-sample residuals.
Another avenue for loss estimation beyond the square risk is \citep{rad2018scalable,wang2018approximate}, wherein the proposed Approximate Leave-One-Out handles the non-ensemble setting.
Whether one can further additively correct these residuals and construct consistent estimators of the out-of-sample distribution for the ensemble of penalized estimators is an interesting direction of future work.

\paragraph{Risk estimation for generic ensembles.\hspace{-0.5em}}
The subsamples $I_1,\ldots,I_M$ are assumed to be sampled without replacement throughout this paper.
It is of interest to investigate risk estimators for other ensemble techniques like bagging (where we sample with replacement) or other random weighting schemes, i.e., where the data-fitting loss in \eqref{eq:def-hbeta} is replaced by the loss $\sum_{i=1}^n w_{m,i}(y_i-\bx_i^\top\bbeta)^2$ for weights $(w_{m,i})_{i\in[n]}$ sampled independently of the data $(\bX,\by)$. 
This includes weights $w_{m,i} \overset{\text{i.i.d.}}{\sim}\mathrm{Poisson}(1)$ or $(w_{m,1},\ldots,w_{m,n}) \sim $ Multinomial$(n,n,n^{-1})$ typically used in the pair bootstrap; see \cite{karoui2016can} for a study of such weights in unregularized estimation in the proportional regime.

\section*{Acknowledgements}

We warmly thank Arun {K.} Kuchibhotla for orchestrating this ensemble as well as many insightful discussions and feedback while developing the ideas in the paper.
We also warmly thank Ryan {J.} Tibshirani, Alessandro Rinaldo, Yuting Wei, Matey Neykov, Daniel LeJeune for many insightful discussions surrounding generalized cross-validation and ensembles in general (and especially of the ridge and lasso estimators).
Finally, we thank the anonymous reviewers for their constructive feedback and suggestions, leading to several technical and presentational improvements throughout the paper.

P.C.\ Bellec's research was partially supported by the NSF Grant
DMS-1945428.
P.\ Patil thanks the computing support provided by grant MTH230020, which enabled experiments at the Pittsburgh Supercomputing Center Bridge-2 RM-partition.

\bibliographystyle{apalike}

\clearpage

\newcommand{\removelinebreaks}[1]{%
      \def\\{\relax}#1}
\def\titleRLB{\removelinebreaks{\titletext}}

\appendix
\clearpage

\def\titleRLB{\removelinebreaks{\titletext}}

\newgeometry{left=0.75in,top=0.6in,right=0.75in,bottom=0.25in,head=.1in,foot=.1in}

\begin{center}
\Large
{\bf
\framebox{Supplement}
}
\end{center}

\bigskip

This document serves as a supplement to the paper ``Correcting generalized cross-validation for arbitrary ensembles of penalized estimators.''
The initial (unnumbered) section outlines the supplement's structure and summarizes the general notation used in the main paper and supplement in \Cref{tab:notation}. 

\subsection*{Organization}

The content of this appendix is organized as follows.

\begin{itemize}
[leftmargin=1mm]
    \item In \Cref{app:finite-sample-analysis}, we provide proofs of results in \Cref{sec:finite-sample-analysis}.

    \begin{table}[!ht]
    \centering
    \begin{tabular}{l  l  l}
        \toprule
        \textbf{Section} & \textbf{Content} & \textbf{Purpose} \\
        \midrule
        \Cref{sec:derivatives_formulae} & \Cref{lem:derivative} &  Preparatory derivative formulae \\
        \addlinespace[0.2ex] \arrayrulecolor{black!25}
        \midrule
        \Cref{proof:th:sub_full} & \Crefrange{lem:psi_bound}{lem:denom_lower_bound} & Proof of \Cref{th:sub_full}  \\
        \addlinespace[0.2ex] \arrayrulecolor{black!25}
        \midrule
        \Cref{proof:th:correction_gcv} & \Crefrange{lem:cgcv_risk_close}{lem:concentration_df} & 
        Proof of \Cref{th:correction_gcv}
        \\
        \midrule
        \Cref{app:proof:prop:gcv_inconsistency}
        & \cellcolor{lightgray!25}
        & Proof of \Cref{prop:gcv_inconsistency} \\
        \addlinespace[0.2ex] \arrayrulecolor{black!25}
        \midrule
        \Cref{sec:technical-lemmas-nonasymp} & \Crefrange{prop:sure_bound}{lem:contraction_ratio} & Helper lemmas (and their proofs) used in the proofs of \Cref{th:sub_full,th:correction_gcv,prop:gcv_inconsistency} \\
        \midrule
        \Cref{app:elementary-inequality} & \Crefrange{lem:gaussian_oper_bound}{lem:concentration-intersection} & Miscellaneous useful facts used in the proofs of \Cref{th:sub_full,th:correction_gcv,prop:gcv_inconsistency} \\
        \midrule
        \Cref{sec:relaxing_strong_convexity} & \Cref{thm:relaxing-strong-convexity} & Relaxing \Cref{assu:penalty} in the underparameterized regime \\
        \addlinespace[0.2ex] \arrayrulecolor{black}
        \bottomrule
    \end{tabular}
    \end{table}
     
    \item In \Cref{app:asm-consistency}, we provide proofs of results in \Cref{sec:asm-consistency}.

    \begin{table}[!ht]
    \centering
    \begin{tabular}{l  l  l}
        \toprule
        \textbf{Section} & \textbf{Content} & \textbf{Purpose} \\
        \midrule
        \Cref{sec:preparation-asm-consistency}
        & \cellcolor{lightgray!25}
        & Preparatory definitions \\
        \addlinespace[0.2ex] \arrayrulecolor{black!25}
        \midrule
        \Cref{app:subsec:asym-thm:risk-est} & \Crefrange{lem:risk}{lem:boundary} &  Proofs of \Cref{thm:risk-est,thm:unifrom-consistency-lambda} (and proofs of \Crefrange{lem:risk}{lem:boundary}) \\
        \addlinespace[0.2ex] \arrayrulecolor{black!25}
        \midrule
        \Cref{app:subsec:asym-prop:correct-gcv} & \cellcolor{lightgray!25} & Proof of \Cref{prop:correct-gcv}  \\
        \addlinespace[0.2ex] \arrayrulecolor{black!25}
        \addlinespace[0.2ex] \arrayrulecolor{black!25}
        \midrule
        \Cref{app:subsec:asym-technical-lemmas} & \Crefrange{lem:risk-bias}{lem:quad-uncorr} & 
        Helper lemmas (and their proofs)
        used in the proofs 
        of \Cref{thm:risk-est,thm:unifrom-consistency-lambda,prop:correct-gcv} 
        \\
        \addlinespace[0.2ex] \arrayrulecolor{black}
        \bottomrule
    \end{tabular}
    \end{table}
    
    \item In \Cref{app:experiments-gaussian}, we provide additional illustrations for \Cref{sec:finite-sample-analysis} with Gaussian data.

    \begin{table}[!ht]
    \centering
    \begin{tabular}{l  l  l}
        \toprule
        \textbf{Section} & \textbf{Content} & \textbf{Purpose} \\
        \midrule
        \Cref{app:experiments-gaussian-sub-vs-full} & \Crefrange{fig:Fig2_sub_full_ElasticNet.pdf}{fig:Fig2_sub_full_k.pdf} &  Intermediate \ovlp- vs.\ full-estimators in $k$ and $\lambda$ for elastic net and lasso \\
        \addlinespace[0.2ex] \arrayrulecolor{black!25}
        \midrule
        \Cref{app:experiments-gaussian-sub-vs-full-CGCV} & \Crefrange{fig:Fig2_sub_full_ridge-cgcv.pdf}{fig:Fig2_sub_full_lasso.pdf} &  CGCV \ovlp- vs.\ full-estimators in $k$ and $\lambda$ for ridge, elastic net, and lasso \\
        \addlinespace[0.2ex] \arrayrulecolor{black!25}
        \midrule
        \Cref{app:experiments-gaussian-cgcv-vs-gcv-in-k} & \Cref{fig:surface-plot,fig:ElasticNet_k.pdf,fig:lasso_k.pdf} & CGCV vs.\ GCV in $k$ for ridge, elastic net, and lasso \\
        \addlinespace[0.2ex] \arrayrulecolor{black!25}
        \midrule
        \Cref{app:experiments-gaussian-cgcv-vs-gcv-in-M} & \Cref{fig:lasso-surface-plot,fig:ElasticNet_M.pdf,fig:lasso_M.pdf} & CGCV vs.\ GCV in $M$ for ridge, elastic net, and lasso \\
        \addlinespace[0.2ex] \arrayrulecolor{black!25}
        \midrule
        \Cref{app:experiments-gaussian-cgcv-vs-gcv-in-lambda} & \Cref{fig:lasso-lam,fig:lam-ridge,fig:lam-lasso} & CGCV vs.\ GCV in $\lambda$ for ridge, elastic net, and lasso \\
        \addlinespace[0.2ex] \arrayrulecolor{black!25}
        \midrule
        \Cref{sec:simulations-large-M} & \Cref{fig:fun_M_large_M} & CGCV vs. GCV with a large ensemble size $M$ for ridge, elastic net, and lasso \\
        \addlinespace[0.2ex] \arrayrulecolor{black!25}
        \midrule
        \Cref{sec:naive-gcv-inconsistency} & \Cref{fig:gcv_naive_lam} & Inconsistency for the GCV variant \eqref{eq:gcv-naive}  \\
        \addlinespace[0.2ex] \arrayrulecolor{black}
        \bottomrule
    \end{tabular}
    \end{table}

    \clearpage

    \item
    In \Cref{sec:numerical-non-gaussian}, we provide additional illustrations for \Cref{sec:finite-sample-analysis} with non-Gaussian data.

    \begin{table}[!ht]
    \centering
    \begin{tabular}{l  l  l}
        \toprule
        \textbf{Section} & \textbf{Content} & \textbf{Purpose} \\
        \midrule
        \Cref{app:experiments-rademacher} & \Cref{fig:Fig1_ridge_lasso_rademacher,fig:k_rademacher,fig:ridge-k_rademacher,fig:M_rademacher} &  Experiments with Rademacher distribution \\
        \addlinespace[0.2ex] \arrayrulecolor{black!25}
        \midrule
        \Cref{app:experiments-uniform} & \Cref{fig:Fig1_ridge_lasso_uniform,fig:k_uniform,fig:ridge-k_uniform,fig:M_uniform} & Experiments with uniform distribution  \\
        \addlinespace[0.2ex] \arrayrulecolor{black}
        \bottomrule
    \end{tabular}
    \end{table}

    \item
    In \Cref{app:experiments-nongaussian}, we provide additional numerical illustrations for \Cref{sec:asm-consistency} with non-Gaussian data.

    \begin{table}[!ht]
    \centering
    \begin{tabular}{l  l  l}
        \toprule
        \textbf{Section} & \textbf{Content} & \textbf{Purpose} \\
        \midrule
        \Cref{app:experiments-nongaussian-cgcv-vs-gcv-in-lambda} & \Cref{fig:elasticnet-lambda,fig:lasso-lambda} & CGCV vs. GCV in $\lambda$ for ridge, elastic net, and lasso \\
        \addlinespace[0.2ex] \arrayrulecolor{black}
        \bottomrule
    \end{tabular}
    \end{table}

\end{itemize}

\clearpage

\subsection*{Notation}

The notational conventions used in this paper are as follows.
Any section-specific notation is introduced in-line as needed.

\begin{table}[ht]
\centering
\caption{Summary of general notation used in the paper and the supplement.}
\begin{tabularx}{\textwidth}{L{3.5cm}L{16cm}}
\toprule
\textbf{Notation} & \textbf{Description} \\
\midrule
Non-bold lower or upper case & Denotes scalars (e.g., $k$, $\lambda$, $M$) \\
Bold lower case & Denotes vectors (e.g., $\bx$, $\by$, $\bbeta_0$) \\
Bold upper case & Denotes matrices (e.g., $\bX$, $\bSigma$, $\bL$) \\
Script font & Denotes certain limiting functions (e.g., $\sR$). \\
\arrayrulecolor{black!25}\midrule
$\mathbb{R}$, $\mathbb{R}_{\ge 0}$ & Set of real and non-negative real numbers \\
$[n]$ & Set $\{1, \dots, n\}$ for a natural number $n$ \\
\midrule
$|I|$ & Cardinality of a set $I$ \\
$(x)_{+}$ & Positive part of a real number $x$ \\
$\nabla f$, $\nabla^2 f$ & Gradient and Hessian of a function $f$ \\
$\ind_A$, $\PP(A)$ & Indicator random variable associated with an event $A$ and probability of $A$ \\
$\EE[X], \Var(X)$ & Expectation and variance of a random variable $X$ \\
\midrule
$\mathcal{V}^\perp$ & Orthogonal complement of a vector space $\cV$ \\
$\tr[\bA]$, $\bA^{-1}$ & Trace and inverse (if invertible) of a square matrix $\bA \in \RR^{p \times p}$ \\
$\mathrm{rank}(\bB)$, $\bB^\top$, $\bB^{+}$ & Rank, transpose and Moore-Penrose inverse of matrix $\bB \in \RR^{n \times p}$ \\
$\bC^{1/2}$ & Principal square root of a positive semidefinite matrix $\bC$ \\
$\bI_p$ or $\bI$ & The $p \times p$ identity matrix \\
\midrule
$\langle \bu, \bv \rangle$ & Inner product of vectors $\bu$ and $\bv$ \\
$\| \bu \|_{p}$ & The $\ell_p$ norm of a vector $\bu$ for $p \ge 1$ \\
$\|f\|_{L_p}$ & The $L_p$ norm of a function $f$ for $p \ge 1$ \\
\midrule
$\| \bA \|_{\mathrm{op}}$ & Operator (or spectral) norm of a real matrix $\bA$ \\
$\| \bA \|_{\mathrm{tr}}$ & Trace (or nuclear) norm of a real matrix $\bA$ \\
$\| \bA \|_F$ & Frobenius norm of a real matrix $\bA$ \\
\midrule
$X \lesssim Y$ & $X \le C Y$ for some absolute constant $C$ \\
$X \lesssim_\alpha Y$ & $ X \le C_\alpha Y$ for some constant $C_\alpha$ that may depend on ambient parameter $\alpha$ \\
$X = \cO_\alpha(Y)$ & $| X | \le C_\alpha Y$ for some constant $C_\alpha$ that may depend on ambient parameter $\alpha$  \\
$\bu \le \bv$ & Lexicographic ordering for vectors $\bu$ and $\bv$  \\
$\bA \preceq \bB$ & Loewner ordering for symmetric matrices $\bA$ and $\bB$ \\
\midrule
$\Op$ & Probabilistic big-O notation \\
$\op$ & Probabilistic little-o notation \\
$\bC \asympequi \bD$ & Asymptotic equivalence of matrices $\bC$ and $\bD$ (see \Cref{app:asm-consistency} for more details) \\
$\stackrel{\mathrm{d}}{=}$ & Equality in distribution \\
\midrule
$\dto$ & Denotes convergence in distribution \\
$\pto$ & Denotes convergence in probability \\
$\asto$ & Denotes almost sure convergence \\
\arrayrulecolor{black}\bottomrule
\end{tabularx}
\label{tab:notation}
\end{table}

Note: Throughout, $C$, $C'$ denote positive absolute constants.
If no subscript is present for norm $\| \bu \|$ of a vector $\bu$, then it is assumed to be the $\ell_2$ norm of $\bu$.
If a proof of a statement is separated from the statement, the statement is restated (while keeping the original numbering) along with the proof for convenience.

\restoregeometry

\newpage

\section[Proofs of \Cref{sec:finite-sample-analysis}]{Proofs in \Cref{sec:finite-sample-analysis}} \label{app:finite-sample-analysis}

\subsection{Preparatory derivative formulae}
\label{sec:derivatives_formulae}
For a pair of possibly overlapping subsamples $(I_m, I_\ell)$ (written as $(I, \Tilde{I})$ in this proof for easy of writing) of $\{1,2, \ldots, n\}$ with $|I_m|= |I_\ell|=k \le n$, 
we define 
$\hbeta$ and $\tbeta$ as the estimators using subsamples $(\bx_i, y_i)_{i\in I}$ and $(\bx_i, y_i)_{i\in \tilde I}$ respectively, i.e., 
\begin{align*}
    \hbeta := \argmin_{\bbeta\in \R^p} \frac{1}{k} \sum_{i\in I} (y_i - \bx_i^\top \bbeta)^2 + g(\bbeta), 
    \quad
    \text{and}
    \quad
    \tbeta := \argmin_{\bbeta\in \R^p} \frac{1}{k} \sum_{i\in \tilde I} (y_i - \bx_i^\top \bbeta)^2 + \tilde{g}(\bbeta),
\end{align*}
and we denote the corresponding residual and the degree of freedom by
$$
\br := \by - \bX \hbeta, \quad \tilde\br := \by-\bX\tbeta, \quad  
\df := \tr[\frac{\partial}{\partial \by} \bX\hbeta],\quad  \tdf := \tr[\frac{\partial}{\partial\by}\bX \tbeta]. 
$$
For clarity, we use the notation below throughout \Cref{app:finite-sample-analysis}:
$$
\bZ := \biggl[\bX \bSigma^{-1/2} \mathrel{\Big|} \sigma^{-1} \bepsilon\biggr] \in \R^{n\times(p+1)},\quad
\bh := \begin{pmatrix}
    \bSigma^{1/2}(\hbeta - \bbeta_0)
    \\
    -\sigma
\end{pmatrix} \in \R^{p+1}
, \quad \tilde\bh := \begin{pmatrix}
    \bSigma^{1/2}(\tbeta - \bbeta_0)
    \\
    -\sigma
\end{pmatrix}\in \R^{p+1}
$$
Under \Cref{assu:Gaussian-feature,assu:Gaussian-response}, 
the matrix $\bZ$ has i.i.d.\ $\cN(0,1)$ entries. The individual risk component $R_{m,\ell}$ in \eqref{eq:R_M} and the residuals $(\br, \tilde{\br})$ can be written as  functions of $(\bZ, \bh, \tilde\bh)$ as follows:
$$
(\hbeta-\bbeta_0)^\top\bSigma(\tbeta-\bbeta_0) + \sigma^2 = \bh^{\top}\tilde \bh, \quad \br = -\bZ\bh, \quad \tilde\br = -\bZ\tilde\bh. 
$$
Our first lemma below provides the derivative formula of $\bh$ with respect to $\bZ$, which is an important component of our proof. Similar derivative formulas have been studied in M-estimation \citep{bellec2022derivatives}, multi-task linear model \citep{bellec2021chi,tan2022noise}, and multinomial logistic model \citep{tan2023multinomial}.
\begin{restatable}{lemma}{LemDerivative}[Variant of \citet[Theorem 1]{bellec2022derivatives}]\label{lem:derivative}
Suppose the penalty $g$ satisfies \Cref{assu:penalty} with a constant $\mu>0$. Then there exists a matrix $\bB\in \R^{(p+1)\times (p+1)}$ depending on $(\bz_i)_{i\in I}$ such that 
\begin{equation}\label{eq:property_B}
    \|\bB\|_{\oper} \le (k\mu)^{-1}, \quad \rank(\bB) \le  p, \quad \tr[\bB] \ge 0
\end{equation}
and the derivative of $\bh$ with respect to $\bZ=(z_{ij})_{i\in [n], j\in[p+1]}$ is given by
\begin{equation}\label{eq:derivartive_formula}
    \text{for all } i\in [n] \text{ and } j\in [p+1], \quad \frac{\partial \bh}{\partial z_{ij}} =
\bB\be_j \be_i^\top \bL\br-\bB \bZ^{\top} \bL \be_i  \be_j^\top \bh  \quad \text{with} \quad \bL = \bL_I = \sum_{i\in I} \be_i\be_i^\top. 
\end{equation}
Furthermore, $\bV := \bI_n - \bZ \bB \bZ^\top \bL$ satisfies the following:
\begin{align}\label{eq:property_tr}
\df = k - \tr[\bL\bV], \quad 
0 < k \bigg(1+ \frac{\|\bL\bG\|_{op}^2}{k\mu}\bigg)^{-1} \le \tr[\bL\bV] \le k \quad \text{and} \quad \|\bL\bV\|_{\oper} \le 1,
\end{align}
where $\bG=\bX\bSigma^{-1/2}\in\R^{n\times p}$ is a standard Gaussian matrix. 
\end{restatable}
\Cref{lem:derivative} is proved in
\Cref{proof:lem:derivative}.
Note that \Cref{lem:derivative} holds for the derivatives of $\tilde \bh$, using $\tilde\bB$, $\tilde \bL = \bL_{\tilde{I}}$, and $\tilde\bV$ instead.

\subsection{Proof of \Cref{th:sub_full} (restated here for convenience)}\label{proof:th:sub_full}

\bigskip

\ThmSubFull*

\subsubsection{Part 1: Proof of \Cref{eq:sub_full_moment_ineq}}
\begin{proof}[\underline{Proof for the $\ovlp$-estimator}]
Using the notation in \Cref{sec:derivatives_formulae} with $I$ and $\tilde I$ meaning $I_m$ and $I_\ell$, 
the inequality we want to show can be written as 
\begin{equation}\label{eq:moment_ineq_sub}
\EE\Bigl[\Bigl|
\frac{\br^\top \bL \tilde\bL \tilde\br - |I\cap \tilde I|(1-\df/k) (1-\tdf/k) \bh^\top\tilde\bh}{\|\bh\| \|\tilde\bh\|}
\Bigr|\Bigr] \lesssim \sqrt{n}\tau^{-2} c^{-3} \gamma^{7/2},    
\end{equation}
where $c=k/n, \tau=\min(1, \mu), \gamma=\max(1, p/n)$. 
Using $\tr[\bL \bV] = k - \df$ and $\tr[\tilde\bL \tilde\bV] = k - \tdf$ by \eqref{eq:property_tr}, we can rewrite the error inside the expectation in the left-hand side as the sum of three terms
\begin{equation}\frac{\br^\top \bL \tilde\bL \tilde\br - |I\cap \tilde I|(1-\df/k) (1-\tdf/k) \bh^\top\tilde\bh}{\|\bh\| \|\tilde\bh\|} = \Rem_1+\Rem_2+\Rem_3,
    \label{eq:ovlp_Rem1_2_3}
\end{equation}
where $\Rem_1, \Rem_2, \Rem_3$ are defined as
\begin{align*}
    \Rem_1 &= \frac{\bh^\top \tilde\bh}{\|\bh\| \|\tilde\bh\|} \frac{\tr[\tilde \bL\tilde\bV]}{k} (\tr[\tilde\bL\bL\bV] -  \frac{|I\cap\tilde I|\tr[\bL \bV]}{k} ),
    \Rem_2 = \frac{\tr[\tilde\bL\tilde\bV]}{k} \frac{\br^\top \tilde\bL\bL \br (1 + \tr[\tilde\bB]) - \tr[\tilde\bL\bL\bV]\bh^\top\tilde\bh}{\|\bh\|\|\tilde\bh\|},\\
    \Rem_3 &= \frac{\br^\top \tilde\bL\bL \tilde\br}{\|\bh\|\|\tilde\bh\| } \frac{k - \tr[\tilde \bL \tilde\bV] (1+\tr[\tilde{\bB}])}{k}.
\end{align*}
Next, we bound the moment of $|\Rem_1|, |\Rem_2|, |\Rem_3|$ one by one.

{\noindent\textbf{Control of $\Rem_1$.}}
Since $\tr[\tilde\bL\tilde\bV] \in (0,k]$ by \eqref{eq:property_tr} and $|\tilde\bh^\top \bh| \le \|\bh\|\|\tilde\bh\|$ by the Cauchy--Schwarz inequality, we have 
$$
|\Rem_1| \le |\tr[\tilde\bL\bL\bV] - k^{-1} |I\cap\tilde I|\tr[\bL \bV]| = \text{RHS}.
$$
Below we bound the second moment of RHS. Here, the key fact is that conditionally on $(|I\cap \tilde I|, I)$,  $I\cap \tilde I$ is uniformly distributed over all subsets of $I$ of size $|I\cap \tilde I|$. Then, the variance formula of the sampling without replacement (see \Cref{lemma:sample_without_rep} with $x_i = V_{ii}$, $M=|I|=k$, and $m= |I\cap\tilde I|$) implies that the conditional expectation of $(\text{RHS})^2$ given $(|I\cap \tilde I|, I, \bV)$ is bounded from above as
\begin{align*}
&\EE\Bigl[\Bigl(\tr[\tilde\bL\bL\bV] - \frac{|I\cap\tilde I|}{k} \tr[\bL\bV]\Bigr)^2~\Big|~ |I\cap \tilde I|, I, \bV\Bigr] = |I\cap \tilde I|^2 \EE\Bigl[\bigg(\frac{\sum_{i\in I\cap \tilde I} V_{ii}}{|I\cap\tilde I|} - \frac{\sum_{i\in I} V_{ii}}{|I|}\bigg)^2~\Big|~ |I\cap \tilde I|, I, \bV\Bigr]\\
&= |I\cap \tilde I|^2 \Var\Bigl[\frac{\sum_{i\in I\cap \tilde I} V_{ii}}{|I\cap \tilde I|}~\Big|~ |I\cap \tilde I|, I, \bV \Bigr] \le |I\cap \tilde I|^2 \frac{\sum_{i\in I} V_{ii}^2}{|I||I\cap \tilde I|} = |I\cap \tilde I|\frac{\sum_{i\in I} V_{ii}^2}{|I|}. 
\end{align*}
Note that the above inequality holds even when $|I\cap \tilde I|=0$, since the both sides are then $0$. 
Note in passing that $\sum_{i\in I} V_{ii}^2/|I| \le \|\bL\bV\|_{F}^2/|I| \le \|\bL\bV\|_{\oper}^2 \le 1$ by \eqref{eq:property_tr}. Then, we obtain
\begin{equation}\label{eq:var_trace}
    \EE\Bigl[\Bigl|\tr[\tilde\bL\bL\bV] - \frac{|I\cap\tilde I|\tr[\bL \bV]}{k}\Bigr|^2\Bigr] \le \EE\Bigl[|I\cap \tilde I|\frac{\sum_{i\in I} V_{ii}^2}{|I|}\Bigr] \le \EE[|I\cap \tilde I|] \le k. 
\end{equation}
Thus, \eqref{eq:var_trace} and the Cauchy--Schwarz inequality $\EE[\cdot]^2 \le \EE[(\cdot)^2]$ yield
\begin{equation}\label{eq:rem_1_sub}
        \EE|\Rem_1| \le \sqrt{\EE[|\Rem_1|^2]}
 \le  \sqrt{\EE[(\tr[\tilde\bL\bL\bV] -  k^{-1}{|I\cap\tilde I|\tr[\bL \bV]})^2]} \le \sqrt{k}.
\end{equation}

{\noindent\textbf{Control of $\Rem_2$.}} 
Since $\tr[\tilde\bL\tilde\bV]\in (0,k]$ by \eqref{eq:property_tr}, we have
$$
|\Rem_2| \le |\br^\top \tilde\bL\bL \br (1 + \tr[\tilde\bB]) - \tr[\tilde\bL\bL\bV]\bh^\top\tilde\bh|/(\|\bh\|\|\tilde\bh\|).
$$
The lemma below gives the bound of the second moment of the right-hand side. 
\begin{restatable}{lemma}{LemPsiBound}\label{lem:psi_bound}
    For any subset $J \subset [n]$ that is independent of $\bZ$, we have
    $$
    \EE \biggl[
        \Bigl(
        \frac{-\br^{\top} \bL_J \tilde \br
    - 
    \tilde \br^\top \tilde \bL \bL_J \br
    \trace[\tilde \bB]
    + \trace[\bL_J \bV] \bh^\top \tilde \bh
    }{\|\bh\| \|\tilde\bh\|}\Bigr)^2 \ \mathrel{\Big|} J \ \biggr]
    \lesssim \tau^{-2} (|J| + \frac{p(|J|+p)^2}{k^2})
    $$
    where $\tau=\min(1, \mu)$ and $\bL_J = \sum_{i\in J}\be_i\be_i^\top$.
    \end{restatable}
    \Cref{lem:psi_bound} is proved in \Cref{proof:psi_bound}.
    A sketch is as follows:
        \Cref{lem:psi_bound} follows from the derivative formula \eqref{eq:derivartive_formula} and the second order Stein's formula, where $-\br^\top \bL_J\tilde\br = \br^\top \bL_J \bZ \tilde\bh$ is regarded as a component-wise inner product of the Gaussian matrix $\bL_J \bZ\in \R^{n \times (p+1)}$ and $\bL_J \br \tilde\bh^\top \in \R^{n \times (p+1)}$. 
    
\Cref{lem:psi_bound} with $J=I\cap\tilde I$ and 
$\EE|\Rem_2| \le \sqrt{\EE[|\Rem_2|^2]}$ lead to 
\begin{equation}\label{eq:rem_2_sub}
    \EE|\Rem_2| \lesssim \tau^{-1} \sqrt{\EE[|I\cap\tilde I| + p(|I\cap\tilde I| + p)^2 k^{-2}]}
    \le \tau^{-1} \sqrt{k + p(k+p)^2 k^{-2}}
    \lesssim  \sqrt{k} \tau^{-1} (1 + p/k)^{3/2}
\end{equation}
thanks to $\tau = \min(1, \mu)$.

{\noindent\textbf{Control of $\Rem_3$.}} 
Note
$$
|\Rem_3| = \frac{|\br^\top \tilde\bL\bL \tilde\br|}{\|\bh\|\|\tilde\bh\| } \cdot \frac{|k - \tr[\tilde \bL \tilde\bV] (1+\tr[\tilde{\bB}])|}{k}\le  \|\tilde\bL\bZ\|_{\oper}^2 \frac{|k - \tr[\tilde \bL \tilde\bV] (1+\tr[\tilde{\bB}])|}{k},
$$
thanks to $|\br^\top \tilde\bL\bL \tilde\br| = |\bh^\top \bZ^\top \tilde\bL \bL \bZ \tilde\bh| \le \|\bh\|\|\tilde\bh\|\|\tilde\bL\bL\bZ\|_{\oper}^2 \le \|\bh\|\|\tilde\bh\|\|\tilde\bL\bZ\|_{\oper}^2$. 
Let us define the random variable $Y$ and the event $\Omega$ as 
$$
Y = k^{-1} (k - \tr[\tilde \bL \tilde\bV] (1+\tr[\tilde{\bB}])), \quad \Omega =  \{ \|\tilde\bL\bZ\|_{\oper} \le 2(\sqrt{k} + \sqrt{p+1})\}. 
$$
Then, $\EE[|\Rem_3|]$ is bounded from above as
\begin{align*}
    \EE|\Rem_3| &\le \EE [\|\tilde{\bL}\bZ\|_{\oper}^2 |Y|]
    \le  \|Y\|_{\infty} \EE[\ind_{\Omega^c} \|\tilde{\bL}\bZ\|_{\oper}^2] + [2(\sqrt{k} + \sqrt{p+1})]^2 \EE|Y|\\
    &\lesssim \|Y\|_{\infty} \cdot \PP(\Omega^c) \cdot \EE[\|\tilde{\bL}\bZ\|_{\oper}^4]^{1/2} + (k+p) \EE|Y| := \Rem_3^1 + \Rem_3^2
\end{align*}
thanks to H\"older's inequality. 
Below we bound $\Rem_3^1$ and $\Rem_3^2$. 
\begin{enumerate}
    \item Bound of $\Rem_3^1$: Since $\|\tilde{\bL}\bZ\|_{\oper} =^d \|\bG\|_{\oper}$ with $\bG\in\R^{k\times(p+1)}$ being the standard Gaussian matrix, the concentration of the operator norm (see \Cref{lem:gaussian_oper_bound}) implies 
$$
\PP(\Omega^c) \le e^{-(\sqrt{k} + \sqrt{p+1})^2/2} \le e^{-(k+p)/2} \text{ and } \EE\|\tilde \bL \bZ\|_{\oper}^4 \lesssim (k + p)^2.
$$
By contrast, $\|\tilde\bB\|_{\oper} \le (k\mu)^{-1}$ and $\rank(\tilde\bB) \le p$ by \eqref{eq:property_B} and $\tr[\tilde\bL\tilde\bV]\in(0,k]$ by \eqref{eq:property_tr} lead to $$
\|Y\|_{\infty} = |1 - k^{-1}\tr[\tilde{\bL}\tilde{\bV}](1+\tr[\tilde\bB])]| \le 1 + |k^{-1}\tr[\tilde{\bL}\tilde{\bV}]| (1+|\tr[\tilde\bB]|) \le 2 + p(k\mu)^{-1} \le \tau^{-1} (1+p/k)
$$ with $\tau=\min(1, \mu)$. Therefore, we have
\begin{align*}
\Rem_3^1 = \|Y\|_{\infty} \PP(\Omega^c) \EE[\|\tilde{\bL}\bZ\|_{\oper}^4]^{1/2} \lesssim\tau^{-1} (1+p/k)  e^{-(k+p)/2}  (k + p)
\lesssim
\sqrt{k}\tau^{-1} (1+p/k)^{3/2},
\end{align*}
thanks to $e^{x} \ge \sqrt{x/2}$ for all $x>0$. 
\item Bound of $\Rem_3^2$: The following lemma bounds $\Rem_3^2 = (k+p)\EE[|Y|]$. 
    \begin{restatable}{lemma}{LemTrBdf}\label{lem:trB_df}
    Let
    $\xi_1 = (1+\tr[\bB]) \|\bL \br \|^2/\|\bh\|^2 - \tr[\bL\bV]$
    and
    $\xi_2 = k -  (1+\tr[\bB])^2 \|\bL \br \|^2/\|\bh\|^2$.
    Then
$$
\EE  |k - \tr[\bL\bV] (1+\tr[\bB])| 
= \EE| (1+\tr[\bB]) \xi_1 + \xi_2 |
\lesssim \sqrt{k} \tau^{-2}(1+p/k)^{5/2}.
$$
\end{restatable}
\Cref{lem:trB_df} is proved in \Cref{proof:trB_df}.
A sketch is as follows:
The quantity $\xi_1$ is bounded using \Cref{lem:psi_bound} with $J=\tilde I = I$. For $\xi_2$, we use a chi-square type moment inequality \cite[Theorem 7.1]{bellec2020out}. 

Recall $\Rem_3^2 = (k+p) \EE|Y|$ with $Y:= k^{-1} (k - \tr[\tilde \bL \tilde\bV] (1+\tr[\tilde{\bB}]))$. Then, \Cref{lem:trB_df} implies
$$
\Rem_3^2 = (1+p/k) \EE[|k - \tr[\tilde \bL \tilde\bV] (1+\tr[\tilde{\bB}])|] \lesssim \sqrt{k} \tau^{-2} (1+p/k)^{7/2}.
$$
\end{enumerate}
Combining the bounds of $\Rem_3^1$ and $\Rem_3^2$, we have
\begin{equation}\label{eq:rem_3_sub}
    \EE|\Rem_3| \lesssim \Rem_3^1 + \Rem_3^2 \lesssim \sqrt{k} \tau^{-2} (1 + p/{k})^{7/2} + \sqrt{k}\tau^{-1} (1+p/k)^{3/2} \lesssim \sqrt{k} \tau^{-2} (1 + p/{k})^{7/2},
\end{equation}
thanks to $\tau = \min(1, \mu) \le 1$. 

We have now controlled each of $\EE|\Rem_1|,\EE|\Rem_2|,\EE|\Rem_3|$. \eqref{eq:rem_1_sub}, \eqref{eq:rem_2_sub} and \eqref{eq:rem_3_sub} result in
\begin{align*}
    \EE[|\Rem_1|] + \EE[|\Rem_2|] + \EE[|\Rem_3|] &\lesssim \sqrt{k}\tau^{-1} (1+p/k)^{3/2} + \sqrt{k} \tau^{-2} (1 + p/k)^{7/2} + \sqrt{k} \\
    &\lesssim  \sqrt{k} \tau^{-2} (1 + p/k)^{7/2} \lesssim \sqrt{n}\tau^{-2} c^{-3} \gamma^{7/2}
\end{align*}
thanks to $\gamma=\max(1, p/n) \ge 1$, $c=k/n \le 1$ and $\tau=\min(1, \mu) \le 1$. This finishes the proof of \eqref{eq:sub_full_moment_ineq} for the $\sub$-estimator.
\end{proof}

\begin{proof}[\underline{Proof for the $\full$-estimator}]
Using the notation in \Cref{sec:derivatives_formulae}, our goal here is to show
$$
\EE\Bigl[\Bigl|
\frac{(n - \df -\tdf +  \frac{|I\cap \tilde I|}{k^2} \df\tdf) {\bh^\top  \tilde\bh} -{\br^\top\tilde\br}}{\|\bh\|\|\tilde{\bh}\|}
\Bigr|\Bigr] \lesssim  \sqrt{n} \tau^{-2} c^{-2} \gamma^{5/2}.
$$
Here, $\tr[\bV] = n- \df$,  $\tr[\bL\bV] = k-\df$ and $\tr[\tilde\bL \bV] - \tr[\tilde\bL \bL\bV] = k - |I\cap \tilde I|$ implies
\begin{align*}
&\Big(n - \df -\tdf + \frac{|I\cap \tilde I|}{k^2} \df\tdf\Big) \bh^\top  \tilde\bh =
\Bigl(\tr[\bV] - \frac{\tdf}{k}\big\{\tr[\tilde\bL \bV] - \tr[\tilde\bL \bL\bV] +  \frac{|I\cap\tilde I|}{k}\tr[\bL\bV]\big\}\Bigr) \bh^\top \tilde\bh,
\end{align*}
so that, by simple algebra, the error can be decomposed as 
\begin{align*}
    \Big(n - \df -\tdf +  \frac{|I\cap \tilde I|}{k^2} \df\tdf\Big) {\bh^\top  \tilde\bh} -{\br^\top\tilde\br} = \|\bh\| \|\tilde \bh\| (\Rem_1 + \Rem_2 + \Rem_3 + \Rem_4),
\end{align*}
where the four following remainders will be bounded separately:
\begin{align*}
    \Rem_1 &= \frac{- \br^\top \tilde\br - \tilde\br^\top \tilde\bL \br \tr[\tilde\bB] + \tr[\bV] \bh^\top\tilde\bh}{\|\bh\|\|\tilde\bh\|}, 
           &&\Rem_2 = \frac{\tr[\tilde\bB]}{1+\tr[\tilde\bB]} \frac{
    (1+\tr[\tilde\bB]) \tilde\br^\top \tilde\bL \br - \tr[\tilde\bL \bV] \bh^\top \tilde\bh}{\|\bh\|\|\tilde\bh\|},\\
    \Rem_3 &= \frac{(1+\tr[\tilde\bB]) \tr[\tilde\bL \tilde\bV] - k}{(1+\tr[\tilde\bB])} \frac{\tr[\tilde \bL \bV]}{k} \frac{\bh^\top\tilde\bh}{\|\bh \| \| \tilde\bh \|}, 
           &&\Rem_4 = \frac{\tdf}{k} (\tr[\tilde\bL\bL\bV] - \frac{|I\cap \tilde I|}{k} \tr(\bL\bV)) \frac{\bh^\top \tilde\bh}{\|\bh\|\|\tilde\bh\|}.
\end{align*}

{\noindent\textbf{Control of $\Rem_1$ and $\Rem_2$.}} 
\Cref{lem:psi_bound} with $K=[n]$ implies
$$
\EE[|\Rem_1|] \le \sqrt{\EE[|\Rem_1|^2]} \lesssim \sqrt{\tau^{-2} (n + p(n+p)^2/k^2)} \le \tau^{-1} c^{-1} \sqrt{n + p(1+p/n)^2}
\lesssim \sqrt{n} \tau^{-1} c^{-1} (1+p/n)^{3/2},
$$
thanks to $k/n = c\in (0,1]$, while 
$\tr[\tilde\bB] \ge 0$ by \eqref{eq:property_B} and \Cref{lem:psi_bound} with $J=\tilde I$ lead to 
$$
\EE[|\Rem_2| \le \sqrt{\EE[|\Rem_2|^2]} \lesssim \sqrt{\tau^{-2} (k + p(k+p)^2/k^2)} \lesssim \tau^{-1} \sqrt{k} (1 + p/k)^{3/2} \lesssim
\sqrt{n} \tau^{-1} c^{-1} (1+p/n)^{3/2}.
$$
{\noindent\textbf{Control of $\Rem_3$.}} 
According to \eqref{eq:property_tr}, it holds that
$\tr[\tilde\bL \bV] =  \tr[\tilde\bL \bL\bV] + k - |I\cap \tilde I|$ and $\|\bL\bV\|_{\oper} \le 1$. Thus, 
$$
k^{-1} |\tr[\tilde\bL\bV]
| = k^{-1} |\tr[\tilde\bL\bL\bV]| + k^{-1} |k- |I\cap \tilde I|| \le \|\bL\bV\|_{\oper} + 1 \le 2.
$$
Combining the above display and $\tr[\tilde\bB] \ge 0$ by \eqref{eq:property_B}, as well as \Cref{lem:trB_df}, we obtain
$$
\EE[|\Rem_3|] \le 2 \EE[|(1+\tr[\tilde\bB]) \tr[\tilde\bL \tilde\bV] - k|] \lesssim  \tau^{-2} k^{1/2} (1 + p/k)^{5/2} \lesssim \sqrt{n} \tau^{-2} c^{-2} (1+p/n)^{5/2}.
$$

{\noindent\textbf{Control of $\Rem_4$.}} 
\Cref{eq:var_trace} implies
$$\EE[|\Rem_4|] \le \EE|\tr[\tilde\bL\bL\bV] - k^{-1}|I\cap \tilde I| \tr[\bL\bV]| \le \sqrt{k} \le \sqrt{n}.$$
From the above displays, we observe that the dominating upper bound is that of $\EE[|\Rem_3|]$. Thus, we obtain 
\[
    \EE[|\Rem_1|] + \EE[|\Rem_2|] + \EE[|\Rem_3|] + \EE[|\Rem_4|] \lesssim  \sqrt{n} \tau^{-2} c^{-2} (1+p/n)^{5/2} \lesssim \sqrt{n} \tau^{-2} c^{-2} \gamma^{5/2}
\]
thanks to $\gamma=\max(1, p/n)$. This finishes the proof of \eqref{eq:sub_full_moment_ineq} for the $\full$-estimator.
\end{proof}

\subsubsection{Part 2: Proof of \Cref{eq:sub_full_relative_error}}
For $\est=\sub$ and $\full$, by algebra, the relative error is bounded from above as
\begin{align}
    \Big|\frac{\tR_M^\est}{R_M} - 1\Big| &= \frac{|\sum_{m,\ell} (\hat{R}_{m,\ell}^\est - \bh_m^\top \bh_\ell)|}{\|\sum_{m'} \bh_{m'}\|^2} \le \frac{(\sum_{m'} \|\bh_{m'}\|)^2}{\|\sum_{m'} \bh_{m'}\|^2} \cdot \sum_{m,\ell} \frac{\|\bh_m\|\|\bh_\ell\|}{(\sum_{m'} \|\bh_{m'}\|)^2}
   \cdot \frac{|\hat{R}_{m, \ell}^\est - \bh_m^\top \bh_\ell|}{\|\bh_m\|\|\bh_\ell\|} \notag \\
   &\le \frac{(\sum_{m'} \|\bh_{m'}\|)^2}{\|\sum_{m'} \bh_{m'}\|^2} \cdot
   \max_{m, \ell} \frac{|\hat{R}_{m, \ell}^\est - \bh_m^\top \bh_\ell|}{\|\bh_m\|\|\bh_\ell\|} 
   \le M \cdot \frac{\sum_{m'} \|\bh_{m'}\|^2}{\|\sum_{m'} \bh_{m'}\|^2} \cdot
   \max_{m, \ell} \frac{|\hat{R}_{m, \ell}^\est - \bh_m^\top \bh_\ell|}{\|\bh_m\|\|\bh_\ell\|}\notag\\
   &= M \cdot \text{Ratio} \cdot \max_{m, \ell} \frac{n} {|d_{m,\ell}^\est|} |E_{m,\ell}^\est|
   \label{eq:relative_error_ineq}.
\end{align}
Here, we have defined $\text{Ratio}$ and $E_{m,\ell}^\est$ by
\begin{equation}\label{eq:df_ratio_E_ml_est}
    \text{Ratio} := \frac{\sum_{m'} \|\bh_{m'}\|^2}{\|\sum_{m'} \bh_{m'}\|^2}, \quad E_{m,\ell}^\est:= \frac{d_{m,\ell}^\est(\hat{R}_{m,\ell}^\est-\bh_m^\top\bh_\ell)}{n\|\bh_m\|\|\bh_\ell\|}, 
\end{equation}
where $d_{m,\ell}^\est$ is the denominator of the corresponding estimator for $\est\in\{\sub, \full\}$ defined in \eqref{eq:d_sub_and_full}.  We restate their expressions for convenience: 
\begin{equation}\label{eq:df_denominator_sub_full}
    d_{m,\ell}^\sub = |I_m\cap I_\ell| (1-\df_m/k)(1-\df_\ell/k), \quad d_{m,\ell}^\full=
n-\df_m - \df_\ell + k^{-2}{|I_m\cap I_\ell|} \df_m\df_\ell
\end{equation}
Below we bound $|E_{m,\ell}^\est|$,  $|d_{m,\ell}^\est|^{-1}$, and Ratio. 

\noindent\textbf{Control of $E_{m, \ell}^\est$.} 
Markov's inequality applied with the moment bound \eqref{eq:sub_full_moment_ineq} yields
\begin{equation}\label{eq:bound_E_ml_est}
\text{ for all $\epsilon > 0$}, \quad \PP(|E_{m,\ell}^\est| > \epsilon) \le \epsilon^{-1} \EE [|E_{m,\ell}^\est|] \lesssim \frac{1}{\epsilon \sqrt{n} \tau^2} \times \begin{dcases}
    c^{-3} \gamma^{7/2} & \est=\sub\\
    c^{-2} \gamma^{5/2} & \est=\full
\end{dcases}    
\end{equation}

\noindent\textbf{Control of Ratio.}
Here the key lemma to bound Ratio := ${\sum_{m'} \|\bh_{m'}\|^2}/{\|\sum_{m'} \bh_{m'}\|^2}$ is the concentration of the correlation $\bh_m^\top \bh_\ell/\|\bh_m\|\|\bh_\ell\|$:
\begin{restatable}{lemma}{LemContractionCor}\label{lem:contraction_cor}
    If the same penalty is used for all $m\in[M]$, i.e., $g_m=g$ in \eqref{eq:def-hbeta}, then
$$
\eta_{m,\ell} := \EE\bigg[\frac{\bh_m^\top \bh_\ell}{\|\bh_m\|\|\bh_\ell \| } \mid (\bz_i)_{i\in I_m\cap I_\ell} \bigg] \ge 0, \quad 
\text{and}
\quad
\EE\bigg[\Bigl(\frac{\bh_m^\top \bh_\ell}{\|\bh_m\|\|\bh_\ell\|} -\eta_{m,\ell}\Bigr)^2\bigg]\lesssim 
 \frac{\gamma }{nc^2\tau^2}.
$$
\end{restatable}
\Cref{lem:contraction_cor} is proved in \Cref{proof:contraction_cor}.
The proof is based on a symmetry argument for the non-negativity of $\eta_{m,\ell}$ and the Gaussian Poincar\'e inequality for the upper bound.

\noindent
Now we provide a high-probability bound for the Ratio. 
Let $U_{m,\ell} := -\bh_m^\top\bh_\ell/\|\bh_m\|\|\bh_\ell\| + \eta_{m,\ell}$. Then, using the positiveness of $\eta_{m,\ell}$ by \Cref{lem:contraction_cor}, we find
\begin{align*}
    \sum_{m}\|\bh_m\|^2 -\|\sum_{m} \bh_m\|^2 
    &=  \sum_{m\neq \ell} (-\bh_m^\top\bh_\ell + \eta_{m, \ell}\|\bh_m\|\|\bh_\ell\| - \eta_{m, \ell}\|\bh_m\|\|\bh_\ell\|) 
    \le \sum_{m\neq \ell} U_{m, \ell} \|\bh_m\|\|\bh_\ell\| \\
    &\le \Bigl(\sum_{m\neq\ell}U_{m\ell}^2\Bigr)^{1/2} \Bigl(\sum_{m\neq \ell}\|\bh_m\|^2\|\bh_\ell\|^2\Bigr)^{1/2} 
    \le \Bigl(\sum_{m,\ell} U_{m\ell}^2\Bigr)^{1/2} \sum_{m}\|\bh_m\|^2.
\end{align*}
Dividing both sides by $\sum_m\|\bh_m\|^2$, we have
$$
\frac{\|\sum_{m} \bh_m\|^2}{\sum_{m}\|\bh_m\|^2} \ge 1-\Bigl(\sum_{m, \ell} U_{m, \ell}^2\Bigr)^{1/2}.
$$
Then, Markov's inequality applied with $\EE[U_{m,\ell}^2] \lesssim \gamma/(nc^2\tau^2)$ by \Cref{lem:contraction_cor} yields
\begin{equation}\label{eq:correlation_bound}
    \PP\bigg(\frac{\sum_{m} \|\bh_{m}\|^2}{\|\sum_{m} \bh_{m}\|^2} \ge  2\bigg) \le
    \PP\bigg( 1- \Bigl(\sum_{m,\ell} U_{m,\ell}^2\Bigr)^{1/2} \le \frac{1}{2}\bigg) 
    \le  \PP\Big(\sum_{m, \ell} U_{m, \ell}^2 \ge 1/4\Big) \le 4\EE\Big[\sum_{m, \ell} U_{m, \ell}^2\Big] \lesssim \frac{M^2\gamma} {nc^2\tau^2},
\end{equation}
which gives the upper bound of Ratio $= \sum_{m} \|\bh_{m}\|^2/{\|\sum_{m} \bh_{m}\|^2}$.

\noindent\textbf{Control of $|d_{m,\ell}^\est|^{-1}$.}
The lemma below gives an lower bound of $d_{m,\ell}^\est$ for $\est\in\{\sub, \full\}$.
\begin{restatable}{lemma}{LemDenomLowerBound}\label{lem:denom_lower_bound}
Let $d_{m, \ell}^\sub$ and $d_{m,\ell}^\full$ be the random variables defined in \eqref{eq:df_denominator_sub_full}. Then, 
there exists a positive absolute constant $C >0$ such that for $\est\in\{\sub,\full\}$, 
\begin{align*}
    \PP\Bigl(\frac{d_{m,\ell}^\est}{n} < C d^\est(\tau, c, \gamma) \Bigr) \lesssim \frac{\gamma^4}{n c^2 \tau^4}
    \text{ where }
    d^\sub(\tau, c,\gamma) = \tau^2 c^4\gamma^{-2}, \text{ and } d^\full(\tau,c,\gamma)=\tau^2\gamma^{-2}.
\end{align*}
\end{restatable}
\Cref{lem:denom_lower_bound} is proved in \Cref{proof:df_lower_bound}.
The proof uses \eqref{eq:property_tr} and some property of the hypergeometric distribution (\Cref{lem:concentration-intersection}).

Combining
\eqref{eq:relative_error_ineq}, \eqref{eq:bound_E_ml_est},
\eqref{eq:correlation_bound}, and \Cref{lem:denom_lower_bound} yield the following for all $\epsilon>0$:
\begin{align*}
\PP\Big(\Big|\frac{\tR_M^\est}{R_M} - 1\Big|>\epsilon\Big) &\le \PP\Big(\text{Ratio} \cdot  \max_{m,\ell} \frac{n}{|d_{m,\ell}^\est|} |E_{m,\ell}^\est| > \frac{\epsilon}{M}\Big) \\
&\lesssim \frac{M^2\gamma}{n\tau^2 c^2} + \PP\Big(2 \max_{m,\ell} \frac{n}{|d_{m,\ell}^\est|} |E_{m,\ell}^\est| > \frac{\epsilon}{M}\Big)\\
&\le \frac{M^2\gamma}{n\tau^2 c^2} + \sum_{m, \ell} \PP \Big(\frac{n}{|d_{m,\ell}^\est|}|E_{m,\ell}^\est| > \frac{\epsilon}{2M}\Big) \\
&\le \frac{M^2\gamma}{n\tau^2 c^2} + \sum_{m,\ell} \bigg\{ \PP\Big(\frac{d_{m,\ell}^\est}{n} < Cd^\est(\tau, c, \gamma)\Big) + \PP\Big(|E_{m,\ell}^\est| > \frac{\epsilon C d^\est(\tau, c, \gamma)}{2M}\Big)
\bigg\}\\
&\lesssim \frac{M^2\gamma}{n\tau^2 c^2} + M^2 \Bigl[\frac{\gamma^4}{n c^2 \tau^4} + \frac{2M}{\epsilon C d^\est(\tau, c, \gamma)} \frac{1}{\sqrt{n}\tau^2} \cdot \begin{dcases}
    c^{-3} \gamma^{7/2} & \est=\sub\\
    c^{-2} \gamma^{5/2} & \est=\full,
\end{dcases} \Bigr]
\end{align*}
where $C>0$ is an absolute constant, and $d^\est(c, \tau, \gamma)$ is given by
$d^\sub(\tau, c, \gamma) = \tau^2 c^4 \gamma^{-2}$ and  $d^\full(\tau, c, \gamma) = \tau^2 \gamma^{-2}$. Substituting this expression to the last bound, we obtain
\begin{align}
    \text{for all } \epsilon>0, \quad \PP\bigg(\Big|\frac{\tR_M^\est}{R_M} - 1\Big|>\epsilon\bigg) &\lesssim \frac{M^2\gamma}{n\tau^2 c^2} + \frac{M^2\gamma^4}{nc^2\tau^4} + \frac{M^3}{\epsilon\sqrt{n} \tau^{2+2}} \cdot  \begin{dcases}
    c^{-3-4} \gamma^{7/2+2} &\est=\sub\\
    c^{-2} \gamma^{5/2+2} & \est=\full
    \end{dcases} \notag\\
    &\lesssim \frac{M^2\gamma^4}{nc^2\tau^4} + \frac{M^3}{\epsilon\sqrt{n} \tau^{4}} \times \begin{dcases}
    c^{-7} \gamma^{11/2} &\est=\sub\\
    c^{-2} \gamma^{9/2} & \est=\full
    \end{dcases} \label{eq:relative_error_epsilon_any},
\end{align}
thereby when $\epsilon\in(0,1)$, the second term dominates the first term. This concludes the proof of \eqref{eq:sub_full_relative_error}.

\subsection{Proof of \Cref{th:correction_gcv} (restated here for convenience)}
\label{proof:th:correction_gcv}

\bigskip

\ThmCorrectionGCV*

The goal in this section is to show $\tR_M^{\cgcv, \est}/R_M \approx 1$ for $\est\in\{\sub, \full\}$. The definition of $\tR_M^{\cgcv, \est}$ is recalled for convenience:
\begin{equation}\label{eq:df_cgcv_est}
\tR_M^{\cgcv, \est} = \frac{\|M^{-1}\sum_m\br_m\|^2}{n(1-\tdf_M/n)^2} - (c^{-1}-1) \frac{(\tdf_M/n)^2}{(1-\tdf_M/n)^2} \frac{1}{M^2} \sum_m \hat{R}_{m,m}^\est.
\end{equation}
Now, we define $\tR^\cgcv_M$ by
\begin{equation}\label{eq:df_cgcv_without_est}
\tR^\cgcv_M := \frac{\|M^{-1}\sum_m\br_m\|^2}{n(1-\tdf_M/n)^2} - (c^{-1}-1) \frac{(\tdf_M/n)^2}{(1-\tdf_M/n)^2} \frac{1}{M^2} \sum_m \|\bh_m\|^2.
\end{equation}
The difference between  $\tR_M^\cgcv$ and $\tR_M^{\cgcv,\est}$  is whether it uses $\|\bh_m\|^2$ or its estimate $\hat{R}^\est_{m,m}$ in the rightmost sum. Considering $\hat{R}_{m,m}^\est \approx \|\bh_m\|^2$ from  \eqref{eq:sub_full_moment_ineq}, $\tR_M^\cgcv$ is naturally expected to be close to $\tR_M^{\cgcv, \est}$. We state this approximation more precisely as follows:
\begin{equation}\label{eq:cgcv_cgcv_est_close}
 \text{for all } \epsilon\in(0,1),\quad \PP\bigg(\Big|\frac{\tR_M^{\cgcv} - \tR_M^{\cgcv, \est}}{R_M}\Big| > \epsilon\bigg) \lesssim \frac{M^2}{\sqrt{n} \epsilon \tau^6} \times \begin{dcases}
        c^{-6} \gamma^{15/2} & \est=\sub\\
        c^{-2} \gamma^{13/2} & \est = \full 
    \end{dcases}.
\end{equation}
\Cref{eq:cgcv_cgcv_est_close} is proved in \Cref{proof:lem:cgcv_cgcv_est_close}.
The second task is to show $\tR^\cgcv_M \approx R_M$. 
\begin{restatable}{lemma}{LemCgcvRiskClose}\label{lem:cgcv_risk_close}
For $\tR_M^{\cgcv}$ defined in \eqref{eq:df_cgcv_without_est}, we have
$$
\text{for all } \epsilon\in(0,1), \quad \PP\bigg(\Big|\frac{\tR_{M}^\cgcv -R_M}{R_M}\Big| > \epsilon\bigg) \lesssim \frac{M^4}{\sqrt{n}\epsilon \tau^5} \cdot c^{-4} \gamma^{13/2}.
$$
\end{restatable}
\Cref{lem:cgcv_risk_close} is proved in \Cref{proof:lem:cgcv_risk_close}, where we use some concentration of $\df_m$ around its average $\tdf_M = M^{-1}\sum_{m=1}^M\df_m$. Now we assume \Cref{lem:cgcv_risk_close}. Then, 
\Cref{eq:cgcv_cgcv_est_close} and \Cref{lem:cgcv_risk_close} together yield
\begin{align*}
    &\PP\Bigl(|{\tR_M^{\cgcv,\est}}/{R_M} -1| > \epsilon\Bigr) \le \PP \Big(|({\tR^{\cgcv,\est}_M-\tR_M^\cgcv})/{R_M}| > \epsilon/2 \Big) + \PP\Big(|{\tR_M^\cgcv}/{R_M} - 1| > \epsilon/2\Bigr)
    \\
    &\lesssim
\frac{M^4}{\sqrt{n}\epsilon \tau^5} \cdot c^{-4} \gamma^{13/2} + \frac{M}{\sqrt{n} \epsilon \tau^6} \cdot \begin{dcases}
        c^{-6} \gamma^{15/2} & \est=\sub\\
        c^{-2} \gamma^{13/2} & \est = \full 
    \end{dcases}
    \lesssim \frac{M^4}{\sqrt{n}\epsilon\tau^6} \cdot \begin{dcases}
        c^{-6} \gamma^{15/2} & \est=\sub\\
        c^{-4} \gamma^{13/2} & \est = \full 
    \end{dcases},
\end{align*}
for all $\epsilon\in(0,1)$. 
This completes the proof of \Cref{th:correction_gcv}. 
In the following sections, we prove \Cref{eq:cgcv_cgcv_est_close} and \Cref{lem:cgcv_risk_close}.
\subsubsection{Proof of \Cref{eq:cgcv_cgcv_est_close}}
\label{proof:lem:cgcv_cgcv_est_close}
By the definition of $\tR^{\cgcv, \est}$ in \eqref{eq:df_cgcv_est} and $\tR^\cgcv$ in \eqref{eq:df_cgcv_without_est}, 
$$
\frac{\tR_M^{\cgcv} - \tR_M^{\cgcv, \est}}{R_M} = (c^{-1}-1) \bigg(\frac{\tdf_M/n}{1-\tdf_M/n}\bigg)^2 \frac{\sum_m(\hat{R}_{mm}^\est - \|\bh_m\|^2)}{\|\sum_m\bh_m\|^2}.
$$
Here, thanks to $(c^{-1}-1) (\tdf_M/n)^2 \le (c^{-1}-1) (k/n)^2 = (c^{-1}-1) c^2 \le c$ by \eqref{eq:property_tr}, we have
\begin{align*}
    \bigg|\frac{\tR_M^{\cgcv} - \tR_M^{\cgcv, \est}}{R_M}\bigg| 
    &\le \frac{c}{(1-\tdf_M/n)^2} \frac{\sum_m\|\bh_m\|^2}{\|\sum_m\bh_m\|^2}\max_{m} |\frac{\hat{R}_{mm}^\est}{\|\bh_m\|^2}-1| = c \cdot U \cdot \max_{m} V_m^\est,
\end{align*}
where $U$ and $V_m^\est$ are defined as 
\begin{equation}\label{eq:df_U_V_m}
U:= \frac{1}{(1-\tdf_M/n)^2} \frac{\sum_m\|\bh_m\|^2}{\|\sum_m\bh_m\|^2}, 
\quad\qquad
V_{m}^\est := \bigg|\frac{\hat{R}_{mm}^\est}{\|\bh_m\|^2}-1\bigg|.    
\end{equation}

\noindent\textbf{Control of $U$.}
By \Cref{lem:df_upperbound} (introduced later),
there exists an absolute constant $C\in(0,1)$ such that 
$$
\PP(1-\df_m/n \le C\tau\gamma^{-1}) \le e^{-nc/2},
$$
Applying the above display to $1-\tdf_M/n = M^{-1} \sum_{m=1}^M (1-\df_m/n)$ with the union bound, we have
$$
\PP\Bigl(\frac{1}{(1-\tdf_M/n)^2}\ge \frac{\gamma^2}{C^2 \tau^2}\Bigr) = \PP(1-\tdf_M/n \le C \tau \gamma^{-1}) \le 
\sum_{m=1}^M \PP(1-\df_m/n \le C\tau\gamma^{-1}) \lesssim M e^{-nc/2}.
$$ 
Combining the above display and the upper bound of ${\sum_{m} \|\bh_{m}\|^2}/{\|\sum_{m} \bh_{m}\|^2}$ given by \eqref{eq:correlation_bound}, we have
\begin{equation}\label{eq:U_upper_bound}
\PP\Bigl(U > \frac{2\gamma^2}{C^2 \tau^2}\Bigr) \le \PP\Bigl(\frac{\sum_{m} \|\bh_{m}\|^2}{\|\sum_{m} \bh_{m}\|^2}>2\Bigr) + \PP\Bigl(\frac{1}{(1-\tdf_M/n)^2}\ge \frac{\gamma^2}{C^2 \tau^2}\Bigr)
\lesssim \frac{M^2\gamma^2}{n \tau^2 c^2} +  \frac{M}{e^{nc/2}} \lesssim\frac{M^2\gamma^2}{n \tau^2 c^2}
\end{equation}

\noindent\textbf{Control of $V_m^\est$.}
\Cref{eq:relative_error_epsilon_any} with $M=1$ implies
\begin{equation}\label{eq:V_m_upper_bound}
\text{for all } \epsilon>0, \quad \PP(V_m^\est>\epsilon) \lesssim \frac{\gamma^4}{nc^2\tau^4} +  \frac{1}{\epsilon \sqrt{n}\tau^4} \cdot 
\begin{dcases}
    c^{-7} \gamma^{11/2} & \est=\sub\\
    c^{-2} \gamma^{9/2} & \est = \full 
\end{dcases}    
\end{equation}

Now we have controlled $U$ and $V_m^\est$. \eqref{eq:U_upper_bound} and \eqref{eq:V_m_upper_bound} together yield, for all $\epsilon\in(0,1)$, 
\begin{align*}
\PP\bigg( \Big|\frac{\tR_M^{\cgcv} - \tR_M^{\cgcv, \est}}{R_M}\Big|>\epsilon\bigg) &\le \PP(c U \max_m V_{m}^\est> \epsilon) 
\le \PP\Bigl(U > \frac{2\gamma^2}{C^2 \tau^2}\Bigr) + \PP\Bigl(c\frac{2\gamma^2}{C^2\tau^2} \max_{m} V_M^\est > \epsilon\Bigr)\\
&\lesssim \frac{M^2\gamma^2}{n \tau^{2} c^2}   + \sum_m \PP\Big(V_m^\est > \frac{\epsilon C^2\tau^2}{2c \gamma^2}\Big) \text{ (by \eqref{eq:U_upper_bound})}\\
&\lesssim \frac{M^2\gamma^2}{n \tau^{2} c^2} + \frac{M\gamma^4}{n\tau^4c^2} + \frac{M}{\sqrt{n} \tau^4} \bigg(\frac{\epsilon C^2\tau^2}{2c \gamma^2}\bigg)^{-1} \cdot
    \begin{dcases}
        c^{-7} \gamma^{11/2} & \est=\sub\\
        c^{-2} \gamma^{9/2} & \est = \full 
    \end{dcases} \text{ (by \eqref{eq:V_m_upper_bound})} \\
&\lesssim \frac{M^2\gamma^4}{n \tau^{4} c^2} + \frac{M}{\sqrt{n} \epsilon \tau^6} \times \begin{dcases}
        c^{-6} \gamma^{15/2} & \est=\sub\\
        c^{-1} \gamma^{13/2} & \est = \full 
    \end{dcases} \text{ (thanks to $\tau\in(0,1]$ and $\gamma\ge 1$)}\\
    &\lesssim \frac{M^2}{\sqrt{n} \epsilon \tau^6} \times \begin{dcases}
        c^{-6} \gamma^{15/2} & \est=\sub\\
        c^{-2} \gamma^{13/2} & \est = \full 
    \end{dcases} \text{ (thanks to $\epsilon\in(0,1)$)}
\end{align*}
This concludes the proof.

\subsubsection{Proof of \Cref{lem:cgcv_risk_close} (restated here for convenience)}\label{proof:lem:cgcv_risk_close}
\LemCgcvRiskClose*
\begin{proof}
If we define $d_{m,\ell}^\cgcv$ by
$$
d_{m,\ell}^\cgcv := n\Bigl\{\Big(1-\frac{\tdf_M}{n}\Big)^2 + (c^{-1}-1) \ind_{\{m=\ell\}}\Big(\frac{\tdf_M}{n}\Big)^2\Bigr\}
$$
the relative error can be written as 
$$
\frac{\tR_M^\cgcv}{R_M} - 1 =\frac{\|\sum_m \br_m\|^2 - n(c^{-1}-1) (\tdf_M/n)^2 \sum_m \|\bh_m\|^2}{n(1-\tdf_M/n)^2\|\sum_m \bh_m\|^2} - 1= \frac{\sum_{m, \ell} (\br_m^\top \br_\ell - d_{m,\ell}^\cgcv\bh_m^\top\bh_\ell)}{n(1-\tdf_M/n)^2 \|\sum_m\bh_m\|^2}.
$$
By the same argument in \eqref{eq:relative_error_ineq}, the relative error is bounded from above as
\begin{align*}
    \Bigl|\frac{\tR_M^\cgcv}{R_M} - 1 \Bigr| 
    &\le \frac{1}{(1-\tdf_M/n)^2}  \frac{(\sum_{m=1}^M\|\bh_m\|)^2}{\|\sum_{m=1}^M\bh_m\|^2} \cdot \max_{m, \ell} \frac{|\br_m^\top\br_\ell - d_{m,\ell}^\cgcv\bh_m^\top\bh_\ell|}{n\|\bh_m\|\|\bh_\ell|}\\
    &\le M \cdot \frac{1}{(1-\tdf_M/n)^2}  \frac{\sum_m\|\bh_m\|^2}{\|\sum_m\bh_m\|^2} \cdot \max_{m,\ell} (\frac{|\br_m^\top\br_\ell - d_{m,\ell}^\full \bh_m^\top \bh_\ell |}{n \|\bh_m\|\|\bh_\ell\|} + \frac{|d_{m,\ell}^\cgcv - d_{m,\ell}^\full|}{n})\\
    &= M \cdot U \cdot \max_{m,\ell}\Big(|E_{m,\ell}^\full| + |d_{m,\ell}^\cgcv-d_{m,\ell}^\full|/n\Big),
\end{align*}
where $E_{m,\ell}^\full$ and $U$ were defined before by \eqref{eq:df_ratio_E_ml_est} and \eqref{eq:df_U_V_m}. Their expressions are recalled here for convenience:
$$
    U := \frac{1}{(1-\tdf_M/n)^2} \frac{\sum_m\|\bh_m\|^2}{\|\sum_m\bh_m\|^2}, \quad E_{m,\ell}^\full := \frac{\br_m^\top \br_\ell - d_{m,\ell}^\full \bh_m^\top\bh_\ell}{n\|\bh_m\|\bh_\ell\|} = \frac{d_{m,\ell}^\full(\hat{R}_{m,\ell}^\full-\bh_m^\top\bh_\ell)}{n\|\bh_m\|\bh_\ell|},
$$
where $d_{m,\ell}^\full = n-\df_m-\df_\ell + k^{-2}{|I_m\cap I_\ell|}\df_m\df_\ell$ is the denominator in the $\full$-estimator.
Note in passing that $U$ and $E_{m,\ell}^\full$ have been already bounded by 
$\eqref{eq:bound_E_ml_est}$ with $\est=\full$ and \eqref{eq:U_upper_bound}: there exists a absolute constant $C$ such that 
\begin{equation}\label{eq:bound_E_ml_full_U}
 \PP(U > C \tau^{-2}\gamma^2) \lesssim  \frac{M^2\gamma^2}{n\tau^2 c^2}, \quad\quad 
\PP(|E_{m,\ell}^\full| > \epsilon) \lesssim\frac{\tau^{5/2}}{\epsilon\sqrt{n}\tau^2 c^2} \text{ for all $\epsilon>0$}.
\end{equation}
It remains to bound  $|d_{m,\ell}^\cgcv-d_{m,\ell}^\full|/n$. 
Here, we will argue that
\begin{equation}\label{eq:concentrate_denom_full}
     \text{for all } m, \ell\in[M], \quad \text{for all }\epsilon>0, \quad \PP\Big(\frac{|d_{m,\ell}^\full - d_{m,\ell}^\cgcv|}{n} > \epsilon\Big) \lesssim M \Bigl(e^{-nc/2} + \frac{\gamma^{9/2}}{\epsilon\sqrt{n} \tau^3 c^4}\Bigr).
\end{equation}
The proof of \eqref{eq:concentrate_denom_full} is given at the end of this section. Now we assume it. Then \eqref{eq:bound_E_ml_full_U} and \eqref{eq:concentrate_denom_full}  together yield, for all $\epsilon\in(0,1)$, 
\begin{align*}
    \PP\Big(|\frac{\tR_M^\cgcv}{R_M} - 1|>\epsilon\Big) &\le \PP \Big( U \cdot \max_{m,\ell} (|E_{m,\ell}^\full| + \frac{|d_{m,\ell}^\cgcv-d_{m,\ell}^\full|}{n}) > \frac{\epsilon}{M}\Big)\\
    &\le \PP\Big(U > \frac{C \gamma^2}{\tau^2}\Big) + \PP\Big(\max_{m,\ell} (|E_{m,\ell}^\full| + \frac{|d_{m,\ell}^\cgcv-d_{m,\ell}^\full|}{n}) > \frac{\epsilon}{M} \frac{\tau^2}{C\gamma^2}\Big)\\
    &\lesssim\frac{M^2\gamma^2}{n\tau^2 c^2} + \sum_{m,\ell} \PP\Big(|E_{m,\ell}^\full|> \frac{\epsilon\tau^2}{2CM\gamma^2}\Big) + \PP(\frac{|d_{m,\ell}^\cgcv-d_{m,\ell}^\full|}{n}> \frac{\epsilon\tau^2}{2CM\gamma^2})\\
    &\lesssim \frac{M^2\gamma^2}{n\tau^2 c^2} + M^2 \bigg\{
    \frac{2CM\gamma^2}{\epsilon\tau^2} \cdot 
    \frac{\gamma^{5/2}}{\sqrt{n}\tau^2 c^2}
    +  M e^{-nc/2} + \frac{2CM\gamma^2}{\epsilon\tau^2} \cdot 
    \frac{M\gamma^{9/2}}{\sqrt{n}\tau^3 c^4}
    \bigg\}\\
    &\lesssim \frac{M^4 \gamma^{13/2}}{\sqrt{n}\epsilon\tau^5 c^4} \quad \text{(thanks to  $\epsilon, c, \tau \in(0,1]$ and $\gamma\ge1$}),
\end{align*}
which concludes the proof of \Cref{lem:cgcv_risk_close}.
\end{proof}

\noindent \textbf{\underline{Proof of Equation~\eqref{eq:concentrate_denom_full}}.}
The expressions of $d_{m,\ell}^\full$ and $d_{m,\ell}^\cgcv$ are recalled here for convenience:
$$
    \frac{d_{m,\ell}^\full}{n} = 1 - \frac{\df_m}{n} - \frac{\df_\ell}{n} + \frac{|I_m\cap I_\ell|}{nk^2}\df_m\df_\ell, \quad \frac{d_{m,\ell}^\cgcv}{n} =  \Big(1-\frac{\tdf_M}{n}\Big)^2 + (c^{-1}-1) \ind_{\{m=\ell\}}\Big(\frac{\tdf_M}{n}\Big)^2. 
$$
Below we prove $d_{m,\ell}^\full/n \approx d_{m,\ell}^\cgcv/n$. 
The key lemma is the concentration of $\df_m$ around its average $\tdf_M = \sum_m \df_m/M$. 
\begin{restatable}{lemma}{LemConcentrationDf}\label{lem:concentration_df}
Suppose the same penalty is used for $(\bh_m)_{m=1}^M$. Then, we have
$$
\text{for all } m\in[M], \quad \text{for all }\epsilon>0, \quad \PP\Big(\frac{|\df_m - \tdf_M|}{n} > \epsilon\Big) \lesssim M\Big(e^{-nc/2} + 
\frac{\gamma^{9/2}}{\epsilon\sqrt{n}\tau^3 c^4}\Big).
$$
\end{restatable}
The proof of \Cref{lem:concentration_df} is given in  \Cref{proof:concentrate_denom_full}.
From \Cref{lem:concentration_df}, it suffices to bound $|d_{m,\ell}^\full-d_{m,\ell}^\cgcv|/n$ from above by $|\df_m-\tdf_M|/n$ and $|\df_\ell-\tdf_M|/n$ up to an absolute constant. Below, we prove \eqref{eq:concentrate_denom_full} for $m=\ell$ and $m\neq \ell$ separately. 

\noindent\textbf{When $m=\ell$. }
Letting $f$ be the function
$
f(x) = 1 -2x + c^{-1} x^2 = (1-x)^2 + (c^{-1}-1)x^2
$, we have
$$
\frac{d_{m,m}^\full}{n} = 1- 2 \frac{\df_m}{n} + \frac{\df_m^2}{nk}= f(\frac{\df_m}{n}), \quad \frac{d_{m,m}^\cgcv}{n} = (1-\frac{\tdf_M}{n})^2 + (c^{-1}-1) (\frac{\tdf_M}{n})^2 = f(\frac{\tdf_M}{n}).
$$
Here, $\df_m/n, \tdf_M/n \in [0, c)$ by \eqref{eq:property_tr}, while 
$f$ is $2$-Lipschitz on $[0,c]$ since $\sup_{x \in [0,c]} |f'(x)| = \sup_{x\in [0,c]} 2(1-x/c) = 2$. Thus, we obtain that, for all $\epsilon>0$, 
$$
\PP\Big(\frac{|d_{m,m}^\full - d_{m,m}^\cgcv|}{n} >\epsilon\Big) = \PP\Big(\Big|f\big(\frac{\df_m}{n}\big)-f\big(\frac{\tdf_M}{n}\big)\Big|> \epsilon\Big) \le \PP\Big( 2\Big|\frac{\df_m}{n}-\frac{\tdf_M}{n}\Big| > \epsilon\Big)\lesssim M\Big(e^{-nc/2} + 
\frac{\gamma^{9/2}}{\epsilon\sqrt{n}\tau^3 c^4}\Big),
$$
thanks to the Lipschitz property of $f$ and \Cref{lem:concentration_df}. 
This completes the proof of \eqref{eq:concentrate_denom_full} for $m=\ell$. 

\noindent\textbf{When $m\neq \ell$.}  Letting $g$ be the function $g(x,y)=(1-x)(1-y)$, we have 
\begin{align*}
\frac{d_{m, \ell}^\full}{n} = g\Big(\frac{\df_m}{n}, \frac{\df_\ell}{n}\Big) +  \frac{\df_m}{n}\frac{\df_\ell}{n}  \Big(\frac{|I_m\cap I_\ell| \cdot n}{k^2}-1\Big), \qquad  \frac{d_{m,\ell}^\cgcv}{n} = g \Big(\frac{\tdf_M}{n}, \frac{\tdf_M}{n}\Big).
\end{align*}
Here, $g(x,y)=(1-x)(1-y)$ satisfies the following inequality:
for all $x,x', y,y'\in [0,1]$,
\begin{align*}
|g(x, y) - g(x',y')| 
\le |g(x, y)-g(x',y)| + |g(x', y)-g(x', y')| \le |x-x'| + |y-y'|. 
\end{align*}
From this property of $g$ and $\df_m/n, \df_\ell/n, \tdf_M/n \in [0,c) \subset[0,1]$ by \eqref{eq:property_tr}, we find
\begin{align*}
\frac{|d_{m, \ell}^\full - d_{m,\ell}^\cgcv|}{n} &\le \frac{\df_m}{n}\frac{\df_\ell}{n}  \Bigl|\frac{|I_m\cap I_\ell| \cdot n}{k^2}-1\Bigr| + \Big|g\Big(\frac{\df_m}{n}, \frac{\df_\ell}{n}\Big)-g\Big(\frac{\tdf_M}{n}, \frac{\tdf_M}{n}\Big)\Big| \\
&\le c^2 \Bigl|\frac{|I_m\cap I_\ell| \cdot n}{k^2}-1\Bigr| + 
\frac{|\df_m-\tdf_M|}{n} + \frac{|\df_\ell-\tdf_M|}{n}.
\end{align*}
Here, thanks to the bound of of $\Var[|I_m \cap I_\ell|]$ (see \Cref{lem:concentration-intersection}), 
the moment of the first term on the right-hand side is bounded from above as
\[
\EE\Bigl[c^2 \Bigl|\frac{|I_m\cap I_\ell| \cdot n}{k^2}-1\Bigr|\Bigr] = c^2 \frac{n}{k^2} \EE\Bigl[\Bigl||I_m\cap I_\ell|-\frac{k^2}{n}\Bigr|\Bigr] 
\le \frac{1}{n} \sqrt{ \EE\Bigl[\Bigl||I_m\cap I_\ell|-\frac{k^2}{n}\Bigr|^2\Bigr] } \le \frac{1}{n} \sqrt{\frac{k^2}{n}} = \frac{c}{\sqrt{n}}.
\]
Therefore, Markov's inequality applied with the above moment bound \Cref{lem:concentration_df} results in
\begin{align*}
\PP\Big(\frac{|d_{m, \ell}^\full - d_{m,\ell}^\cgcv|}{n}>\epsilon\Big) &\le \PP\Big(c^2 \Bigl|\frac{|I_m\cap I_\ell| \cdot n}{k^2}-1\Bigr|>\frac{\epsilon}{3}\Big) + \PP\Big(\frac{|\df_m-\tdf_M|}{n} > \frac{\epsilon}{3}\Big) + \PP\Big(\frac{|\df_\ell-\tdf_M|}{n} > \frac{\epsilon}{3}\Big)\\
&\lesssim \frac{c}{\epsilon\sqrt{n}} + M\Big(e^{-nc/2} + 
\frac{\gamma^{9/2}}{\epsilon\sqrt{n}\tau^3 c^4}\Big) 
\lesssim M\Big(e^{-nc/2} + 
\frac{\gamma^{9/2}}{\epsilon\sqrt{n}\tau^3 c^4}\Big),
\end{align*}
for all $\epsilon>0$. This completes the proof of Equation~\eqref{eq:concentrate_denom_full} for $m\neq\ell$. 
\qedsymbol{}

\subsection{Proof of \Cref{prop:gcv_inconsistency}}
\label{app:proof:prop:gcv_inconsistency}
Define $\text{Correction}$ as 
$$
 \text{Correction} :=  (c^{-1}-1) \frac{(\tdf_M/n)^2}{(1-\tdf_M/n)^2} \frac{\sum_m \|\bh_m\|^2}{\|\sum_m\bh_m\|^2},
$$
so that \Cref{lem:cgcv_risk_close} can be written as 
\begin{equation}\label{eq:gcv_C_n_1}
\PP\bigg(\Big|\frac{\tR_M^\gcv}{R_M} - \text{Correction} -  1\Big| > \epsilon\bigg) \lesssim \frac{M^4  \gamma^{13/2}}{\epsilon \sqrt{n}\tau^5 c^4} \text{ for all } \epsilon\in(0,1). 
\end{equation}
Since $\|\sum_m\bh_m\|^2 \le (\sum_m \|\bh_m\|)^2 \le M\sum_m \|\bh_m\|^2$ by the triangle inequality, it holds that
$$
\forall \delta\in(0,1), \quad \Bigl\{{\tdf_M}/{k} \ge \delta\Bigr\} = \Bigl\{{\tdf_M}/{n} \ge c\delta\Bigr\} \subset 
\Bigl\{\text{Correction} \ge (c^{-1}-1) \bigg(\frac{c\delta}{1-c\delta}\bigg)^2 \frac{1}{M} \ge \frac{\delta^2 c(1-c)}{M}\Bigr\}.
$$
Therefore, we obtain that, for all $\delta\in(0,1)$, 
\begin{align*}
    \PP\bigg(\frac{\tdf_M}{k} \ge \delta\bigg) &\le \PP\bigg(\text{Correction} \ge
    \frac{\delta^2 c(1-c)}{M}\bigg)\\
    &\le \PP\bigg(\frac{\tR_M^\gcv}{R_M} - \text{Correction} -  1  \le -\frac{\delta^2 c(1-c)}{2M}\bigg) + \PP\bigg(\frac{\tR^\gcv_M}{R_M}-1 > \frac{\delta^2 c(1-c)}{2M}\bigg)\\
    &\le C \bigg(\frac{\delta^2 c(1-c)}{2M}\bigg)^{-1} \frac{M^4  \gamma^{13/2}}{\sqrt{n}\tau^5 c^4} + \PP\bigg(\frac{\tR^\gcv_M}{R_M} > 1 + \frac{c(1-c)\delta^2}{2M}\bigg) \text{ (by \eqref{eq:gcv_C_n_1} with $\epsilon=\frac{\delta^2c(1-c)}{2M}$)} \\
    &\le C' \frac{M^5}{\sqrt{n}\tau^5 \delta^2} \frac{\gamma^{13/2}}{c^5(1-c)} + \PP\bigg(\frac{\tR^\gcv_M}{R_M} > 1 + \frac{c(1-c)\delta^2}{2M}\bigg),
\end{align*}
where $C,C'>0$ are absolute constants. This finishes the proof.

\subsection{Technical lemmas and their proofs}
\label{sec:technical-lemmas-nonasymp}

\subsubsection{Proof of \Cref{lem:derivative} (restated here for convenience)}
\label{proof:lem:derivative}

\bigskip

\LemDerivative*

\begin{proof}
By the change of variable $\bbeta \mapsto \bu = \bSigma^{1/2}(\bbeta-\bbeta_0)$ and $\bG = \bX\bSigma^{-1/2}$, we have
$$
\bSigma^{1/2}(\hat{\bbeta} -\bbeta_0) = \hat{\bu}, \quad \br = \by - \bX\hbeta =\bepsilon - \bG\hat{\bu}, \quad  \df = \tr[\bX (\partial/\partial \by) \hat{\bbeta}] = \tr[\bG (\partial/\partial\bepsilon)\hat{\bu}], 
$$
where $\hat{\bu}$ is a penalized estimator with an isotropic design $\bG = \bX\bSigma^{-1/2}$:
$$
\hat{\bu} = \hat{\bu}(\bepsilon, \bG) := \argmin_{\bu\in\R^{p}} \frac{1}{k}\sum_{i\in I} (\epsilon_i - \bg_i^\top \bu)^2 + f(\bu) \text{ with } f(\bu) := g(\bSigma^{-1/2} \bu + \bbeta_0).
$$
Note in passing that the map $\bu \in\R^{p} \mapsto f(\bu) - \mu \|\bu\|^2/2$ is convex thanks to \Cref{assu:penalty}. Then, \citet[Theorem 1]{bellec2022derivatives} with $\bSigma=\bI_p$ implies the followings:
there exists a matrix $\bA\in \R^{p\times p}$ depending on $(\epsilon_i, \bg_i)_{i\in I}$ such that 
$$
\|\bA\|_{\oper} \leq (|I| \mu)^{-1} = (k\mu)^{-1},
$$ and the derivative of $\hat{\bu}$ with respect to $\bG=(g_{ij})_{ij}\in \R^{n\times p}$ and $\bepsilon=(\epsilon_i)\in\R^n$ are  given by
\begin{align*}
  \text{for all } i\in [n], j\in [p], \quad  &\frac{\partial\hat{\bu}}{\partial g_{ij}}(\bepsilon,\bG)
= \begin{dcases}
    \bA [\be_j \be_{i}^\top\br - \bG^\top \be_i \be_j \hat{\bu}] & (i\in I)\\
    \bm{0}_p & (i\notin I)
\end{dcases} = \bA [{\be_j} (\bL\br)_i - \bG^\top\bL \be_i \hat{u}_j ], 
\\
\text{for all } i\in[n], \quad &\frac{\partial\hat{\bu}}{\partial \epsilon_i}(\bu,\bG)
= \begin{dcases}
    \bA \bG^\top \be_i & (i\in I)\\
    \bm{0}_p & (i\notin I)
\end{dcases} = \bA \bG^\top \bL \be_i,
\end{align*}
where $\bL=\sum_{i\in I} \be_i\be_i^\top$. 
Furthermore, equation (D.8) (and the following argument) in \cite{bellec2022derivatives} imply that 
$\bV := \bI_n - \bG \bA \bG^\top \bL$ satisfies
$$
0 < k(1+\|\bG\|_{\oper}^2/(k\mu))^{-1} \le \tr[\bL\bV] = k - \df \le k, \quad \|\bL\bV\|_{\oper} \le 1.
$$
Thus, if we define the matrix $\bB\in\R^{(p+1)\times(p+1)}$ as
$
\bB := \begin{pmatrix}
\bA & \bm{0}_p \\
\bm{0}_p^\top & 0
\end{pmatrix}, 
$
the derivative of $\bh = (\hat{\bu}^\top, -\sigma)^\top\in\R^{p+1}$
with respect to $\bZ = [\bG| \sigma^{-1}\bepsilon] = (z_{ij})_{ij} \in \R^{n\times (p+1)}$ can be written as 
\begin{align*}
i\in [n], \ 1\le j\le p, \quad \frac{\partial\bh}{\partial z_{ij}} &= \frac{\partial}{\partial g_{ij}} \begin{pmatrix}
    \hat{\bu}\\
    -\sigma
\end{pmatrix} = 
\begin{pmatrix}
  \bA [{\be_j} (\bL\br)_i - \bG^\top\bL \be_i \hat{h}_j ]\\
    0
\end{pmatrix}
=  \bB \Bigl[\begin{pmatrix}
    \be_j\\
   0
\end{pmatrix} (\bL\br)_i - \bZ^\top \bL\be_i \hat{h}_j\Bigr],\\
i\in[n], \quad \frac{\partial\bh}{\partial z_{i, p+1}} &=\sigma \frac{\partial}{\partial \epsilon_{i}} \begin{pmatrix}
    \hat{\bu}\\
    -\sigma
\end{pmatrix}  = \begin{pmatrix}
      \sigma \bA\bG^\top \bL \be_i\\
      0
  \end{pmatrix}
=  \bB \Bigl[\begin{pmatrix}
    \bm{0}_p \\
   1
\end{pmatrix} (\bL\br)_i - \bZ^\top \bL\be_i  \hat{h}_{p+1}\Bigr].
\end{align*}
This finishes the proof of the derivative formula \eqref{eq:derivartive_formula}. 
It remains to prove $\tr[\bB] \ge 0$ in \eqref{eq:property_B}. By the definition of $\bB$, it suffices to show $\tr[\bA] \ge 0$.
Equation (7.4) in \cite{bellec2022observable} implies 
$
\bv^\top \bA \bv \ge 0
$ for all $\bv\in \ker(\bA)^\perp$. Then, letting $\bV=(\bv_1, \bv_2, \dots, \bv_p)\in\R^{p\times p}$ be an orthogonal matrix with its columns including an orthogonal basis of $\ker(\bA)^\perp$, we have
$
0 \le \sum_{i=1}^p \bv_i^\top \bA \bv_i = \tr[\bV^\top \bA \bV] = \tr[\bA \bV\bV^\top] = \tr[\bA].
$ This completes the proof.
\end{proof}

\subsubsection{Proof of \Cref{lem:psi_bound} (restated here for convenience)}\label{proof:psi_bound}

\bigskip

\LemPsiBound*

\begin{proof}
The proof is based on the moment inequality in \Cref{prop:sure_bound} ahead. We will bound $\Xi_J$ in \Cref{prop:sure_bound} using the derivative formula \eqref{eq:derivartive_formula} and the bound of the operator norm $\|\bB\|_{\oper} \le (k\mu)^{-1}$ by \eqref{eq:property_B}. 
Note in passing that the derivative formula \eqref{eq:derivartive_formula} implies 
$$
\sum_{i\in J} \sum_{j=1}^{p+1} \Big\|\frac{\partial \bh}{\partial z_{ij}}\Big\|^2 = \sum_{i\in J}\sum_{j=1}^{p+1} \|\bB \be_j (\bL \br)_i - \bB \bZ^\top \bL \be_i  h_j\|^2 \lesssim \|\bB\|_F^2 \|\bL_J \br\|^2 + \|\bL_J\bZ \bB^\top\|_F^2\|\bh\|^2,
$$
so that $\br = -\bZ\bh$ and $\|\bB\|_{F}^2 \le \rank(\bB)\|\bB\|_{\oper}^2 \le p (k\mu)^{-2}$ lead to 
\begin{equation}\label{eq:Xi_bound}
       \Xi_J = \sum_{i\in J} \|\bh\|^{-2} \sum_{j=1}^{p+1} \Big\| \frac{\partial \bh}{\partial z_{ij}} \Big\|^2
        \lesssim \|\bB\|_F^2 \|\bL_J \bZ\|_{\oper}^2 \lesssim p (k\mu)^{-2}\|\bL_J \bZ\|_{\oper}^2.
\end{equation}
Thus, \Cref{prop:sure_bound} and $\EE[\|\bL_J\bZ\|_{\oper}^4] \lesssim (|J|+p)^2$ by \eqref{lem:gaussian_oper_bound} yield
$$
\EE\bigg[
    \frac{1}{\|\bh\|^2 \|\tilde\bh\|^2} \Bigl(\br \bL_J \tilde\br + \sum_{i\in J}\sum_{j=1}^{q} \frac{\partial r_i \tilde h_j }{\partial z_{ij}}
    \Bigr)^2 \mathrel{\Big |} J
    \bigg] \lesssim |J| +  p (|J|+p)^2 (k\mu)^{-2} \lesssim \tau^{-2} (|J| + p(|J|+p)^2/k^2), 
$$
thanks to $\tau=\min(1, \mu)$. It remains to bound the error inside the square in the LHS; the derivative formula \eqref{eq:derivartive_formula} leads to
$$
\frac{1}{\|\bh\|\|\tilde\bh\|} \sum_{i\in K}\sum_{j=1}^{p+1}\Bigl(
\frac{\partial (\tilde h_j r_i)}{\partial z_{ij}} - \tilde \br^{\top}\tilde \bL \bL_J \br  \trace [\tilde \bB] + \tr[\bL_J \bV] \tilde \bh^{\top}\bh\Bigr)
= \ddfrac{- \tilde \bh^{\top} \tilde \bB \bZ^{\top} \tilde \bL \bL_J \br
- \br^{\top} \bL \bL_J \bZ \bB \tilde \bh}{\|\bh\|\|\tilde\bh\|}.
$$
Here, the square of the RHS can be bounded from above by
$4\mu^{-2} k^{-2} \|\bL_J \bZ\|_{\oper}^4$.
From this and  $\EE[\|\bL_J\bZ\|_{\oper}^4] \lesssim (|J|+p)^2$ by \eqref{lem:gaussian_oper_bound}, we find that the expectation of the RHS is bounded from above by $(|J|+p)^2/(k\mu)^{-2}$ up to some absolute constant. This finishes the proof.
\end{proof}

\begin{lemma}[{Variant of \citet[Proposition 6.1]{bellec2020out}}]\label{prop:sure_bound}
Assume that $\bh$ and $\tilde\bh$ are locally Lipschitz function from $\R^{n\times q} \to \R^q$ such that $\|\bh\|^2, \|\tilde\bh\|^2 \neq 0$, and let $\br = -\bZ\bh $ and $\tilde\br = -\bZ \tilde\bh$, where 
$\bZ\in \R^{n\times q}$ has i.i.d.\ $\cN(0,1)$ entries. If 
$J\subset[n]$ is independent of $\bZ$, we have
\begin{align*}
\EE\bigg[
    \frac{1}{\|\bh\|^2 \|\tilde\bh\|^2} \Bigl(\br \bL_J \tilde\br + \sum_{i\in J}\sum_{j=1}^{q} \frac{\partial r_i \tilde h_j }{\partial z_{ij}}
    \Bigr)^2 \mathrel{\Big |} J
    \bigg] 
    \lesssim |J| +  \EE [\|\bL_J \bZ\|_{\oper}^2 (1 + \Xi_J + \tilde\Xi_J)],
\end{align*}
where $\bL_J = \sum_{i\in J} \be_i\be_i^\top$, $\Xi_J = \sum_{i\in J} \sum_{j=1}^{q} \|\bh\|^{-2} \Big\|\frac{\partial \bh}{\partial z_{ij}}\Big\|^2$, and $\tilde\Xi_J = \sum_{i\in J} \sum_{j=1}^{q} \|\tilde \bh\|^{-2} \Big\|\frac{\partial \tilde\bh}{\partial z_{ij}}\Big\|^2$.
\end{lemma}
\begin{proof}
    Let $\brho = \bL_J \br/\|\bh\| = - \bL_J \bZ \bh /\|\bh\|\in \R^n$ and $\tilde{\bm \eta}:=  \tilde\bh/\|\tilde\bh\|\in \R^{q}$. Then,  Proposition 6.1. in \cite{bellec2020out} implies
    \begin{equation}\label{eq:sure_quadratic}
        \EE \bigg[\Big(-\brho^\top \bZ \tilde{\bm\eta} + \sum_{i,j} \frac{\partial (\rho_i \tilde\eta_j)}{\partial z_{ij}}\Big)^2\bigg] \le \EE[\|\brho\|^2 \|\tilde{\bm\eta}\|^2] + \sum_{ij} \EE\bigg[\|\brho\|^2 \Big\|\frac{\partial\tilde{\bm \eta}}{\partial z_{ij}}\Big\| + \|\tilde{\bm\eta}\|^2 \Big\|\frac{\partial \brho}{\partial z_{ij}}\Big\|^2\bigg],
    \end{equation}
    where 
    $\EE[\|\brho\|^2 \|\tilde{\bm\eta}\|^2] \le \EE[\|\bL_J\bZ\|_{\oper}^2] $ thanks to $\|\tilde{\bm\eta}\|^2 = 1$ and $\|\brho\|^2 \le \|\bL_J \bZ\|_{\oper}^2$. It remains to bound 
    the second term. Here, we use the following identity: if $\bh: \R^{n\times q}\to \R^{q}$ is locally Lipschitz with $\|\bh\|^2\neq 0$, we have
    \begin{equation}\label{eq:identity_frac}
        \frac{\partial}{\partial z_{ij}} \Big(\frac{\bh}{\|\bh\|}\Big) = \frac{\bP^\perp}{\|\bh\|} \frac{\partial \bh}{\partial z_{ij}}\text{ with }\bP^\perp = \bI_{p+1} - \frac{\bh \bh^\top}{\|\bh\|^2} \text{ so that } \Big\| \frac{\partial}{\partial z_{ij}} \Big(\frac{\bh}{\|\bh\|}\Big) \Big\|^2 \le \frac{1}{\|\bh\|^2} \Big\|\frac{\partial\bh}{\partial z_{ij}}\Big\|^2,
    \end{equation}
    where the inequality follows from $\|\bP^\perp\|_{\oper} \le 1$.  Then, \eqref{eq:identity_frac} and $\|\brho\|^2 \le \|\bL_J \bZ\|_{\oper}^2$ yield
    $$
    \sum_{i\in J}\sum_{j=1}^q \|\brho\|^2 \Big\|\frac{\partial\tilde{\bm \eta}}{\partial z_{ij}}\Big\| \le \|\bL_J\bZ\|_{\oper}^2 \sum_{i\in J}\sum_j \frac{1}{\|\bh\|^2} \Big\|\frac{\partial \tilde \bh}{\partial z_{ij}}\Big\|^2 = \|\bL_J \bZ\|_{\oper}^2 \Xi_J.
    $$
    In the same way, 
    $\brho=-\bL_J\bZ\bh/\|\bh\|$, $\|\tilde{\bm\eta}\|^2 = 1$, and \eqref{eq:identity_frac} lead to
    $$
    \sum_{i\in J}\sum_{j=1}^q\|\tilde{\bm\eta}\|^2 \Big\|\frac{\partial \brho}{\partial z_{ij}}\Big\|^2 =
    \sum_{i\in J}\sum_{j=1}^q \Big\|\bL_J \Big(\be_i \frac{h_j}{\|\bh\|} + \bZ \frac{\partial}{\partial z_{ij}}  \Big(\frac{\bh}{\|\bh\|}\Big)\Big)\Big\|^2
    \le 2|J| + 2\|\bL_J\bZ\|_{\oper}^2 \Xi_J.
    $$
    Therefore, the RHS of \eqref{eq:sure_quadratic} is bounded from above by $|J| +  \EE [\|\bL_J \bZ\|_{\oper}^2 (1 + \Xi_J + \tilde\Xi_J)]$ up to some absolute constant. It remains to control the error inside the square:
\begin{align*}
& \frac{1}{\|\bh\|\|\tilde\bh\|}\sum_{i\in J} \sum_{j\in[q]} \frac{\partial r_i\tilde h_j}{\partial z_{ij}} - \sum_{i\in J}\sum_{j\in[q]} \frac{\partial \rho_i \tilde\eta_j}{\partial z_{ij}}  =\sum_{i\in J}\sum_{j=1}^q \frac{r_i \tilde h_j}{\|\bh\|\|\tilde\bh\|} \bigg(
\frac{\bh^\top}{\|\bh\|^2}\frac{\partial\bh}{\partial z_{ij}} + \frac{\tilde\bh^\top }{\|\tilde\bh\|^2} \frac{\partial \tilde\bh}{\partial z_{ij}}
\bigg)
\end{align*}
By multiple applications of the Cauchy-Schwartz inequality, the square of the RHS is bounded from above by $2\|\bL_J \bZ\|_{\oper}^2 (\Xi_J + \tilde{\Xi}_J)$. This finishes the proof.
\end{proof}

\subsubsection{Proof of \Cref{lem:trB_df} (restated here for convenience)}\label{proof:trB_df}

\bigskip

\LemTrBdf*

\begin{proof}
\Cref{lem:psi_bound} with $J=I=\tilde I$ and $0 \le \tr[\bB] \le p \|\bB\|_{\oper} \le p/(k\mu)$ by \eqref{eq:property_B} yield
$$
\EE|(1+\tr[\bB]) \xi_1| \lesssim \mu^{-1} (p/k ) \EE|\xi_1| \lesssim \tau^{-1} (p/k) \sqrt{k} \tau^{-1} (1 + p/k)^{3/2} \lesssim \sqrt{k}\tau^{-2} (1+p/k)^{5/2},
$$
while $\EE|\xi_2| \lesssim \sqrt{k} \tau^{-2} (1+p/k)^2$ by \Cref{lem:psi_bound_2} below. Thus, we obtain 
\[
\EE[|k - \tr[\bL\bV](1+\tr[\bB])|] \lesssim \sqrt{k}\tau^{-2} (1+p/k)^{5/2} +  \sqrt{k} \tau^{-2} (1+p/k)^2 \lesssim \sqrt{k}\tau^{-2} (1+p/k)^{5/2}.
\]
This finishes the proof.
\end{proof}

\begin{lemma}\label{lem:psi_bound_2}
We have
$\EE[|k - (1+\tr[\bB])^2 \|\bL\br\|^2/\|\bh\|^2|] 
\lesssim k^{1/2} \tau^{-2} (1 + p/k)^2$. 
\end{lemma}
\begin{proof}
Define $\ba\in \R^{n}$ and $\bb \in \R^{n}$ as
\[
\text{for all } i\in [n], \quad 
a_i = \frac{1}{\|\bh\|} (1+\tr[\bB])(\bL\br)_i
\quad
\text{and} 
\quad
b_i = \frac{1}{\|\bh\|} (\bL\br)_i + \frac{1}{\|\bh\|} \sum_{j=1}^{p+1} \frac{\partial h_j}{\partial z_{ij}}
\]
so that  $(1+\tr[\bB])^2 \|\bL\br\|^2/\|\bh\|^2 = \|\ba\|^2$. \Cref{lem:basic_triangle_inequality} ahead and the Cauchy--Schwarz inequality yield
\begin{align}
    \EE[|k-\|\ba\|^2|] 
    &\le \EE[2 (\sqrt{k}\|\ba-\bb\| + |k-\|\bb\|^2| + \|\ba-\bb\|^2] \notag \\
    &\lesssim \sqrt{k \EE[\|\ba-\bb\|^2]} + \EE[\|\ba-\bb\|^2] + \EE[|k-\|\bb\|^2|]. \label{eq:basic_triangle_cauchy}
\end{align}
Thus, it suffices to bound $\EE[\|\ba-\bb\|^2]$ and $\EE[|k-\|\bb\|^2]$. For
$\EE\|\ba-\bb\|^2$, \eqref{eq:property_B},  \eqref{eq:derivartive_formula}, and \Cref{lem:gaussian_oper_bound} yield 
\[
\EE \|\ba-\bb\|^2 = \EE[\|\bh\|^{-2} \|\bL\bZ \bB^\top \bh\|^2]  \le \EE[\|\bB\|_{\oper}^2 \|\bL\bZ\|_{\oper}^2] \lesssim  (k\mu)^{-2} (k+p) \lesssim 
k^{-1} \tau^{-2} (1+p/k),
\]
where $\tau = \min(1, \mu)$.
If we denote $\|\bh\|^{-2} \sum_{i\in I} \sum_{j=1}^{p+1} \|(\partial/\partial z_{ij})\bh\|^2$ by $\Xi_I$, 
\Cref{prop:chi_square} below leads to 
\begin{align*}
    \EE[|k-\|\bb\|^2|] &\lesssim \sqrt{k(1+\EE[\Xi_I])} + \EE[\Xi_I] \text{ (by \Cref{prop:chi_square})}\\
    &\lesssim \sqrt{k(1+ \mu^{-2} (1+p/k)^2)} + \mu^{-2} (1+p/k)^2 \text{ (thanks to \eqref{eq:Xi_bound} with $J=I$)}\\
    &\lesssim \sqrt{k}\tau^{-2} (1+p/k)^2 \text{ (thanks to $\tau=\min(1, \mu)$)}.
\end{align*}
This concludes the proof.
\end{proof}

\begin{lemma}[{Variant of \citet[Theorem 7.1]{bellec2020out}}]\label{prop:chi_square}
Let $\bh: \R^{k \times q}\to \R^{q}$ be a locally Lipschitz functions. If $\bZ \in \R^{k\times q}$ has i.i.d.\ $\cN(0,1)$ entries, we have 
$$
\EE\bigg[\bigg|k - \frac{1}{\|\bh\|^2}\sum_{i=1}^k \Big(\be_i^\top \bZ \bh - \sum_{j=1}^{q} \frac{\partial h_j}{\partial z_{ij}}\Big)^2\bigg|\bigg] \lesssim \sqrt{k(1 + \EE\Xi)} + \EE \Xi \text{ with } \Xi = \frac{1}{\|\bh\|^2} \sum_{i=1}^k \sum_{ j=1}^{q} \Big\|\frac{\partial \bh}{\partial z_{ij}}\Big\|^2.
$$
\end{lemma}
\begin{proof}
Define vectors $\ba, \bb\in \R^{k}$ by
\[
\text{for all } i\in [k], \quad  a_i = \frac1{\|\bh\|} \Big(\be_i^\top \bZ\bh - \sum_j \frac{\partial h_j}{\partial z_{ij}}\Big)
\quad
\text{and}
\quad
b_i = \be_i^\top \bZ\frac{\bh}{\|\bh\|} - \sum_j \frac{\partial}{\partial z_{ij}} \Bigl(\frac{h_j}{\|\bh\|}\Bigr) ,
\] 
so that the LHS of the assertion is $\EE|k - \|\ba\|^2|$. The same argument in \eqref{eq:basic_triangle_cauchy} leads to
$$
\EE[|k-\|\ba\|^2|] \lesssim \sqrt{k \EE[\|\ba-\bb\|^2]} + \EE[\|\ba-\bb\|^2] + \EE[|k-\|\bb\|^2|]. 
$$
Below we bound $\EE[\|\ba-\bb\|^2]$ and $\EE[|k-\|\bb\|^2]$. 
For $\|\ba-\bb\|^2$, multiple applications of the Cauchy--Schwarz inequality lead to 
$$
\|\ba-\bb\|^2 = \sum_{i=1}^k \Bigr\{-\frac{1}{\|\bh\|} \sum_{j}\frac{\partial h_j}{\partial z_{ij}} + \sum_j \frac{\partial}{\partial z_{ij}} \Big(\frac{h_j}{\|\bh\|}\Big)\Bigr\}^2 = \sum_{i=1}^k \Bigl( - \sum_{j} \frac{\bh^\top}{\|\bh\|^3} \frac{\partial \bh}{\partial z_{ij}} h_j\Bigr)^2 \le \sum_{i,j} \frac{1}{\|\bh\|^2} \Big\|\frac{\partial \bh}{\partial z_{ij}}\Big\|^2 = \Xi.
$$
For $\EE[|k-\|\bb\|^2]$, Theorem 7.1 in \cite{bellec2020out} applied to the unit vector $\bh/\|\bh\|\in \R^q$ implies 
\[
\EE|k -\|\bb\|^2| \lesssim  \sqrt{k(1 + \EE \Xi')} +  \EE\Xi'
\quad
\text{with}
\quad
\Xi' := \EE\sum_{i=1}^k \sum_{j=1}^q \|\frac{\partial}{\partial z_{ij}} \Bigl(\frac{\bh}{\|\bh\|} \Bigr)\|^2.
\]
Since $\Xi' \le \Xi$ by \eqref{eq:identity_frac}, we have 
$
\EE[|k-\|\bb\|^2|] \lesssim \sqrt{k (1 + \EE[\Xi])}  + \EE[\Xi]. 
$
This finishes the proof.
\end{proof}

\begin{lemma}\label{lem:basic_triangle_inequality}
We have
$|k-\|\ba\|^2| \le 2 (\sqrt{k}\|\ba-\bb\| + |k-\|\bb\|^2| + \|\ba-\bb\|^2)$ for all vector $\ba, \bb$ in the same Euclidean space. 
\end{lemma}
\begin{proof}
 By multiple applications of the triangle inequality and $2ab \le a^2 + b^2$, we have
\begin{align*}
    \bigl||k- \|\ba\|^2| - |k-\|\bb\|^2|\bigr| 
    &\le \|\ba - \bb\|\|\ba + \bb\| \le \|\ba-\bb\|^2 + 2\|\ba-\bb\|\|\bb\| \\
    &\le \|\ba-\bb\|^2 + 2\|\ba-\bb\| (\sqrt{|\|\bb\|^2 - k |} + \sqrt{k}) \\
    &\le 2\|\ba-\bb\|^2 + |\|\bb\|^2-k| + 2\|\ba-\bb\| \sqrt{k}.
\end{align*}
Adding $|\|\bb\|^2 - k|$ to both sides and using the triangle inequality, we conclude the proof.
\end{proof}

\subsubsection{Proof of \Cref{lem:contraction_cor} (restated here for convenience)}
\label{proof:contraction_cor}

\bigskip

\LemContractionCor*

\begin{proof}
Letting $\bv = \bh/\|\bh\|$ and $\tilde{\bv}=\tilde{\bh}/\|\tilde\bh\|$, the statement of \Cref{lem:contraction_cor} can be written as follows;
if the same penalty is used for $\bh$ and $\tilde\bh$, we have
$$
\EE \Bigl[\bv^\top \tilde\bv \mathrel{\big |} (\bz_i)_{i\in I\cap \tilde I}\Bigr] \ge 0, 
\quad
\text{and}
\quad
\EE\Bigl[\Var \Bigl[\bv^\top\tilde\bv \mathrel{\big |} (\bz_i)_{i\in I\cap \tilde I}\Bigr]\Bigr] \lesssim \frac{\gamma}{nc^2\tau^2},
$$
Below we prove this claim. Here, the key fact is that conditionally on $(\bz_i)_{i\in I \cap \tilde I}$, the random vectors $\bv$ and $\tilde{\bv}$ are independent and identically distributed. Then, it immediately follows from this fact that 
$$
\EE \Bigl[\bv^\top\tilde\bv \mathrel{\big |} (\bz_i)_{i \in I\cap \tilde I}\Bigr] = \Big\Vert\EE\Bigl[\bv \mathrel{\big |} (\bz_{i})_{i\in I\cap \tilde I} \Bigr]\Big\Vert^2 \ge 0.
$$
Next, we derive the bound of the variance:
first using the Gaussian Poincar\'e inequality for the inequality below,
second using that $\bv$ does not depend on $(\bz_i)_{i\in \tilde{I}\setminus I}$ and $\tilde{\bv}$ does not depend on $(\bz_{i})_{i\in I\setminus\tilde I}$ for the equality below,
$$
\Var \Bigl[\bv^\top\tilde\bv \mid (\bz_i)_{i\in I\cap \tilde I}\Bigr] \le \bar\EE \sum_{j\in[p+1]} \sum_{i\in (I\setminus\tilde I) \cup (\tilde I\setminus I)} \Bigl(\frac{\partial }{\partial z_{ij}} \bv^\top \tilde\bv \Bigr)^2 = \bar\EE \sum_{j\in[p+1]} \Big[\sum_{i\in I\setminus\tilde I}\Bigl(
       \tilde\bv^\top
       \frac{\partial\bv}{\partial z_{ij}}
       \Bigr)^2 + \sum_{i\in \tilde I\setminus I}\Bigl(
       \bv^\top
       \frac{\partial\tilde \bv}{\partial z_{ij}}
       \Bigr)^2\Big],
$$ 
where $\bar\EE$ is conditional expectation given $(\bz_i)_{i\in I\cap \tilde I}$. 
By the symmetry of $(\bh, \tilde\bh)$, it suffices to bound the first term.
Letting $\bP^\perp = \bI_{p+1} - \bv\bv^\top$, we obtain the following from the identity \eqref{eq:identity_frac}:
\begin{align*}
\sum_{j\in[p+1]}\sum_{i\in I\setminus\tilde I}\Bigl(
       \tilde\bv^\top
       \frac{\partial\bv}{\partial z_{ij}}
       \Bigr)^2 &= \frac{1}{\|\bh\|^2} \sum_{j\in[p+1]}\sum_{i\in I} (\tilde\bv^\top \bP^\perp \bB \be_j (\bL\br)_i - \tilde\bv^\top \bP^\perp \bB \bZ^\top \bL \be_i h_j)^2\\
       &\lesssim
        \|\bB\|_{\oper}^2 \|\bL \bZ\|_{\oper}^2 + \|\bL\bZ \bB^\top\|_{\oper}^2 \lesssim (k\mu)^{-2} \|\bL\bZ\|_{\oper}^2.
   \end{align*}
   Thus, the moment bound $\EE[\|\bL\bZ\|_{\oper}^2]\lesssim (k+p)$ by \Cref{lem:gaussian_oper_bound} leads to 
   $$
   \EE\Bigl[\Var \Bigl[\bv^\top\tilde\bv \mid (\bz_i)_{i\in I\cap \tilde I}\Bigr]\Bigr] \lesssim  (k\mu)^{-2} \EE\Bigl[\|\bL\bZ\|_{\oper}^2 + \|\tilde{\bL}\bZ\|_{\oper}^2\Bigr] \lesssim (k\mu)^{-2} (k+p) \lesssim n^{-1} \tau^{-2} c^{-2}\gamma,
   $$
   thanks to $\tau=\min(1, \mu)$, $c=k/n\in(0,1)$, and $\gamma=\max(1, p/n)$. This completes the proof.
\end{proof}

\subsubsection{Proof of \Cref{lem:denom_lower_bound} (restated here for convenience)}
\label{proof:df_lower_bound}

\bigskip

\LemDenomLowerBound*

\begin{proof}
Below, we prove this assertion for $\est=\sub$ and $\est=\full$, separately.

\noindent\textbf{\underline{Proof for $\est=\sub$.}} 
Recall that $d_{m,\ell}^\sub = |I_m\cap I_\ell|(1-\df_m/k)(1-\df_\ell/k)$. \Cref{lem:df_upperbound} with $I=I_m$ and $I_\ell$ implies that there exists an absolute constant $C\in(0,1)$ such that 
\begin{align}
\text{for all } m, \ell, \quad \PP((1-\df_m/k)(1-\df_\ell/k) \le  C^2 \tau^{2} c^2 \gamma^{-2}) \le 2 e^{-nc/2} \label{eq:df_k_df_ell}. 
\end{align}
\Cref{lem:concentration-intersection} (introduced later) and Markov's inequality lead to the following:
$$
\text{for all } m\neq \ell, \quad  \PP(|I_m\cap I_\ell|n/k^2-1| > 1/2) \le 4 \EE[(I_m\cap I_\ell|n/k^2-1)^2] \le 4 n^2 k^{-4} k^2 n^{-1} = 4 n^{-1} c^{-2},
$$
which implies
$$\PP(|I_m\cap I_\ell| \le k^2/(2n) = 2^{-1} nc^2) \le 4n^{-1} c^{-2}.$$ 
Note in passing that the above inequality also holds for $m=\ell$ since $|I_m\cap I_\ell| = k \ge k^2/2n$ with probability $1$.  Therefore, we have, for all $m, \ell$, 
\begin{align*}
    \PP(d_{m, \ell}^{\sub} \le 2^{-1} C^2 n c^4\tau^2 \gamma^{-2}) &\le \PP(|I_m\cap I_\ell| \le 2^{-1}nc^2) + \PP((1-\df_m/k)(1-\df_\ell/k) \le  C^2 \tau^{2} c^2 \gamma^{-2})\\
    &\le 4n^{-1} c^{-2} + 2e^{-nc/2} \lesssim n^{-1}c^{-2}.
\end{align*}

\noindent\textbf{\underline{Proof for $\est=\full$.}} 
Recall that $d_{m,\ell}^\full = n - \df_m-\df_\ell + k^{-2} |I_m\cap I_\ell| \df_m \df_\ell$. 
We consider $m=\ell$ and $m\neq \ell$ separately. 

When $m=\ell$, \eqref{eq:df_k_df_ell} with $m=\ell$ implies that the following holds with probability at least $1-2e^{-nc/2}$:
$$
n^{-1} d_{m,m}^\full = c (1-\df_m/k)^2 + 1-c \ge c C^2 \tau^2 c^2 \gamma^{-2} + 1 - c \ge C^2 \tau^2 \gamma^{-2} (c^3 + 1- c) \ge C' \tau^2\gamma^{-2},
$$
where $C' = C^2 \min_{c\in[0,1]}(c^3+1-c) = C^2(1-2\sqrt{3}/9)$ is a positive absolute constant.

When $m\neq \ell$, we decompose
$$
\frac{d_{m, \ell}^\full}{n} = \Big(1-\frac{\df_m}{n}\Big)\Big(1-\frac{\df_\ell}{n}\Big) +  \frac{\df_m}{n}\frac{\df_\ell}{n}  \Big(\frac{|I_m\cap I_\ell| \cdot n}{k^2}-1\Big) =: A_{m, \ell} + B_{m,\ell}.
$$
\Cref{lem:df_upperbound} implies that there exists an absolute constant $C\in(0,1)$ such that
$$
\PP(A_{m,\ell} \le C^2 \tau^2 \gamma^{-2}) \le 2e^{-nc/2}.
$$
For $B_{m,\ell}$, $0<\df_m <k = nc$ and \Cref{lem:concentration-intersection} imply
\[
\EE[|B_{m,\ell}|^2] \le (\frac{k^2}{n^2} \frac{n}{k^2})^2 \EE[||I_m\cap I_\ell|-\frac{k^2}{n}|^2] 
\le \frac{1}{n^2} \frac{k^2}{n} = \frac{c^2}{n}.
\]
Thus, Markov's inequality leads to
\begin{align*}
\PP(d_{m, \ell}^\full > n 2 C^2 \tau^2\gamma^{-2}) 
&\le \PP(A_{m, \ell} > C^2 \tau^2\gamma^{-2}) + \PP(B_{m, \ell} > C^2 \tau^2\gamma^{-2}) \\ 
&\le 2e^{-nc/2} + (C^2\tau^2\gamma^{-2})^{-2} \EE[|B_{m,\ell}|^2] \lesssim (nc)^{-1} + \tau^{-4}\gamma^{4} c^2 n^{-1}  \\
&\lesssim n^{-1} c^{-1} \tau^{-4} \gamma^4.
\end{align*}
This concludes the proof.
\end{proof}

\begin{lemma}\label{lem:df_upperbound}
    There exists an absolute constant $C\in(0,1)$ such that
    $$
    \PP(1- \df/k \ge C \tau c \gamma^{-1}) \ge 1- e^{-nc/2}, \quad \PP(1-\df/n \ge C\tau\gamma^{-1}) \ge 1 - e^{-nc/2}
    $$
\end{lemma}
\begin{proof}
Equation~\eqref{eq:property_tr} implies
$$
(1-\df/k)^{-1} \le 1+\mu^{-1} \|\bL\bG\|_{\oper}^2/k \le \tau^{-1} (1+\|\bL\bG\|_{\oper}^2/k),
$$
thanks to $\tau=\min(1,\mu)$, 
where $\bG\in\R^{n\times p}$ has i.i.d.\ Gaussian entries. Since 
$\|\bL\bG\|_{\oper} \stackrel{\mathrm{d}}{=} \|\bG'\|_{\oper}$ with $\bG'\in \R^{k\times p}$ having i.i.d.\ $\cN(0,1)$ entries, \Cref{lem:gaussian_oper_bound} with $t=\sqrt{k} + \sqrt{p}$ implies
$$
\PP(\|\bL\bZ\|_{\oper}^2 > \{2(\sqrt{k} + \sqrt{p})\}^2) \le e^{-(\sqrt{k} + \sqrt{p})^2/2} \le e^{-(k+p)/2} \le e^{-k/2}. 
$$
Since $\{2(\sqrt{k} + \sqrt{p})\}^2 \le 8(k+p)$. the following holds with probability at least $1-e^{-k/2}$:
\[
    (1-\df/k)^{-1} \le \tau^{-1} (1 + 8(1+p/k)) \le 9 \tau^{-1} (1+p/k) \le 9 \tau^{-1} (1+c^{-1}p/n) \le 18 \tau^{-1} c^{-1}\gamma,
\]
thanks to $\gamma=\max(1, p/n)$.
This completes the proof for  $1-\df/k$. 
For $1-\df/n$, $0 < \df_m < k = nc$ and the above display lead to
    $$
    \PP(1-\df_m/n \ge (1-c) \vee \{(18)^{-1} \tau c \gamma^{-1}\}) \ge 1-e^{-nc/2},
    $$
    where 
    $(1-c) \vee \{(18)^{-1} \tau c \gamma^{-1}\} \ge (18)^{-1} \tau \gamma^{-1} ((1-c) \vee c) \ge (18)^{-1} \tau \gamma^{-1} 2^{-1}$ thanks
    to $c,\tau\in(0,1]$ and $\gamma>1$.
    Therefore
    taking $C=(36)^{-1}$ concludes the proof.
\end{proof}

\subsubsection{Proof of \Cref{lem:concentration_df} (restated here for convenience)}
\label{proof:concentrate_denom_full}

\bigskip

\LemConcentrationDf*

\begin{proof}
Now we suppose that exists a deterministic scalar $d_n$ that does not depend on $m$ such that 
\begin{equation}\label{eq:df_concentrate_constant}
\PP(|\df_m/n - d_n| > \epsilon) \lesssim e^{-nc/2} + 
\frac{\gamma^{9/2}}{\epsilon\sqrt{n}\tau^3 c^4}.    
\end{equation}
Then, multiple applications of the triangle inequality lead to
\begin{align*}
\PP(|\df_m - \tdf_M| > n \epsilon) 
&\le \PP\bigg(\Big|\frac{\df_m}{n} - d_n\Big|> \frac\epsilon2\bigg) + \PP\bigg(\Big|\frac{M^{-1}\sum_m\df_m}{n} - d_n\Big|>\frac\epsilon2\bigg)\\
&\le \PP\bigg(\Big|\frac{\df_m}{n} - d_n\Big|> \frac\epsilon2\bigg) + \sum_{m=1}^M \PP\bigg(\Big|\frac{\df_m}{n} - d_n\Big|>\frac\epsilon2\bigg) 
\lesssim M \Big(e^{-nc/2} + 
\frac{\gamma^{9/2}}{\epsilon\sqrt{n}\tau^3 c^4}\Big),
\end{align*}
which completes the proof. 
Below we prove \eqref{eq:df_concentrate_constant}. 
Let $\bL_m=\bL_{I_m}$ for brevity.
Define $(W_m, d_n)$ by
\[
W_m := \frac{k\|\bL_m \br_m\|^2}{n^2 \|\bh_m\|^2} = \frac{c\|\bL_m \br_m\|^2}{n\|\bh_m\|^2}, \quad d_n = \frac{k}{n} - \sqrt{\EE[W_m]}.
\]
Note that $d_n$ does not depend on $m$
by symmetry. Since $\df_m = k - \tr[\bL_m\bV_m]$, our goal is to bound
$$
{\df_m}/{n} - d_n = {\tr[\bL_m\bV_m]}/{n} - \sqrt{\EE[W_m]}.
$$
\Cref{eq:moment_ineq_sub} with $I = \tilde{I} = I_m$ implies
$$
\EE\bigg[\Big|W_m -\frac{\tr[\bL_m\bV_m]^2}{n^2}\Big|\bigg] = \frac{k}{n^2} \EE\bigg[\Big|\frac{\|\bL_m\br_m\|^2}{\|\bh_m\|^2} - k \tr[\bL_m\bV_m]^2\Big|\bigg]  
\lesssim \frac{k}{n^2}\cdot \sqrt{n} \tau^{-2} c^{-3} \gamma^{7/2} = \frac{\gamma^{7/2}}{\sqrt{n}\tau^2 c^2},
$$
while the Gaussian Poincar\'e inequality leads to (see \Cref{lem:contraction_ratio} below) 
$$
\Var[W_m] = \Var\bigg[\frac{c}{n} \frac{\|\bL_m\br_m\|^2}{\|\bh_m\|^2}\bigg] = \frac{c^2}{n^2} \Var\bigg[\frac{\|\bL_m\br_m\|^2}{\|\bh_m\|^2}\bigg] \lesssim \frac{c^2}{n^2} \cdot \frac{n\gamma^3}{c^2 \tau^2} = \frac{\gamma^3}{n\tau^2}.
$$
By the above displays, we obtain
\begin{align*}
\EE \bigg[\Big|\frac{\tr[\bL_m\bV_m]^2}{n^2} - \EE[W_m]\Big|\bigg] 
&\le \EE \bigg[\Big|\frac{\tr[\bL_m\bV_m]^2}{n^2} - W_m\Big|\bigg] + \EE \Big[\big|W_m - \EE[W_m]\big|\Big] \\
&\le  \EE\bigg[\Big|\frac{\tr[\bL_m\bV_m]^2}{n^2} - W_m\Big|\bigg] + \sqrt{\Var[W_m]}
\\
&\lesssim \frac{\gamma^{7/2}}{\sqrt{n}\tau^2 c^2} + \sqrt{\frac{\gamma^3}{n\tau^2}} \lesssim \frac{\gamma^{7/2}}{\sqrt{n}\tau^2 c^2}.
\end{align*}
Now, \Cref{lem:df_upperbound} implies that there exists an absolute constant $C>0$ such that
$$
\PP(\Omega^c) \le  e^{-(k+p)/2} \le e^{-nc/2} \text{ with } \Omega = \{\tr[\bL_m\bV_m]/n \ge C \gamma^{-1} c\tau\}
$$
Notice that under $\Omega$, 
\begin{align*}
\Bigl|\frac{\tr[\bL_m\bV_m]^2}{n^2} - \EE[W_m]\Bigr| &= \Bigl|\frac{\tr[\bL_m\bV_m]}{n} - \sqrt{\EE[W_m]}\Bigr| \cdot \Bigl|\frac{\tr[\bL_m\bV_m]}{n} + \sqrt{\EE[W_m]}\Bigr|\\
&= \Bigl|\frac{\df_m}{n} - d_n\Bigr| \cdot \Big|\frac{\tr[\bL_m\bV_m]}{n} + \sqrt{\EE[W_m]}\Big| \ge  \Bigl|\frac{\df_m}{n} - d_n\Bigr| C \gamma^{-1} c^2 \tau.
\end{align*}
Combining the above displays together, we obtain
\begin{align*}
\PP\bigg(\Big|\frac{\df_m}{n} - d_n\Big| > {\epsilon}\bigg) &\le \PP(\Omega^c) + \PP\bigg(\Big|\frac{\tr[\bL_m\bV_m]^2}{n^2} - \EE[W_m]\Big| >  \epsilon C \gamma^{-1} c^2\tau\bigg) \\ 
&\le e^{-nc/2} + \frac{\gamma}{\epsilon C c^2\tau} \EE\bigg[\Big|\frac{\tr[\bL_m\bV_m]^2}{n^2} - \EE[W_m]\Big|\bigg]\\
&\lesssim e^{-nc/2} + \frac{\gamma}{\epsilon c^2\tau} \cdot \frac{\gamma^{7/2}}{\sqrt{n}\tau^2 c^2} = e^{-nc/2} + 
\frac{\gamma^{9/2}}{\epsilon\sqrt{n}\tau^3 c^4}.
\end{align*}
This finishes the proof.
\end{proof}

\begin{lemma}\label{lem:contraction_ratio}
We have
$\Var(\|\bL\br\|^2/\|\bh\|) \lesssim n c^{-2} \tau^{-2} \gamma^3$. 
\end{lemma}
\begin{proof}
    Let $\bu = \bL\br
    /\|\bh\| = -\bL\bZ\bh/\|\bh\| = -\bL\bZ \bv$ with $\bv = \bh/\|\bh\|$. The Gaussian Poincar\'e inequality yields
    \[
    \Var(\|\bu\|^2) 
    \le \EE \bigg[ \sum_{ij} \bigg(\frac{\partial \|\bu\|^2}{\partial z_{ij}}\bigg)^2 \bigg]
    = 4 \sum_{ij} \EE \bigg(\bu^\top \frac{\partial \bu}{\partial z_{ij}}\bigg)^2,
    \]
    where
    $$
    \bu^\top \frac{\partial \bu}{\partial z_{ij}} = \bv^\top \bZ^\top \bL( \be_i v_j +\bZ (\frac{\partial}{\partial z_{ij}} \frac{\bh}{\|\bh\|})) =  \bv^\top \bZ^\top \bL \be_i v_j +\bv^\top \bZ^\top \bL \bZ \bP^\perp \frac{1}{\|\bh\|} \frac{\partial \bh}{\partial z_{ij}} = \Rem^1_{ij} + \Rem^2_{i,j}.
    $$
    Note that $\sum_{ij} (\Rem^1_{ij})^2 \le \|\bL\bZ\|_{\oper}^2$. For $\Rem^2_{ij}$, the derivative formula \eqref{eq:derivartive_formula} leads to 
    \begin{align*}
    \sum_{ij} (\Rem^2_{ij})^2 
    &= \frac{1}{\|\bh\|^2} \sum_{ij} (\bv^\top \bZ^\top \bL \bZ \bP^\perp (\bB \be_j (\bL\br)_i - \bB\bZ^\top \bL\be_i h_j))^2 \\
    &\lesssim \frac{1}{\|\bh\|^2} (\|
    \bB^\top \bP^\perp \bZ^\top \bL \bZ \bv\|^2 \|\bL\br\|^2 + \|
    \bL \bZ \bB^\top \bP^\perp \bZ^\top \bL \bZ \bv\|^2\|\bh\|^2)\\
    &\lesssim \|\bB\|_{\oper}^2 \|\bL\bZ\|_{\oper}^6 \lesssim  (k\mu)^{-2} \|\bL\bZ\|_{\oper}^6.
    \end{align*}
    Thus, using \Cref{lem:gaussian_oper_bound}, we obtain
    $\Var(\|\bu\|^2) \lesssim (k+p) + k^{-2}\mu^{-2} (k^3 + p^3) \lesssim n c^{-2} \tau^{-2} \gamma^3$, thanks to $\tau=\min(1, \mu), c=k/n \in(0,1]$ and $\gamma=\max(1,p/n)$. 
\end{proof}

\subsection{Miscellaneous useful facts}
\label{app:elementary-inequality}

\bigskip

\begin{lemma}\label{lem:gaussian_oper_bound}
If $\bG\in \R^{k\times q}$ has i.i.d.\ $\cN0,1)$ entries, the tail 
$\PP(\|\bG\|_{\oper} > \sqrt{k} + \sqrt{q} + t) \le \Phi(-t) \le e^{-t^2/2}$
holds for all $t\in \R$, where $\Phi(\cdot)$ is the CDF of the standard normal. Therefore, we have $\EE[\|\bG\|_{\oper}^m] \le C(m) (\sqrt{k} + \sqrt{q})^r$ for all $m\ge 1$, where $C(m)>0$ is a constant depending only on $m$. 
\end{lemma}
\begin{proof}
    See Theorem 2.13 of \citet{davidson2001local} for the tail bound. The moment bound is obtained by integrating the tail bound. 
\end{proof}

\begin{lemma}[Simple random sampling properties; see, e.g., page 13 of \cite{chaudhuri2014modern}]\label{lemma:sample_without_rep}
Fix an array $(x_i)_{i=1}^M$ of length $M\ge 1$ and let $\mu_M$ be the mean $M^{-1}\sum_{i\in [M]} x_i$ and $\sigma_M^2$ be the variance $M^{-1} \sum_{i\in M} x_i^2 - \mu_M^2$. 
Suppose $J$ is uniformly distributed on  $\{J \subset [M]: |J|=m\}$ for a fixed integer $m\le M$. Then, the mean and variance of the sample mean $\hat{\mu}_J = \sum_{i\in J} x_i$ are given by
$$
\EE [\hat{\mu}_J] = \mu_M, 
\quad 
\text{and}
\quad
\Var[\hat{\mu}_J] =  \frac{\sigma_M^2}{m} \left(1-\frac{m-1}{M-1}\right) \le \frac{\sum_{i\in M} x_i^2}{m M}. 
$$
\end{lemma}

\begin{lemma}\label{lem:concentration-intersection}
If $I$ and $\tilde I$ are independent and uniformly distributed over $[n]$ with cardinality $k\le n$, 
we have
$
\EE[(|I \cap \tilde I| - n^{-1}k^2)^2] \le n^{-1} k^2
$. 
\end{lemma}
\begin{proof}
A random variable $X$ is said to follow a hypergeometric distribution, denoted as $X \sim \text{Hypergeometric} (n,K,N)$, if its probability mass function can be expressed as:
    \[
    \PP(X=k) = \frac{
    \binom{K}{k}\binom{N-K}{n-k}
    }{\binom{N}{n}
    }, \text{ where } \max\{0, n+K-N\} \leq k \leq \min\{n,K\}.
    \]
    The mean and variance of $X$ are respectively given by (e.g., \citet{greene2017exponential}):
    \[
    \EE[X] = \frac{nK}{N}, \quad \text{and} \quad \Var(X) = \frac{nK(N-K)(N-n)}{N^2(N-1)} \le \frac{nK}{N}.
    \]
    Since $|I \cap \tilde I|\sim \text{Hypergeometric} (k,k,n)$, we have
    $\EE[(|I \cap \tilde I| - n^{-1}k^2)^2] =
        \Var(|I\cap \tilde I|) \le k^2/n$.
\end{proof}

\subsection{Relaxing \Cref{assu:penalty} in the underparameterized regime}
\label{sec:relaxing_strong_convexity}

    In the underparameterized regime when $p/n < 1$, the strongly convex assumption on the penalty function in \Cref{assu:penalty} is not needed; we only need the penalty to be convex. 
    We make this precise in this section.
    
    \begin{theorem}
    \label{thm:relaxing-strong-convexity}
        Let \Cref{assu:sampling,assu:Gaussian-feature,assu:Gaussian-response} be fulfilled and assume $p/k\le \gamma < 1$ for a constant $\gamma$ independent of $k,n,p$. For any convex function $g_m$, define $$\check\bb_m = 
         \argmin_{\bb \in \RR^{p}}
        \frac{1}{2|I_m|}
        \|\bL_{I_m}(\by - \bX\bb)\|_2^2
        + g_m(\bb) + \mu_n\|\bSigma^{1/2}(\bb-\bbeta_0)\|^2$$ for some deterministic $\mu_n\ge 0$;
    which is the same as $\hbeta_m$ in \eqref{eq:def-hbeta}
    with an additional quadratic penalty term
    $\mu_n\|\bSigma^{1/2}(\bb-\bbeta_0)\|^2$.
    Let also $\cdf_m$ be the degrees of freedom of $\check\bb_m$.
    Then as long as $\mu_n\to 0$ as $n,k,p\to+\infty$ we have
    $$
    \frac{\sigma^2 + \|\bSigma^{1/2}(\check\bb_m-\bbeta_0)\|^2}{\sigma^2 + \|\bSigma^{1/2}(\hbeta_m-\hbbeta_0)\|^2}\pto 1,~
    \frac{1-\cdf_m/k}{1-\df_m/k}\pto 1,~
    \frac{\|\by-\bX\check\bb_m\|}{\|\by-\bX\hbeta_m\|}\pto 1,~
    \frac{\|\bL_{I_m}(\by-\bX\check\bb_m)\|}{\|\bL_{I_m}(\by-\bX\hbeta_m)\|}\pto 1
    $$
    as well as
    $$\frac{\|\bSigma^{1/2}(\hbeta_m-\check\bb_m)\|}{\min\{\|\bSigma^{1/2}(\hbeta_m-\bbeta_0)\|, \|\bSigma^{1/2}(\check\bb_m - \bbeta_0)\|\}}\pto 0.$$
    If $g_m$ is convex and the same for all $m\in[M]$, taking $\tau=\mu_n\to0$ sufficiently slowly so that
    the right-hand sides of \eqref{eq:cgcv_relative_error} converge to 0 (for instance, $\mu_n=1/\log n$ works), then
    $$
    \frac{\tR_M^{\cgcv,\sub}}{R_M}
    \pto 1,
    \qquad
    \frac{\tR_M^{\cgcv,\full}}{R_M}
    \pto 1,
    $$ 
    i.e., the proposed estimates are consistent without requiring strong convexity
    assumption on $g_m$.
    \end{theorem}
    \begin{proof}
        By the optimality condition of $\check\bb_m$, its objective value
        is smaller than that of $\hbeta_m$, so that 
        leaving aside $\mu_n\|\bSigma^{1/2}(\check\bb_m-\bbeta_0)\|^2$
        which is non-negative,
        \begin{equation*}
    \frac{
    \|\bL_{I_m}(\by - \bX\check\bb_n)\|_2^2
    }{2|I_m|}
    + g_m(\check\bb_m)
    \le
    \frac{
    \|\bL_{I_m}(\by - \bX\hbeta_m)\|_2^2
    }{2|I_m|}
    + g_m(\hbeta_m) + \mu_n\|\bSigma^{1/2}(\hbeta_m-\bbeta_0)\|^2.
    \end{equation*}
       Similarly by optimality of $\hbeta_m$, noting that $\hat{\bbeta}_m$ still minimizes the convex function 
       $$\bbeta \mapsto \frac{
    \|\bL_{I_m}(\by - \bX\bbeta)\|_2^2
    }{2|I_m|} - \frac{\|\bL_{I_m}\bX(\bbeta-\hbbeta_m)\|^2}{2|I_m|}
    + g_m(\bbeta),$$ thanks to the KKT (Karush-Kuhn-Tucker) condition for the original minimization problem, we have
    \begin{equation*}
    \frac{
    \|\bL_{I_m}\bX(\check\bb_m - \hbeta_m)\|_2^2
    }{2|I_m|}
    +
    \frac{
    \|\bL_{I_m}(\by - \bX\hbeta_m)\|_2^2
    }{2|I_m|}
    + g_m(\hbeta_m)
    \le
    \frac{
    \|\bL_{I_m}(\by - \bX\check\bb_n)\|_2^2
    }{2|I_m|}
    + g_m(\check\bb_m) .
        \end{equation*}
    Summing the above two displays, most terms cancel, and we find
    $$
    \frac{
    \|\bL_{I_m}\bX(\check\bb_m - \hbeta_m)\|_2^2
    }{2|I_m|}
    \le
    \mu_n\|\bSigma^{1/2}(\hbeta_m-\bbeta_0)\|^2.
    $$
    Since $p/k\le \gamma < 1$, 
    the left-hand
    side is bounded from below by $\frac12 C(\gamma)\|\bSigma^{1/2}(\check\bb_m - \hbeta_ m)\|^2$ for $C(\gamma)= ((1-\sqrt \gamma)/2)^2$
    with exponentially large probability. In this event,
    by the triangle inequality,
    \begin{equation}
    \begin{aligned}
    &|
    (|\|\bSigma^{1/2}(\check\bb_m - \bbeta_0)\|^2+\sigma^2)^{1/2}
    -
    (|\|\bSigma^{1/2}(\bbeta_0 - \hbeta_m)\|^2+\sigma^2)^{1/2}
    |
    \\
    &\le\big|
    \|\bSigma^{1/2}(\check\bb_m - \bbeta_0)\|
    -
    |\|\bSigma^{1/2}(\bbeta_0 - \hbeta_m)\| \big|
    \\&\le
    \|\bSigma^{1/2}(\check\bb_m - \hbeta_m)\|
  \\&\le \sqrt{2\mu_n/C(\gamma)} \|\bSigma^{1/2}(\hbeta_m-\bbeta_0)\|.
    \end{aligned}
    \label{eq:above_extra_gamma<1}
    \end{equation}
    Dividing by $(\|\bSigma^{1/2}(\hbeta_m-\bbeta_0)\|^2+\sigma^2)^{1/2}$,
    we obtain the desired convergence
    $$\frac{\sigma^2 + \|\bSigma^{1/2}(\check\bb_m-\bbeta_0)\|^2}{\sigma^2 + \|\bSigma^{1/2}(\hbeta_m-\hbbeta_0)\|^2}\pto 1,$$ thanks to $\mu_n\to 0$.
    For the convergence of the ratio of norms of residuals,
    we may bound from below the denominators by
    $$\|\by-\bX\hbeta_m\|
    \ge
    \|\bL_{I_m}(\by-\bX\hbeta_m)\|
    \ge
    \sqrt{|I_m| C(\gamma)}(\|\bSigma^{1/2}(\hbeta_m-\bbeta_0)\|^2+\sigma^2)^{1/2}$$ with probability approaching one, since the eigenvalues of
    $|I_m|^{-1/2}\bL_{I_m}[\bX\mid \bepsilon/\sigma]$ are bounded  away from 0.
    In the numerators,
    $\|\by-\bX\hbeta_m\| - \|\by-\bX\check\bb_m\|$ is bounded
    from above by $C'(\gamma)\|\bSigma^{1/2}(\check\bb_m-\hbeta_m)\|$
    and the ratio $\|\bSigma^{1/2}(\check\bb_m-\hbeta_m)\|/\|\bSigma^{1/2}(\hbeta_m-\bbeta_0)\|$ converges to 0 by
    \eqref{eq:above_extra_gamma<1}.
    This proves that the ratios of norms of residuals converge to 1 in probability.
    The ratio $\frac{1-\df_m/k}{1-\cdf_m/k}$ converges to 1
    in probability because GCV is consistent \cite[Section 3]{bellec2020out}
    for both $\check\bb_m$ and $\hbeta_m$, which gives
    $$(1-\df_m/k) \cdot \frac{\|\bL_{I_m}(\by-\bX\hbeta_m)\|}{
    \sqrt{
    |I_m|(\sigma^2+\|\bSigma^{1/2}(\hbeta_m-\bbeta_0)\|^2)}
    } \pto 1,$$
    and similarly for $\check\bb_m$.
    Finally, \eqref{eq:above_extra_gamma<1} provides
    $$\|\bSigma^{1/2}(\hbeta_m-\check\bb_m)\|/\min\{\|\bSigma^{1/2}(\hbeta_m-\bbeta_0)\|, \|\bSigma^{1/2}(\check\bb_m - \bbeta_0)\|\}\pto 0.$$

    By \eqref{eq:correlation_bound}, we further have, for $\check\bb_1,\dots\check\bb_M$ thanks to the strongly convex penalty,
    that $$\PP\Bigl(\sum_m(\sigma^2+\|\bSigma^{1/2}(\check\bb_m-\bbeta_0)\|^2)
    \le 2(\sigma^2+\|\sum_m\bSigma^{1/2}(\check\bb_m-\bbeta_0)\|^2)
    \Bigr) \to 1.$$
    Combining these inequalities, we get that all quantities involved in the definition
    of \eqref{eq:cgcv-proposal-sub}-\eqref{eq:cgcv-proposal} and their targets
    are close to the corresponding quantities  with $(\hbbeta_m,\df_m)$
    replaced by ($\check\bb_m,\cdf_m)$. 
    Thus the consistency of the estimates for $(\check\bb_m,\cdf_m)$
    implies the consistency of the estimates for $(\hbeta_m,\df_m)$.
    \end{proof}

\section[Proofs in \Cref{sec:asm-consistency}]{Proofs in \Cref{sec:asm-consistency}}\label{app:asm-consistency}

\subsection{Preparatory definitions}
\label{sec:preparation-asm-consistency}

    \noindent\textbf{Fixed-point parameters.}
    In the study of ridge ensembles under proportional asymptotics, a key quantity 
    that appears is the solution of a fixed-point equation. 
    For any $\lambda > 0$ and $\theta > 0$, define $v_p(-\lambda; \theta)$ as the unique nonnegative solution to the fixed-point equation
    \begin{align}
        v_p(-\lambda;\theta)^{-1} &= \lambda + \theta\int r(1 + v_p(-\lambda;\theta)r)^{-1}\rd H_p(r) ,\label{eq:v_ridge}
    \end{align}
    where $H_{p}(r) = p^{-1}\sum_{i=1}^{p} \ind_{\{r_i \le r\}}$ is the empirical spectral distribution of $\bSigma$ and $r_i$'s are the eigenvalues of $\bSigma$.
    For $\lambda=0$, define $v_p(0; \theta):=\lim_{\lambda \to 0^{+}} v_p(-\lambda; \theta)$ for $\theta > 1$ and $\infty$ otherwise.

    \noindent\textbf{Best linear projection.}
    Since the ensemble ridge estimators are linear estimators, we evaluate their performance relative to the oracle parameter: 
    \[\bbeta_0 = \EE[\bx\bx^{\top}]^{-1} \EE[\bx y],\] which is the best (population) linear projection of $y$ onto $\bx$ and minimizes the linear regression error.
    Note that we can decompose any response $y$ into: 
    \begin{align*}
        y = \fLI(\bx) + \fNL(\bx),
    \end{align*}
    where $\fLI(\bx)=\bbeta_0^{\top}\bx$ is the oracle linear predictor, and $\fNL(\bx)=y-\fLI(\bx)$ is the nonlinear component that is not explained by $\fLI(\bx)$.
    The best linear projection has the useful property that $\fLI(\bx)$ is (linearly) uncorrelated with $\fNL(\bx)$, although they are generally dependent.
    It is worth mentioning that this does not imply that $y$ and $\bx$ follow a linear regression model. Indeed, our framework allows any nonlinear dependence structure between them and is model-free for the joint distribution of $(\bx,y)$.
    For $n$ i.i.d. samples from the same distribution as $(\bx,y)$, we define analogously the vector decomposition:
    \begin{align}
        \label{eq:li-nl-decomposition}
        \by = \bfLI +\bfNL,
    \end{align}
    where $\bfLI=\bX\bbeta_0$ and $\bfNL = [\fNL(\bx_i)]_{i\in[n]}$.

    \noindent\textbf{Covariance and resolvents.}
    For $j\in[M]$, let $\bL_j$ be a diagonal matrix with $i$-th diagonal entry being 1 if $i\in I_m$ and 0 otherwise. Let $\hSigma_j = \bX^{\top}\bL_{j}\bX/k$ and $\bM_j = (\hSigma_j + \lambda\bI_p)^{-1}$.

    \noindent\textbf{Risk and risk estimator.} 
    Recall that for $m,\ell\in[M]$, the risk component is defined as
    \begin{align}
        R_{m,\ell} &=  \|\fNL\|_{L_2}^2+ (\hbeta_m - \bbeta_0)^\top \bSigma (\hbeta_{\ell} - \bbeta_0),          
    \end{align}
    while its two estimators are defined as
    \begin{align}
        \hR_{m,\ell}^{\sub}    &= \frac{N_{m,\ell}^{\sub}}{D_{m,\ell}^{\sub}},
        \quad
        \text{and}
        \quad
        \hR_{m,\ell}^{\full}=\frac{N_{m,\ell}^{\full}}{D_{m,\ell}^{\full}},          
    \end{align}
    where 
    \begin{align*}
        N_{m,\ell}^{\sub} &= |I_m\cap I_{\ell}|^{-1}(\by - \bX\hbeta_m)^{\top}\bL_{m\cap \ell}(\by - \bX\hbeta_{\ell}\\
        D_{m,\ell}^{\sub} &= 1  - \frac{\tr(\bS_m)}{|I_m|} - \frac{\tr(\bS_{\ell})}{|I_{\ell}|} + \frac{\tr(\bS_{m})\tr(\bS_{\ell})}{|I_m||I_{\ell}|}\\
        N_{m,\ell}^{\full} &= n^{-1} (\by - \bX \hbeta_m)^\top (\by - \bX \hbeta_{\ell})\\
         D_{m,\ell}^{\full} &=1 - \frac{\tr(\bS_m)}{n} - \frac{\tr(\bS_{\ell})}{n} + \frac{\tr(\bS_{m}) \tr(\bS_{\ell}) | I_m \cap I_\ell |}{n |I_m| |I_{\ell}|},
    \end{align*}
    and $\bL_{m\cap \ell}=\bL_{I_m\cap I_\ell}$. 
    
    \noindent\textbf{Asymptotic equivalence.}
    Let $\bA_p$ and $\bB_p$ be sequences of additively conformable random matrices with arbitrary dimensions (including vectors and scalars as special cases). 
    We define $\bA_p$ and $\bB_p$ to be \emph{asymptotically equivalent}, denoted as $\bA_p \asympequi \bB_p$, if $\lim_{p \to \infty} \left| \tr[\bC_p (\bA_p - \bB_p)] \right| = 0$ almost surely for any sequence of random matrices $\bC_p$ with bounded trace norm that are multiplicatively conformable to $\bA_p$ and $\bB_p$ and are independent of $\bA_p$ and $\bB_p$.
    Observe that when dealing with sequences of scalar random variables, this definition simplifies the standard notion of almost sure convergence for the involved sequences.

    \subsection{Proofs of \Cref{thm:risk-est,thm:unifrom-consistency-lambda} (restated here for convenience)}\label{app:subsec:asym-thm:risk-est}

    \bigskip

    \ThmRiskEst*

    \ThmUniformConsistencyLambda*
    
        \Cref{thm:risk-est} is a direct consequence of \Cref{thm:unifrom-consistency-lambda}. Thus, below, we focus on proving \Cref{thm:unifrom-consistency-lambda}.
        The proof of \Cref{thm:unifrom-consistency-lambda} builds on the following lemmas.

        \begin{enumerate}
            \item
            \emph{Risk.}
            \Cref{lem:risk} establishes the asymptotics $\sR_{m,\ell}$ of the prediction risk $R_{m,\ell}$.

            \item 
            \emph{\Ovlp-estimator.}
            \Cref{lem:num-est-1} and \Cref{lem:den-est-1} establishes the asymptotics $\sN_{m,\ell}^{\sub}$ and $\sD_{m,\ell}^{\sub}$ of the numerator $N_{m,\ell}^{\sub}$ and the denominator $D_{m,\ell}^{\sub}$, respectively, for the \ovlp-estimator.
            
            \item 
            \emph{Full-estimator.}
            \Cref{lem:num-est-2} and \Cref{lem:den-est-2} establishes the asymptotics $\sN_{m,\ell}^{\full}$ and $\sD_{m,\ell}^{\full}$ of the numerator $N_{m,\ell}^{\full}$ and the denominator $D_{m,\ell}^{\full}$, respectively, for the full-estimator.
            
        \end{enumerate}

        For $\lambda>0$ and $k\in\cK_n$ such that $p/k\rightarrow[\phi,\infty)$, the consistency of the two estimators follows by showing that the ratio of the numerator asymptotics and the denominator asymptotics in (1) and (2) matches with the risk asymptotics in \Cref{lem:risk}, respectively:
        \[
           R_{m,\ell}
            \asympequi \sR_{m,\ell} = \frac{\sN_{m,\ell}^{\sub}}{\sD_{m,\ell}^{\sub}} \asympequi \frac{N_{m,\ell}^{\sub}}{D_{m,\ell}^{\sub}} = \hR_{m,\ell}^{\sub},
        \]
        and 
        \[
            R_{m,\ell} \asympequi \sR_{m,\ell} = \frac{\sN_{m,\ell}^{\full}}{\sD_{m,\ell}^{\full}} \asympequi \frac{N_{m,\ell}^{\full}}{D_{m,\ell}^{\full}} = \hR_{m,\ell}^{\full}.
        \]
        Next, \Cref{lem:boundary} takes care of the boundary cases when $\psi=\infty$ and $\lambda=0$.

        Finally, from \Cref{lem:boundary}, we also have that the sequences of functions $\{R_{m,\ell}-\sR_{m,\ell}\}_{p\in\NN}$, $\{\hR_{m,\ell}^{\sub}-\sR_{m,\ell}\}_{p\in\NN}$, and $\{\hR_{m,\ell}^{\full}-\sR_{m,\ell}\}_{p\in\NN}$ are uniformly equicontinuous on $\lambda\in\Lambda=[0,\infty]$ when $\psi\neq 1$.
        This implies that the sequences of functions $\{R_{m,\ell}-\hR_{m,\ell}^{\sub}\}_{p\in\NN}$ and $\{R_{m,\ell}-\hR_{m,\ell}^{\full}\}_{p\in\NN}$ are also uniformly equicontinuous on $\lambda\in\Lambda=[0,\infty]$ almost surely.
        From Theorem 21.8 of \cite{davidson1994stochastic}, it further follows that the sequences converge to zero uniformly over $\Lambda$ almost surely.
    
    The rest of the current subsection is committed to presenting and proving the important lemmas \Cref{lem:risk,lem:num-est-1,lem:den-est-1,lem:num-est-2,lem:den-est-2,lem:boundary}.

    \begin{lemma}[Asymptotic equivalence for prediction risk]\label{lem:risk}
        Under \Crefrange{asm:model}{asm:prop-asym}, for $\psi\in[\phi,\infty)$ and $\lambda>0$, it holds that
        \[
            R_{m,\ell}
            \asympequi \sR_{m,\ell} := (\|\fNL\|_{L_2}^2 + \tc_p(-\lambda;\psi)) (1+\tv_p(-\lambda;\varphi_{m\ell},\psi)),
        \]
        where $\varphi_{m\ell}=\psi\ind_{\{m=\ell\}}+\phi\ind_{\{m\neq\ell\}}$ and $\tv_p(-\lambda;\phi,\psi),\tc_p(-\lambda;\psi)$ are defined in \Cref{lem:risk-bias}.
    \end{lemma}
    \begin{proof}
        From the definition of the ridge estimator \eqref{eq:ingredient-estimator} and the decomposition of the response \eqref{eq:li-nl-decomposition}, we have
        \begin{align*}
            \hbeta_m - \bbeta_0 &=
            \bM_m\frac{\bX^{\top}\bL_m}{k} (\bX\bbeta_0 + \bfNL) - \bbeta_0
            \\
            &= - \lambda \bM_m\bbeta_0 + \bM_m\frac{\bX^{\top}\bL_m}{k}
            \bfNL.
        \end{align*}
        Then it follows that
        \begin{align*}
            &\|\fNL\|_{L_2}^2
            + 
            (\hbeta_{m} - \bbeta_0)^\top \bSigma (\hbeta_{\ell} - \bbeta_0) \\
            &= \lambda^2 \bbeta_0^{\top}\bM_m\bSigma\bM_{\ell}\bbeta_0\\
            &\qquad+ (\|\fNL\|_{L_2}^2
            +  \bfNL^{\top} \frac{\bL_m\bX}{k} \bM_m \bSigma \bM_{\ell}\frac{\bX^{\top}\bL_{\ell}}{k}
            \bfNL) \\
            &\qquad + (\lambda \bbeta_0^{\top}\bM_{\ell} \bSigma\bM_{\ell}\frac{\bX^{\top}\bL_m}{k}
            \bfNL + \lambda \bbeta_0^{\top}\bM_{\ell} \bSigma\bM_m\frac{\bX^{\top}\bL_m}{k}
            \bfNL)\\
            &= T_B + T_V + T_C.
        \end{align*}
        \noindent\textbf{Bias term.} 
        From \Cref{lem:risk-bias}, we have
        \begin{align*}
            T_B \asympequi  \tc_p(-\lambda;\psi) (1+\tv_p(-\lambda;\varphi_{m\ell},\psi)),
        \end{align*}
        where $\varphi_{m\ell}=\psi\ind_{\{m=\ell\}}+\phi\ind_{\{m\neq\ell\}}$ and $\tv_p(-\lambda;\phi,\psi),\tc_p(-\lambda;\psi)$ are defined in \Cref{lem:risk-bias}.

        \noindent\textbf{Variance term.}
        By conditional independence and Lemma D.2 of \citet{patil2023generalized}, the variance term converges to the the quadratic term of $\bbf_0=\bL_{m\cap\ell}\bfNL$:
        \begin{align*}
            T_V &\asympequi \|\fNL\|_{L_2}^2 +  \frac{1}{k^2}\bbf_0^{\top} \bX\bM_m \bSigma \bM_{\ell}\bX^{\top} \bbf_0 \asympequi \|\fNL\|_{L_2}^2 ( 1 + \tv_p(-\lambda;\varphi_{m\ell},\psi) ),
        \end{align*}
        where the last equivalence is from \Cref{lem:risk-var}.

        \noindent\textbf{Cross term.} 
        From \citet[Lemma A.3]{patil2023generalized}, the cross term vanishes, i.e., $T_C \asto 0$.

        This completes the proof.
    \end{proof}
    
    \begin{lemma}[Asymptotic equivalence for numerator of the \ovlp-estimator]\label{lem:num-est-1}
        Under \Cref{asm:model,asm:feat,asm:prop-asym}, for $\psi\in[\phi,\infty)$ and $\lambda>0$, and index sets $I_m,I_{\ell}\overset{\SRS}{\sim}\cI_k$, it holds that        
        \begin{align*}
                N_{m,\ell}^{\sub}
                &\asympequi\sN_{m,\ell}^{\sub}:=
                \lambda^2v_p(-\lambda;\psi)^2(\|\fNL\|_{L_2}^2 + \tc_p(-\lambda;\psi)) (1+\tv_p(-\lambda;\varphi_{m\ell},\psi)),
        \end{align*}
        where $\varphi_{m\ell}=\psi\ind_{\{m=\ell\}}+\phi\ind_{\{m\neq\ell\}}$ and $v_p(-\lambda;\psi)$, $\tv_p(-\lambda;\phi,\psi)$ and $\tc_p(-\lambda;\psi)$ are defined in \Cref{lem:risk-bias}.
    \end{lemma}
    \begin{proof}
        From the definition of the ridge estimator \eqref{eq:ingredient-estimator} and the decomposition of the response \eqref{eq:li-nl-decomposition}, we have
        \begin{align*}
            \bL_{m\cap \ell}(\by - \bX\hbeta_m) &=
            \bL_{m\cap \ell}\left(\bI_p-\bX\bM_m\frac{\bX^{\top}\bL_m}{k}\right)
            \by\\
            &= \lambda\bL_{m\cap \ell}\bX \bM_m\bbeta_0 + \bL_{m\cap \ell}\left(\bI_p-\bX\bM_m\frac{\bX^{\top}\bL_m}{k}\right)
            \bfNL.
        \end{align*}
        Then it follows that
        \begin{align*}
            &\frac{1}{|I_m\cap I_{\ell}|} (\by - \bX\hbeta_m)^{\top}\bL_{m\cap \ell}(\by - \bX\hbeta_{\ell}) \\
            &= \lambda^2\bbeta_0^{\top}\bM_m\hSigma_{m\cap\ell}\bM_{\ell}\bbeta_0\\
            &\qquad + \frac{1}{|I_m\cap I_{\ell}|}(\bL_m\bfNL)^{\top}\left(\bI_p-\frac{\bX\bM_m\bX^{\top}}{k}\right)\bL_{m\cap\ell}\left(\bI_p-\frac{\bX\bM_{\ell}\bX^{\top}}{k}\right)(\bL_{\ell}\bfNL) \\
            &\qquad + \frac{\lambda}{|I_m\cap I_{\ell}|} \left(
            \bbeta_0^{\top} \bM_{\ell}\bX^{\top}\bL_{m\cap\ell}\left(\bI_p-\bX\bM_m\frac{\bX^{\top}\bL_m}{k}\right)
            \bfNL + \bbeta_0^{\top} \bM_{\ell}\bX^{\top}\bL_{m\cap\ell}\left(\bI_p-\bX\bM_{\ell}\frac{\bX^{\top}\bL_{\ell}}{k}\right)
            \bfNL
            \right)\\
            &= T_B + T_V + T_C.
        \end{align*}
        Next, we analyze the three terms separately.

        \noindent\textbf{Bias term.}
        From \Cref{lem:risk-bias}, \begin{align*}
            T_B \asympequi \lambda^2 v_p(-\lambda;\psi)^2 \tc_p(-\lambda;\psi) (1+\tv_p(-\lambda;\varphi_{m\ell},\psi)),
        \end{align*}
        where $\varphi_{m\ell}=\psi\ind_{\{m=\ell\}}+\phi\ind_{\{m\neq\ell\}}$ and $v_p(-\lambda;\psi)$, $\tv_p(-\lambda;\phi,\psi)$ and $\tc_p(-\lambda;\psi)$ are defined in \Cref{lem:risk-bias}.
        
        \noindent\textbf{Variance term.}
        By conditional independence and Lemma D.2 of \citep{patil2023generalized}, the variance term converges to the the quadratic term of $\bbf_0=\bL_{m\cap\ell}\bfNL$:
        \begin{align*}
            T_V &\asympequi \frac{1}{|I_m\cap I_{\ell}|}(\bL_{m\cap\ell}\bfNL)^{\top}\left(\bI_p-\frac{\bX\bM_m\bX^{\top}}{k}\right)\bL_{m\cap\ell}\left(\bI_p-\frac{\bX\bM_{\ell}\bX^{\top}}{k}\right)(\bL_{m\cap\ell}\bfNL) \\
            &= \frac{1}{|I_m\cap I_{\ell}|}\|\bbf_0\|_2^2 -  \frac{1}{k}\sum_{j\in\{m,\ell\}} \bbf_0^{\top} \frac{\bX\bM_j\bX^{\top}}{|I_m\cap I_{\ell}|} \bbf_0 + \frac{|I_m\cap I_{\ell}|}{k^2} \bbf_0^{\top} \frac{\bX\bM_m\hSigma_{m\cap\ell} \bM_{\ell}\bX^{\top}}{|I_m\cap I_{\ell}|} \bbf_0\\
            & \asympequi \lambda^2 v_p(-\lambda;\psi)^2 (1+\tv_p(-\lambda;\varphi_{m\ell},\psi)),
        \end{align*}
        where the last equivalence is from \Cref{lem:risk-var}.
        
        \noindent\textbf{Cross term.}
        From \citet[Lemma A.3]{patil2023generalized}, the cross term vanishes, i.e., $T_C \asto 0$.        
    \end{proof}

    \begin{lemma}[Asymptotic equivalence for denominator of the \ovlp-estimator]\label{lem:den-est-1}
        Under \Cref{asm:feat,asm:prop-asym}, for $\psi\in[\phi,\infty)$ and $\lambda>0$, and index sets $I_m,I_{\ell}\overset{\SRS}{\sim}\cI_k$, it holds that
        \[
            D_{m,\ell}^{\sub}
                \asympequi\sD_{m,\ell}^{\sub}:= \lambda^2v_p(-\lambda;\psi)^2,
        \]
        where $v_p(-\lambda;\psi)$ is the unique nonnegative solution to fixed-point equation \eqref{eq:v_ridge}.  
    \end{lemma}
    \begin{proof}
        Since $I_m,I_{\ell}\overset{\SRS}{\sim}\cI_k$, we have that $|I_m|=|I_{\ell}|=k$.                
        From \citet[Lemma C.1]{du2023subsample}, we have that
        \[\frac{\tr(\bS_j)}{|I_j|} = \frac{1}{k^2}\tr(\bX \bM_j\bX^{\top}\bL_j) = \frac{1}{k}\tr( \bM_j\bSigma_j) \asympequi 1 - \lambda v_p(-\lambda;\psi).\]
        By continuous mapping theorem,
        \begin{align*}
            1  - \frac{\tr(\bS_m)}{|I_m|} - \frac{\tr(\bS_{\ell})}{|I_{\ell}|} + \frac{\tr(\bS_{m})\tr(\bS_{\ell})}{|I_m||I_{\ell}|} &\asympequi 1 - 2(1 - \lambda v_p(-\lambda;\psi)) + (1 - \lambda v_p(-\lambda;\psi)^2 = \lambda^2v_p(-\lambda;\psi)^2,
        \end{align*}
        which completes the proof.
    \end{proof}

    \begin{lemma}
        [Asymptotic equivalence for the numerator of the full-estimator]
        \label{lem:num-est-2}
        Under \Crefrange{asm:model}{asm:prop-asym}, for $\psi\in[\phi,\infty)$ and $\lambda>0$, it holds that
        \[
            N_{m,\ell}^{\full}
            \asympequi\sN_{m,\ell}^{\full}:=
            d_p(-\lambda;\varphi_{m\ell},\psi) (\|\fNL\|_{L_2}^2 + \tc_p(-\lambda;\psi)) (1+\tv_p(-\lambda;\varphi_{m\ell},\psi)),
        \]
        where 
        \[d_p(-\lambda;\varphi_{m\ell},\psi) = \begin{dcases}
                \frac{\psi-\phi}{\psi} + \frac{\phi}{\psi}\lambda^2 v_p(-\lambda;\psi)^2, &\text{ if }\varphi_{m\ell}=\psi,\\
                \left(\frac{\psi-\phi}{\psi} + \frac{\phi}{\psi} \lambda v_p(-\lambda;\psi) \right)^2,    &\text{ if }\varphi_{m\ell} = \phi,
            \end{dcases}\]
        $\varphi_{m\ell}=\psi\ind_{\{m=\ell\}}+\phi\ind_{\{m\neq\ell\}}$, and $\tv_p(-\lambda;\phi,\psi),\tc_p(-\lambda;\psi)$ are defined in \Cref{lem:risk-bias}.
    \end{lemma}
    \begin{proof}
        Analogous to the proof of \Cref{lem:num-est-1}, the numerator splits into
        \begin{align*}
            &\frac{1}{n}
            {(\by - \bX \hbeta_m)^\top (\by - \bX \hbeta_{\ell})} \\
            &= \frac{1}{n}\lambda^2\bbeta_0^{\top}\bM_m\hSigma\bM_{\ell}\bbeta_0\\
            &\qquad + \frac{1}{n}\bfNL^{\top}\left(\bI_p-\frac{\bL_m\bX\bM_m\bX^{\top}}{k}\right)\left(\bI_p-\frac{\bX\bM_{\ell}\bX^{\top}\bL_{\ell}}{k}\right)\bfNL \\
            &\qquad \qquad + \frac{\lambda}{n} \left(
            \bbeta_0^{\top} \bM_{\ell}\bX^{\top}\left(\bI_p-\bX\bM_m\frac{\bX^{\top}\bL_m}{k}\right)
            \bfNL + \bbeta_0^{\top} \bM_{\ell}\bX^{\top}\left(\bI_p-\bX\bM_{\ell}\frac{\bX^{\top}\bL_{\ell}}{k}\right)
            \bfNL
            \right)\\
            &= T_B + T_V + T_C.
        \end{align*}
        Next, we analyze the three terms separately.

    The bias and the cross term are analyzed as in the proof of \Cref{lem:num-est-1}, by involving \Cref{lem:risk-bias} and \citet[Lemma A.3]{patil2023generalized}, respectively:
    \begin{align*}
        T_B & \asympequi d_p(-\lambda;\varphi_{m\ell},\psi) \tc_p(-\lambda;\psi) (1+\tv_p(-\lambda;\varphi_{m\ell},\psi)) ,\\
        T_C & \asympequi 0.
    \end{align*}
    For the variance term, by conditional independence and Lemma D.2 of \citep{patil2023generalized}, the variance term converges to the quadratic terms of $\bbf_1=\bL_{m\cup\ell}\bfNL$ and $\bbf_0=\bL_{m\cap\ell}\bfNL$, respectively:
    \begin{align*}
        &\frac{1}{n}\bfNL^{\top} \left(\bI_p-\frac{\bL_m\bX\bM_m\bX^{\top}}{k}\right)\left(\bI_p-\frac{\bX\bM_{\ell}\bX^{\top}\bL_{\ell}}{k}\right) \bfNL \\
        &=\frac{|I_m\cup I_{\ell}|}{n}\bfNL^{\top} \left(\bI_p-\frac{\bL_m\bX\bM_m\bX^{\top}}{k}\right)\bL_{m\cup\ell}\left(\bI_p-\frac{\bX\bM_{\ell}\bX^{\top}\bL_{\ell}}{k}\right) \bfNL +\\
        &\qquad + \frac{|(I_m\cup I_{\ell})^c|}{n}\bfNL^{\top} \left(\bI_p-\frac{\bL_m\bX\bM_m\bX^{\top}}{k}\right)\bL_{(m\cup\ell)^c}\left(\bI_p-\frac{\bX\bM_{\ell}\bX^{\top}\bL_{\ell}}{k}\right) \bfNL\\
        &=\frac{|I_m\cup I_{\ell}|}{n}\bbf_1^{\top} \left(\bI_p-\frac{\bX\bM_m\bX^{\top}}{k}\right)\bL_{m\cup\ell}\left(\bI_p-\frac{\bX\bM_{\ell}\bX^{\top}}{k}\right) \bbf_1 \\
        &\qquad + \frac{|(I_m\cup I_{\ell})^c|}{n} \|\bbf_0\|_2^2 +  \frac{|(I_m\cup I_{\ell})^c|^2}{nk^2} \bbf_0^{\top} \bX\bM_m\hSigma_{(m\cup\ell)^c } \bM_{\ell}\bX^{\top}\bbf_0.
    \end{align*}
    From \Cref{lem:risk-var}, it follows that
    \begin{align*}
        T_V &\asympequi d_p(-\lambda;\varphi_{m\ell},\psi) \|\fNL\|_{L_2}^2 (1+\tv_p(-\lambda;\varphi_{m\ell},\psi)).
    \end{align*}
    Combining the above results completes the proof.
    \end{proof}

    \begin{lemma}[Asymptotic equivalence for denominator of the full-estimator]\label{lem:den-est-2}
        Under \Cref{asm:feat,asm:prop-asym}, for $\psi\in[\phi,\infty)$ and $\lambda>0$, and index sets $I_m,I_{\ell}\overset{\SRS}{\sim}\cI_k$, it holds that
        \[
            D_{m,\ell}^{\full}
                \asympequi\sD_{m,\ell}^{\full}:=
            \begin{dcases}
                \frac{\psi-\phi}{\psi} + \frac{\phi}{\psi}\lambda^2 v_p(-\lambda;\psi)^2, &m=\ell,\\
                \left(\frac{\psi-\phi}{\psi} + \frac{\phi}{\psi} \lambda v_p(-\lambda;\psi) \right)^2,    &m\neq\ell,
            \end{dcases}            
        \]
        where $v_p(-\lambda;\psi)$ is the unique nonnegative solution to fixed-point equation \eqref{eq:v_ridge}.
    \end{lemma}
    \begin{proof}
        From the proof of \Cref{lem:den-est-1}, we have $\tr(\bS_j)/|I_j| \asto 1 - \lambda v_p(-\lambda;\psi),$ for $j\in\{m,\ell\}$.
        From \citet[Lemma G.2]{du2023subsample}, we also have $k/n\asto \phi/\psi$.
        
        For $m=\ell$, we have
        \begin{align*}
            1 - \frac{\tr(\bS_m)}{n} - \frac{\tr(\bS_{\ell})}{n} + \frac{\tr(\bS_{m}) \tr(\bS_{\ell})}{n k}
            &\asympequi 1 - \frac{2\phi}{\psi}(1-\lambda v_p(-\lambda;\psi)) + \frac{\phi}{\psi}(1-\lambda v_p(-\lambda;\psi))^2\\
            &= \frac{\psi-\phi}{\psi} + \frac{\phi}{\psi}(\lambda v_p(-\lambda;\psi))^2
        \end{align*}
        For $m\neq \ell$, from \citet[Lemma G.2]{du2023subsample}, we also have $|I_m\cap I_{\ell}|/k\asto \phi/\psi$.
        By continuous mapping theorem, it follows that
        \begin{align*}
            1 - \frac{\tr(\bS_m)}{n} - \frac{\tr(\bS_{\ell})}{n} + \frac{\tr(\bS_{m}) \tr(\bS_{\ell}) | I_m \cap I_\ell |}{n |I_m| |I_{\ell}|} &\asympequi 1 - \frac{2\phi}{\psi}(1-\lambda v_p(-\lambda;\psi)) +  \frac{\phi^2}{\psi^2}(1 - \lambda v_p(-\lambda;\psi))^2\\
            &= \left(\frac{\psi-\phi}{\psi} + \frac{\phi}{\psi} \lambda v_p(-\lambda;\psi) \right)^2,
        \end{align*}
        which finishes the proof.
    \end{proof}

    \begin{lemma}[Boundary cases: diverging subsample aspect ratio and ridgeless]\label{lem:boundary}
        Under \Crefrange{asm:model}{asm:prop-asym}, the conclusions in \Cref{thm:risk-est} hold for $\psi = \infty$ or $\lambda=0$, if $k=\omega(\sqrt{n})$ for the first estimator and no restriction on $k$ is needed for the second estimator.

        Furthermore, for $\sR_{m,\ell},\sN_{m,\ell}^{\sub},\sD_{m,\ell}^{\sub},\sN_{m,\ell}^{\full},\sD_{m,\ell}^{\full}$ defined in \Cref{lem:risk,lem:den-est-1,lem:num-est-2,lem:den-est-2}, the sequences of functions $\{R_{m,\ell}-\sR_{m,\ell}\}_{p\in\NN}$, $\{N_{m,\ell}^{\sub}/D_{m,\ell}^{\sub}-\sN_{m,\ell}^{\sub}/\sD_{m,\ell}^{\sub}\}_{p\in\NN}$, and $\{N_{m,\ell}^{\full}/D_{m,\ell}^{\full}-\sN_{m,\ell}^{\full}/\sD_{m,\ell}^{\full}\}_{p\in\NN}$ are uniformly bounded, equicontinuous and approaching zero on $\lambda\in[0,\infty]$ almost surely.
    \end{lemma}
    \begin{proof}
        We split the proof into two parts.
    
        \noindent\textbf{Part (1) Diverging subsample aspect ratio for $\lambda>0$.}
        Recall that
        \begin{align*}
            &\frac{1}{|I_m\cap I_{\ell}|} (\by - \bX\hbeta_m)^{\top}\bL_{m\cap \ell}(\by - \bX\hbeta_{\ell}) \\
            &= (\bbeta_0 - \hbeta_m)^{\top}\hSigma_{m\cap \ell}(\bbeta_0 - \hbeta_m)  + \frac{1}{|I_m\cap I_{\ell}|} \|\bL_{m\cap\ell}\bfNL\|_2^2 + \frac{1}{|I_m\cap I_{\ell}|}\sum_{j\in\{m,\ell\}} \bfNL^{\top}\bL_{m\cap\ell}\bX(\bbeta_0 - \hbeta_m).
        \end{align*}
        From law of large numbers, $ \|\bL_{m\cap\ell}\bfNL\|_2^2 / |I_m\cap I_{\ell}| \asto \|\fNL\|_{L_2}^2$ as $|I_m\cap I_{\ell}|$ tends to infinity, which is guaranteed when $k=\omega(\sqrt{n})$.
        From \citet[Lemma A.3]{patil2023generalized}, $\bfNL^{\top}\bL_{m\cap\ell}\bX\bbeta_0/|I_m\cap I_{\ell}| \asto 0$.
        
        For the other term, note that,
        \begin{align*}
            \|\hbeta_m\|_2 &\leq \|(\bX^{\top}\bL_m\bX/k+\lambda\bI_p)^{-1}(\bX^{\top}\bL_m\by/k)\|_2  \\
            &\leq \|(\bX^{\top}\bL_m\bX/k+\lambda\bI_p)^{-1}\bX^{\top}\bL_m/\sqrt{k}\|\cdot\|\bL_m\by/\sqrt{k}\|_2\\
            &\leq C\sqrt{\EE[y_1^2]}\cdot  \|(\bX^{\top}\bL_m\bX/k+\lambda\bI_p)^{-1}\bX^{\top}\bL_m/\sqrt{k}\|_{\oper},
        \end{align*}
        where the last inequality holds eventually almost surely since Assumption \ref{asm:model} implies that the entries of $\by$ have bounded $4$-th moment, and thus from the strong law of large numbers, $\| \bL_m \by / \sqrt{k} \|_2$ is eventually almost surely
        bounded above by $C\sqrt{\EE[y_1^2]}$ for some constant $C$.
        On the other hand, the operator norm of the matrix $(\bX^\top \bL_m\bX / k+ \lambda\bI_p)^{-1} \bX\bL_m / \sqrt{k}$ is upper bounded $\max_i s_i/(s_i^2+\lambda)\leq 1/s_{\min}$ where $s_i$'s are the singular values of $\bX$ and $s_{\min}$ is the smallest nonzero singular value.
        As $k, p \to \infty$ such that $p / k \to \infty$, $s_{\min} \to \infty$ almost surely (e.g., from results of \citet{bloemendal2016principal}) and therefore, $\|\hbeta_m\|_2 \to 0$ almost surely.
        Because $\|\hSigma\|_{\oper}$ is upper bounded almost surely, we further have $\bfNL^{\top}\bL_{m\cap\ell}\bX\hbeta_m/|I_m\cap I_{\ell}| \asto 0$.
        Consequently we have $\bfNL^{\top}\bL_{m\cap\ell}\bX(\bbeta_0-\hbeta_m)/|I_m\cap I_{\ell}| \asto 0$ and
        \begin{align*}
            \frac{1}{|I_m\cap I_{\ell}|} (\by - \bX\hbeta_m)^{\top}\bL_{m\cap \ell}(\by - \bX\hbeta_{\ell}) & \asto \bbeta_0^{\top}\bSigma \bbeta_0 + \|\fNL\|_{L_2}^2 .
        \end{align*}
        Since $\bS_{j} = \bX_{I_j} \hbeta_j$ for $j\in\{m,\ell\}$, we have that $\tr(\bS_{j} )/k\asto 0$.
        So the denominator converges to 1, almost surely.

        On the other hand, the asymptotic formula for $\lambda>0$ satisfies that
        \begin{align*}
            \lim_{\psi\rightarrow\infty} \frac{\lambda^2v_p(-\lambda;\psi)^2(\|\fNL\|_{L_2}^2 + \tc_p(-\lambda;\psi)) (1+\tv_p(-\lambda;\varphi_{m\ell},\psi))}{\lambda^2v_p(-\lambda;\psi)^2}= \|\fNL\|_{L_2}^2 + \bbeta_0^{\top}\bSigma \bbeta_0,
        \end{align*}
        where we use the property that $\lim_{\psi\rightarrow\infty}v_p(-\lambda;\psi)=0$.
        Thus, the asymptotic formula is also well-defined and right continuous at $\psi=\infty$.
        
        \noindent\textbf{Part (2) Ridgeless predictor when $\lambda=0$.}
        Below, we analyze the numerator and the denominator separately for the first estimator.
        To indicate the dependency on $p$ and $\lambda$, we denote the denominator and its asymptotic equivalent by
        \begin{align*}
            P_{p,\lambda}:= D_{m,\ell}^{\sub}, \qquad Q_{p,\lambda}:= \sD_{m,\ell}^{\sub},
        \end{align*}
        where we view $k$ and $n$ as sequences $\{ k_p \}$ and $\{ n_p \}$ that are indexed by $p$.
        For $j\in\{m,\ell\}$, since $\bS_j\succeq \zero_{n\times n}$ and $\|\bS_j\|_{\oper} = \|\bL_j\bX(\bX^{\top}\bL_j\bX/k+\lambda\bI_p)^{+}\bX^{\top}\bL_j/|I_j|\|_{\oper}\leq 1$ is upper bounded for $\lambda\geq 0$, with equality holds only if $\lambda=0$.         
        When $\lambda=0$ and $\psi< 1$, we have $\tr(\bS_j)/|I_j|=\tr(\bL_j\bX(\bX^{\top}\bL_j\bX/k)^{-1}\bX^{\top}\bL_j/|I_j|)/|I_j| = p/|I_j|\leq\psi$ almost surely.
        When $\lambda=0$ and $\psi>1$, we have 
        that $\tr(\bS_j)/|I_j|\geq r_{\min}r_{\max}^{-1} p/|I_j|\asto r_{\min}r_{\max}^{-1}\psi>\psi$.
        Thus, we have that $0<P_{p,\lambda}<(1-\psi)^2$ for $\lambda>0$ and $\lambda=0$, $\psi\neq 1$.

        Next we inspect the boundedness of the derivative of $P_{p,\lambda}$:
        \begin{align*}
            \frac{\partial}{\partial \lambda}P_{p,\lambda} &= - \frac{\tr(\frac{\partial}{\partial \lambda}\bS_m)}{|I_m|} - \frac{\tr(\frac{\partial}{\partial \lambda}\bS_{\ell})}{|I_{\ell}|} + \frac{\tr(\frac{\partial}{\partial \lambda}\bS_{m})\tr(\bS_{\ell}) + \tr(\bS_{m})\tr(\frac{\partial}{\partial \lambda}\bS_{\ell})}{|I_m||I_{\ell}|}.
        \end{align*}
        For $j\in\{m,\ell\}$, note that 
        \[\frac{\partial}{\partial \lambda}\bS_j  = \bL_j\bX \left(\frac{\bX^{\top}\bL_j\bX}{k}+\lambda\bI\right)^{-2} \frac{\bX^{\top}\bL_j}{|I_j|} .\]
        We also have $\|\partial \bS_j/\partial \lambda\|_{\oper}$ is upper bounded almost surely on $\Lambda$, and thus so does $|\partial P_{p,\lambda}/\partial \lambda|$.
        We know that $P_{p,\lambda} - Q_{p,\lambda} \asto 0$ for $\lambda>0$.
        Define $Q_{p,0}:=\lim_{\lambda\rightarrow0^+}Q_{p,\lambda} = (1-\psi)^2$, which is well-defined according to \citet[Proposition E.2]{du2023subsample}.
        The equicontinuity property, together with the Moore-Osgood theorem, we have that almost surely,
        \begin{align*}
            \lim_{p\rightarrow\infty}|P_{p,0}-Q_{p,0}| &= \lim_{\lambda\rightarrow0^+}\lim_{p\rightarrow\infty}|P_{p,\lambda}-Q_{p,\lambda}| = \lim_{\lambda\rightarrow0^+} 0 =0,
        \end{align*}
        which proves that the denominator formula $Q_{p,0}$ is valid for $\lambda=0$.

        For the numerator, similarly, we define
        \begin{align*}
            P_{p,\lambda}'&:=N_{m,\ell}^{\sub} = \frac{1}{|I_m\cap I_{\ell}|} \by^{\top}L_{m\cap\ell}\by - \sum_{j\in\{m,\ell\}}\by^{\top}\bS_jL_{m\cap\ell}\by + \by^{\top}\bS_m\bS_{\ell}\by,
        \end{align*}
        and $Q_{p,\lambda}'=\sN_{m,\ell}^{\sub}$.
        By the similar argument above, the sequence of function $P_{p,\lambda}'-Q_{p,\lambda}'$ is equicontinuous on $\Lambda$ and thus $Q_{p,0}'=\lim_{\lambda\rightarrow0^+}Q_{p,\lambda}'$ is well-defined.
        Finally, since the proof for the second estimator and the risk are similar, we omit it here for simplicity.

        \noindent\textbf{Part (3) Equicontinuity (in $\lambda$) of the ratio.}
        Again, we focus on proof for the first estimator. Define $P_{p,\lambda}'' := N_{m,\ell}^{\sub}/D_{m,\ell}^{\sub}$ and $Q_{p,\lambda}'' := \sN_{m,\ell}^{\sub}/\sD_{m,\ell}^{\sub}$.
        From the proof of Part (1), we know that 
        $D_{m,\ell}^{\sub}$ is positive and upper bounded almost surely, and $N_{m,\ell}^{\sub}$ is nonnegative and upper bounded almost surely.
        Thus, we have that $P_{p,\lambda}''$ is upper bounded almost surely over $\Lambda$.

        On the other hand, from the monotonicity and boundedness of fixed-point quantities \citep[Lemma F.12]{du2023subsample}, it follows that when $\lambda=0$ and $\psi\neq 1$, $\sD_{m,\ell}^{\sub}=(1 - \psi)^2$;
        when $\lambda>0$, $\sD_{m,\ell}^{\sub}$ is a continuous function, with left limit at $\lambda=0$ bounded away from zero and infinity and right limit $\lim_{\lambda\rightarrow\infty}\lambda^2v_p(-\lambda;\psi)=1$, and is therefore also bounded away from zero and infinity over $\Lambda$.
        This implies that $Q_{p,\lambda}''$ is upper bounded over $\Lambda$.
        
        Next, we examine the derivative of $P_{p,\lambda}''$ and $Q_{p,\lambda}''$:
        \begin{align*}
            \frac{\partial }{\partial \lambda} P_{p,\lambda}'' &= \frac{\frac{\partial N_{m,\ell}^{\sub}}{\partial \lambda} D_{m,\ell}^{\sub} - N_{m,\ell}^{\sub}\frac{\partial D_{m,\ell}^{\sub}}{\partial \lambda}  }{(D_{m,\ell}^{\sub })^2},
            \quad
            \text{and}
            \quad
            \frac{\partial }{\partial \lambda} Q_{p,\lambda}'' = \frac{\frac{\partial \sN_{m,\ell}^{\sub}}{\partial \lambda} \sD_{m,\ell}^{\sub} - \sN_{m,\ell}^{\sub}\frac{\partial \sD_{m,\ell}^{\sub}}{\partial \lambda}  }{(\sD_{m,\ell}^{\sub })^2}.
        \end{align*}
        By a similar argument as above, we can also show that $|\partial P_{p,\lambda}''/\partial \lambda|$ and $|\partial Q_{p,\lambda}''/\partial \lambda|$ are upper bounded over $\Lambda$.
        Applying the Moore-Osgood theorem, the conclusion follows.
    \end{proof}
    
\subsection{Proof of \Cref{prop:correct-gcv} (restated here for convenience)}\label{app:subsec:asym-prop:correct-gcv}

\bigskip

\PropCorrectGCV*

\subsubsection*{Pointwise consistency}
    From \Cref{lem:risk} and \Cref{lem:num-est-2}, we have that
    \begin{align}
        R_M &=  \|\fNL\|_{L_2}^2 + \frac{1}{M^2}\sum_{m=1}^M \|\hbeta(\cD_{I_m}) - \bbeta_0\|_{\bSigma}^2\notag\\
        &\qquad + 
        \frac{1}{M^2} \sum_{m,\ell=1}^M (\hbeta(\cD_{I_m}) - \bbeta_0)^\top \bSigma (\hbeta_{\ell} - \bbeta_0)) \notag\\
        &\asympequi \left( \frac{1}{M}  (1+\tv_p(-\lambda;\psi,\psi)) \right.\notag\\
        &\qquad\left.+ \frac{M-1}{M} (1+\tv_p(-\lambda;\phi,\psi))\right) (\|\fNL\|_{L_2}^2 + \tc_p(-\lambda;\psi)), \label{eq:prop:correct-gcv-eq-1}\\
         \frac{1}{n}\|\by - \bX \hbeta_m)\|_2^2 &\asympequi d_p(-\lambda;\psi,\psi)  (1+\tv_p(-\lambda;\psi,\psi)) (\|\fNL\|_{L_2}^2 + \tc_p(-\lambda;\psi)), \label{eq:prop:correct-gcv-eq-2}
    \end{align}
    and
    \begin{align}
        &\frac{1}{n}\|\by - \bX \tbeta_M\|_2^2 \notag\\
        &= \frac{1}{M^2}\sum_{m=1}^M \frac{1}{n}
        \|\by - \bX \hbeta_m)\|_2^2 + \frac{1}{M^2}\sum_{m,\ell=1}^M \frac{1}{n}
            {(\by - \bX \hbeta_m)^\top (\by - \bX \hbeta_{\ell})} \notag\\
        &\asympequi \left( \frac{1}{M}d_p(-\lambda;\psi,\psi)  (1+\tv_p(-\lambda;\psi,\psi)) \right. \notag\\
        &\qquad \left.+ \frac{M-1}{M}d_p(-\lambda;\phi,\psi)  (1+\tv_p(-\lambda;\phi,\psi))\right) (\|\fNL\|_{L_2}^2 + \tc_p(-\lambda;\psi)). \label{eq:prop:correct-gcv-eq-3}
    \end{align}
    
    On the other hand, from \citet[Lemma 3.4.]{du2023subsample}, the average degrees of freedom $\tdf$ and the denominator $(1-\tdf/n)^2$ of the naive GCV estimator satisfy that
    \begin{align}
        \frac{1}{n}\tdf \asympequi \frac{\phi}{\psi}(1-\lambda v_p(-\lambda;\psi)),\label{eq:prop:correct-gcv-eq-4}
    \end{align}
    and
    \begin{align}
        (1-\tdf/n)^2 \asympequi \left(\frac{\psi-\phi}{\psi} + \frac{\phi}{\psi}\lambda v_p(-\lambda;\psi)\right)^2 = d_p(-\lambda;\phi,\psi).\label{eq:prop:correct-gcv-eq-5}
    \end{align}
    Thus, from \eqref{eq:prop:correct-gcv-eq-1}, \eqref{eq:prop:correct-gcv-eq-3} and \eqref{eq:prop:correct-gcv-eq-5}, the difference between the prediction risk and the naive GCV estimator admits the following asymptotic representations:
    \begin{align}
        &R_M - \frac{\|\by - \bX \tbeta_M\|_2^2/n}{(1-\tdf/n)^2 }  \notag\\
        &\asympequi \frac{1}{M}\left( 1 - \frac{d_p(-\lambda;\psi,\psi)}{d_p(-\lambda;\phi,\psi)} \right)(1+\tv_p(-\lambda;\psi,\psi))(\|\fNL\|_{L_2}^2 + \tc_p(-\lambda;\psi)).\label{eq:prop:correct-gcv-eq-6}
    \end{align}
    On the other hand, from \eqref{eq:prop:correct-gcv-eq-2}, \eqref{eq:prop:correct-gcv-eq-4} and \eqref{eq:prop:correct-gcv-eq-5}, we also have that for all $m\in[M]$,
    \begin{align}
         & \frac{1}{(1-\tdf/n)^2} \cdot \frac{1}{M}\cdot \frac{(\psi-\phi) (\tdf/n)^2}{\phi (1-\tdf/n)^2 + (\psi-\phi) (\tdf/n)^2} \cdot \frac{\|\by-\bX\hbeta_m\|_2^2}{n} \notag\\
         &\asympequi \frac{1}{M}\frac{d_p(-\lambda;\psi,\psi)}{d_p(-\lambda;\phi,\psi)}\frac{(\psi - \phi)\frac{\phi^2}{\psi^2}(1-\lambda v_p(-\lambda;\psi))^2}{\phi d_p(-\lambda;\phi,\psi) + (\psi - \phi)\frac{\phi^2}{\psi^2}(1-\lambda v_p(-\lambda;\psi))^2 }  (1+\tv_p(-\lambda;\psi,\psi)) (\|\fNL\|_{L_2}^2 + \tc_p(-\lambda;\psi)) \notag\\
         &= \frac{1}{M}  \frac{\frac{\phi(\psi - \phi)}{\psi^2}\left(1-\lambda v_p(-\lambda;\psi)\right)^2}{d_p(-\lambda;\phi,\psi)} (1+\tv_p(-\lambda;\psi,\psi)) (\|\fNL\|_{L_2}^2 + \tc_p(-\lambda;\psi)) \notag\\
         &= - \frac{1}{M} \left(1 - \frac{d_p(-\lambda;\psi,\psi)}{d_p(-\lambda;\phi,\psi)}\right) (1+\tv_p(-\lambda;\psi,\psi)) (\|\fNL\|_{L_2}^2 + \tc_p(-\lambda;\psi)),\label{eq:prop:correct-gcv-eq-7}
    \end{align}
    by noting that $d_p(-\lambda;\psi,\psi) - d_p(-\lambda;\phi,\psi) = \phi(\psi - \phi)\psi^{-2}(1-\lambda v_p(-\lambda;\psi))^2$.
    Matching \eqref{eq:prop:correct-gcv-eq-6} and \eqref{eq:prop:correct-gcv-eq-7} finishes the proof when $\hR_{m,m}$ is the full-estimator.
    The proof when $\hR_{m,m}$ is the \ovlp-estimator follows similarly.

\subsubsection*{Uniform consistency}

    From the proof of \Cref{lem:den-est-2}, we have that
    \begin{align*}
        (1-\tdf/n)^2\asympequi \left(1-\frac{\phi}{\psi}(1-\lambda v_p(-\lambda;\psi)\right)^2=:\sD_{m,\ell}^{\full},\qquad \text{for all }\ m\neq\ell,
    \end{align*}
    and
    \begin{align*}
        \tR^{\gcv}_M&\asympequi \frac{1}{M^2}\sum_{m,\ell\in[M]}\frac{\sN_{m,\ell}^{\full}}{\sD_{1,2}^{\full}} \asympequi  \frac{1}{M^2}\sum_{m,\ell\in[M]} R_{m,\ell} - \frac{1}{M^2}\sum_{m\in[M]} \frac{\sD_{m,m}^{\full}}{\sD_{1,2}^{\full}} R_{m}\\
        \hR_{m,m}^\est &\asympequi R_{m,m}.
    \end{align*}
    Then we have    
    \begin{align*}
        \hR_M^{\cgcv,\est}&=\tR^{\gcv}_M- \frac{1}{M} \frac{(\tdf/n)^2}{(1-\tdf/n)^2} \frac{\psi-\phi}{\phi}\frac{1}{M}\sum_{m=1}^M \hR_{m,m}^\est\\
        &\asympequi\tR^{\gcv}_M- \frac{1}{M^2}\sum_{m=1}^{M} \frac{\sD_{1,1}^{\full} - \sD_{1,2}^{\full}}{\sD_{1,2}^{\full}} R_{m,m}\\
        &\asympequi \frac{1}{M^2}\sum_{m,\ell\in[M]}R_{m,\ell}\\
        &= R_M,
    \end{align*}
    which establishes the point-wise consistency.

    Similar to the proof of \Cref{thm:unifrom-consistency-lambda}, the uniform equicontinuity of $\tR_M^{\gcv}$ and $R_{m,\ell}$ to their asymptotic limits follows from \Cref{lem:boundary}.
    And the uniformity for $|\hR_M^{\cgcv,\est} - R_M|$ follows similarly.

\subsection{Technical lemmas and their proofs}\label{app:subsec:asym-technical-lemmas}

\bigskip

    \begin{lemma}[Bias term of risk]\label{lem:risk-bias}
        Suppose the same assumptions in \Cref{thm:risk-est} hold and let $\varphi_{m\ell}=\psi\ind_{\{m=\ell\}}+\phi\ind_{\{m\neq\ell\}}$, then it holds that:
        \begin{enumerate}[(1),leftmargin=7mm]
            
            \item $\lambda^2\bbeta_0^{\top} \bM_{m} \bSigma\bM_{\ell} \bbeta_0 \asympequi \tc_p(-\lambda;\psi) (1+\tv_p(-\lambda;\varphi_{m\ell},\psi))$.
            
            \item $\lambda^2\bbeta_0^{\top} \bM_{m} \hSigma_{m\cap\ell}\bM_{\ell} \bbeta_0 \asympequi \lambda^2v_p(-\lambda;\psi)^2\tc_p(-\lambda;\psi) (1+\tv_p(-\lambda;\varphi_{m\ell},\psi))$

            \item $\lambda^2\bbeta_0^{\top} \bM_{m} \hSigma \bM_{\ell} \bbeta_0 \asympequi 
            \begin{dcases}
                \left(\frac{\psi-\phi}{\psi} + \frac{\phi}{\psi } \lambda^2v_p(-\lambda;\psi)^2  \right) \tc_p(-\lambda;\psi) (1+\tv_p(-\lambda;\varphi_{m\ell},\psi)), & m=\ell,\\
                \left(\frac{\psi-\phi}{\psi} + \frac{\phi}{\psi}\lambda v_p(-\lambda;\psi)\right)^2 \tc_p(-\lambda;\psi) (1+\tv_p(-\lambda;\varphi_{m\ell},\psi)), & m\neq\ell.
            \end{dcases} $
        \end{enumerate}
        Here the nonnegative constants $v_p(-\lambda;\psi)$, $\tv_p(-\lambda;\phi,\psi)$ and $\tc_p(-\lambda;\psi)$ are defined through the following equations:
        \begin{align*}
            \frac{1}{v_p(-\lambda;\psi)} &= \lambda+\psi \int\frac{r}{1+v_p(-\lambda;\psi)r }\rd H_p(r),\\
            \tv_p(-\lambda;\phi,\psi) &= \ddfrac{\phi \int\frac{r^2}{(1+v_p(-\lambda;\psi)r)^2}\rd H_p(r) }{v_p(-\lambda;\psi)^{-2}-\phi \int\frac{r^2}{(1+v_p(-\lambda;\psi)r)^2}\rd H_p(r)},\\
            \tc_p(-\lambda;\psi) &=
            \bbeta_0^{\top}(v_p(-\lambda;\psi)\bSigma+\bI_p)^{-1}\bSigma (v_p(-\lambda;\psi)\bSigma+\bI_p)^{-1}\bbeta_0 .
        \end{align*}    
    \end{lemma}
    \begin{proof}[Proof of \Cref{lem:risk-bias}]
        Note that $\bbeta_0$ is independent of $\bM_{m} \bSigma\bM_{\ell}$, $\bM_{m} \hSigma_{m\cap\ell}\bM_{\ell}$ and $\bM_{m} \hSigma\bM_{\ell}$. %
        We analyze the deterministic equivalents of the latter for the three cases.

    \noindent\textbf{Part (1)}
    From \citet[Lemma S.2.4]{patil2022bagging}, we have that
    \begin{align*}
        \lambda^2\bM_m\bSigma\bM_{\ell} \asympequi \left(v_p(-\lambda;\psi) \bSigma + \bI_p\right)^{-1} (1+\tv_p(-\lambda;\varphi_{m\ell},\psi))\bSigma\left(v_p(-\lambda;\psi) \bSigma + \bI_p\right)^{-1}. 
    \end{align*}
    Since $\|\bbeta_0\|_2$ is almost surely bounded from \citet[Lemma D.5.]{patil2023generalized}, by the trace property of deterministic equivalents in \citet[Lemma E.3 (4)]{patil2023generalized}, we have
    \begin{align*}
        \bbeta_0^{\top} \bM_{m} \bSigma\bM_{\ell}  \bbeta_0
        &\xlongequal{\as}   \bbeta_0^{\top} \left(v_p(-\lambda;\psi) \bSigma + \bI_p\right)^{-1} (1+\tv_p(-\lambda;\varphi_{m\ell},\psi))\bSigma\left(v_p(-\lambda;\psi) \bSigma + \bI_p\right)^{-1} \bbeta_0 \\
        & = \tc_p(-\lambda;\psi) (1+\tv_p(-\lambda;\varphi_{m\ell},\psi)).
    \end{align*}

    \noindent\textbf{Part (2)} From \citet[Lemma D.6 (1) and Lemma F.8 (3)]{du2023subsample}, we have that
    \begin{align}
        \lambda^2\bM_m\hSigma_{m\cap\ell}\bM_{\ell} \asympequi \lambda^2v_p(-\lambda;\psi)^2\left(v_p(-\lambda;\psi) \bSigma + \bI_p\right)^{-1} (1+\tv_p(-\lambda;\varphi_{m\ell},\psi))\bSigma\left(v_p(-\lambda;\psi) \bSigma + \bI_p\right)^{-1}. \label{eq:bias-2}
    \end{align}
    Then, by the trace property of deterministic equivalents in \citet[Lemma E.3 (4)]{patil2023generalized}, we have
    \begin{align*}
        \bbeta_0^{\top} \bM_{m} \hSigma_{m\cap\ell}\bM_{\ell}  \bbeta_0
        &\xlongequal{\as}   \lambda^2v_p(-\lambda;\psi)^2\bbeta_0^{\top} \left(v_p(-\lambda;\psi) \bSigma + \bI_p\right)^{-1} (1+\tv_p(-\lambda;\varphi_{m\ell},\psi))\bSigma\left(v_p(-\lambda;\psi) \bSigma + \bI_p\right)^{-1} \bbeta_0 \\
        &= \lambda^2v_p(-\lambda;\psi)^2\tc_p(-\lambda;\psi) (1+\tv_p(-\lambda;\varphi_{m\ell},\psi)).
    \end{align*}

    \noindent\textbf{Part (3)} Note that 
        $\hSigma = \frac{|I_m\cup I_\ell|}{n}\hSigma_{m\cup \ell} + \frac{|(I_m\cup I_\ell)^c|}{n}\hSigma_{(m\cup \ell)^c}$,
    we have
    \begin{align}
        \bM_{m} \hSigma \bM_{\ell}  = \frac{|I_m\cup I_\ell|}{n} \bM_{m} \hSigma_{m\cup \ell} \bM_{\ell}  + \frac{|(I_{m}\cup I_\ell)^c|}{n} \bM_{m} \hSigma_{(m\cup \ell)^c} \bM_{\ell} .  \label{eq:lem:risk-bias-1}
    \end{align}
    For the first term in \eqref{eq:lem:risk-bias-1}, from \citet[Equation (54)]{du2023subsample}, when $m\neq \ell$, it holds that
    \begin{align*}
        \lambda^2\bM_m\hSigma_{m\cup \ell}\bM_{\ell} &\asympequi  \lambda^2v_p(-\lambda; \psi)^2(1+\tv_p(-\lambda;\phi,\psi))\left(\frac{2(\psi-\phi)}{2\psi-\phi} \frac{1}{\lambda v_p(-\lambda;\psi)} + \frac{\phi}{2\psi-\phi}\right)( v_p(-\lambda; \psi) \bSigma + \bI_p)^{-2}\bSigma,
    \end{align*}
    From Part (2), when $m=\ell$, it holds that
    \begin{align*}
        \lambda^2\bM_m\hSigma_{m\cup \ell}\bM_{\ell} &\asympequi \lambda^2v_p(-\lambda;\psi)^2\left(v_p(-\lambda;\psi) \bSigma + \bI_p\right)^{-1} (1+\tv_p(-\lambda;\varphi_{m\ell},\psi))\bSigma\left(v_p(-\lambda;\psi) \bSigma + \bI_p\right)^{-1},
    \end{align*}

    For the second term in \eqref{eq:lem:risk-bias-1}, from \citet[Lemma F.8 (1)]{du2023subsample}, we have
    \begin{align*}
        \lambda^2\bM_m\hSigma_{(m\cup \ell)^c}\bM_{\ell} &\asympequi \lambda^2\bM_1\bSigma\bM_2
        \asympequi \left(v_p(-\lambda;\psi) \bSigma + \bI_p\right)^{-1} (1+\tv_p(-\lambda;\varphi_{m\ell},\psi))\bSigma\left(v_p(-\lambda;\psi) \bSigma + \bI_p\right)^{-1},
    \end{align*}
    where the last equivalence is from Part (1).

    When $m=\ell$, the coefficients in \eqref{eq:lem:risk-bias-1} concentrate $|I_m\cup I_\ell|/n\asto \phi/\psi$ and $|(I_m\cup I_\ell)^c|/n\asto (\psi-\phi)/\psi$ from \citet[Lemma G.6]{du2023subsample}.
    Then \eqref{eq:lem:risk-bias-1} implies that
    \begin{align*}
        \lambda^2\bM_m\hSigma\bM_{\ell} &\asympequi \left(\frac{\psi-\phi}{\psi} + \frac{\phi}{\psi } \lambda^2v_p(-\lambda;\psi)^2  \right) \left(v_p(-\lambda;\psi) \bSigma + \bI_p\right)^{-1} (1+\tv_p(-\lambda;\varphi_{m\ell},\psi))\bSigma\left(v_p(-\lambda;\psi) \bSigma + \bI_p\right)^{-1}.
    \end{align*}

    When $m\neq\ell$, the coefficients in \eqref{eq:lem:risk-bias-1} concentrate $|I_m\cup I_\ell|/n\asto \phi(2\psi-\phi)/\psi^2$ and $|(I_m\cup I_\ell)^c|/n\asto (\psi-\phi)^2/\psi^2$ from \citet[Lemma G.6]{du2023subsample}.
    Then \eqref{eq:lem:risk-bias-1} implies that
    \begin{align*}
        \lambda^2\bM_m\hSigma\bM_{\ell} &\asympequi \left(\frac{\psi-\phi}{\psi} + \frac{\phi}{\psi}\lambda v_p(-\lambda;\psi)\right)^2 \left(v_p(-\lambda;\psi) \bSigma + \bI_p\right)^{-1} (1+\tv_p(-\lambda;\varphi_{m\ell},\psi))\bSigma\left(v_p(-\lambda;\psi) \bSigma + \bI_p\right)^{-1}.
    \end{align*}
    Finally, applying the trace property of deterministic equivalents in \citet[Lemma E.3 (4)]{patil2023generalized} as in the previous parts completes the proof.    
    \end{proof}
    
\begin{lemma}[Variance term of risk]\label{lem:risk-var}
    Suppose the same assumptions in \Cref{thm:risk-est} hold and let $\varphi_{m\ell}=\psi\ind_{\{m=\ell\}}+\phi\ind_{\{m\neq\ell\}}$, then it holds that:
    \begin{enumerate}[(1),leftmargin=7mm]
        \item $k^{-2}\bbf_0^{\top} \bX\bM_m \bSigma \bM_{\ell}\bX^{\top} \bbf_0 \asympequi  \|\fNL\|_{L_2}^2 \tv_p(-\lambda;\varphi_{m\ell},\psi)$.

        \item $\frac{1}{|I_m\cap I_{\ell}|}\bbf_0^{\top} \left(\bI_p-\frac{\bX\bM_m\bX^{\top}}{k}\right)\bL_{m\cap\ell}\left(\bI_p-\frac{\bX\bM_{\ell}\bX^{\top}}{k}\right) \bbf_0 \asympequi \|\fNL\|_{L_2}^2 \lambda^2v_p(-\lambda;\psi))^2(1+\tv_p(-\lambda;\varphi_{m\ell},\psi))$.
        
        \item
        \begin{minipage}[t]{\linewidth}
        $\frac{|I_m\cup I_{\ell}|}{n}\bbf_1^{\top} \left(\bI_p-\frac{\bX\bM_m\bX^{\top}}{k}\right)\bL_{m\cup\ell}\left(\bI_p-\frac{\bX\bM_{\ell}\bX^{\top}}{k}\right) \bbf_1$\\ 
        \quad + $\frac{|(I_m\cup I_{\ell})^c|}{n} \|\bbf_0\|_2^2 +  \frac{|(I_m\cup I_{\ell})^c|^2}{nk^2} \bbf_0^{\top} \bX\bM_m\hSigma_{(m\cup\ell)^c } \bM_{\ell}\bX^{\top}\bbf_0$\\
         $~~~~\asympequi 
        \begin{dcases}
            \|\fNL\|_{L_2}^2 \left(\frac{\psi-\phi}{\psi} + \frac{\phi}{\psi } \lambda^2v_p(-\lambda;\psi)^2  \right) (1+\tv_p(-\lambda;\varphi_{m\ell},\psi)), & m=\ell,\\
            \|\fNL\|_{L_2}^2 \left(\frac{\psi-\phi}{\psi} + \frac{\phi}{\psi}\lambda v_p(-\lambda;\psi)\right)^2(1+\tv_p(-\lambda;\varphi_{m\ell},\psi)), & m\neq\ell.
            \end{dcases}$
        \end{minipage}
    \end{enumerate}
    Here $\bbf_0=\bL_{m\cap\ell}\bfNL$, $\bbf_1=\bL_{m\cup\ell}\bfNL$ and the nonnegative constants $v_p(-\lambda;\psi)$ and $\tv_p(-\lambda;\phi,\psi)$ are defined in \Cref{lem:risk-bias}.
\end{lemma}
\begin{proof}
    Since $\|\fNL\|_{4+\delta}<\infty$ from \citet[Lemma D.5]{patil2023generalized}, we have that $\|\bbf_0\|_2^2/|I_m\cap I_{\ell}| \asto \|\fNL\|_{L_2}^2$ by strong law of large number.
    Then, from \Cref{lem:quad-uncorr}, we have that the quadratic term 
    \begin{align}
        \frac{1}{|I_m\cap I_{\ell}|}\bbf_0^{\top}\bX\bA\bX^{\top}\bbf_0  &\asympequi \frac{1}{|I_m\cap I_{\ell}|}\|\fNL\|_{L_2}^2 \tr(\bL_{m\cap\ell}\bX\bA\bX^{\top}\bL_{m\cap\ell}) = \|\fNL\|_{L_2}^2 \tr(\bA\hSigma_{m\cap\ell}) \label{eq:lem:risk-var-eq-1}
    \end{align}
    for any symmetric matrix $\bA$ with bounded operator norm.
    We next apply this result for different values of $\bA$.

    \noindent\textbf{Part (1)} Let $\bA=\bM_m \bSigma \bM_{\ell} $. Since from \eqref{eq:bias-2} and the product rule \citep[Lemma S.7.4 (3)]{patil2022bagging}, we have that 
    \[\bM_{\ell} \hSigma_{m\cap \ell} \bM_m \bSigma \asympequi v_p(-\lambda;\psi)^2\left(v_p(-\lambda;\psi) \bSigma + \bI_p\right)^{-1} (1+\tv_p(-\lambda;\varphi_{m\ell},\psi))\bSigma\left(v_p(-\lambda;\psi) \bSigma + \bI_p\right)^{-1}\bSigma .\]
    Then by the trace property of deterministic equivalents in \citet[Lemma E.3 (4)]{patil2023generalized}, we have
    \begin{align*}
        \frac{p}{k\ind_{\{m=\ell\}}+n\ind_{\{m\neq\ell\}}}\cdot\frac{1}{p}\tr(\bA\hSigma_{m\cap\ell}) & \asympequi \varphi_{m\ell}(1+\tv_p(-\lambda;\varphi_{m\ell},\psi)) \int \left(\frac{v_p(-\lambda;\psi) r}{1 + v_p(-\lambda;\psi) r}\right)^2 \rd H_p(r) \\
        &= \tv_p(-\lambda;\varphi_{m\ell},\psi)).
    \end{align*}
    Finally, note that $|I_m\cap I_{\ell}|(k\ind_{\{m=\ell\}}+n\ind_{\{m\neq\ell\}}) \asympequi k^2$.
    This implies that 
    \[
        k^{-2}\bbf_0^{\top} \bX\bM_m \bSigma \bM_{\ell}\bX^{\top} \bbf_0 \asympequi  \|\fNL\|_{L_2}^2 \tv_p(-\lambda;\varphi_{m\ell},\psi).
    \]

    \noindent\textbf{Part (2)}
    Note that 
    \begin{align*}
        &\frac{1}{|I_m\cap I_{\ell}|}\bbf_0^{\top} \left(\bI_p-\frac{\bX\bM_m\bX^{\top}}{k}\right)\bL_{m\cap\ell}\left(\bI_p-\frac{\bX\bM_{\ell}\bX^{\top}}{k}\right) \bbf_0 \\
        &= \frac{1}{|I_m\cap I_{\ell}|}\bbf_0^{\top} \bbf_0 - \frac{1}{|I_m\cap I_{\ell}|}\sum_{j\in\{m,\ell\}}\bbf_0^{\top} \frac{\bX\bM_j\bX^{\top}}{k}\bbf_0 + \frac{1}{k^2}\bbf_0^{\top} \bX\bM_m\hSigma_{m\cap\ell}\bM_{\ell}\bX^{\top} \bbf_0.
    \end{align*}
    We next analyze the three terms separately for $m\neq \ell$.
    By the law of large numbers, the first term converges as $\bbf_0^{\top} \bbf_0/|I_m\cap I_{\ell}| \asto \|\fNL\|_{L_2}^2$.
    For the second term, let $\bA = \bM_m$ and $\bM_{\ell}$. From \citet[Corollary F.5 and Lemma F.8 (4)]{patil2023generalized}, it follows that for $j\in\{m,\ell\}$,
    \begin{align*}
        \bM_j\hSigma_{m\cap\ell} \asympequi \bI_p - (v_p(-\lambda;\psi) \bSigma + \bI_p)^{-1},
    \end{align*}
    and thus, we have
    \begin{align*}
        \frac{1}{p}\tr(\bM_j \hSigma_{m\cap\ell} ) &\asympequi \int \frac{v_p(-\lambda;\psi)r}{1+v_p(-\lambda;\psi)r}\rd H_p(r) = \psi^{-1}(1 -\lambda v_p(-\lambda;\psi)).
    \end{align*}
    It then follows that
    \begin{align*}
        \frac{1}{|I_m\cap I_{\ell}|}\sum_{j\in\{m,\ell\}}\bbf_0^{\top} \frac{\bX\bM_j\bX^{\top}}{k}\bbf_0 \asympequi 2\|\fNL\|_{L_2}^2(1 -\lambda v_p(-\lambda;\psi)).
    \end{align*}
    For the third term, let $\bA = \bM_m \hSigma_{m\cap \ell} \bM_{\ell}$.
    When $m\neq \ell$, from \citet[Lemma D.7 (1) and Lemma F.8 (5)]{patil2023generalized}, we have that
    \begin{align*}
        \bM_m \hSigma_{m\cap \ell} \bM_{\ell}\hSigma_{m\cap \ell} &\asympequi \frac{\psi}{\phi}\left(v_p(-\lambda;\psi)-\frac{\psi-\phi}{\psi}\lambda \tv_v(-\lambda;\phi,\psi) \right)(v_p(-\lambda;\psi)\bSigma+\bI_p)^{-1}\bSigma \notag\\
        &\qquad - \lambda \tv_v(-\lambda;\phi,\psi)(v_p(-\lambda;\psi)\bSigma+\bI_p)^{-2}\bSigma,
    \end{align*}
    where $\tv_v(-\lambda;\psi) = v_p(-\lambda;\psi)^2(1+\tv_p(-\lambda;\phi,\psi))$.
    Then from \eqref{eq:lem:risk-var-eq-1}, we have
    \begin{align*}
        &\frac{1}{|I_m\cap I_{\ell}|}\bbf_0^{\top} \left(\bI_p-\frac{\bX\bM_m\bX^{\top}}{k}\right)\bL_{m\cap\ell}\left(\bI_p-\frac{\bX\bM_{\ell}\bX^{\top}}{k}\right) \bbf_0 \\
        &\asympequi \|\fNL\|_{L_2}^2 \left(1 - 2(1-\lambda v_p(-\lambda;\psi)) + \psi\left(v_p(-\lambda;\psi)-\frac{\psi-\phi}{\psi}\lambda \tv_v(-\lambda;\phi,\psi) \right) \int \frac{r}{1+v_p(-\lambda;\psi) r}\rd H_p(r) \right.\\
        &\qquad \left. - \phi \lambda\tv_v(-\lambda;\phi,\psi) \int \frac{r}{(1+v_p(-\lambda;\psi) r)^2}\rd H_p(r) \right)\\
        &= \|\fNL\|_{L_2}^2 \lambda \tv_v(-\lambda;\phi,\psi) \left(\frac{1}{v_p(-\lambda;\psi)} + \phi  \int \left(\frac{r}{1+v_p(-\lambda;\psi) r}\right)^2\rd H_p(r) - (\psi-\phi)  \int \frac{r}{1+v_p(-\lambda;\psi) r}\rd H_p(r)\right.\\
        &\qquad\left.- \phi  \int \frac{r}{(1+v_p(-\lambda;\psi) r)^2}\rd H_p(r)\right)\\
        &=\|\fNL\|_{L_2}^2 \lambda \tv_v(-\lambda;\phi,\psi) \left(\frac{1}{v_p(-\lambda;\psi)} - \psi\int \frac{r}{1+v_p(-\lambda;\psi) r}\rd H_p(r) \right)\\
        &= \|\fNL\|_{L_2}^2 \lambda^2 \tv_v(-\lambda;\phi,\psi)\\
        &= \|\fNL\|_{L_2}^2 \lambda^2v_p(-\lambda;\psi)^2(1+\tv_p(-\lambda;\phi,\psi)),
    \end{align*}
    when $m\neq \ell$.
    
    When $m=\ell$, from \eqref{eq:lem:risk-var-eq-1} and \citet[Lemma D.7 (1)]{du2023subsample}, we have
    \begin{align*}
        &\frac{1}{|I_m\cap I_{\ell}|}\bbf_0^{\top} \left(\bI_p-\frac{\bX\bM_m\bX^{\top}}{k}\right)\bL_{m\cap\ell}\left(\bI_p-\frac{\bX\bM_{\ell}\bX^{\top}}{k}\right) \bbf_0 \\
        &= \frac{1}{k} \bbf_0^{\top} \left(\bI_p-\frac{\bX\bM_m\bX^{\top}}{k}\right)\bL_{m}\left(\bI_p-\frac{\bX\bM_{m}\bX^{\top}}{k}\right) \bbf_0\\
        &\asympequi \|\fNL\|_{L_2}^2 (1 - \frac{2}{k}\tr(\bM\hSigma_m) + \frac{1}{k}\tr(\bM_m\hSigma_m\bM_m\hSigma_m))\\
        &\asympequi \|\fNL\|_{L_2}^2 \lambda^2v_p(-\lambda;\psi)^2(1+\tv_p(-\lambda;\psi,\psi)).
    \end{align*}
    Combining the above results finish the proof for Part (2).

    \noindent\textbf{Part (3)} We analyze the three terms separately for $m\neq\ell$.    
    From \citet[Lemma D.7 (3)]{du2023subsample}, we have
    \begin{align*}
        &\frac{1}{|I_m\cup I_{\ell}|}\tr\left( \left(\bI_p-\frac{\bX\bM_m\bX^{\top}}{k}\right)\bL_{m\cup\ell}\left(\bI_p-\frac{\bX\bM_{\ell}\bX^{\top}}{k}\right) \right) \\
        &\asympequi \lambda^2v_p(-\lambda;\psi)^2\left(\frac{2(\psi-\phi)}{2\psi-\phi} \frac{1}{\lambda v_p(-\lambda;\psi)}+ \frac{\phi}{2\psi-\phi}\right)(1+\tv_p(-\lambda;\psi,\psi)) .
    \end{align*}
    From \Cref{lem:quad-uncorr}, it then follows that
    \begin{align*}
        &\frac{|I_m\cup I_{\ell}|}{n}\bbf_1^{\top} \left(\bI_p-\frac{\bX\bM_m\bX^{\top}}{k}\right)\bL_{m\cup\ell}\left(\bI_p-\frac{\bX\bM_{\ell}\bX^{\top}}{k}\right) \bbf_1 \\
        &\asympequi \|\fNL\|_{L^2}^2 \frac{\phi(2\psi-\phi)}{\psi^2}\lambda^2v_p(-\lambda;\psi)^2\left(\frac{2(\psi-\phi)}{2\psi-\phi} \frac{1}{\lambda v_p(-\lambda;\psi)}+ \frac{\phi}{2\psi-\phi}\right)(1+\tv_p(-\lambda;\psi,\psi)) .
    \end{align*}
    From strong law of large numbers, the second term converges $\frac{|(I_m\cup I_{\ell})^c|}{n} \|\bbf_0\|_2^2 \asto \|\fNL\|_{L^2}^2$. 
    For the third term, let $\bA = \bM_m\hSigma_{(m\cup\ell)^c } \bM_{\ell}$, then we have 
    \begin{align*}
        \frac{1}{p}\tr(\bM_m\hSigma_{(m\cup\ell)^c } \bM_{\ell}\hSigma_{m\cap \ell}  \asympequi \frac{1}{p}\tr(\bM_m\bSigma \bM_{\ell}\hSigma_{m\cap \ell} \asympequi \psi^{-1}\tv_p(-\lambda;\phi,\psi),
    \end{align*}
    where the first equality is from the conditional independence property and the second is from \citet[Lemma F.8 (3)]{du2023subsample}.
    Again, from \Cref{lem:quad-uncorr}, it follows that
    \[\frac{|(I_m\cup I_{\ell})^c|^2}{nk^2} \bbf_0^{\top} \bX\bM_m\hSigma_{(m\cup\ell)^c } \bM_{\ell}\bX^{\top}\bbf_0 \asympequi \|\fNL\|_{L_2}^2\left(\frac{\psi-\phi}{\psi}\right)^2\tv_p(-\lambda;\phi,\psi).\]
    Combining the above results finishes the proof of Part (3) for $m\neq\ell$.

    When $m=\ell$, the formula simplifies to
    \[
        \frac{1}{n}\bbf_0^{\top} \left(\bI_p-\frac{\bX\bM_m\bX^{\top}}{k}\right)\left(\bI_p-\frac{\bX\bM_{\ell}\bX^{\top}}{k}\right) \bbf_0.
    \]
    From \eqref{eq:lem:risk-var-eq-1}, Part (2) and \citet[Lemma S.2.5 (1)]{patil2022bagging}, we have
    \begin{align*}
        &\frac{1}{n}\bbf_0^{\top} \left(\bI_p-\frac{\bX\bM_m\bX^{\top}}{k}\right)\left(\bI_p-\frac{\bX\bM_{\ell}\bX^{\top}}{k}\right) \bbf_0 \\
        &= \frac{1}{n} \bbf_0^{\top} \left(\bI_p-\frac{\bX\bM_m\bX^{\top}}{k}\right)\left(\bI_p-\frac{\bX\bM_{m}\bX^{\top}}{k}\right) \bbf_0\\
        &\asympequi \|\fNL\|_{L_2}^2 \frac{k}{n}\left(1 - \frac{2}{k}\tr(\bM\hSigma_{m}) + \frac{1}{k}\tr(\bM_m\hSigma_{m}\bM_{m}\hSigma_{m})\right) \\
        &\qquad + \|\fNL\|_{L_2}^2 \frac{n-k}{n}\left( 1 + \frac{1}{n}\tr(\bM_m\hSigma_{m}\bM_{m}\hSigma_{m^c})\right)\\
        &\asympequi \|\fNL\|_{L_2}^2 \frac{\phi}{\psi } \lambda^2v_p(-\lambda;\psi)^2 (1+\tv_p(-\lambda;\varphi_{m\ell},\psi)) + \frac{\psi-\phi}{\psi} \left( 1 + \frac{1}{k}\tr(\bM_m\hSigma_{m}\bM_{m}\bSigma) \right)\\
        &\asympequi \|\fNL\|_{L_2}^2 \frac{\phi}{\psi } \lambda^2v_p(-\lambda;\psi)^2 (1+\tv_p(-\lambda;\varphi_{m\ell},\psi)) + \frac{\psi-\phi}{\psi} \left( 1 +  \tv_p(-\lambda;\psi,\psi)\right) \\
        &\asympequi \|\fNL\|_{L_2}^2 \left(\frac{\psi-\phi}{\psi} + \frac{\phi}{\psi } \lambda^2v_p(-\lambda;\psi)^2  \right) (1+\tv_p(-\lambda;\varphi_{m\ell},\psi)).
    \end{align*}
   Combining the above results finish the proof for Part (3).
\end{proof}

\begin{lemma}[Quadratic concentration with uncorrelated components]\label{lem:quad-uncorr}
    Let $\bX\in\RR^{n\times p}$ be the design matrix whose rows $\bx_i$'s are independent samples drawn according to \Cref{asm:feat}.
    Let $\bbf \in \RR^{p}$ be a random vector with i.i.d.\ entries $f_i$'s, where $f_i$ has bounded $L_2$ norm and is uncorrelated with $\bx_i$.
    Let $\bA\in\RR^{n\times n}$ be a symmetric matrix such that $\limsup\|\bA\|_{\oper}\leq M_0$ almost surely as $p\rightarrow\infty$ for some constant $M_0<\infty$.
    Then as $n,p \to \infty$ such that $p/n\rightarrow\phi\in(0,\infty)$, it holds that    
    \begin{align}
        \frac{1}{n}\bbf^{\top}\bX \bA \bX^{\top} \bbf \asympequi \frac{1}{n}\|\fNL\|_{L_2}^2\tr(\bX \bA \bX^{\top}).
    \end{align}
\end{lemma}
\begin{proof}
    Let $\hSigma=\bX^{\top}\bX/n$ and $\bM=(\hSigma+\lambda\bI_p)^{-1}$ be the resolvent.
    Note that 
    \begin{align}
        \frac{1}{n}\bbf^{\top}\bX \bA \bX^{\top} \bbf = \frac{1}{n}\bbf^{\top}\bX \bM (\bM^{-1}\bA\bM^{-1}) \bM \bX^{\top} \bbf \label{eq:lem:quad-uncorr-eq-1}
    \end{align}
    Since $\bX=\bZ\bSigma^{1/2}$, we have $\bM\bX^{\top} = \bSigma^{-1/2}(\bZ^{\top}\bZ/n+\lambda\bSigma^{-1})^{-1}\bZ^{\top} = \bSigma^{1/2}\bZ^{\top}(\bZ\bSigma\bZ^{\top}/n+\lambda\bI_p)^{-1}$.
    Let $\bB_1= (\bZ\bSigma\bZ^{\top}/n+\lambda\bI_p)^{-1}$ and $\bB_2=\bZ\bSigma^{1/2}\bM^{-1}\bA\bM^{-1} \bSigma^{1/2}\bZ^{\top}/n$, then \eqref{eq:lem:quad-uncorr-eq-1} becomes
    \begin{align}
        \frac{1}{n}\bbf^{\top}\bX \bA \bX^{\top} \bbf =  \bbf^{\top}\bB_1 \bB_2 \bB_1 \bbf .
    \end{align}
    Next, we adapt the idea of \citet[Lemma A.16]{bartlett2021deep}, to show the diagonal concentration and trace concentration successively.
    
    \noindent\textbf{Diagonal concentration.}
    From a matrix identity in \citet[Lemma D.4]{patil2023generalized}, we have that, for any $t>0$,
    \begin{align*}
        \bB_1^{-1}\bB_2\bB_1^{-1} &= \frac{1}{t}(\bB_1^{-1} - (\bB_1 + t \bB_2)^{-1})  + t \bB_1^{-1} \bB_2 (\bB_1 + t \bB_2)^{-1} \bB_2 \bB_1^{-1} .
    \end{align*}    
    Let $\bU\in\RR^{n\times n}$ with $U_{ij}=[\bbf]_i[\bbf]_j\ind\{i\neq j\}$.
    We then have
    \begin{align}
        &\left|\sum_{1\leq i\neq j\leq n}[\bB_1^{-1}\bB_2\bB_1^{-1}]_{ij}[\bbf]_i[\bbf]_j\right| \notag\\
        &= |\langle \bB_1^{-1}\bB_2\bB_1^{-1} ,\bU\rangle| \notag\\
        &\leq \frac{1}{t}|\langle \bB_1^{-1},\bU\rangle| - \frac{1}{t}|\langle(\bB_1 + t \bB_2)^{-1},\bU\rangle| + t \|\bB_1^{-1}\|_{\oper}^2 \|\bB_2\|_{\oper}^2 \|(\bB_1 + t \bB_2)^{-1}\|_{\oper}\|\bU\|_{\tr}. \label{eq:lem:gen-risk-nonlinear-eq-4}
    \end{align}
    For the first two terms, \citet[Lemma D.3]{patil2023generalized} implies that
    \begin{align*}
        \frac{1}{n}|\langle \bB_1^{-1} ,\bU\rangle| \asto 0,
        \quad 
        \text{and}
        \quad
        \frac{1}{n}|\langle (\bB_1 + t\bB_2)^{-1},\bU\rangle|\asto 0,
    \end{align*}
    where the second convergence is due to $|\langle (\bB_1 + t\bB_2)^{-1},\bU\rangle|\lesssim t|\langle \bB_1 ^{-1},\bU\rangle|$ because of Woodbury matrix identity and the bounded spectrums of $\bB_1$ and $\bB_2$.
    For the last term, note that $\|\bB_1\|_{\oper}\leq \lambda^{-1}$ and $\|\bB_2\|_{\oper}\leq \|\bA\|_{\oper}\|\hSigma\|_{\oper}$, where $\|\bA\|_{\oper}$ is almost surely bounded as assumed, and $\|\bZ\bSigma^{1/2}\|_{\oper}^2/n=\|\hSigma\|_{\oper}\leq r_{\max}(1+\sqrt{\phi})^2$ almost surely as $n,p\rightarrow\infty$ and $p/n\rightarrow\phi\in(0,\infty)$ (see, e.g., \citet{bai2010spectral}).
     Also, $\|\bU\|_{\tr}/n\leq 2\|\bbf\|_2^2/n \asto 2\|\fNL\|_{L_2}^2<\infty$ from the strong law of large numbers.
    Thus, the last term is almost surely bounded.
    By choosing $t=\sqrt{|\langle \bB_1^{-1},\bU\rangle|/\|\bU\|_{\tr}}$, it then follows that $ n^{-1}|\langle \bB_1^{-1}\bB_2\bB_1^{-1} ,\bU\rangle| \asto 0$.
    Therefore,
    \begin{align*}
        \left|\frac{1}{n}\bbf^{\top}\bB_1^{-1}\bB_2\bB_1^{-1}\bbf - \frac{1}{n}\sum_{i=1}^n [\bB_1^{-1}\bB_2\bB_1^{-1}]_{ii}[\bbf]_i^2\right| \asto 0.
    \end{align*}

    \noindent\textbf{Trace concentration.} From the results in \citet{knowles2017anisotropic}, it holds that
    \begin{align*}
        \max_{1\leq i\leq n}\left|[\bB_1^{-1}\bB_2\bB_1^{-1}]_{ii} - \frac{1}{n}\tr[\bB_1^{-1}\bB_2\bB_1^{-1}]\right| \asto 0.
    \end{align*}
    Further, since $n^{-1}\|\bbf\|_2^2\asto \|\fNL\|_{L^2}^2$, we have
    \begin{align*}
        \frac{1}{n}|\bbf^{\top} \bB_1^{-1}\bB_2\bB_1^{-1}\bbf - \tr[\bB_1^{-1}\bB_2\bB_1^{-1}]\|\fNL\|_{L^2}^2|\asto 0,
    \end{align*}
    which finishes the proof.
\end{proof}

\newpage

\section{Additional numerical illustrations in \Cref{sec:finite-sample-analysis}}
\label{app:experiments-gaussian}

\subsection{Comparison of intermediate \ovlp- and full-estimators for elastic net and lasso}
\label{app:experiments-gaussian-sub-vs-full}

\noindent\underline{Elastic net:}

\begin{figure}[!ht]
    \centering
    \includegraphics[width=0.75\textwidth]{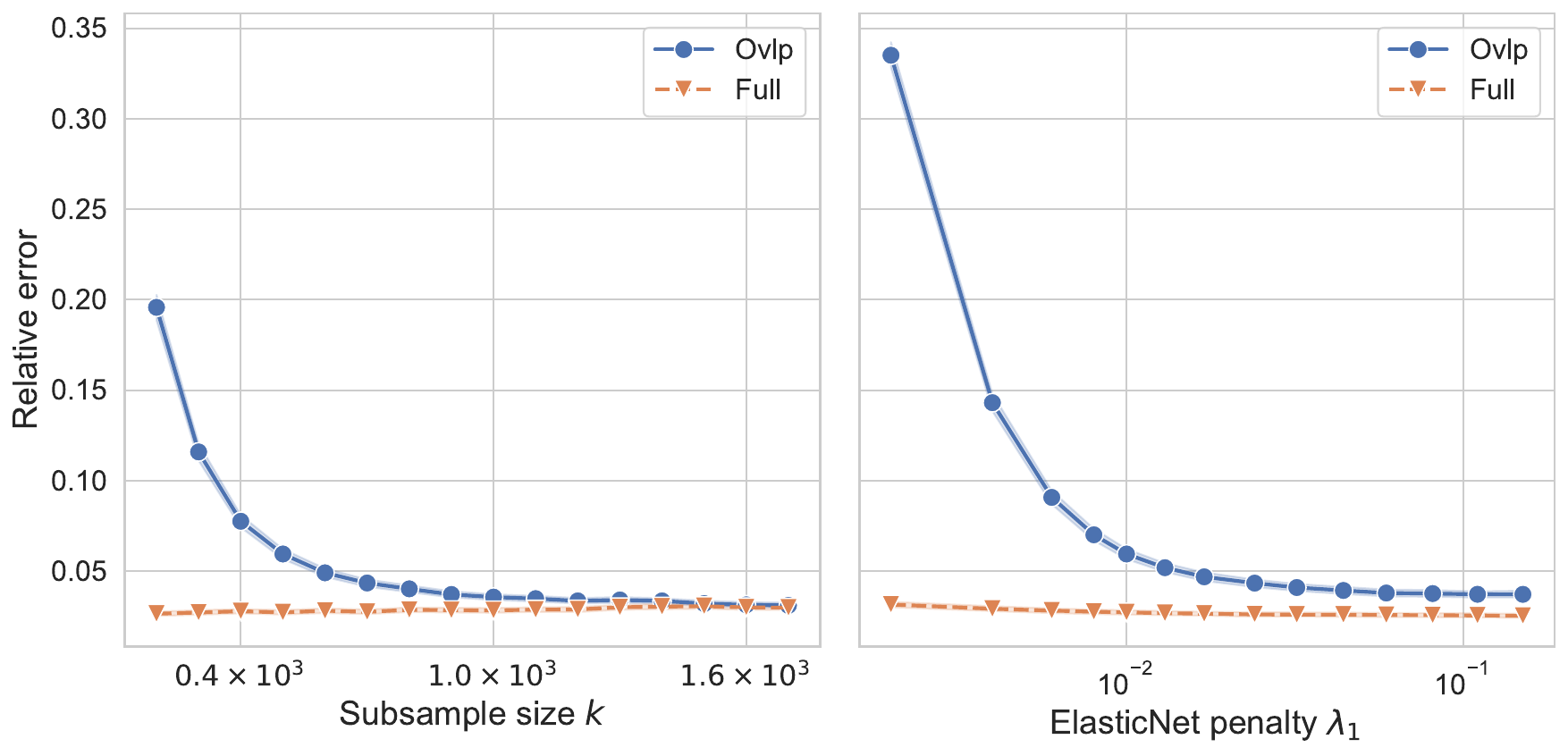}
    \caption{
    \textbf{Full-estimator performs better than \ovlp-estimator across different subsample sizes.}
    The relative errors of the ``\ovlp'' and ``full'' estimators for ridge ensemble with ensemble size $M=10$.
    The left panel shows the results with ridge penalty $\lambda=1$ and varying subsample size $k$.
    The right panel shows the results with subsample size $k=300$ and varying ridge penalty $\lambda$.
    The data generating process is the same as in \Cref{fig:Fig1_ridge_lasso}.
    }
    \label{fig:Fig2_sub_full_ElasticNet.pdf}
\end{figure}

\noindent\underline{Lasso:}

\begin{figure}[!ht]
    \centering
    \includegraphics[width=0.75\textwidth]{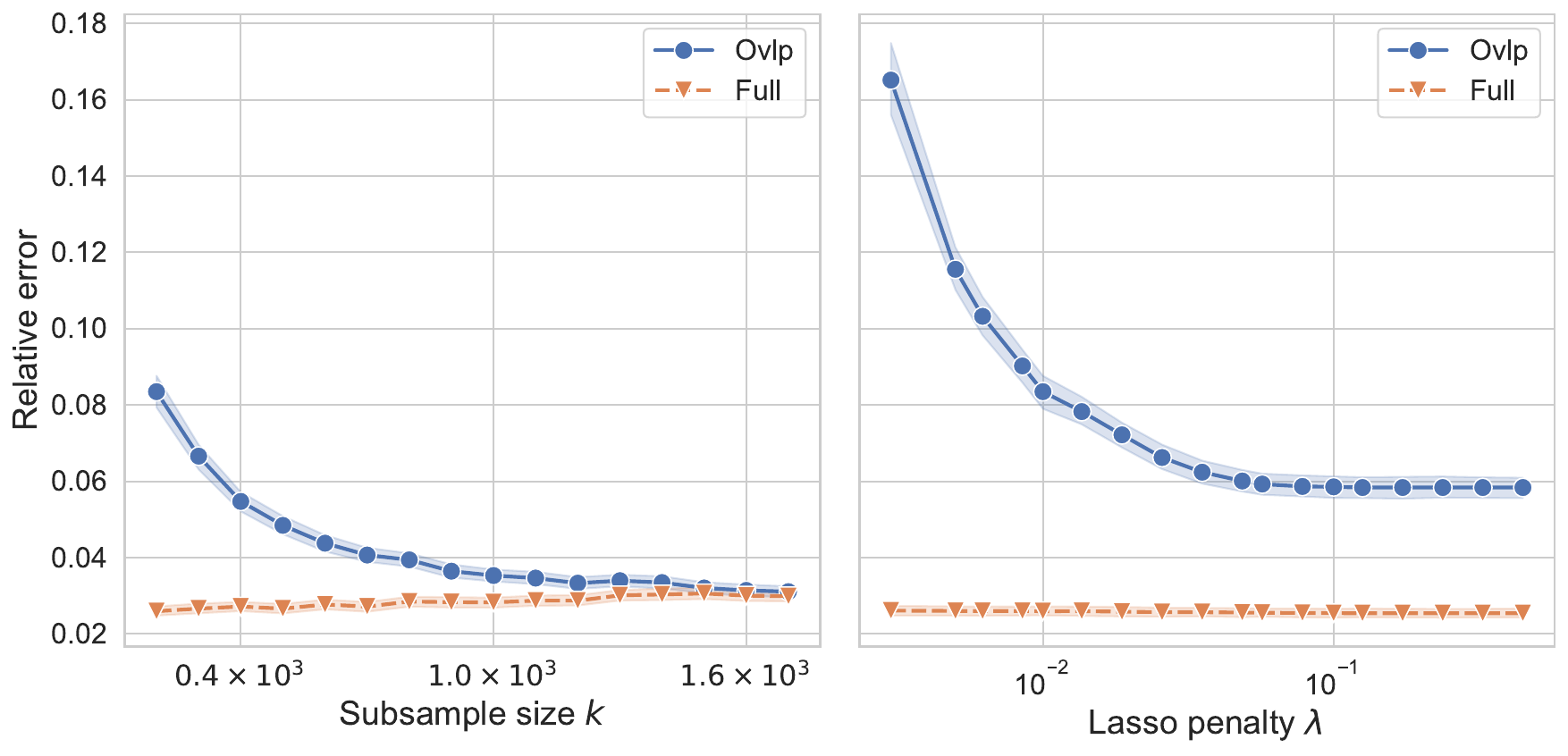}
    \caption{
    \textbf{Full-estimator performs better than \ovlp-estimator across different subsample sizes.}
    The relative errors of the ``\ovlp'' and ``full'' estimators for ridge ensemble with ensemble size $M=10$.
    The left panel shows the results with ridge penalty $\lambda=1$ and varying subsample size $k$.
    The right panel shows the results with subsample size $k=300$ and varying ridge penalty $\lambda$.
    The data generating process is the same as in \Cref{fig:Fig1_ridge_lasso}.
    }
    \label{fig:Fig2_sub_full_k.pdf}
\end{figure}

\newpage

\subsection{Comparison of \ovlp- and full-CGCV for ridge, elastic net, and lasso}
\label{app:experiments-gaussian-sub-vs-full-CGCV}

\noindent\underline{Ridge:}

\begin{figure}[!ht]
    \centering
    \includegraphics[width=0.75\textwidth]{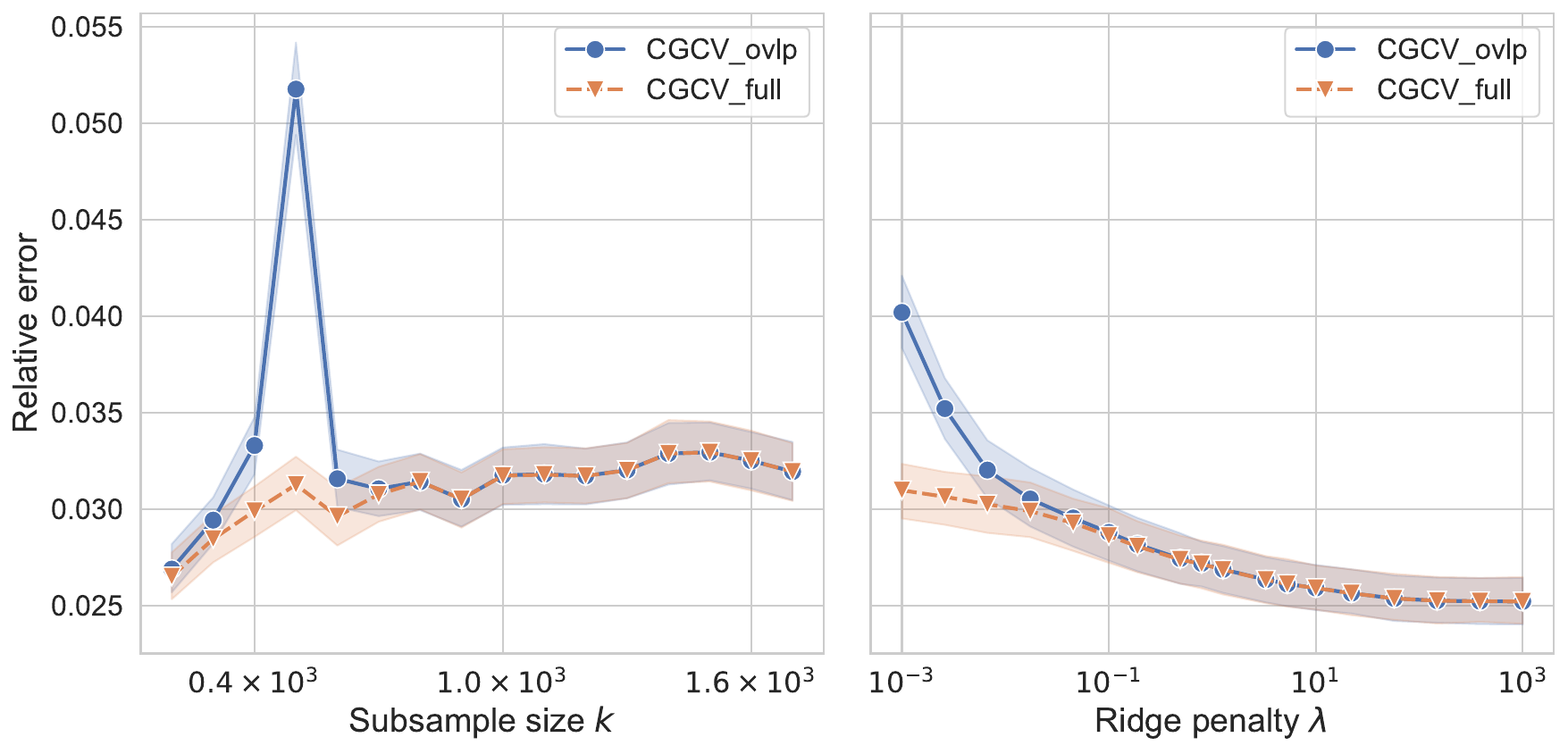}
    \caption{
    \textbf{
    The CGCV{\_}full estimator ($\tR_M^{\cgcv,\full}$) performs better than CGCV\_ovlp 
 estimator ($\tR_M^{\cgcv,\sub}$). }
    Plots of relative errors of the $\tR_M^{\cgcv,\full}$ and $\tR_M^{\cgcv,\sub}$ for ridge ensemble with ensemble size $M=10$.
    The left panel shows the results with ridge penalty $\lambda=0.001$ and varying subsample size $k$.
    The right panel shows the results with subsample size $k=500$ and varying ridge penalty $\lambda$.
    The data generating process is the same as in \Cref{fig:Fig1_ridge_lasso}.
    }
    \label{fig:Fig2_sub_full_ridge-cgcv.pdf}
\end{figure}

\noindent\underline{Elastic net:}

\begin{figure}[!ht]
    \centering
    \includegraphics[width=0.75\textwidth]{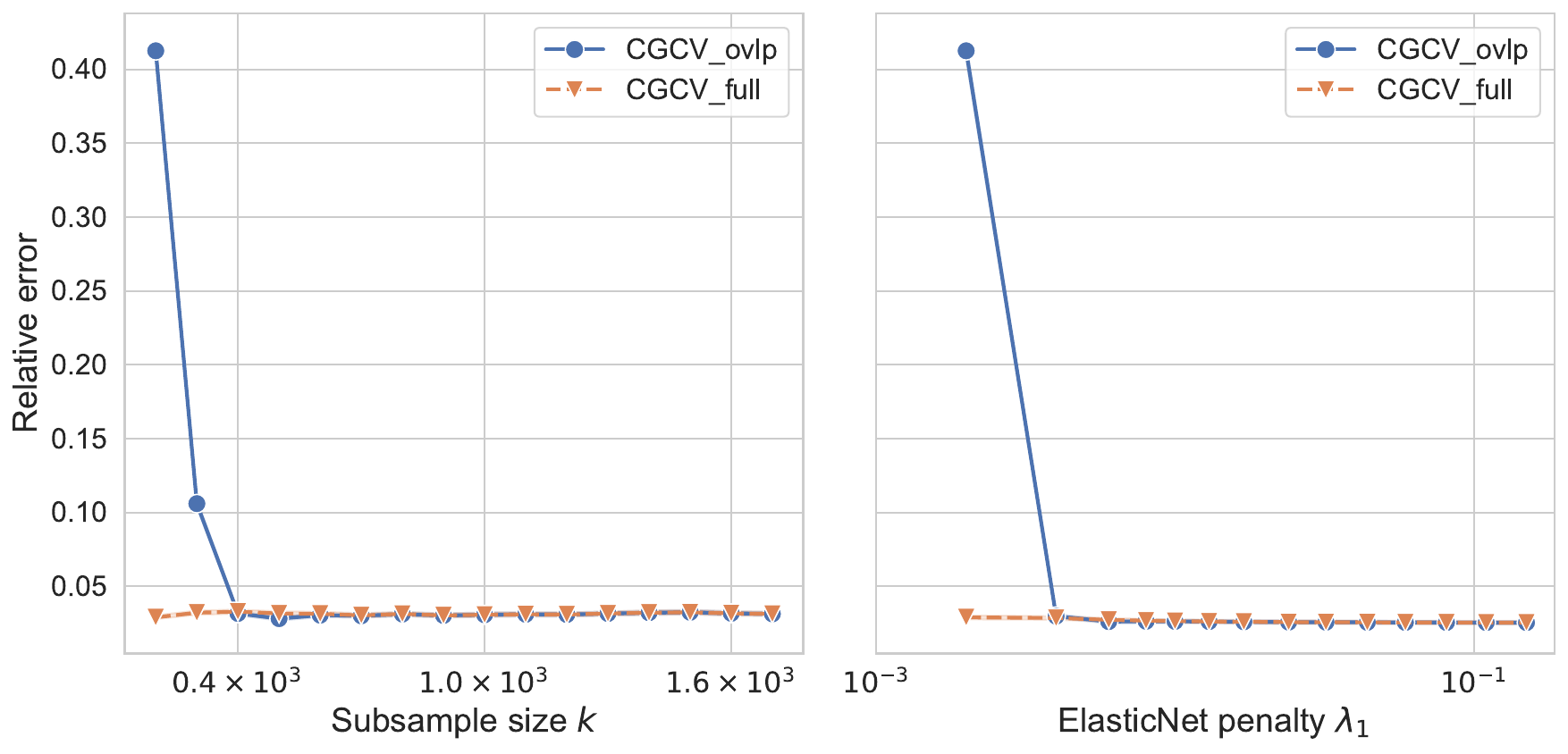}
    \caption{
    \textbf{The CGCV{\_}full estimator ($\tR_M^{\cgcv,\full}$) performs better than CGCV\_ovlp estimator ($\tR_M^{\cgcv,\sub}$). }
    Plots of relative errors of the $\tR_M^{\cgcv,\full}$ and $\tR_M^{\cgcv,\sub}$ for elastic net ensemble with ensemble size $M=10$ and $\lambda_2=0.01$.
    The left panel shows the results with elastic net penalty $\lambda_1=0.002$ and varying subsample size $k$.
    The right panel shows the results with subsample size $k=200$ and varying elastic net penalty $\lambda_1$.
    The data generating process is the same as in \Cref{fig:Fig1_ridge_lasso}.
    }
    \label{fig:Fig2_sub_full_ElasticNet-cgcv.pdf}
\end{figure}

\clearpage
\noindent\underline{Lasso:}

\begin{figure}[!ht]
    \centering
    \includegraphics[width=0.75\textwidth]{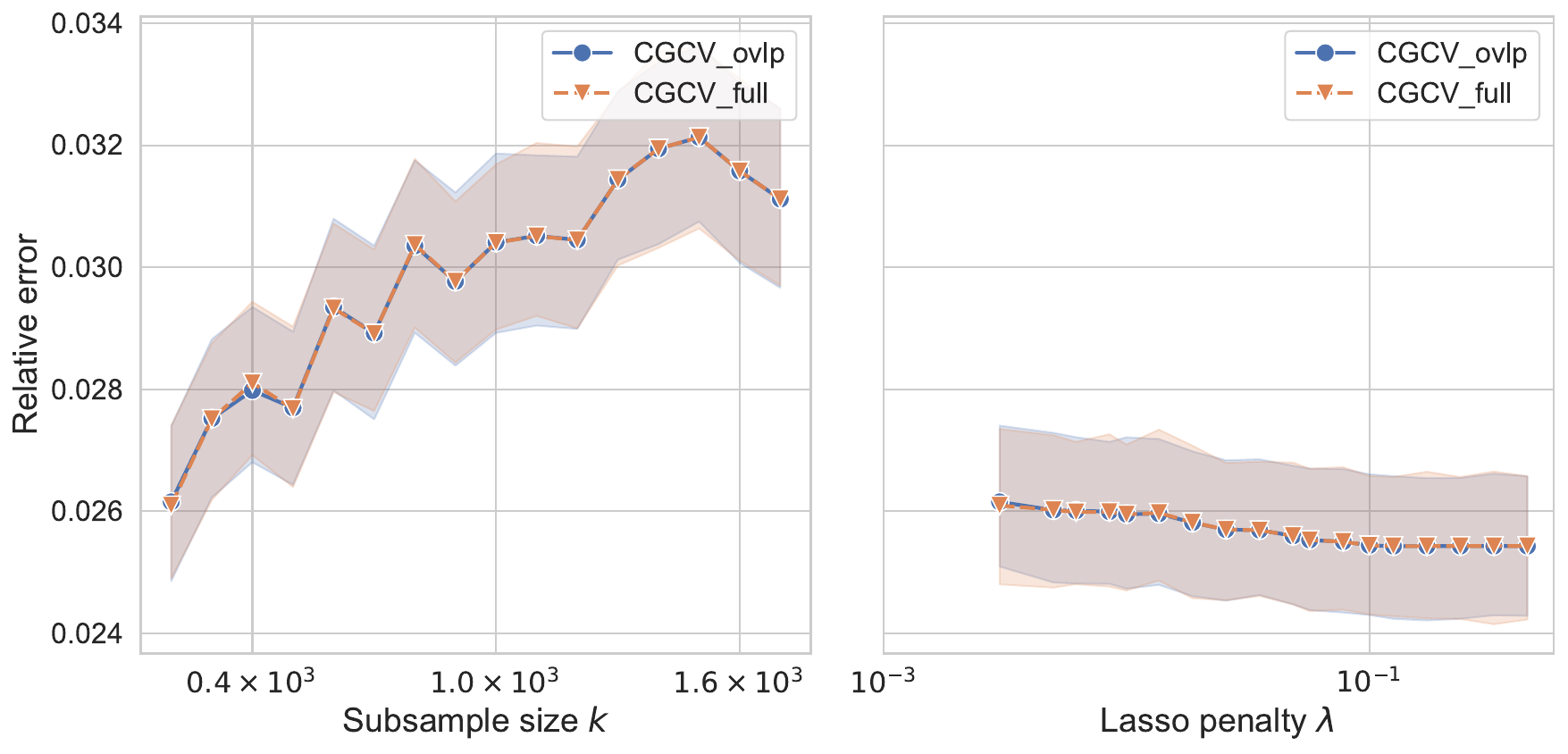}
    \caption{
    \textbf{The CGCV{\_}full estimator ($\tR_M^{\cgcv,\full}$) performs similar to CGCV\_ovlp estimator ($\tR_M^{\cgcv,\sub}$). }
    Plots of relative errors of the $\tR_M^{\cgcv,\full}$ and $\tR_M^{\cgcv,\sub}$ for lasso ensemble with ensemble size $M=10$.
    The left panel shows the results with lasso penalty $\lambda=0.003$ and varying subsample size $k$.
    The right panel shows the results with subsample size $k=200$ and varying lasso penalty $\lambda$.
    The data generating process is the same as in \Cref{fig:Fig1_ridge_lasso}.
    }
    \label{fig:Fig2_sub_full_lasso.pdf}
\end{figure}

\subsection[Comparison of CGCV and GCV in subsample size for different ensemble size]{Comparison of CGCV and GCV in $k$ for ridge, elastic net, and lasso}
\label{app:experiments-gaussian-cgcv-vs-gcv-in-k}

\noindent\underline{Ridge:}

\begin{figure}[!ht]
    \centering
    \includegraphics[width=0.95\textwidth]{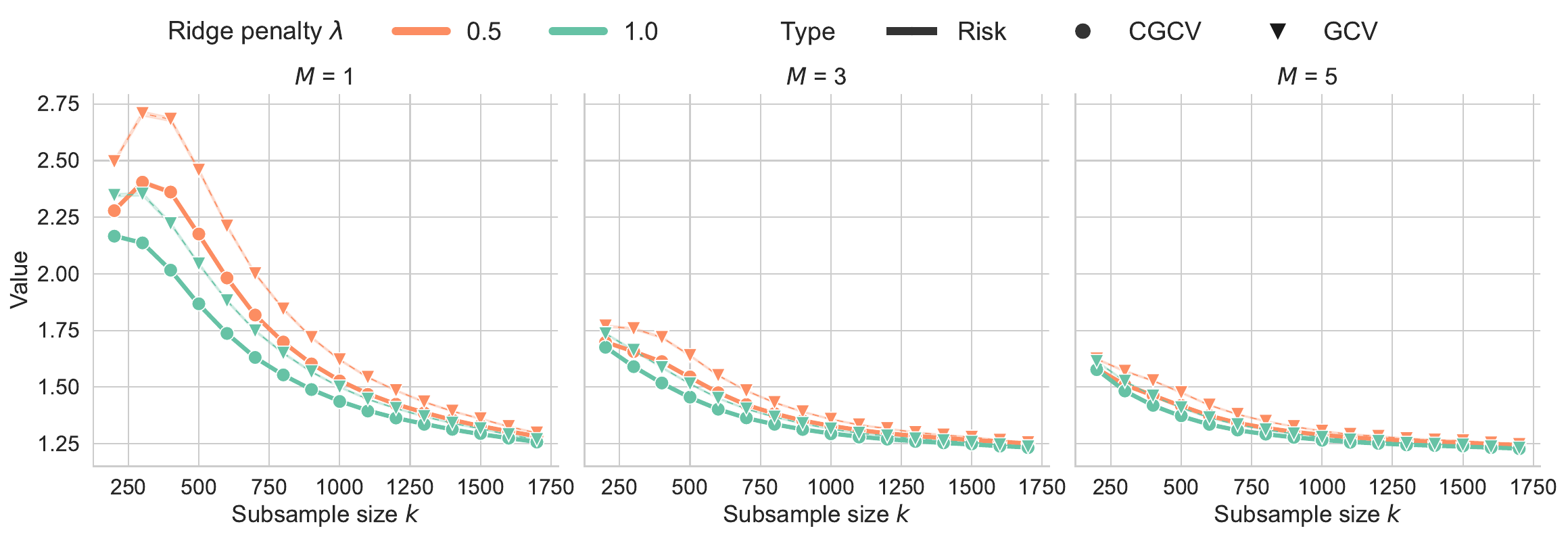}
    \caption{
    \textbf{CGCV is consistent for finite ensembles across different subsample sizes.}
    Plot of risks versus the subsample size $k$ for the ridge ensemble with different $\lambda$ and $M$.
    The data generating process is the same as in \Cref{fig:Fig1_ridge_lasso}.
    }
    \label{fig:surface-plot}
\end{figure}

\clearpage
\noindent\underline{Elastic Net:}

\begin{figure}[!ht]
    \centering
    \includegraphics[width=0.95\textwidth]{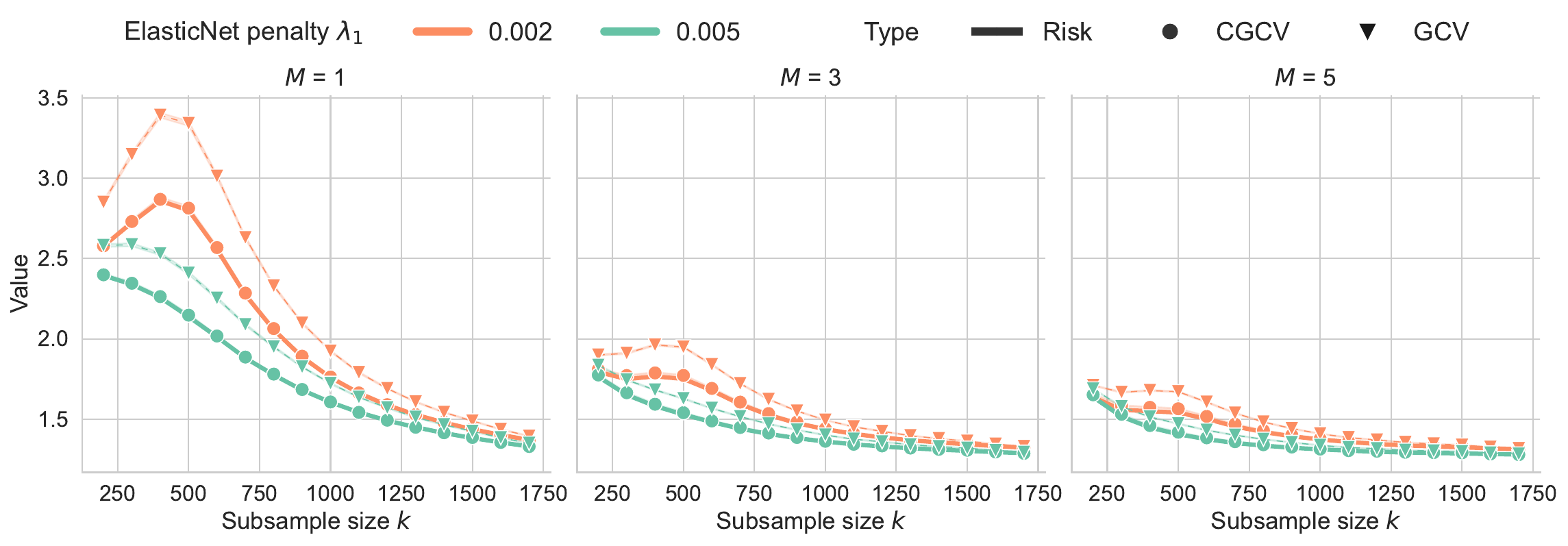}
    \caption{
    \textbf{CGCV is consistent for finite ensembles across different subsample sizes.}
    Plot of risks versus the subsample size $k$ for elastic net ensemble with different $\lambda$ and $M$.
    }
    \label{fig:ElasticNet_k.pdf}
\end{figure}

\noindent\underline{Lasso:}

\begin{figure}[!ht]
    \centering
    \includegraphics[width=0.95\textwidth]{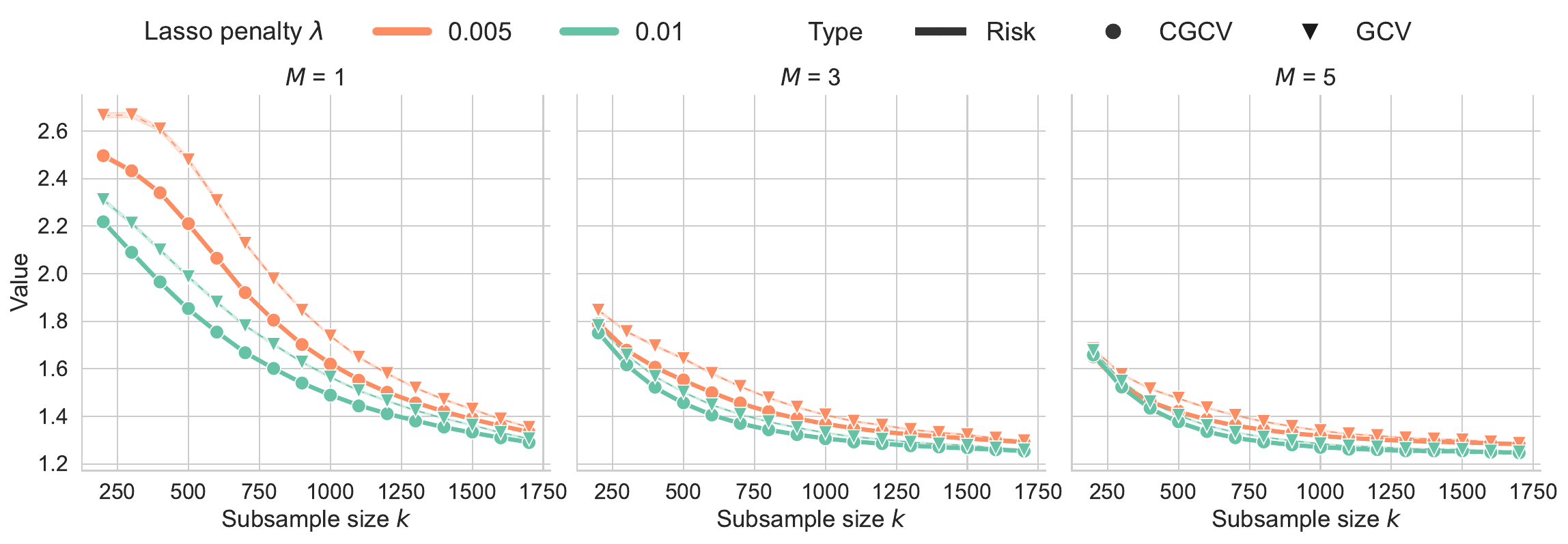}
    \caption{
    \textbf{CGCV is consistent for finite ensembles across different subsample sizes.}
    Plot of risks versus the subsample size $k$ for elastic net ensemble with different $\lambda$ and $M$.
    }
    \label{fig:lasso_k.pdf}
\end{figure}

\clearpage
\subsection[Comparison of GCV and GCV in ensemble size for ridge, elastic net, and lasso]{Comparison of GCV and GCV in $M$ for ridge, elastic net, and lasso}
\label{app:experiments-gaussian-cgcv-vs-gcv-in-M}

\noindent\underline{Ridge:}

\begin{figure}[!ht]
    \centering
    \includegraphics[width=0.95\textwidth]{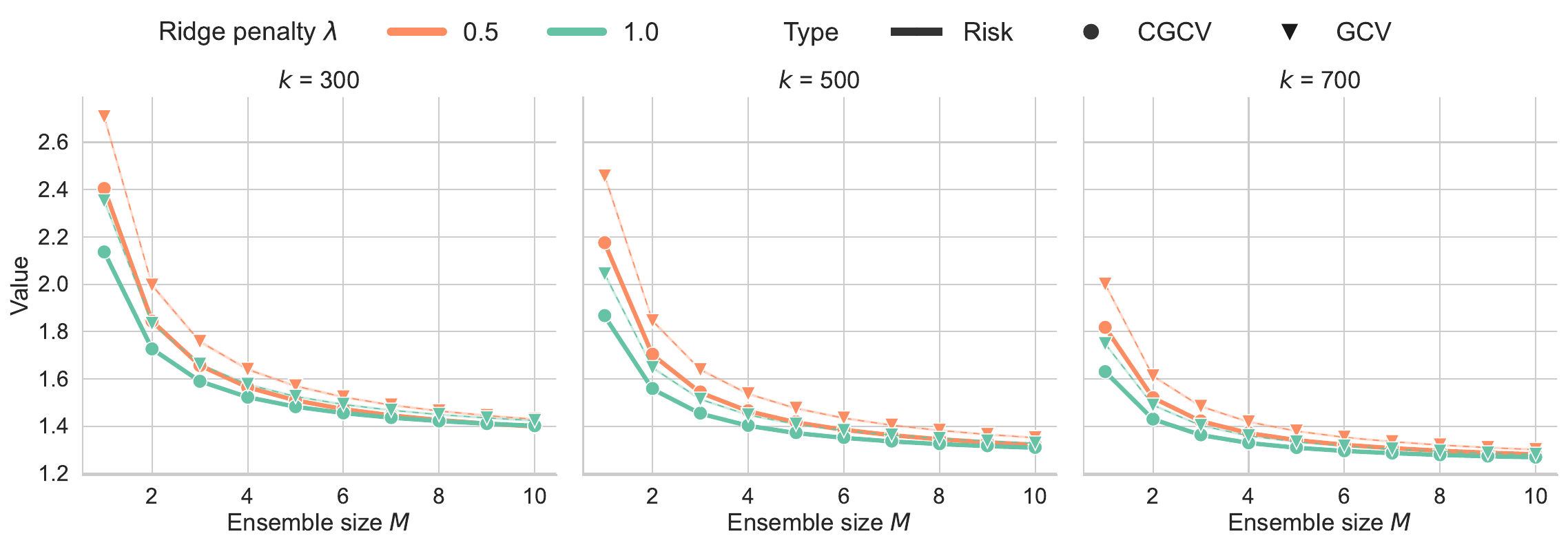}
    \caption{
    \textbf{GCV gets closer to the risk as ensemble size increases across different subsample sizes.}
    Plot of risks versus the subsample size $k$ for the lasso ensemble with different $\lambda$ and $M$.
    The data generating process is the same as in \Cref{fig:Fig1_ridge_lasso}.
    }
    \label{fig:lasso-surface-plot}
\end{figure}

\noindent\underline{Elastic Net:}

\begin{figure}[!ht]
    \centering
    \includegraphics[width=0.95\textwidth]{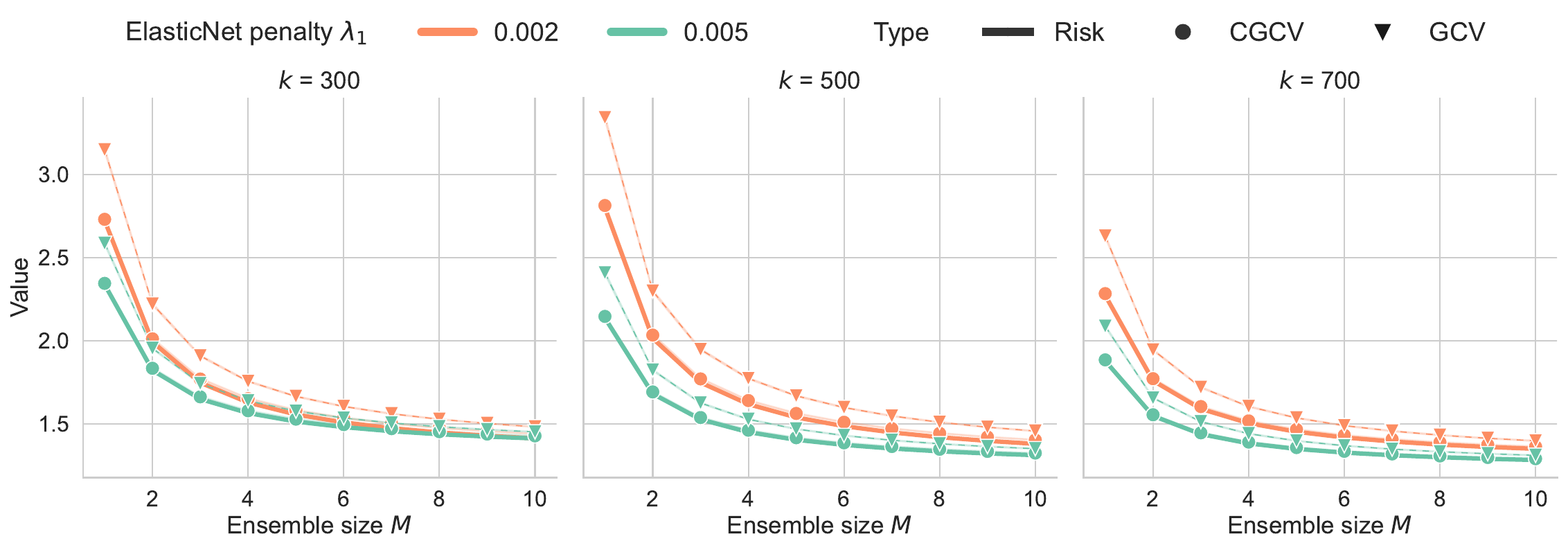}
    
    \caption{
    \textbf{GCV gets closer to the risk as ensemble size increases across different subsample sizes.}
    Plot of risks versus the ensemble size $M$ for ridge ensemble with different $\lambda$ and $k$.
    }
    \label{fig:ElasticNet_M.pdf}
\end{figure}

\clearpage
\noindent\underline{Lasso:}

\begin{figure}[!ht]
    \centering
    \includegraphics[width=0.95\textwidth]{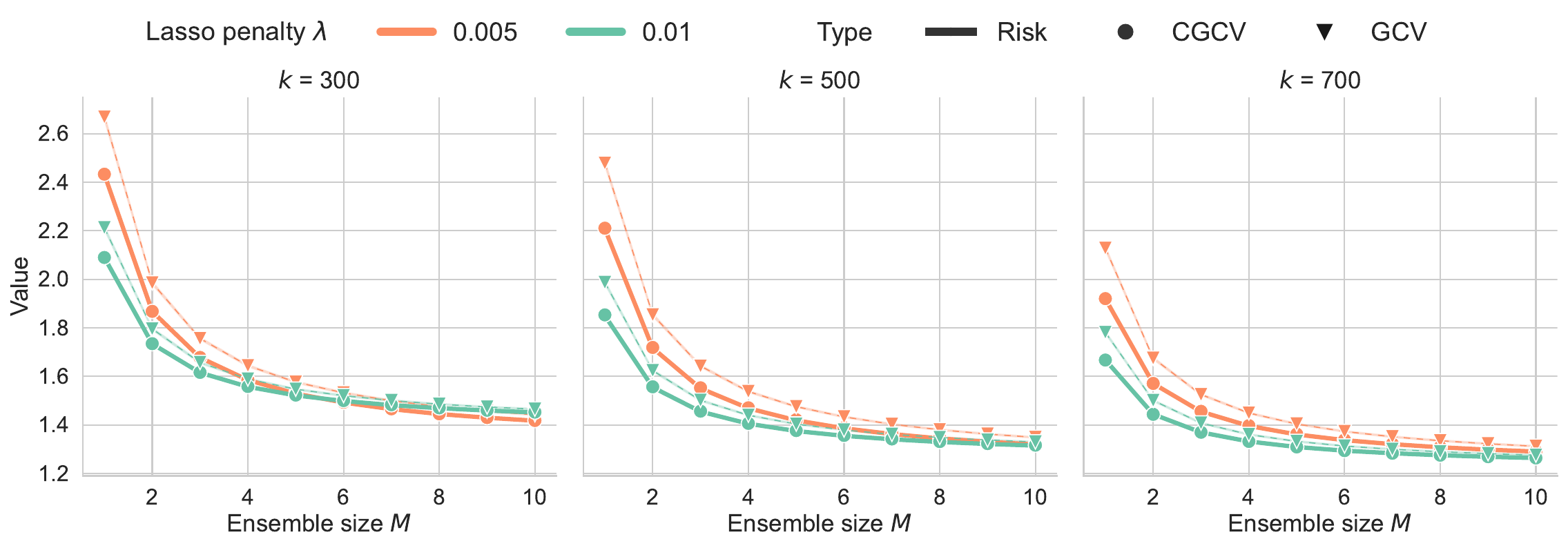}
    
    \caption{
    \textbf{GCV gets closer to the risk as ensemble size increases across different subsample sizes.}
    Plot of risks versus the ensemble size $M$ for ridge ensemble with different $\lambda$ and $k$.
    }
    \label{fig:lasso_M.pdf}
\end{figure}

\subsection[Comparison of CGCV and GCV in regularization level]{Comparison of CGCV and GCV in $\lambda$ for ridge, elastic net and lasso}
\label{app:experiments-gaussian-cgcv-vs-gcv-in-lambda}

\noindent\underline{Ridge:}

\begin{figure}[!ht]
    \centering
    \includegraphics[width=0.95\textwidth]{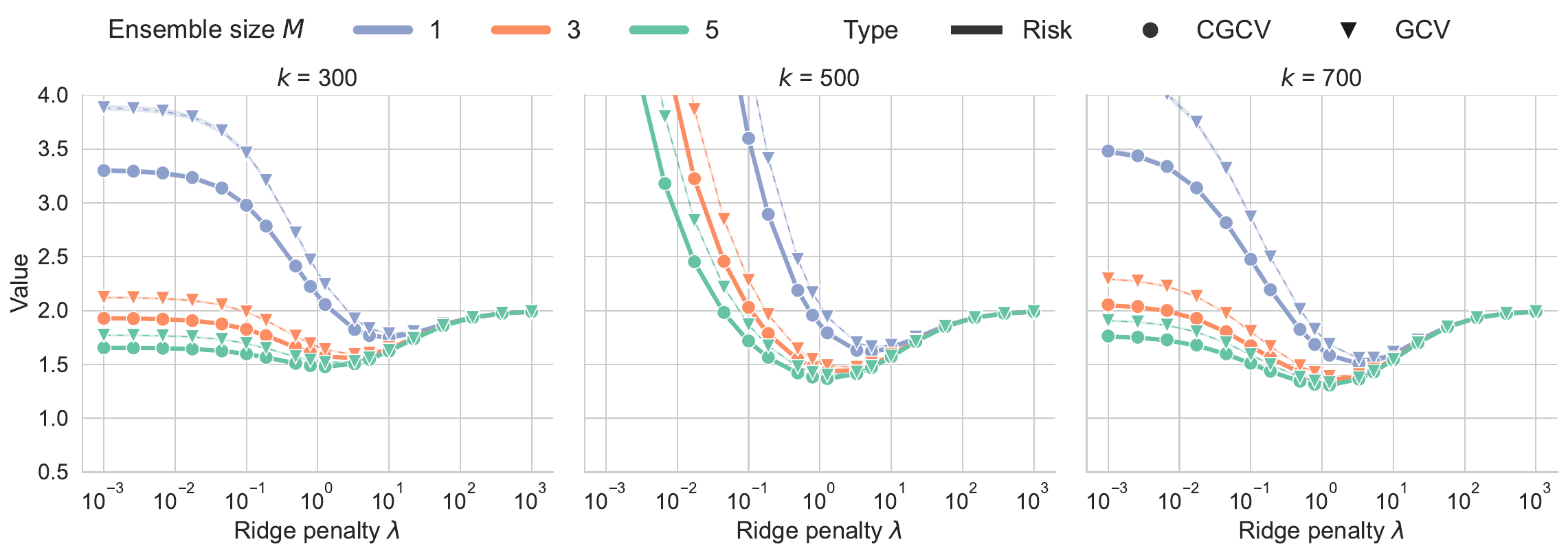}
    \caption{
    \textbf{CGCV is consistent for finite ensembles across different regularization levels.}
    Plot of risks versus the ensemble size $M$ for ridge ensemble with different $\lambda$ and $k$.
    }
    \label{fig:lasso-lam}
\end{figure}

\clearpage
\noindent\underline{Elastic Net:}

\begin{figure}[!ht]
    \centering
    \includegraphics[width=0.95\textwidth]{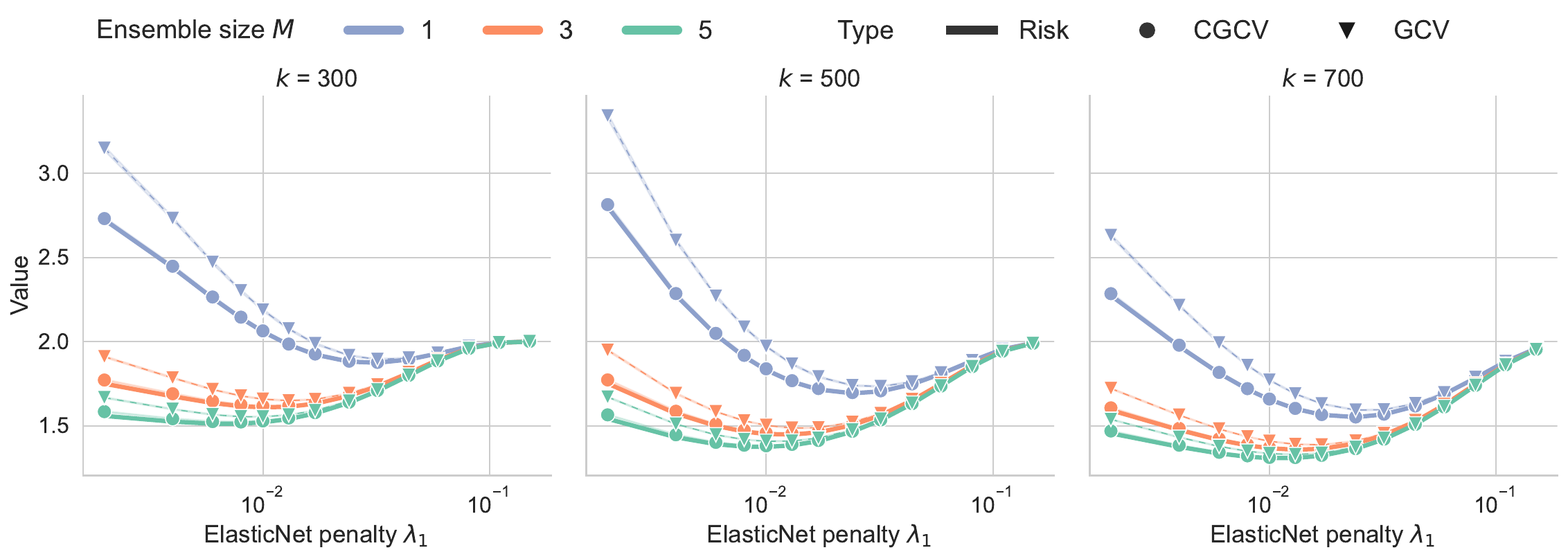}
    
    \caption{
    \textbf{GCV gets closer to the risk as ensemble size increases across different subsample sizes.}
    Plot of risks versus the ensemble size $M$ for elastic net ensemble with different $\lambda$ and $k$.
    }
    \label{fig:lam-ridge}
\end{figure}

\noindent\underline{Lasso:}

\begin{figure}[!ht]
    \centering
    \includegraphics[width=0.95\textwidth]{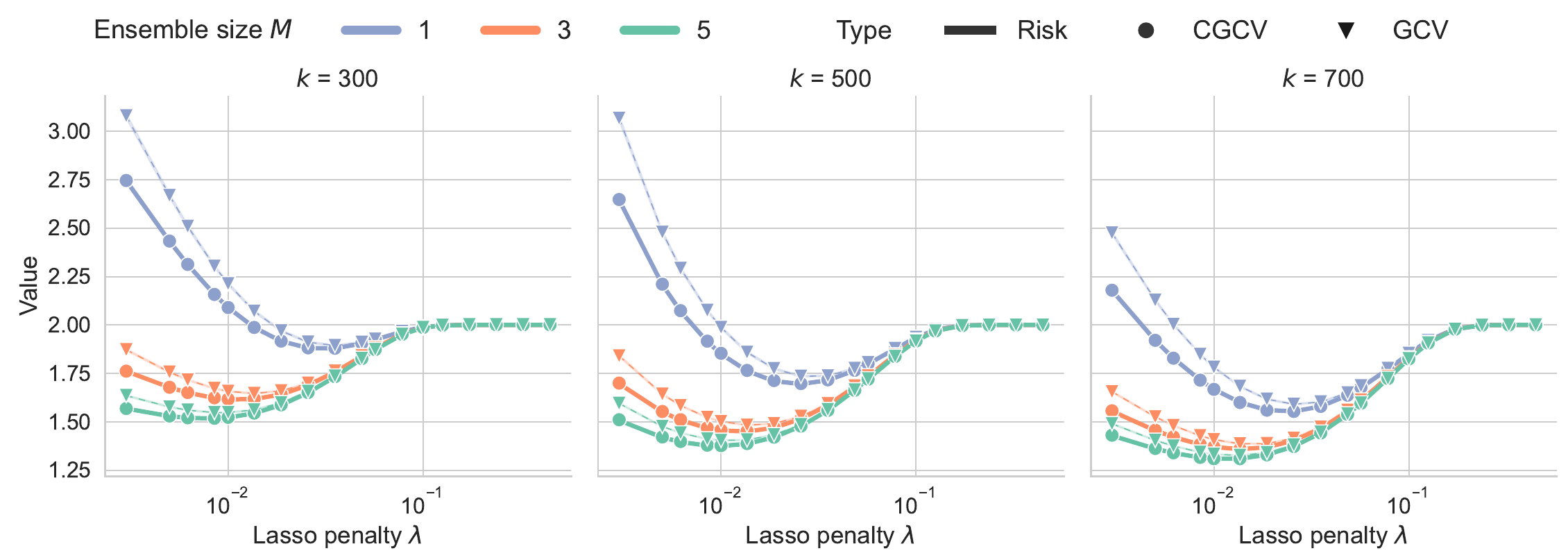}
    
    \caption{
    \textbf{GCV gets closer to the risk as ensemble size increases across different subsample sizes.}
    Plot of risks versus the ensemble size $M$ for lasso ensemble with different $\lambda$ and $k$.
    }
    \label{fig:lam-lasso}
\end{figure}

\subsection[Numerical simulations with large ensemble size M]{Numerical simulations with large ensemble size $M$}
\label{sec:simulations-large-M}

Recall that \Cref{sec:numerical-illustrations-gaussian} presents simulation results when $M$ is at most $10$. 
In this section, we provide simulation results in the same setting as described in \Cref{sec:numerical-illustrations-gaussian}, but with $M$ as large as $100$. The goal is to illustrate that our proposed CGCV estimator is efficient not only for small $M$ but also for large $M$. Since the bias of the GCV is decreasing with $M$, we expect the GCV works well for large $M$. 
It is evident from \Cref{fig:fun_M_large_M} that the CGCV accurately estimates the actual risk for all values of $M$, while the bias of GCV diminishes only when $M$ is large. 

\begin{figure}[!ht]
    \centering
    \includegraphics[width=0.95\textwidth]{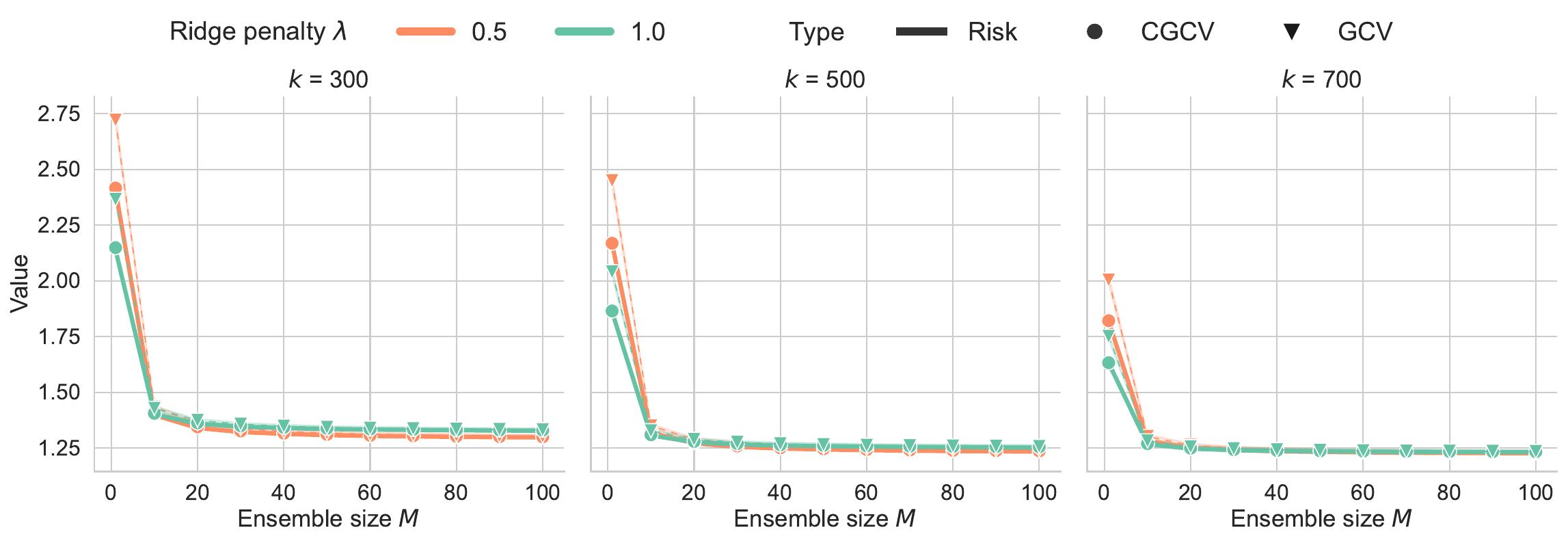}\\
    \includegraphics[width=0.95\textwidth]{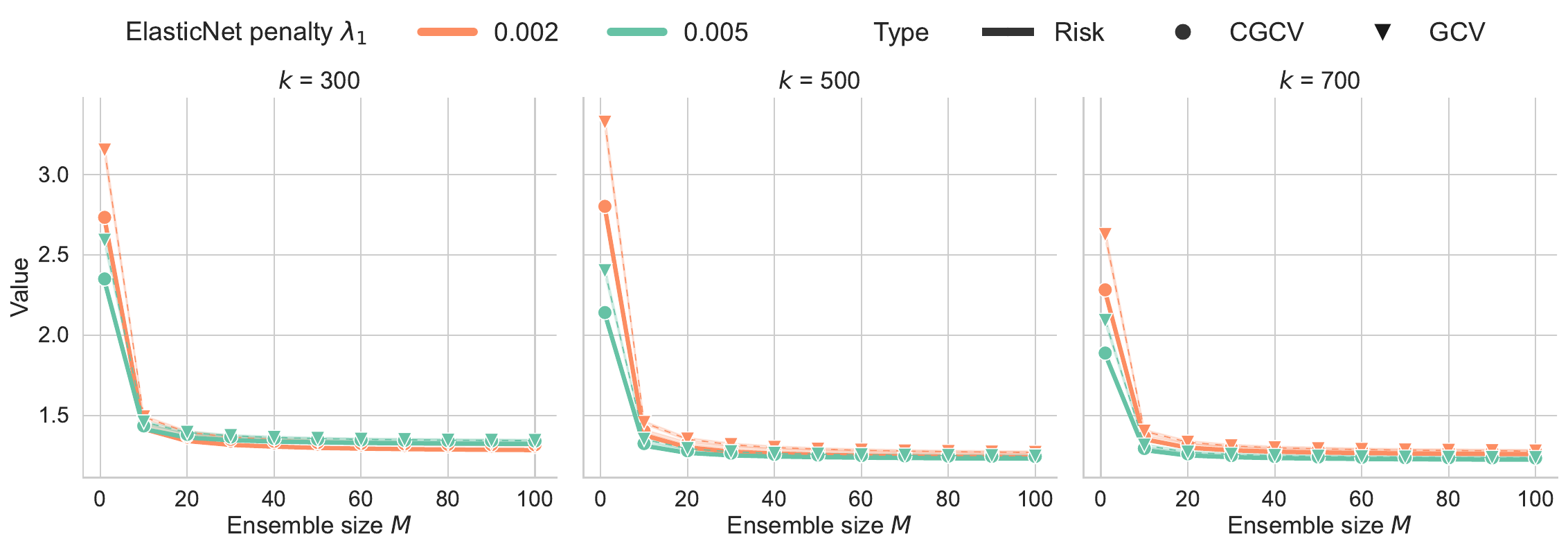}\\
    \includegraphics[width=0.95\textwidth]{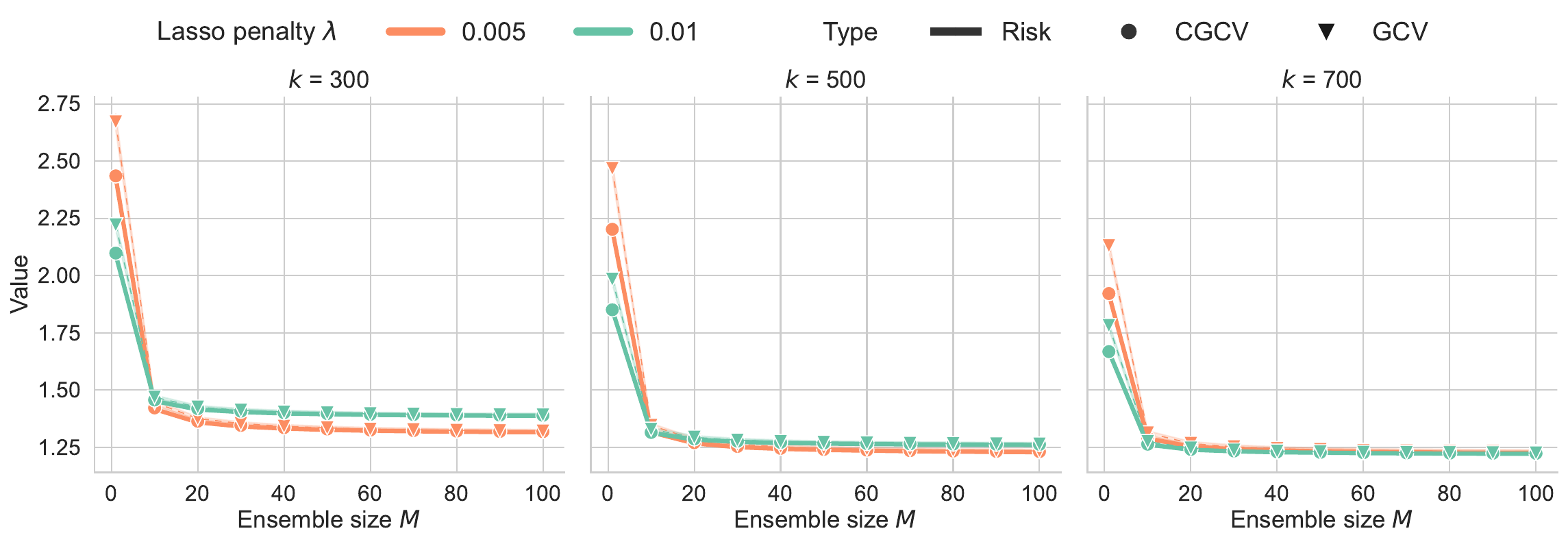}\\
    \caption{
    \textbf{CGCV efficiently estimates the actual risk for all values of $M$.
    GCV gets closer to the actual risk as $M$ increases.}
    Top row: Ridge; Middle row: Elastic Net; Bottom row: Lasso. 
    The data generating process is the same as in \Cref{fig:Fig1_ridge_lasso}.
    }
    \label{fig:fun_M_large_M}
\end{figure}

\clearpage
\subsection{Illustration of inconsistency for GCV variant}
\label{sec:naive-gcv-inconsistency}

\begin{figure}[!ht]
    \centering
    \includegraphics[width=0.95\textwidth]{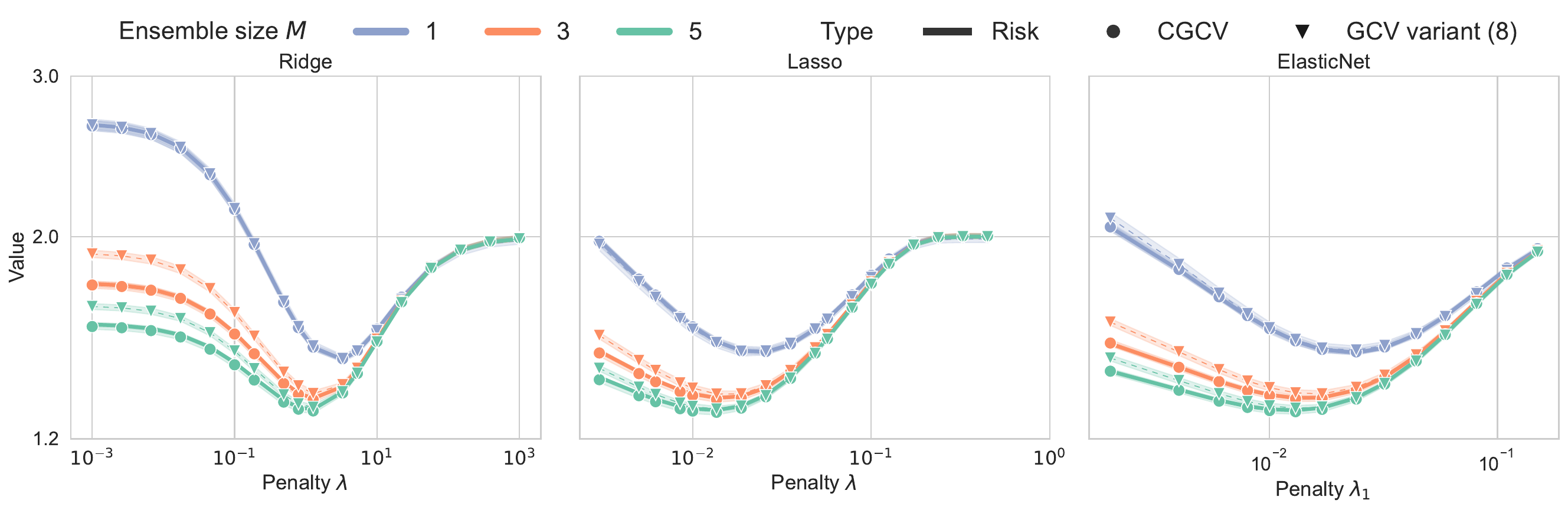}
    
    \caption{
    \textbf{The GCV variant \eqref{eq:gcv-naive} is not consistent for finite $M>1$.}
    Plot of risks versus the regularization parameter for ridge, lasso, and elastic net ensembles,
    under different $M$ and fixed $k=800$. The data generating process is the same as in \Cref{fig:Fig1_ridge_lasso}.
    }
    \label{fig:gcv_naive_lam}
\end{figure}

\clearpage
\section{Additional numerical illustrations for \Cref{sec:finite-sample-analysis} with non-Gaussian data}
\label{sec:numerical-non-gaussian}

In this section, we conduct the same numerical experiments as in \Cref{sec:finite-sample-analysis} except that here the design matrix is not Gaussian distributed. 
We consider two non-Gaussian distributions:
the Rademacher distribution and the uniform distribution in $[-\sqrt{3}, \sqrt{3}]$. 
The design matrix $\bX$ is then generated by a scale transformation to ensure the covariance matrix is the same as $\bSigma$ in 
\Cref{sec:numerical-illustrations-gaussian}. 
In the following subsections, we present counterpart plots to \Cref{fig:Fig1_ridge_lasso,fig:k,fig:ridge-k,fig:M} in the same setting, except here $\bX$ is non-Gaussian distributed.

\subsection{Results for Rademacher distribution}
\label{app:experiments-rademacher}

\begin{figure}[H]
    \centering
    \includegraphics[width=0.8\textwidth]{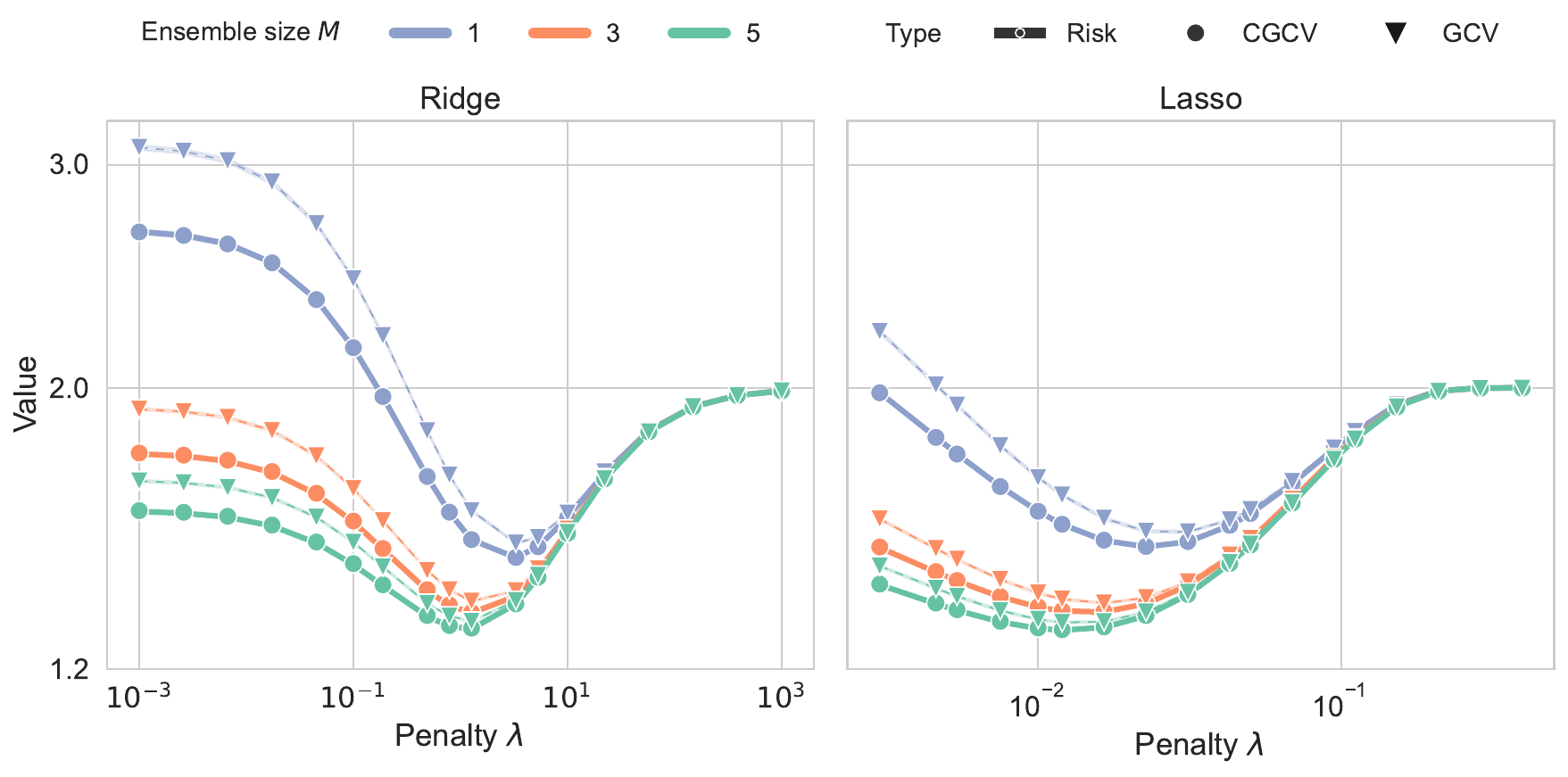}
    \caption{
    \textbf{CGCV is consistent for finite ensembles of penalized estimators while GCV is not.}
    The simulation setting is the same as \Cref{fig:Fig1_ridge_lasso}, except that here $\bX$ is Rademacher distributed. 
    }
    \label{fig:Fig1_ridge_lasso_rademacher}
\end{figure}

\begin{figure}[H]
    \centering
    \includegraphics[width=0.8\textwidth]{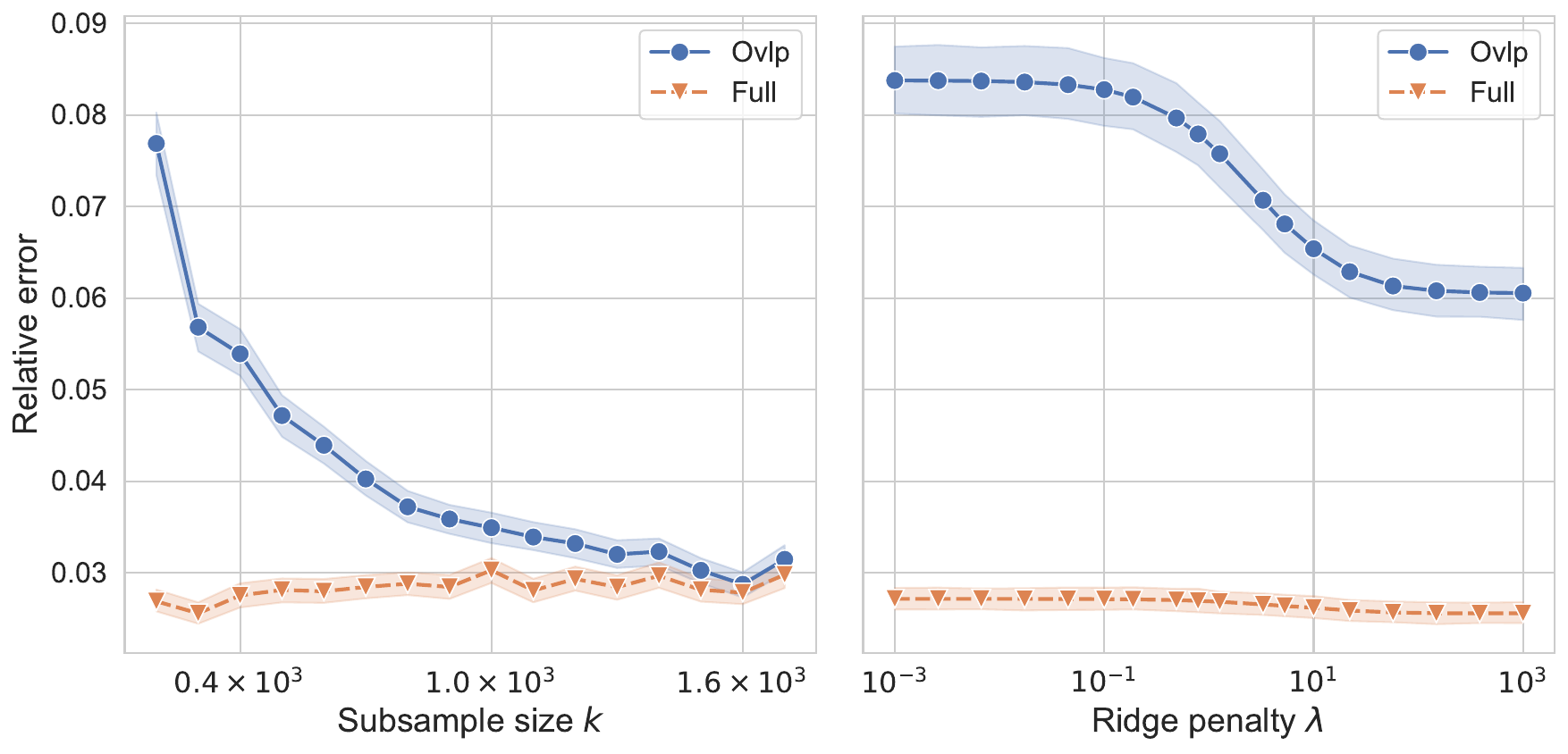}
    \caption{
    \textbf{Full-estimator $\tR_{M}^{\full}$  performs better than \ovlp-estimator $\tR_{M}^{\sub}$ across different regularization and for small subsample size $k$.}
    The simulation setting is the same as \Cref{fig:k}, except that here $\bX$ is Rademacher distributed.
    }\label{fig:k_rademacher}
\end{figure}

\begin{figure}[H]
    \centering
    \includegraphics[width=0.95\textwidth]%
    {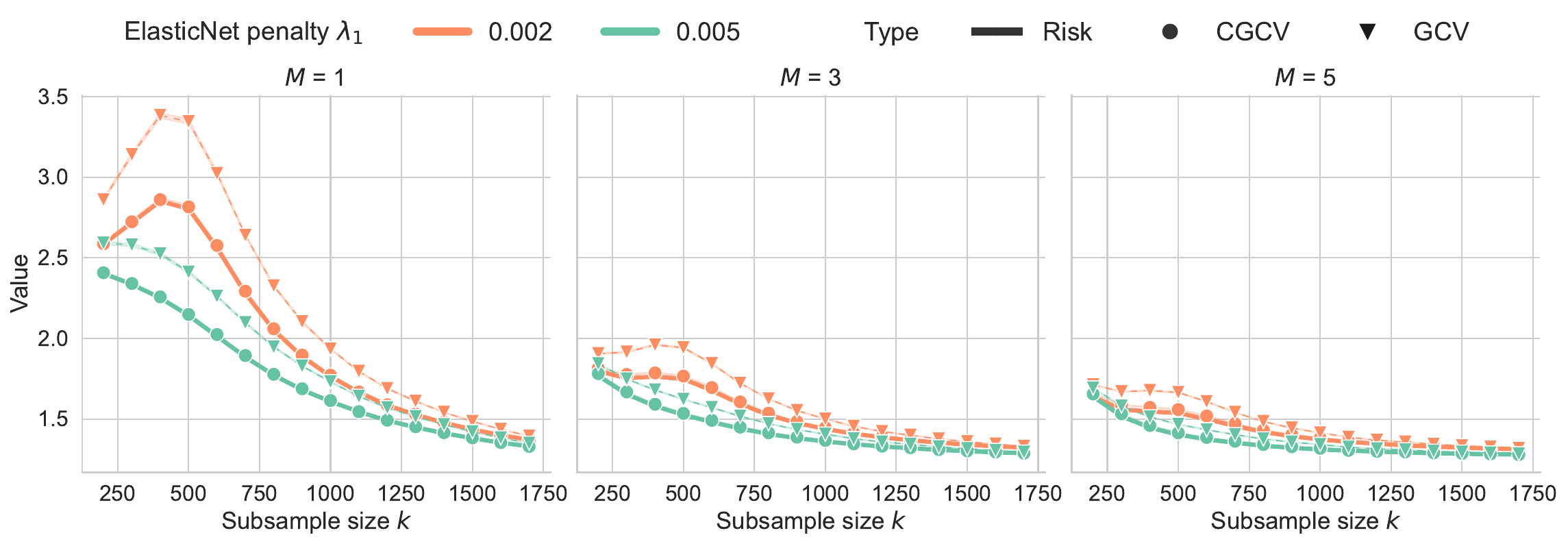}
    \caption{
    \textbf{CGCV is consistent for finite ensembles across different subsample sizes.}
    The plot of risks versus the subsample size $k$ for Elastic Net ensemble with different $\lambda$ and $M$ and fixed $\lambda_2 = 0.01$. 
    The simulation setting is the same as \Cref{fig:ridge-k}, except that here $\bX$ is Rademacher distributed.
    }\label{fig:ridge-k_rademacher}
\end{figure}

\begin{figure}[H]
    \centering
    \includegraphics[width=0.95\textwidth]%
    {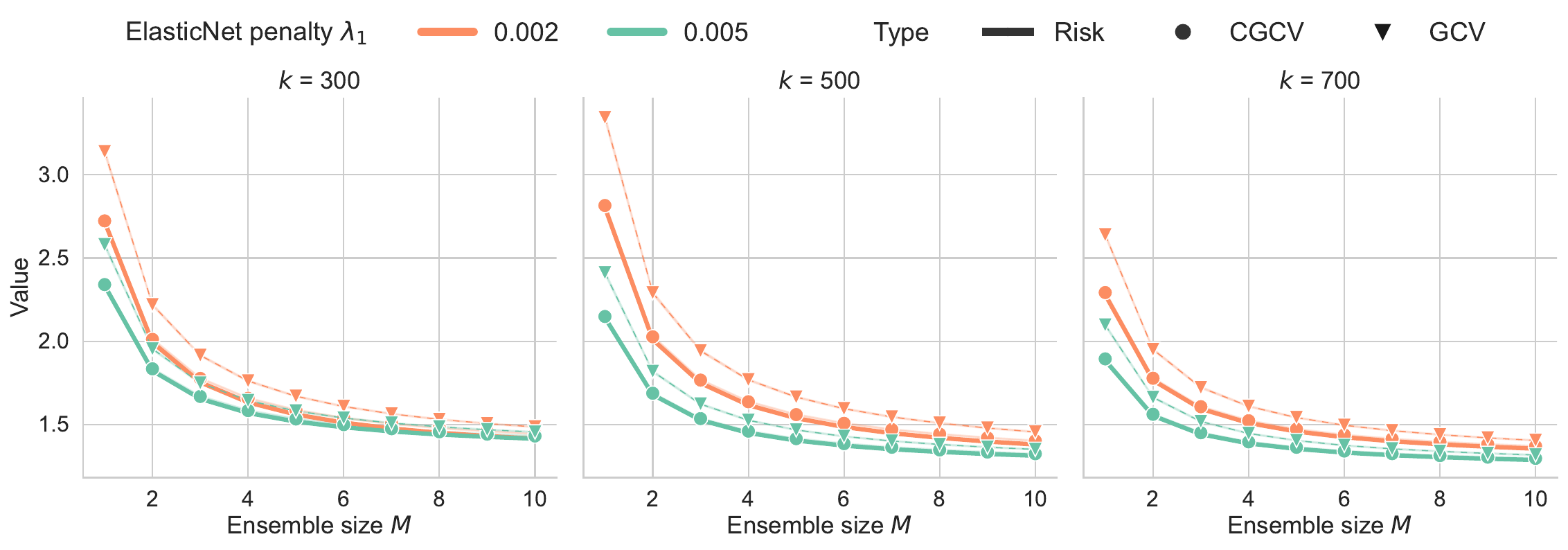}
    \caption{
    \textbf{GCV gets closer to the risk as ensemble size increases across different subsample sizes.}
    The plot of risks versus the Elastic Net ensemble size $M$ for ridge ensemble with different $\lambda$ and $k$ and fixed $\lambda_2 = 0.01$.
    The simulation setting is the same as \Cref{fig:M}, except that here $\bX$ is Rademacher distributed.
    }
    \label{fig:M_rademacher}
\end{figure}

\subsection{Results for uniform distribution}
\label{app:experiments-uniform}

\begin{figure}[H]
    \centering
    \includegraphics[width=0.8\textwidth]{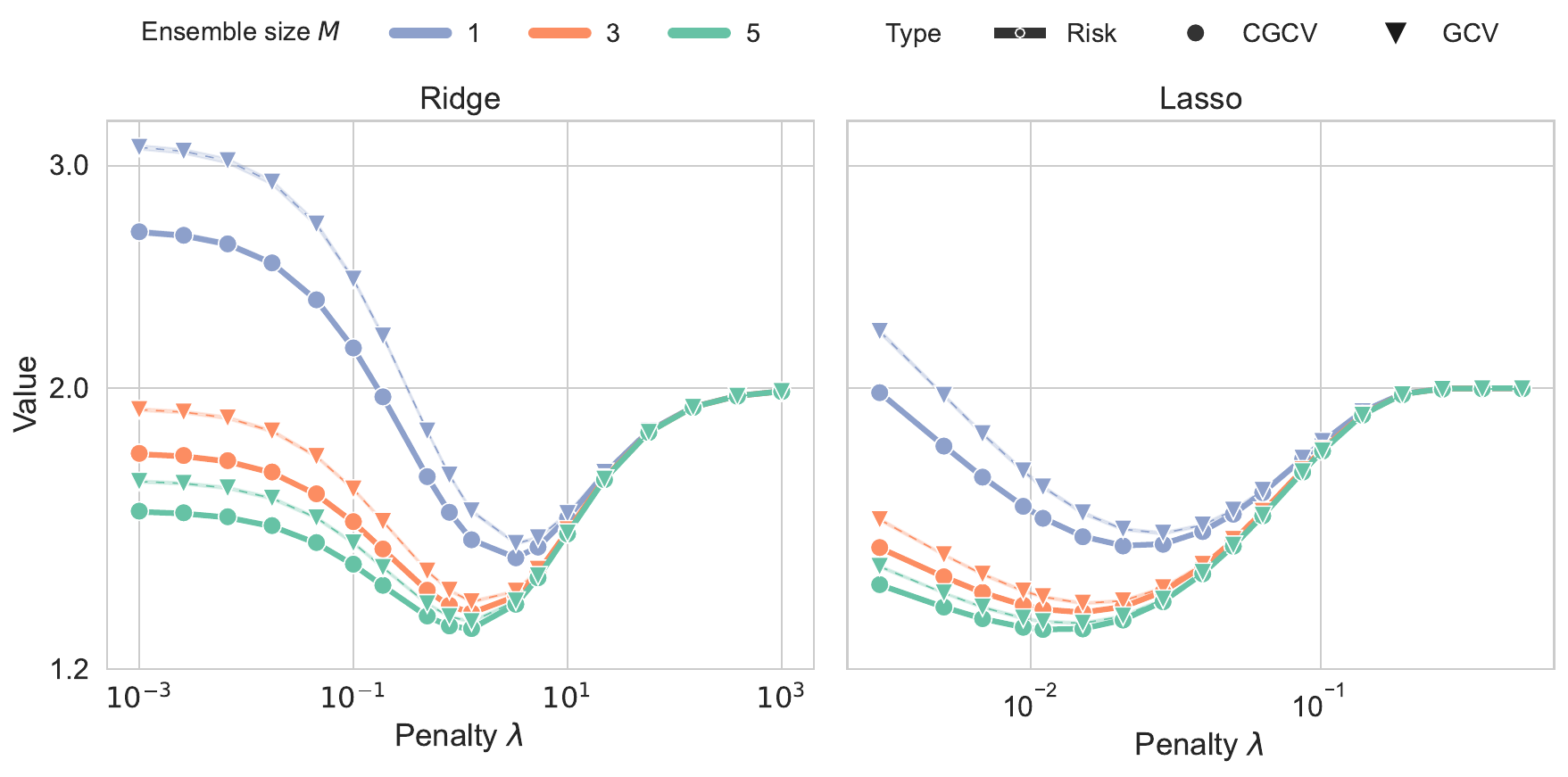}
    \caption{
    \textbf{CGCV is consistent for finite ensembles of penalized estimators while GCV is not.}
    The simulation setting is the same as \Cref{fig:Fig1_ridge_lasso}, except that here $\bX$ is uniformly distributed. 
    }
    \label{fig:Fig1_ridge_lasso_uniform}
\end{figure}

\begin{figure}[H]
    \centering
    \includegraphics[width=0.8\textwidth]{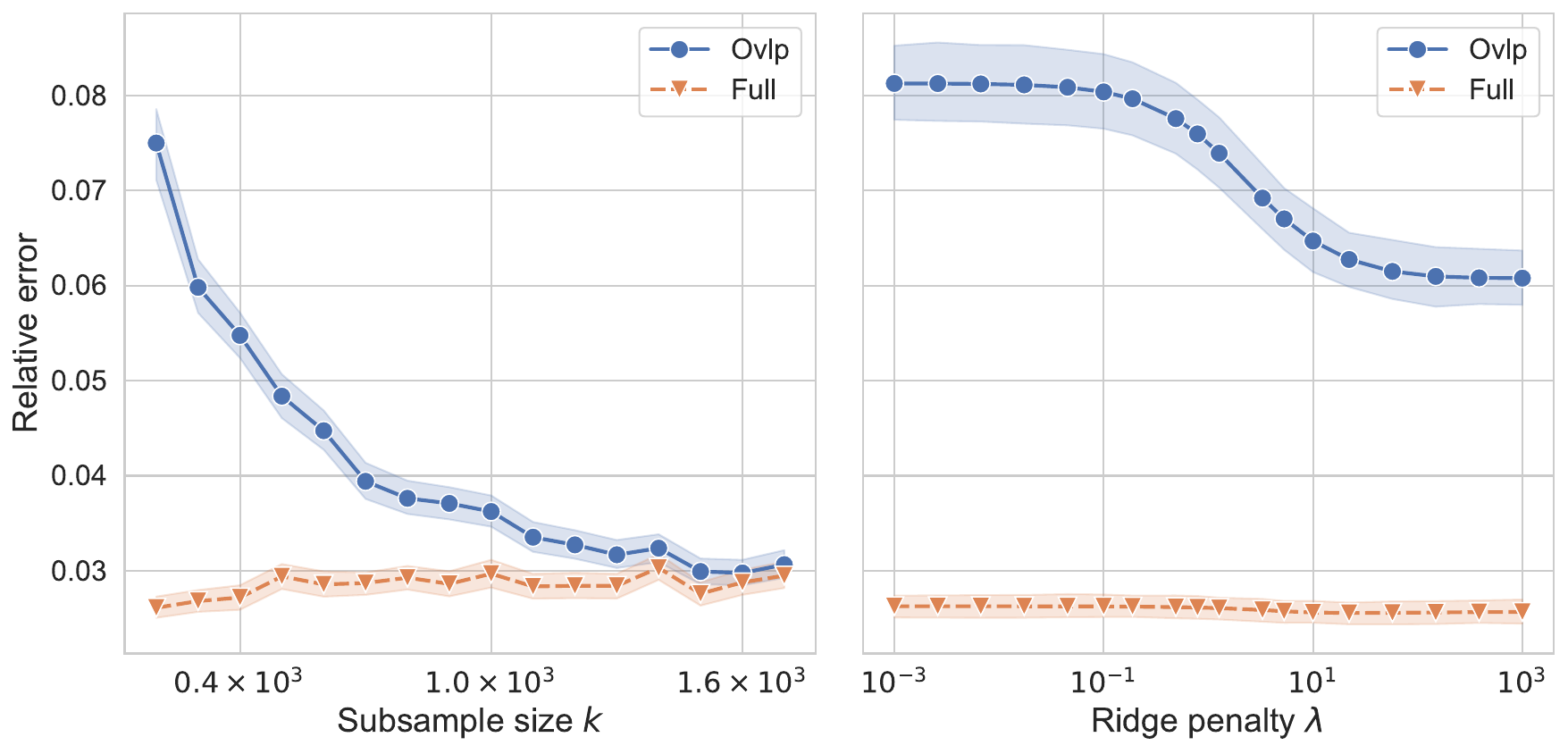}
    \caption{
    \textbf{Full-estimator $\tR_{M}^{\full}$  performs better than \ovlp-estimator $\tR_{M}^{\sub}$ across different regularization and for small subsample size $k$.}
    The simulation setting is the same as \Cref{fig:k}, except that here $\bX$ is uniformly distributed.
    }\label{fig:k_uniform}
\end{figure}

\begin{figure}[H]
    \centering
    \includegraphics[width=0.95\textwidth]%
    {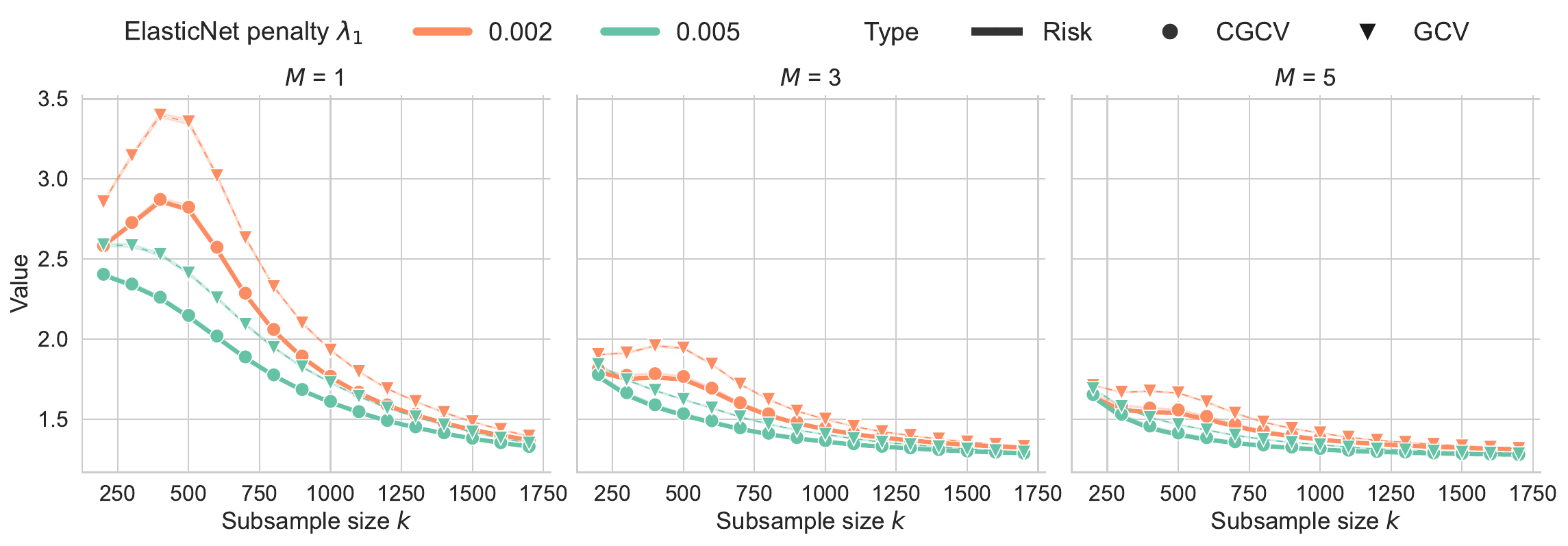}
    \caption{
    \textbf{CGCV is consistent for finite ensembles across different subsample sizes.}
    The plot of risks versus the subsample size $k$ for Elastic Net ensemble with different $\lambda$ and $M$ and fixed $\lambda_2 = 0.01$. 
    The simulation setting is the same as \Cref{fig:ridge-k}, except that here $\bX$ is uniformly distributed.
    }\label{fig:ridge-k_uniform}
\end{figure}

\begin{figure}[H]
    \centering
    \includegraphics[width=0.95\textwidth]%
    {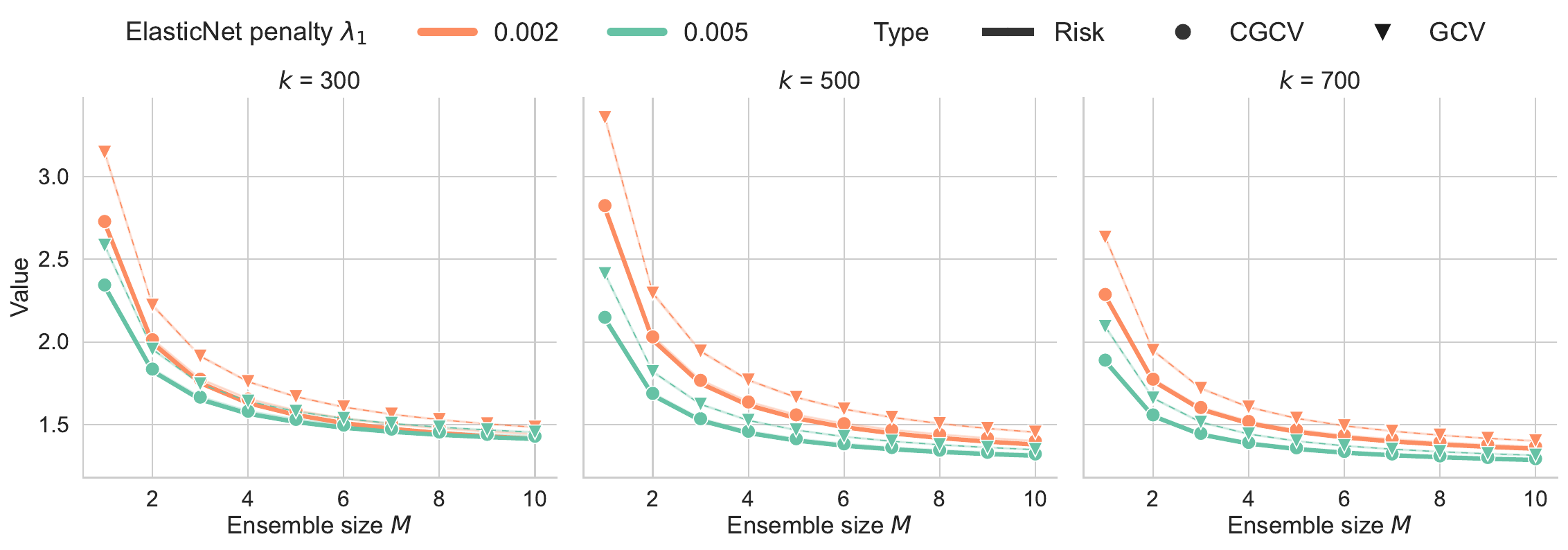}
    \caption{
    \textbf{GCV gets closer to the risk as ensemble size increases across different subsample sizes.}
    The plot of risks versus the Elastic Net ensemble size $M$ for ridge ensemble with different $\lambda$ and $k$ and fixed $\lambda_2 = 0.01$.
    The simulation setting is the same as \Cref{fig:M}, except that here $\bX$ is uniformly distributed.
    }
    \label{fig:M_uniform}
\end{figure}

\clearpage

\section{Additional numerical illustrations for \Cref{sec:asm-consistency} with non-Gaussian data}
\label{app:experiments-nongaussian}

\subsection[Comparison of CGCV and GCV in lambda for elastic net and lasso]{Comparison of CGCV and GCV in $\lambda$ for elastic net and lasso}
\label{app:experiments-nongaussian-cgcv-vs-gcv-in-lambda}

\noindent\underline{Elastic net:}

\begin{figure}[!ht]
    \centering
    \includegraphics[width=0.9\textwidth]{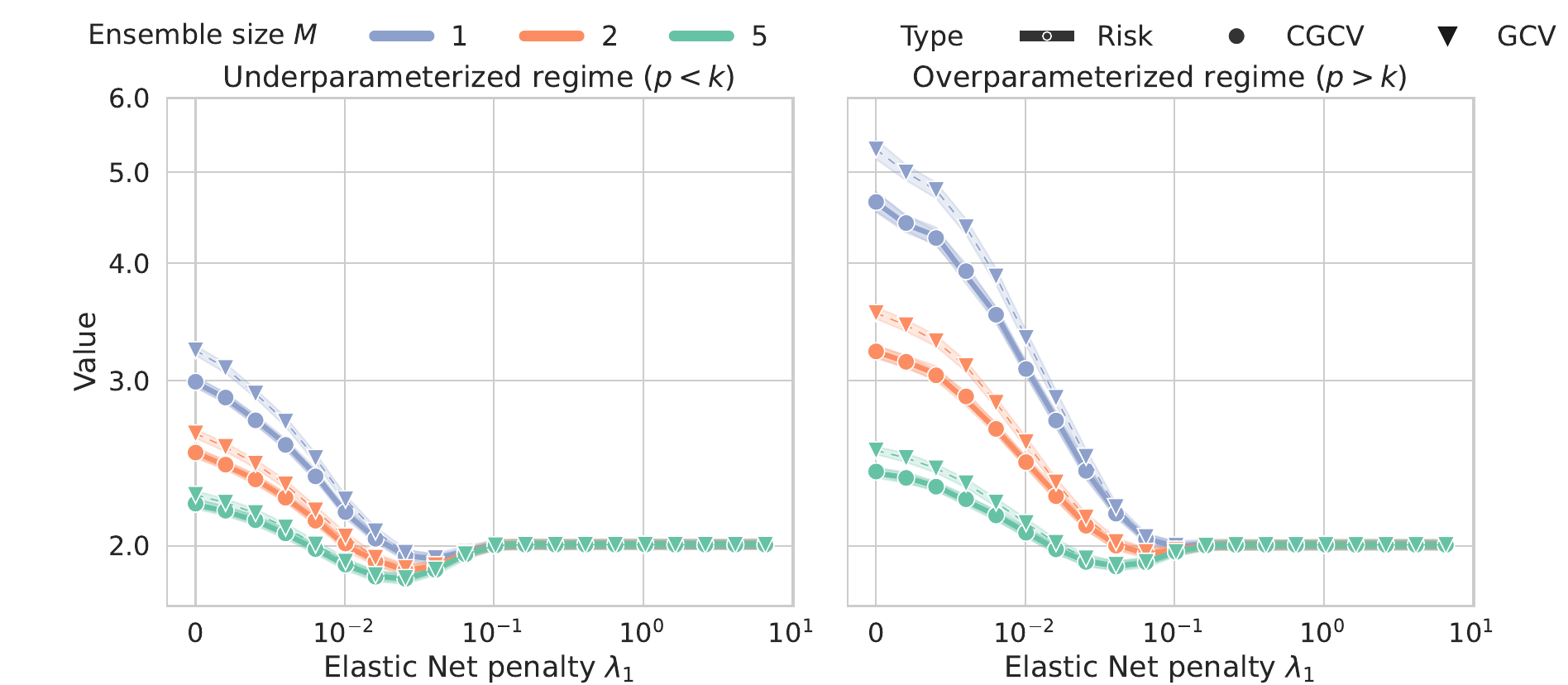}
    \caption{
    \textbf{CGCV is consistent for finite ensembles across different regularization levels.}
    The GCV estimates for the Elastic Net ensemble with varying lasso penalty $\lambda_1$ and fixed $\lambda_2=0.01$ in a problem setup of \Cref{subsec:asym-numerical} over 50 repetitions of the datasets.
    Here the feature dimension is $p=1200$.
    The left and the right panels show the cases when the subsample sizes are $k=2400$ and $k=800$, respectively.}\label{fig:elasticnet-lambda}
\end{figure}

\noindent\underline{Lasso:}

\begin{figure}[!ht]
    \centering
    \includegraphics[width=0.9\textwidth]{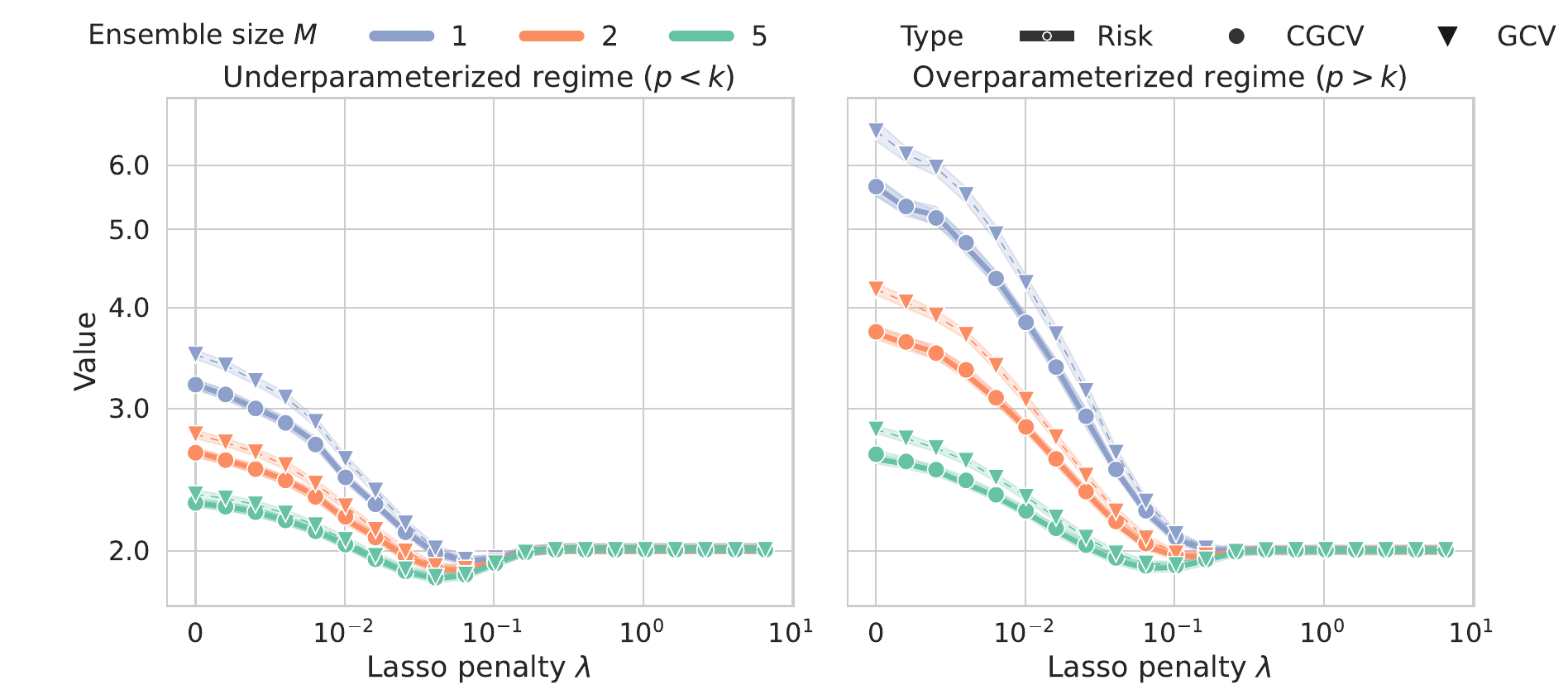}
    \caption{
    \textbf{CGCV is consistent for finite ensembles across different regularization levels.}
    The GCV estimates for the lasso ensemble with varying lasso penalty $\lambda$ in a problem setup of \Cref{subsec:asym-numerical} over 50 repetitions of the datasets.
    Here the feature dimension is $p=1200$.
    The left and the right panels show the cases when the subsample sizes are $k=2400$ and $k=800$, respectively.
    }
    \label{fig:lasso-lambda}
\end{figure}


\begin{thebibliography}{}

\bibitem[Adlam and Pennington, 2020a]{adlam_pennington_2020neural}
Adlam, B. and Pennington, J. (2020a).
\newblock The neural tangent kernel in high dimensions: Triple descent and a multi-scale theory of generalization.
\newblock In {\em International Conference on Machine Learning}.

\bibitem[Adlam and Pennington, 2020b]{adlam2020understanding}
Adlam, B. and Pennington, J. (2020b).
\newblock Understanding double descent requires a fine-grained bias-variance decomposition.
\newblock In {\em Advances in Neural Information Processing Systems}.

\bibitem[Ando and Komaki, 2023]{ando2023high}
Ando, R. and Komaki, F. (2023).
\newblock On high-dimensional asymptotic properties of model averaging estimators.
\newblock {\em arXiv preprint arXiv:2308.09476}.

\bibitem[Auddy et~al., 2023]{auddy2023approximate}
Auddy, A., Zou, H., Rad, K.~R., and Maleki, A. (2023).
\newblock Approximate leave-one-out cross validation for regression with $\ell_1$ regularizers (extended version).
\newblock {\em arXiv preprint arXiv:2310.17629}.

\bibitem[Bai and Silverstein, 2010]{bai2010spectral}
Bai, Z. and Silverstein, J.~W. (2010).
\newblock {\em Spectral Analysis of Large Dimensional Random Matrices}.
\newblock Springer.
\newblock Second edition.

\bibitem[Bartlett et~al., 2021]{bartlett2021deep}
Bartlett, P.~L., Montanari, A., and Rakhlin, A. (2021).
\newblock Deep learning: a statistical viewpoint.
\newblock {\em Acta numerica}, 30:87--201.

\bibitem[Bayati et~al., 2013]{bayati2013estimating}
Bayati, M., Erdogdu, M.~A., and Montanari, A. (2013).
\newblock Estimating lasso risk and noise level.
\newblock {\em Advances in Neural Information Processing Systems}, 26.

\bibitem[Bayati and Montanari, 2011]{bayati2011lasso}
Bayati, M. and Montanari, A. (2011).
\newblock The lasso risk for gaussian matrices.
\newblock {\em IEEE Transactions on Information Theory}, 58(4):1997--2017.

\bibitem[Bellec, 2022]{bellec2022observable}
Bellec, P.~C. (2022).
\newblock Observable adjustments in single-index models for regularized {M}-estimators.
\newblock {\em arXiv preprint arXiv:2204.06990}.

\bibitem[Bellec, 2023]{bellec2020out}
Bellec, P.~C. (2023).
\newblock Out-of-sample error estimation for m-estimators with convex penalty.
\newblock {\em Information and Inference: A Journal of the IMA}, 12(4):2782--2817.

\bibitem[Bellec and Romon, 2021]{bellec2021chi}
Bellec, P.~C. and Romon, G. (2021).
\newblock Chi-square and normal inference in high-dimensional multi-task regression.
\newblock {\em arXiv preprint arXiv:2107.07828}.

\bibitem[Bellec and Shen, 2022]{bellec2022derivatives}
Bellec, P.~C. and Shen, Y. (2022).
\newblock Derivatives and residual distribution of regularized {M}-estimators with application to adaptive tuning.
\newblock In {\em Conference on Learning Theory}.

\bibitem[Bellec and Zhang, 2021]{bellec2021second}
Bellec, P.~C. and Zhang, C.-H. (2021).
\newblock Second-order stein: Sure for sure and other applications in high-dimensional inference.
\newblock {\em The Annals of Statistics}, 49(4):1864--1903.

\bibitem[Bloemendal et~al., 2016]{bloemendal2016principal}
Bloemendal, A., Knowles, A., Yau, H.-T., and Yin, J. (2016).
\newblock On the principal components of sample covariance matrices.
\newblock {\em Probability theory and Related Fields}, 164(1):459--552.

\bibitem[Breiman, 1996]{breiman_1996}
Breiman, L. (1996).
\newblock Bagging predictors.
\newblock {\em Machine Learning}, 24(2):123--140.

\bibitem[Breiman, 2001]{breiman2001random}
Breiman, L. (2001).
\newblock Random forests.
\newblock {\em Machine learning}, 45(1):5--32.

\bibitem[B{\"u}hlmann and Yu, 2002]{buhlmann2002analyzing}
B{\"u}hlmann, P. and Yu, B. (2002).
\newblock Analyzing bagging.
\newblock {\em The Annals of Statistics}, 30(4):927--961.

\bibitem[Celentano and Montanari, 2021]{celentano2021cad}
Celentano, M. and Montanari, A. (2021).
\newblock Cad: Debiasing the lasso with inaccurate covariate model.
\newblock {\em arXiv preprint arXiv:2107.14172}.

\bibitem[Celentano et~al., 2023]{celentano2020lasso}
Celentano, M., Montanari, A., and Wei, Y. (2023).
\newblock The lasso with general {G}aussian designs with applications to hypothesis testing.
\newblock {\em The Annals of Statistics}, 51(5):2194--2220.

\bibitem[Chaudhuri, 2014]{chaudhuri2014modern}
Chaudhuri, A. (2014).
\newblock {\em Modern Survey Sampling}.
\newblock CRC Press.

\bibitem[Craven and Wahba, 1979]{craven_wahba_1979}
Craven, P. and Wahba, G. (1979).
\newblock Estimating the correct degree of smoothing by the method of generalized cross-validation.
\newblock {\em Numerische Mathematik}, 31:377--403.

\bibitem[Davidson, 1994]{davidson1994stochastic}
Davidson, J. (1994).
\newblock {\em Stochastic Limit Theory: An Introduction for Econometricians}.
\newblock OUP Oxford.

\bibitem[Davidson and Szarek, 2001]{davidson2001local}
Davidson, K.~R. and Szarek, S.~J. (2001).
\newblock Local operator theory, random matrices and banach spaces.
\newblock {\em Handbook of the geometry of Banach spaces}, 1(317-366):131.

\bibitem[Dietterich, 1998]{dietterich1998experimental}
Dietterich, T.~G. (1998).
\newblock An experimental comparison of three methods for constructing ensembles of decision trees: Bagging, boosting and randomization.
\newblock {\em Machine learning}, 32:1--22.

\bibitem[Dietterich, 2000]{dietterich2000ensemble}
Dietterich, T.~G. (2000).
\newblock Ensemble methods in machine learning.
\newblock In {\em International workshop on multiple classifier systems}.

\bibitem[Dobriban and Wager, 2018]{dobriban_wager_2018}
Dobriban, E. and Wager, S. (2018).
\newblock High-dimensional asymptotics of prediction: {R}idge regression and classification.
\newblock {\em The Annals of Statistics}, 46(1):247--279.

\bibitem[Dossal et~al., 2013]{dossal2013degrees}
Dossal, C., Kachour, M., Fadili, M., Peyr{\'e}, G., and Chesneau, C. (2013).
\newblock The degrees of freedom of the lasso for general design matrix.
\newblock {\em Statistica Sinica}, pages 809--828.

\bibitem[Du et~al., 2023]{du2023subsample}
Du, J.-H., Patil, P., and Kuchibhotla, A.~K. (2023).
\newblock Subsample ridge ensembles: {E}quivalences and generalized cross-validation.
\newblock In {\em International Conference on Machine Learning}.

\bibitem[Du et~al., 2024]{du2023extrapolated}
Du, J.-H., Patil, P., Roeder, K., and Kuchibhotla, A.~K. (2024).
\newblock Extrapolated cross-validation for randomized ensembles.
\newblock {\em Journal of Computational and Graphical Statistics}, pages 1--12.

\bibitem[El~Karoui, 2010]{el2010spectrum}
El~Karoui, N. (2010).
\newblock The spectrum of kernel random matrices.
\newblock {\em The Annals of Statistics}, 38(1):1--50.

\bibitem[El~Karoui, 2013]{karoui_2013}
El~Karoui, N. (2013).
\newblock Asymptotic behavior of unregularized and ridge-regularized high-dimensional robust regression estimators: rigorous results.
\newblock {\em arXiv preprint arXiv:1311.2445}.

\bibitem[El~Karoui and Purdom, 2018]{karoui2016can}
El~Karoui, N. and Purdom, E. (2018).
\newblock Can we trust the bootstrap in high-dimensions? {T}he case of linear models.
\newblock {\em The Journal of Machine Learning Research}, 19(1):170--235.

\bibitem[Friedman and Hall, 2007]{friedman_hall_2007}
Friedman, J.~H. and Hall, P. (2007).
\newblock On bagging and nonlinear estimation.
\newblock {\em Journal of Statistical Planning and Inference}, 137(3):669--683.

\bibitem[Golub et~al., 1979]{golub_heath_wabha_1979}
Golub, G.~H., Heath, M., and Wahba, G. (1979).
\newblock Generalized cross-validation as a method for choosing a good ridge parameter.
\newblock {\em Technometrics}, 21(2):215--223.

\bibitem[Greene and Wellner, 2017]{greene2017exponential}
Greene, E. and Wellner, J.~A. (2017).
\newblock Exponential bounds for the hypergeometric distribution.
\newblock {\em Bernoulli}, 23(3):1911.

\bibitem[Han and Xu, 2023]{han2023distribution}
Han, Q. and Xu, X. (2023).
\newblock The distribution of ridgeless least squares interpolators.
\newblock {\em arXiv preprint arXiv:2307.02044}.

\bibitem[Hastie et~al., 2022]{hastie2022surprises}
Hastie, T., Montanari, A., Rosset, S., and Tibshirani, R.~J. (2022).
\newblock Surprises in high-dimensional ridgeless least squares interpolation.
\newblock {\em The Annals of Statistics}, 50(2):949--986.

\bibitem[Hastie et~al., 2009]{hastie2009elements}
Hastie, T., Tibshirani, R., and Friedman, J.~H. (2009).
\newblock {\em The Elements of Statistical Learning: Data Mining, Inference, and Prediction}.
\newblock Springer.
\newblock Second edition.

\bibitem[Karoui and K{\"o}sters, 2011]{karoui2011geometric}
Karoui, N.~E. and K{\"o}sters, H. (2011).
\newblock Geometric sensitivity of random matrix results: {C}onsequences for shrinkage estimators of covariance and related statistical methods.
\newblock {\em arXiv preprint arXiv:1105.1404}.

\bibitem[Knowles and Yin, 2017]{knowles2017anisotropic}
Knowles, A. and Yin, J. (2017).
\newblock Anisotropic local laws for random matrices.
\newblock {\em Probability Theory and Related Fields}, 169:257--352.

\bibitem[Leeb, 2008]{leeb2008evaluation}
Leeb, H. (2008).
\newblock Evaluation and selection of models for out-of-sample prediction when the sample size is small relative to the complexity of the data-generating process.
\newblock {\em Bernoulli}, pages 661--690.

\bibitem[LeJeune et~al., 2020]{lejeune2020implicit}
LeJeune, D., Javadi, H., and Baraniuk, R. (2020).
\newblock The implicit regularization of ordinary least squares ensembles.
\newblock In {\em International Conference on Artificial Intelligence and Statistics}.

\bibitem[Li, 1985]{li_1985}
Li, K.-C. (1985).
\newblock From {Stein's} unbiased risk estimates to the method of generalized cross validation.
\newblock {\em The Annals of Statistics}, pages 1352--1377.

\bibitem[Li, 1986]{li_1986}
Li, K.-C. (1986).
\newblock Asymptotic optimality of {$C_\ell$} and generalized cross-validation in ridge regression with application to spline smoothing.
\newblock {\em The Annals of Statistics}, 14(3):1101--1112.

\bibitem[Li, 1987]{li_1987}
Li, K.-C. (1987).
\newblock Asymptotic optimality for {$C_p$}, {$C_\ell$}, cross-validation and generalized cross-validation: Discrete index set.
\newblock {\em The Annals of Statistics}, 15(3):958--975.

\bibitem[Lopes, 2019]{lopes2019estimating}
Lopes, M.~E. (2019).
\newblock Estimating the algorithmic variance of randomized ensembles via the bootstrap.
\newblock {\em The Annals of Statistics}, 47(2):1088--1112.

\bibitem[Lopes et~al., 2020]{lopes2020measuring}
Lopes, M.~E., Wu, S., and Lee, T.~C. (2020).
\newblock Measuring the algorithmic convergence of randomized ensembles: The regression setting.
\newblock {\em SIAM Journal on Mathematics of Data Science}, 2(4):921--943.

\bibitem[Loureiro et~al., 2021]{loureiro2021learning}
Loureiro, B., Gerbelot, C., Cui, H., Goldt, S., Krzakala, F., Mezard, M., and Zdeborov{\'a}, L. (2021).
\newblock Learning curves of generic features maps for realistic datasets with a teacher-student model.
\newblock {\em Advances in Neural Information Processing Systems}, 34:18137--18151.

\bibitem[Loureiro et~al., 2022]{loureiro2022fluctuations}
Loureiro, B., Gerbelot, C., Refinetti, M., Sicuro, G., and Krzakala, F. (2022).
\newblock Fluctuations, bias, variance and ensemble of learners: {E}xact asymptotics for convex losses in high-dimension.
\newblock In {\em International Conference on Machine Learning}.

\bibitem[Miolane and Montanari, 2021]{miolane2021distribution}
Miolane, L. and Montanari, A. (2021).
\newblock The distribution of the lasso: Uniform control over sparse balls and adaptive parameter tuning.
\newblock {\em The Annals of Statistics}, 49(4):2313--2335.

\bibitem[Misiakiewicz and Saeed, 2024]{misiakiewicz2024non}
Misiakiewicz, T. and Saeed, B. (2024).
\newblock A non-asymptotic theory of kernel ridge regression: deterministic equivalents, test error, and gcv estimator.
\newblock {\em arXiv preprint arXiv:2403.08938}.

\bibitem[M{\"u}cke et~al., 2022]{mucke_reiss_rungenhagen_klein_2022}
M{\"u}cke, N., Reiss, E., Rungenhagen, J., and Klein, M. (2022).
\newblock Data-splitting improves statistical performance in overparameterized regimes.
\newblock In {\em International Conference on Artificial Intelligence and Statistics}.

\bibitem[Patil and Du, 2023]{patil2023generalized}
Patil, P. and Du, J.-H. (2023).
\newblock Generalized equivalences between subsampling and ridge regularization.
\newblock {\em Thirty-seventh Conference on Neural Information Processing Systems}.

\bibitem[Patil et~al., 2023]{patil2022bagging}
Patil, P., Du, J.-H., and Kuchibhotla, A.~K. (2023).
\newblock Bagging in overparameterized learning: Risk characterization and risk monotonization.
\newblock {\em Journal of Machine Learning Research}, 24(319):1--113.

\bibitem[Patil et~al., 2022]{patil2022estimating}
Patil, P., Rinaldo, A., and Tibshirani, R. (2022).
\newblock Estimating functionals of the out-of-sample error distribution in high-dimensional ridge regression.
\newblock In {\em International Conference on Artificial Intelligence and Statistics}.

\bibitem[Patil et~al., 2021]{patil2021uniform}
Patil, P., Wei, Y., Rinaldo, A., and Tibshirani, R. (2021).
\newblock Uniform consistency of cross-validation estimators for high-dimensional ridge regression.
\newblock In {\em International Conference on Artificial Intelligence and Statistics}.

\bibitem[Rad and Maleki, 2020]{rad2018scalable}
Rad, K.~R. and Maleki, A. (2020).
\newblock A scalable estimate of the out-of-sample prediction error via approximate leave-one-out cross-validation.
\newblock {\em Journal of the Royal Statistical Society: Series B (Statistical Methodology)}, 82(4):965--996.

\bibitem[Stein, 1981]{stein1981estimation}
Stein, C.~M. (1981).
\newblock Estimation of the mean of a multivariate normal distribution.
\newblock {\em The Annals of Statistics}, 9(6):1135--1151.

\bibitem[Tan and Bellec, 2023]{tan2023multinomial}
Tan, K. and Bellec, P.~C. (2023).
\newblock Multinomial logistic regression: Asymptotic normality on null covariates in high-dimensions.
\newblock {\em arXiv preprint arXiv:2305.17825}.

\bibitem[Tan et~al., 2022]{tan2022noise}
Tan, K., Romon, G., and Bellec, P.~C. (2022).
\newblock Noise covariance estimation in multi-task high-dimensional linear models.
\newblock {\em arXiv preprint arXiv:2206.07256}.

\bibitem[Thrampoulidis et~al., 2018]{thrampoulidis2018precise}
Thrampoulidis, C., Abbasi, E., and Hassibi, B. (2018).
\newblock Precise error analysis of regularized {M}-estimators in high dimensions.
\newblock {\em IEEE Transactions on Information Theory}, 64(8):5592--5628.

\bibitem[Tibshirani and Taylor, 2012]{tibshirani2012degrees}
Tibshirani, R.~J. and Taylor, J. (2012).
\newblock Degrees of freedom in lasso problems.
\newblock {\em The Annals of Statistics}, 40(2):1198--1232.

\bibitem[Vaiter et~al., 2017]{vaiter2017degrees}
Vaiter, S., Deledalle, C., Fadili, J., Peyr{\'e}, G., and Dossal, C. (2017).
\newblock The degrees of freedom of partly smooth regularizers.
\newblock {\em Annals of the Institute of Statistical Mathematics}, 69:791--832.

\bibitem[Vaiter et~al., 2012]{vaiter2012degrees}
Vaiter, S., Deledalle, C., Peyr{\'e}, G., Fadili, J., and Dossal, C. (2012).
\newblock The degrees of freedom of the group lasso for a general design.
\newblock {\em arXiv preprint arXiv:1212.6478}.

\bibitem[Wang et~al., 2020]{wang2020bridge}
Wang, S., Weng, H., and Maleki, A. (2020).
\newblock {Which bridge estimator is the best for variable selection?}
\newblock {\em The Annals of Statistics}, 48(5):2791 -- 2823.

\bibitem[Wang et~al., 2018]{wang2018approximate}
Wang, S., Zhou, W., Maleki, A., Lu, H., and Mirrokni, V. (2018).
\newblock Approximate leave-one-out for high-dimensional non-differentiable learning problems.
\newblock {\em arXiv preprint arXiv:1810.02716}.

\bibitem[Wasserman, 2006]{wasserman2006all}
Wasserman, L. (2006).
\newblock {\em All of Nonparametric Statistics}.
\newblock Springer.

\bibitem[Waterhouse, 1983]{waterhouse1983symmetric}
Waterhouse, W.~C. (1983).
\newblock Do symmetric problems have symmetric solutions?
\newblock {\em The American Mathematical Monthly}, 90(6):378--387.

\bibitem[Wei et~al., 2022]{wei_hu_steinhardt}
Wei, A., Hu, W., and Steinhardt, J. (2022).
\newblock More than a toy: Random matrix models predict how real-world neural representations generalize.
\newblock In {\em International Conference on Machine Learning}.

\bibitem[Wolpert, 1992]{wolpert1992stacked}
Wolpert, D.~H. (1992).
\newblock Stacked generalization.
\newblock {\em Neural networks}, 5(2):241--259.

\bibitem[Xu et~al., 2019]{xu2019consistent}
Xu, J., Maleki, A., Rad, K.~R., and Hsu, D. (2019).
\newblock Consistent risk estimation in high-dimensional linear regression.
\newblock {\em arXiv preprint arXiv:1902.01753}.

\bibitem[Zou and Hastie, 2007]{zou2007degrees}
Zou, H. and Hastie, T. (2007).
\newblock On the degrees of freedom of the lasso.
\newblock {\em The Annals of Statistics}, 35(5):2173--2192.

\end{thebibliography}
\end{document}